\documentclass[a4paper,10pt]{article}
\usepackage{amssymb,amsmath,amsthm}
\usepackage{amsfonts}
\usepackage[english]{babel}
\usepackage{eufrak, color, mathrsfs}
\usepackage{hyperref}
\usepackage{verbatim} 
\usepackage{bbm} 
\usepackage{color}
\numberwithin{equation}{section}
\usepackage{enumitem}
\usepackage{relsize}
\usepackage{xcolor}

\newcommand{\DC}{{\mathtt D \mathtt C}}
\newcommand{\sM}{{\frak s}_M}

\newcommand{\la}{\langle}
\newcommand{\ra}{\rangle}

\newcommand{\Sbot}{{\mathbb S}^\bot}

\newcommand{\Lip}{\mathrm{Lip}}

\renewcommand{\t}{\tau}
\newcommand{\io}{\iota}
\newcommand{\Splus}{{\mathbb S}_+} 
\renewcommand{\S}{{\mathbb S}} 
\newcommand{\tLm}{{\mathtt \Lambda} }

\newtheorem{theorem}{Theorem}[section]
\newtheorem{proposition}[theorem]{Proposition}
\newtheorem{lemma}[theorem]{Lemma}
\newtheorem{corollary}[theorem]{Corollary}

\newtheorem{remark}[theorem]{Remark}
\newtheorem{remarks}[theorem]{Remark}
\newtheorem{definition}[theorem]{Definition}
\newcommand{\be}{\begin{equation}}
\newcommand{\ee}{\end{equation}}
\newcommand{\teta}{\theta}
\newcommand{\om}{\omega}

\newcommand{\e}{\varepsilon}

\newcommand{\ov}{\overline}

\newcommand{\R}{\mathbb R}
\newcommand{\C}{\mathbb C}

\newcommand{\ac}{\nu}
\newcommand{\Z}{\mathbb Z}

\newcommand{\N}{\mathbb N}
\newcommand{\T}{\mathbb T}

\renewcommand{\a }{\alpha }
\renewcommand{\b }{\beta }
\newcommand{\s }{\sigma }
\newcommand{\ii }{{\rm i} }
\renewcommand{\d }{\delta }

\newcommand{\g }{\gamma}

\newcommand{\vphi}{\varphi }

\renewcommand{\o }{\omega }

\newcommand{\Dom}{\om \cdot \pa_\vphi}

\newcommand{\mG}{\mathcal G}
\newcommand{\mU}{U}
\newcommand{\tOm}{\mathtt \Omega}

\newcommand{\ph}{\varphi}

\newcommand{\mL}{\mathtt{L}}

\newcommand{\mF}{\mathcal{F}}

\newcommand{\mE}{\mathcal{E}}

\newcommand{\mP}{\mathcal{P}}
\newcommand{\pa}{\partial}
\newcommand{\ompaph}{\om \cdot \partial_\ph}

\def\ba{\begin{aligned}}
\def\ea{\end{aligned}}
\def\beginm{\begin{multline}}
\def\endm{\end{multline}}

\newcommand{\mB}{\mathcal{B}}

\newcommand{\Lipg}{{\rm{Lip}(\g)}}

\begin{document}

\title{{\bf Large KAM tori for \\
quasi-linear  perturbations of KdV}}

\date{}

\author{Massimiliano Berti, Thomas Kappeler, Riccardo Montalto}

\maketitle

\noindent
{\bf Abstract.}
In this paper we prove the persistence of 
space periodic multi-solitons  of arbitrary size
under any quasi-linear Hamiltonian perturbation, 
which is smooth and sufficiently small.
This answers positively a longstanding question 
whether KAM techniques can be further developed  
to prove the existence of quasi-periodic solutions of arbitrary size of 
strongly nonlinear perturbations of integrable PDEs.
\\[1mm]


\tableofcontents

\section{Introduction}\label{introduction}

The Korteweg-de Vries (KdV) equation
\begin{equation}\label{kdv}
\partial_t u = - \partial_x^3u + 6 u \partial_x u
\end{equation}
is one of the most important model equations for dispersive phenomena 
with numerous applications in
physics. The seminal discovery 
in the late sixties
that \eqref{kdv} admits infinitely many conservation laws
(\cite{Lax}, \cite{MGK}),
and the development of the 
inverse scattering transform method 
(\cite{GGKM}), 
led to the modern theory of infinite dimensional integrable systems 
(e.g. \cite{DKN}, \cite{FT} and references therein). 

One of the most distinguished features of \eqref{kdv} is the existence of sharply localized travelling waves of arbitrarily large amplitudes and particle like properties.
 Kruskal and Zabusky,
who discovered them in numerical experiments in the early sixties, both on the real line and in the periodic setup (cf. \cite{KZ}),
coined the name solitons for them. More generally, 
they found solutions, which are localized near finitely many points in space. In the periodic setup, these solutions are referred 
to as  {	\it periodic multi-solitons} or {\it finite gap} solutions. Due to their 
importance 
in applications, various stability aspects, in particular 
long time asymptotics, 
have been extensively studied.  
A major question  concerns the persistence of 
the multi-solitons under perturbations. 
In the last thirty years, KAM methods
pioneered by Kolmogorov, Arnold, and Moser to treat perturbations of integrable systems
of finite dimension, were developed for PDEs. 
Most of the work focused on 
small amplitude solutions or semilinear perturbations. 
It has been a longstanding question from experts in PDEs
and  in infinite dimensional dynamical systems
whether KAM results
hold also for solutions of arbitrary size under quasi-linear perturbations, 
called strongly nonlinear  in \cite{K},  of integrable PDEs. 

The aim of this paper is to prove  
the first persistence result of periodic
multi-solitons of KdV with {\it arbitrary} size  under {\it strongly nonlinear} 
perturbations -- see Theorem \ref{KAM finite gap} below.
Note that in this case, 
it was not even known if there exist solutions of the
perturbed equation which are global in time. 

To describe the class of perturbations of the KdV equation considered, 
we recall that 
\eqref{kdv}, with space periodic variable $ x $ in $ \T_1 := \T / \Z$,   
can be written in Hamiltonian form,
\begin{equation}\label{hamiltoniana kdv0}
\partial_t u = \partial_x \nabla {H}^{kdv}(u)\, , \qquad 
{H}^{kdv}(u) := 
\int_{\T_1}  \frac12 (\pa_x u)^2(x) + u^3 (x) \, d x 
\, , 
\end{equation}
where $\nabla H^{kdv}$ denotes the $L^2-$gradient of $H^{kdv}$
and $\partial_x$ is the Poisson structure, corresponding
to the Poisson bracket, defined for functionals $F, G$ 
by  
$$
\{ F , G \} := \int_{\T_1} \nabla F \partial_x \nabla G \, d x \, .
$$
We consider 
{\em quasi-linear}  perturbations of \eqref{kdv} 
of the form
\begin{equation}\label{main equation}
\partial_t u  = - \partial_x^3 u + 6 u \partial_x u + 
\e  a(x, u(x), \pa_x u(x))  \partial_x^3u + \cdots   \, ,
\end{equation}
where $\e \in (0, 1)$ is a small parameter and 
$\, \cdots \,$ stands for terms containing 
$x-$derivatives of  $ u $ up to order two. We assume that the perturbation
is  {\it Hamiltonian}, namely  
$a \partial_x^3u + \cdots  = \partial_x \nabla P $,
where  $\nabla P $ is the $L^2$-gradient  of 
a functional of the form
\begin{equation}\label{hamiltoniana perturbazione}
P (u) := \int_{\T_1} f(x, u(x), u_x(x))\, d x \, ,  \qquad u_x := \partial_x u \, . 
\end{equation}
Note that the nonlinear term 
\begin{align}\label{QLPe} 
\partial_x \nabla P ( u )  =   
 (\pa^2_{u_x} f ) (x, u(x), u_x(x) ) \pa_x^3 u  + \cdots
\end{align}
has  the same order of the linear vector field $ \partial_x^3 u $ in  \eqref{kdv}.
 When written as 
a Hamiltonian PDE, \eqref{main equation} takes the form
\begin{equation}\label{hamiltonian PDE}
\partial_t u = \partial_x \nabla {H}_\e(u)
\end{equation}
with  Hamiltonian 
\begin{equation}\label{hamiltoniana kdv}
H_\e (u) := H^{kdv}(u) + \e P(u)\,. 
\end{equation}
To state our main result, we first need to introduce some more notation. 
Note  that the mean $u \mapsto \int_{\T_1} u(x)\, d x$ is a prime integral 
for \eqref{hamiltonian PDE}. 
We restrict  our attention to  functions with zero average 
(cf. Remark (R2) below) and  
choose as phase spaces for  \eqref{hamiltonian PDE} 
the scale of Sobolev spaces  $H_0^s (\T_1) $, $s \ge 0$, 
$$
H_0^s (\T_1) := \Big\{ u \in H^s (\T_1) : 
\int_{\T_1} u(x)\, d x = 0 \Big\} \, , 
\qquad  L_0^2 (\T_1) \equiv H_0^0 (\T_1)\, , 
$$
where 
$$
H^s(\T_1) := \Big\{ u(x) = 
\sum_{n \in \Z} u_n e^{\ii 2 \pi n x} \, : \, 
 \| u \|_{H^s_x} 
:= \big( \sum_{n \in \Z} \langle n \rangle^{2 s} |u_n|^2 \big)^{\frac12} < \infty \, , \
u_{-n} = \overline{u_{n}}\,\,\,  \forall n \in \Z \Big\} 
$$
and $\langle n \rangle := {\rm max}\{ 1, |n| \}$ for any $n \in \Z$. We also write $L^2 (\T_1)$ for $H^0 (\T_1)$.  The symplectic form on $ L^2_0 (\T_1)$ 
is given by
\be\label{KdV symplectic}
{\mathcal W}_{L^2_0} (u, v) := 
\int_{\T_1} (\partial_x^{- 1} u ) v\, d x \, , \qquad 
\partial_x^{- 1} u = 
\sum_{n \ne 0} \frac{1}{\ii n} u_n e^{\ii 2 \pi n x}\, ,
\qquad
\forall u, v \in L^2_0(\T_1)\, .
\ee
Note that the Hamiltonian vector field $ X_H (u) = \pa_x \nabla H (u) $, associated with the Hamiltonian
$ H $, is determined by 
$ d H (u)[ \cdot ] = {\mathcal W}_{L^2_0} ( X_H , \cdot ) $. 
\\[1mm]
{\em $\Splus-$gap potentials.} 
According to \cite{KP}, the KdV equation \eqref{kdv} 
on the torus is an integrable PDE in the strongest possible 
sense, meaning that it admits globally defined canonical coordinates
on $H^0_0(\T_1)$,  
so that \eqref{kdv} can be solved by quadrature, see 
Theorem \ref{Birkhoff coordinates for KdV} for a
precise statement. These coordinates,  referred to as Birkhoff coordinates, 
are particularly suited to 
describe the finite gap solutions of KdV. 
Each of these solutions is contained in
a finite dimensional integrable subsystem $\mathcal M_{\Splus}$,
of dimension $2 |\Splus|$,  with $\Splus$ being a finite subset of 
$ \N_+ := \{1, 2, \ldots \} $.
Each $\mathcal M_{\Splus}$ 
can be described in terms of
action angle coordinates $\theta := (\theta_n)_{n \in \Splus}$, 
$I := (I_n)_{n \in \Splus}$ :   
there exists a real analytic
canonical diffeomorphism
\begin{equation}\label{finite gap in coordinate originarie}
\Psi_{\Splus} : \T^{\Splus} \times \R^{\Splus}_{> 0}
\to \mathcal M_{\Splus}\, , \quad 
(\theta, I) \,  \mapsto \, q(\theta, \cdot; I) \, , 
\end{equation}
so that the pull-back of the KdV Hamiltonian, 
$H^{kdv} \circ \Psi_{\Splus}$,
is a real analytic function of the actions $I$ alone. 
Elements in $\mathcal M_{\Splus}$ are referred to as
$\Splus-$gap potentials. 
The function 
$ q(\varphi, x) \equiv q(\varphi, x; I) $ is real analytic. 
In action angle coordinates, any solution of \eqref{kdv}
on $\mathcal M_{\Splus}$ is given by
$$
 \theta(t) = \theta^{(0)} - \omega^{kdv} (\ac) t \, , 
 \qquad  I(t) = \ac \, ,
$$
where $\theta^{(0)}$ denote the initial angles, 
$\ac \in \R^{\Splus}_{>0}$ the initial actions, and 
$\omega^{kdv} (\ac)$ the frequency vector
\begin{equation}\label{frequency finite gap}
\omega^{kdv} (\ac) :=
\partial_I (H^{kdv} \circ \Psi_{\Splus}) (\ac)  
\in \R^{\Splus} \, . 
\end{equation}  
The corresponding solution on $\mathcal M_{\Splus}$ is then given by
$$
q\big(  \theta^{(0)} - \omega^{kdv}(\ac)t,\, x; \, \ac \big) 
$$
and hence  is quasi-periodic in time.
The map $\R^{\mathbb S_+}_{> 0} \to \R^{\Splus}\, ,
  \ac \mapsto \omega^{kdv} (\ac) $,  is a local 
diffeomorphism (see Remark \ref{rem:diffeo}).
In the whole paper $\Xi \subset \R^{\Splus}_{> 0}$  is the closure of a bounded open nonempty set
such  that 
$ \om^{kdv} $ defined in \eqref{frequency finite gap}  is a diffeomorphism onto its image. Moreover we require that, for some $\delta > 0 $,  
\begin{equation}\label{azioni come parametri 1}
\Xi \, + \, B_{\mathbb S_+}(\delta) \, 
\subseteq \, \R^{\mathbb S_+}_{> 0} \, , 
\end{equation}
where $B_{\mathbb S_+}(\delta)$ denotes the ball of radius 
$\delta$ in $\R^{\mathbb S_+}$ centered at the origin. 
Furthermore we  introduce
the Sobolev spaces of periodic, real valued functions
\begin{equation} \label{unified norm}
H^s := 
\Big\{ f = \sum_{(\ell,j) \in \Z^{\Splus} \times \Z} f_{\ell, j} \, e^{\ii(\ell \cdot \ph + 2 \pi jx)} : \ 
\| f \|_s^2 := \sum_{(\ell,j) \in \Z^{\Splus} \times \Z} | f_{\ell, j}|^2 \langle \ell,j \rangle^{2s} < \infty,\,\,\,
\overline{f_{\ell, j}} = f_{-(\ell, j)} \Big\}
\end{equation}
where  $\langle \ell,j \rangle := \max \{ 1, |\ell|, |j| \} $ and we recall the Sobolev embedding  
$ H^s \subset C^0 ( \T^{\Splus} \times \T_1)$
for  $ s >  (|\Splus| + 1)/2 $. 

\smallskip

The main result of this paper, Theorem \ref{KAM finite gap} below, 
proves that for $\e$ small enough 
and for $\ac $ in a subset of $ \Xi $
of asymptotically full Lebesgue measure, 
there is a quasi-periodic solution of equation \eqref{hamiltonian PDE} 
close to the finite gap solution 
$q(\theta^{(0)} - \omega^{kdv} (\ac) t, x; \ac)$
of \eqref{kdv}. More precisely, the following holds:

\begin{theorem}\label{KAM finite gap}
Let $ f $ be a function in 
$ {\cal C}^{\infty}(\T_1 \times \R \times \R, \R)$
and $\Splus$ a finite subset of $\N_+$. 
Then there exist $ \bar s > (|{\Splus}| +1) / 2  $ 
 and $\e_0 \in (0, 1)$  so that 
for any $\e \in (0, \e_0)$,  
there exists a measurable subset $\Xi_\e \subseteq \Xi$ 
with asymptotically full measure, i.e. 
$$
\lim_{\e \to 0} |\Xi \setminus \Xi_\e| = 0 \, , 
$$ 
and,  for any $\ac \in \Xi_\e $,  
there exists a quasi-periodic solution $ u_\e ( \omega_\e( \ac) t, x; \nu )$
of the perturbed KdV equation \eqref{hamiltonian PDE} with
$ u_\e (\cdot, \cdot \, ; \nu ) $ in $ H^{\bar s}(\T^{\Splus} \times \T_1)$
and frequency vector $ \omega_\e( \ac) \in \R^{\Splus} $ satisfying  
$$
\lim_{\e \to 0} \| u_\e (\cdot, \cdot \, ; \nu )  - q (\cdot, \cdot \, ; \nu )  \|_{\bar s} = 0 \, ,  \quad
\lim_{\e \to 0} \omega_\e(\ac) = - \omega^{kdv} (\ac) \, , 
$$
where $q (\vphi, x; \ac ) $, defined 
in \eqref{finite gap in coordinate originarie}, is the 
$\Splus-$gap potential in $\mathcal M_{\Splus}$ with frequency vector
$\omega^{kdv} (\ac)$, defined in \eqref{frequency finite gap}. 
The quasi-periodic solution $ u_\e ( \omega_\e( \ac) t, x; \nu ) $ is linearly stable. 
\end{theorem}

We make the following remarks:
\begin{description}
\item (R1) 
The result of Theorem \ref{KAM finite gap} holds for any density 
$ f $ of class $ {\cal C}^{s_*} $ with $s_*  $ large enough 
and for any family of $\Splus-$gap solutions of KdV
with average $c$ 
(cf. \cite[page 112]{KP}).  
We assume in this paper that $ f $ is $ {\cal C}^{\infty} $
and $c=0 $ merely to simplify the exposition.
\item (R2)  The methods developed to prove 
Theorem \ref{KAM finite gap} are quite general. We expect that
analogous results  
can 
also be proved for  equations in the KdV hierarchy  
as well as for the defocusing NLS and equations in 
the NLS hierarchy such as the defocusing mKdV equation. 
\end{description}

\noindent 
Let us now comment on the novelty of our result. 
\begin{enumerate}
\item  The first KAM results for \eqref{kdv} were proved 
by Kuksin \cite{K2-KdV}  (cf. also \cite{K}) and 
Kappeler-P\"oschel  \cite{KP} 
for finite gap solutions  of {\it arbitrary} size, 
subject to {\em semilinear} perturbation. It means that  
the density $f$ of
 \eqref{hamiltoniana perturbazione} does {\em not} 
 depend on $ u_x $, and hence 
$$ 
\partial_x \nabla { P}(u) = 
\partial_u^2 f(x, u(x))u_x + \cdots 
$$ 
depends only on $u$ and $u_x$ (note that
in addition the dependence on $u_x$ is linear). 
Subsequently,   Liu-Yuan \cite{LY} proved KAM results for 
semilinear perturbations of small amplitude solutions
of the derivative NLS  and the Benjamin-Ono equations 
whereas Zhang-Gao-Yuan \cite{ZGY} proved analogous results for
the reversible derivative NLS. 
More recently, Berti-Biasco-Procesi \cite{BBiP1}-\cite{BBiP2} 
proved existence of small quasi-periodic solutions of derivative Klein-Gordon equations. For the NLS and the beam equations in higher space dimension, 
KAM results were obtained by Eliasson-Kuksin \cite{EK} and, repsectively, Eliasson-Gr\'ebert-Kuksin \cite{EGK}.  
In all these works, the perturbations are required to be semilinear. 

On the other hand, the  
results in Baldi-Berti-Montalto \cite{BBM-auto}, \cite{BBMmKdV},  for
{\it quasi-linear} perturbations of the KdV and mKdV equations
concern only {\em small amplitude} solutions. 
The proof of these results makes use of pseudo-differential calculus and relies  in a decisive manner on the differential nature of KdV. The latter property 
cannot be read off  in the action-angle coordinates  outside  a neighborhood of the
origin. 
We also mention that  the results in Giuliani \cite{Giu} for KdV, 
Feola-Procesi \cite{FP} for NLS, 
Berti-Montalto \cite{Berti-Montalto} and Baldi-Berti-Haus-Montalto \cite{BBHM} for water waves concern small amplitude solutions. 
 
Thus the challenging problem of the persistence 
of  the finite gap solutions of  \eqref{kdv} of {\it arbitrary} size 
under {	\it strongly nonlinear} perturbations
\eqref{QLPe}  remained completely open. 
\item 
In \cite{BKM}, we used the ``$ 1 $-smoothing property" of the
Birkhoff coordinates of the defocusing NLS equation on $ \T^1 $,
established in \cite{KST2}, to prove a KAM  result 
for {\em semilinear} perturbations. 
This property 
is used to deal with the difficulties related to the  double
``asymptotic multiplicity" of the frequencies.
For the KdV equation,
a ``$ 1 $-smoothing property" has been proved 
first near the equilibrium  in \cite{KPe} and 
then in general in \cite{KST}. However it 
is not sufficient for dealing with the quasi-linear perturbations \eqref{QLPe}. 
\item
The proof of Theorem \ref{KAM finite gap} uses the canonical coordinates constructed in \cite{Kappeler-Montalto-pseudo}
near any given compact family of $\Splus-$gap potentials
in $\mathcal M_{\Splus}$.
These coordinates admit
an expansion in terms of pseudo-differential operators
up to a remainder of arbitrary negative order. 
Due to its length, this part of the proof of 
Theorem \ref{KAM finite gap} has been published in a 
separate paper \cite{Kappeler-Montalto-pseudo}. 
The important fact that 
the linearized  Hamiltonian vector field of $  H_\e $, 
expressed in these coordinates, 
admits  an expansion in terms of pseudo-differential operators
is proved in Section \ref{espansione linearized}.
\end{enumerate}

\noindent
{\em Ideas of the proof.}
Theorem \ref{KAM finite gap} 
is proved by means of a Nash-Moser iterative scheme to construct 
quasi-periodic solutions near 
a given family of $\Splus-$gap solutions.
 One of the main issues 
 concerns the invertibility of the linearized Hamiltonian operator 
$$ 
\omega \cdot \partial_\vphi - \partial_x d \nabla H_\e (u (\vphi, x) ) 
$$
where $ u (\om t, x) $ is an approximate quasi-periodic 
solution  of \eqref{hamiltonian PDE}, close to the
finite gap solutions \eqref{finite gap in coordinate originarie}. 
In \cite{Kappeler-Montalto-pseudo} a coordinate chart
$$
\Psi : \,(\theta, y, w) \mapsto \Psi(\theta, y, w) \in 
 L^2_0(\T_1) 
$$
is constructed in a 
neighborhood of 
$\T_1^{\Splus} \times \{ \ac \} \times \{ 0 \} $ in  $\T_1^{\Splus} \times \R^{\Splus}_{> 0} \times L^2_\bot(\T_1) $, which  admits a pseudo-differential expansion, up to regularizing operators 
satisfying tame estimates. Here 
\be\label{Def:L2bot}
L^2_\bot(\T_1) := \Big\{ w = \sum_{n \in \Sbot} 
w_n e^{\ii 2\pi n x} \in L^2_0 (\T_1) 
\Big\}\, , \qquad \Sbot := 
\Z \setminus \big( \Splus \cup (-\Splus) \cup \{ 0 \} \big) \, .
\ee
Important properties of the map $\Psi$
are that the set of $\Splus-$gap solutions of \eqref{kdv} 
in the range of $\Psi$ is characterized by the equation $ w = 0$, 
and that
the linearized equation along the manifold 
$\{ w = 0, y=0 \}$  
is in diagonal form with coefficients only depending on $\ac$, see Theorem \ref{modified Birkhoff map}-{\bf (AE3)}.
This allows us to prove (cf. Section \ref{espansione linearized})
that when expressed in these coordinates, 
\begin{itemize}
\item  the linearized Hamiltonian vector field, acting in the subspace normal to the tangent space of $\mathcal M_{\Splus}$ 
at a given $\Splus-$gap potential, admits an expansion in terms of classical 
pseudo-differential operators, up to smoothing remainders which satisfy tame estimates
in $ H^s (\T_1) $ --
see Lemma \ref{differential nabla perturbation-true}  and \ref{differential nabla kdv remainder-true}.
\end{itemize}
We then evaluate the linearized Hamiltonian vector field
at an approximately invariant torus embedding
$ \vphi \mapsto 
 (\theta (\vphi), y(\vphi), w(\vphi)) $,  
obtaining in this way a 
quasi-periodic  operator, 
acting on  
the normal subspace $ L^2_\bot(\T_1) $, 
of the form  (cf. Lemma  \ref{Lemma di partenza riduzione})
\be\label{L0-intro}
{\cal L}^{(0)}_\omega = 
\omega \cdot \pa_\vphi - \Pi_{\bot} \Big( a_3^{(0)} \pa_x^3 + 
2 (a_3^{(0)})_x \pa_x^2 + a_1^{(0)} \pa_x  + 
 \sum_{k=0}^{M} a_{-k}^{(0)} \pa_x^{-k}    +
Q_{-1}^{kdv} (D ; \om) \Big)  + {\mathcal R_M^{(0)}}  
\ee
where $ a_{-k}^{(0)}  (\vphi, x) $, $ k = - 3, \ldots, M $ 
are real valued functions, $ a_3^{(0)} \sim - 1 $,  and 
${\mathcal R_M^{(0)}}  $ is a $ \vphi $-dependent 
operator which satisfies tame estimates in the 
Sobolev spaces $ H^s ( \T^{\Splus}_\varphi \times \T_1 )$. 
The order $ M $ of regularization will be fixed in Section \ref{sec: reducibility}.  
The term $Q_{-1}^{kdv} (D ; \om) $ is not small in $ \e $. 
It is the Fourier multiplier with symbol 
$ \omega^{kdv}_n - (2 \pi n)^3 $ which takes into account the difference 
between the KdV-frequencies and their approximation
by the frequencies of the Airy equation. 
We remark that the pseudo-differential 
operator $   \sum_{k=0}^{M} a_{-k}^{(0)} \pa_x^{-k}  $
is not present in \cite{BBM-auto}.
In order to show 
that the operator $ {\mathcal R_M^{(0)}}$ is tame
(see Lemma \ref{lem:tame2}) 
we prove in Section \ref{sec canonical coordinates}
novel results of independent interest concerning the extensions of the differential of the 
canonical coordinates of \cite{Kappeler-Montalto-pseudo}
to Sobolev spaces of negative order (cf. Corollaries \ref{corollary transpose negative sobolev} and \ref{corollary Birkhoff negative sobolev}).

The form \eqref{L0-intro} suggests to introduce
preliminary transformations 
which diagonalize $ {\cal L}_\omega^{(0) } $ 
up to a pseudo-differential operator of order zero plus 
a regularizing remainder (see Section \ref{linearizzato siti normali}). 
These transformations,  inspired by  \cite{BBM-auto}, 
are Fourier integral operators generated as 
symplectic flows 
of linear Hamiltonian transport PDEs and pseudo-differential maps. 
In order to conjugate the pseudo-differential terms $ a_{-k}^{(0)} \pa_x^{-k} $
we need a quantitative version  of the Egorov theorem that we prove in Section 
\ref{sezione astratta egorov}. 
We remark that in contrast to \cite{BBM-auto} 
we implement in Section \ref{sec:RT}
the time-quasi-periodic reparametrization {\em before}
the conjugation with  the  transport flow
to avoid a technical difficulty 
in the conjugation of the remainders obtained  in the Egorov theorem. 
Furthermore, we mention that related transformations 
have been developed in \cite{BGMR1}
for proving upper bounds for the growth of the Sobolev norms for various classes of PDEs.   

At this point, using properties of the KdV frequencies that we collect in Section \ref{sec:kdvfre},  we are able to perform a KAM reducibility scheme 
to complete the diagonalization of $ {\cal L}^{(0)}_\omega $
 for most values of $ \nu $.
In view of the remainder $ {\mathcal R_M^{(0)}}$  
in \eqref{L0-intro} (and others generated by the Egorov theorem) 
we implement in Section \ref{sec: reducibility} 
an iterative scheme along the lines in  Berti-Montalto \cite{Berti-Montalto}. 
The proofs are by and large self-contained.

\medskip

\noindent
{\em Notation.} We denote by $ \N := \{0, 1, 2, \ldots \} $ the natural numbers and 
set $\N_+ := \{ 1, 2, \ldots \}$. 
Given a Banach space $ X $ with norm $\|\cdot \|_X$, 
we denote by 
by $ H^s_\vphi X = H^s (\T^{\Splus}, X)  $, $ s \in \N $,  
the Sobolev space of functions $ f : \T^{\Splus} \to X $ equipped with the norm 
$$
\| f \|_{H^s_\vphi X} := \| f \|_{L^2_\vphi X} + \max_{|\beta|=s} \| \pa_\vphi^\beta f \|_{L^2_\vphi X}  \, .
$$
We also denote $ H^0_\vphi X = L^2_\vphi X$. We recall that the continuous 
Sobolev embedding theorem is stronger in the case $X$ is a Hilbert space $H$,  namely 
\be\label{SoboX}
H^s (\T^{\Splus}, X) \hookrightarrow  {\cal C}^0 (\T^{\Splus}, X)  \, , \quad \forall 
s >  |\Splus|  \, , \qquad
H^s (\T^{\Splus}, H) \hookrightarrow  {\cal C}^0 (\T^{\Splus}, H)  \, , \quad \forall 
s >  |\Splus|/2  \, .     
\ee  
Let $ H^s_x :=  H^s (\T_1) $, $s \ge 0$, and denote by 
$\big( f, g \big)_{L^2_x}$
the $ L^2-$inner product on $L^2_x \equiv H^0_x$,
\begin{equation}\label{inner product}
\big( f, g \big)_{L^2_x} := \int_{\T_1} f(x) g(x) \,  dx \, . 
\end{equation}
For any $s \geq 0$, let $h^s_0 := \big\{ z = (z_n)_{n \in \Z} \in h^s :  z_0 = 0 \big\}$ where
$h^s $ is the sequence space 
$$
h^s := 
\Big\{ z = (z_n)_{n \in \Z} \, , \ z_n \in \C \, :   \| z \|_s^2 := \sum_{n \in \Z} \langle n \rangle^{2s} |z_n|^2 < \infty\,, \,\,\, 
\overline {z_n} = z_{- n} \, , \,\, \forall n \in \Z \Big\}\, . 
$$ 
By ${\cal F}$ we denote the Fourier transform,
${\cal F} : L^2(\T_1) \to h^0$, $u \mapsto (u_n)_{n \in \Z}$, where $u_n := \int_{\T_1} u(x) e^{- \ii 2 \pi n x}\, d x$ for any $n \in \Z$ 
and by ${\cal F}^{- 1} : h^0 \to L^2(\T_1)$ its inverse.

Furthermore, we denote by $\Pi_\bot$ the $L^2-$orthogonal projector onto the subspace $ L^2_\bot (\T_1) $, defined in \eqref{Def:L2bot},
and by $ \Pi_0^\bot $ the one onto the subspace of functions  with zero average. We set
\be\label{Hsbot}
H^s_\bot(\T_1) := H^s(\T_1) \cap L^2_\bot(\T_1)
\ee
and
$ H^s_\bot \equiv H^s_\bot(\T^{\mathbb S_+} \times \T_1) := \big\{ u \in H^s(\T^{\mathbb S_+} \times \T_1) : u(\vphi , \cdot) \in L^2_\bot(\T_1)\big\} $, 
which is an algebra for   $ s \geq s_0 := [ \frac{|\Splus| + 1}{2}] + 1$. 
The space $H_\bot^0$  is also denoted by $L^2_\bot $.
Let 
\be\label{EsEs}
 {\mathcal E}_s := \T^{\Splus} \times \R^{\Splus} \times H^s_\bot(\T_1)\,, \quad 
{\mathcal E} \equiv {\mathcal E}_0 \, , \qquad \quad
E_s := \R^{\Splus} \times \R^{\Splus} \times H^s_\bot(\T_1)\,, \quad E \equiv E_0\, , 
\ee  
where $ H^s_\bot(\T_1)$ is defined in \eqref{Hsbot}.
Elements of ${\mathcal E}$ are denoted by $\frak x = (\theta, y , w)$
and the ones of its tangent space $E$ by 
$ \widehat{ \frak x} = (\widehat \theta, \widehat y,\widehat w)$.
For $s < 0$, we consider the Sobolev space $ H^s_\bot (\T_1)$ of distributions, 
and the spaces
$ {\mathcal E}_s $ and $ E_s$  are defined in a similar way as in \eqref{EsEs}. 
Note that $ H^{-s}_\bot (\T_1)$ is the dual space of 
$ H^s_\bot (\T_1)$. On $E$, we denote by 
$ \langle \cdot, \cdot \rangle$ the inner  product, defined  
by  
\be\label{bi-form}
\big\langle (\widehat \theta_1, \widehat y_1, \widehat w_1), (\widehat \theta_2, \widehat y_2, \widehat w_2)  \big\rangle := \widehat \theta_1 \cdot \widehat \theta_2 + \widehat y_1 \cdot \widehat y_2  + \big( \widehat w_1, \widehat w_2 \big)_{L^2_x} \, . 
\ee
By a slight abuse of notation, $\Pi_\bot$ 
also denotes the projector of $E_s$  onto its third component,
$$
\Pi_\bot : E_s  \to H^s_\bot(\T_1) \, , \, \quad 
(\widehat \theta, \widehat y, \widehat w) \mapsto \widehat w\, .
$$
For any $0 < \delta < 1$, we denote by
$B_{\Splus}(\delta)$ the open ball in $\R^{\Splus}$ of radius $\delta$  centered at $0$ and by $B_\bot^s(\delta)$, $s \ge 0$, the corresponding one in $H^s_\bot(\T_1)$ where we also write  
$B_\bot(\delta)$ for $B^0_\bot(\delta)$. These balls are used
to define the following open neighborhoods in $\mathcal E_s$, 
$s \in \N$,
\be\label{Vns}
{\cal V}^s(\delta) := \T^{\Splus}_1 \times  B_{\Splus}(\delta) \times 
B_\bot^s(\delta)  \,, \qquad  {\cal V}(\delta) \equiv {\cal V}^0(\delta) \, , \qquad 0 < \delta < 1\,  .
\ee
The space of bounded linear operators between Banach spaces 
$ X_1, X_2 $ is denoted by $ \mathcal B (X_1,X_2)$ and endowed
with the operator norm. 
For two linear operators $ A, B $ we denote by $[ A, B]$
  their commutator, $ [ A, B] := A B - B A $ and by 
  $ A^\top $ the transpose of $A$ 
with respect  to the scalar product \eqref{inner product}. 

\noindent
Throughout the paper, $ \mathtt \Omega \subseteq \R^{\mathbb S_+}$  
denotes a parameter set of frequency vectors.
Given any function $f :  \mathtt \Omega \to X$, we denote by 
$\Delta_\omega f$ the difference function
$$
\Delta_\omega f:   \mathtt \Omega \times 
 \mathtt \Omega \to X \, , \quad
(\omega_1, \omega_2)  \mapsto f(\omega_1) - f (\omega_2)\, . 
$$

\noindent
{\it Acknowledgements.} 
This project was motivated by questions raised by S. Kuksin and V. Zakharov. We would like to thank them for their input.
We would also like to thank M. Procesi for very valuable feedback. 
Part of this work was written during the stay of M. Berti at FIM. We thank FIM for the kind hospitality and support. 
In addition, the research was partially supported by PRIN 
2015KB9WPT005   
(M.B.), by the Swiss National Science Foundation (T.K., R.M.), and by INDAM-GNFM (R.M.).

\section{Preliminaries}

\subsection{Function spaces and linear operators}\label{subsec:function spaces}

In the paper we consider real or complex functions 
$u( \ph, x; \om)$, 
$(\ph,x) \in \T^{\Splus} \times \T_1 $, depending on a parameter $ \om  \in   \mathtt \Omega $ 
in a Lipschitz way,  where ${ \mathtt \Omega}$ is a subset of $\R^{\Splus} $. 
Given  $ 0 < \g < 1$ and $s \ge 0$, we define the norm
\begin{equation} \label{def norm Lip Stein uniform}
\begin{aligned}
& \| u \|_{s,{  \mathtt \Omega}}^\Lipg   := \| u \|_s^\Lipg :=  
\| u \|_s^{\rm sup} + \gamma \| u  \|_s^{\rm lip} \\
& \| u \|^{\rm sup} := \sup_{\om \in {  \mathtt \Omega}}   \| u(\omega) \|_s, \quad \| u \|_s^{\rm lip} :=   
\sup_{\om_1, \om_2 \in {  \mathtt \Omega} \, , \omega_1 \neq \omega_2}  \frac{\| u(\omega_1) - u(\omega_2) \|_{s }}{|\om_1 - \om_2|}  
\end{aligned}
\end{equation}
where $\| \ \|_s$ is the  norm of the Sobolev space $H^s $ 
defined in \eqref{unified norm}. 
For a function
$u:   \mathtt \Omega \to \C$, the sup norm and the Lipschitz semi-norm
are denoted by $|u |^{\rm sup}$ and, 
respectively $|u|^{\rm lip}$. Correpondingly, we write
$| u |^\Lipg := | u |^{\rm sup} + \gamma | u  |^{\rm lip}$. 

By $\Pi_N$, $N \in \N_+$,  we denote the {\it smoothing} operators on $H^s$,
\begin{equation}\label{def:smoothings}
(\Pi_N u)(\ph,x) := \sum_{\la \ell,j \ra \leq N} u_{\ell, j} e^{\ii (\ell\cdot\ph + 2 \pi jx)}  \, , \qquad
\Pi^\perp_N := {\rm Id} - \Pi_N \, .
\end{equation}
They satisfy, for any $ \a \geq 0 $, $ s \in \R $,  the estimates
\begin{align}
\| \Pi_N u \|_{s}^\Lipg 
 \leq N^\alpha \| u \|_{s-\alpha}^\Lipg\, , \qquad
\| \Pi_N^\bot u \|_{s}^\Lipg 
 \leq N^{-\alpha} \| u \|_{s + \alpha}^\Lipg  \, . 
\label{p3-proi}
\end{align}
Furthermore the 
following interpolation inequalities hold: for any  $0 \le s_1<s_2 $ and $0 < \theta < 1 $,  
\begin{equation}\label{2202.3}
\| u \|_{\theta s_1 + (1- \theta) s_2}^\Lipg \leq 2 (\| u \|_{s_1}^\Lipg)^\theta 
(\| u \|_{s_2}^\Lipg)^{1 - \theta}\, .   
\end{equation}
Multiplication  and composition with Sobolev functions satisfy the following tame estimates. 
\begin{lemma}{\bf (Product and composition)}
\label{lemma:LS norms}
(i) For any $ s \geq s_0 = [ (|\Splus| + 1)/2] + 1 $ 
\begin{align}
\| uv \|_{s}^\Lipg
& \leq C(s) \| u \|_{s}^\Lipg \| v \|_{s_0}^\Lipg 
+ C(s_0) \| u \|_{s_0}^\Lipg \| v \|_{s}^\Lipg\,. 
\label{p1-pr}
\end{align}
(ii) Let $\b(\cdot,\cdot; \omega) : 
\T^{\mathbb S_+} \times \T_1 \to \R$ with $ \| \b \|_{2s_0+2}^\Lipg \leq \d (s_0) $ small enough. Then 
the composition operator 
$\mB : u \mapsto \mB u, \,  
(\mB u)(\ph,x) := u(\ph, x + \b (\ph,x))$
satisfies,  for any $ s \geq s_0 + 1$, 
\be\label{pr-comp1}
\| {\cal B} u \|_{s}^\Lipg \lesssim_{s} \| u \|_{s+1}^\Lipg 
+ \| \b \|_{s}^\Lipg \| u \|_{s_0+2}^\Lipg \, .
\ee
The function $ \breve \beta $, obtained by 
solving  $ y = x + \beta (\vphi, x) $ for $x,$
 $ x = y + \breve \b ( \vphi, y ) $,  
satisfies 
\be\label{p1-diffeo-inv}
\| \breve \beta \|_{s}^\Lipg \lesssim_{s}  \| \b \|_{s+1}^\Lipg \, , \quad \forall s \geq s_0\,.
\ee
(iii) Let $\alpha(\cdot; \omega) : \T^{\mathbb S_+} \to \R$ with $ \| \alpha \|_{2s_0+2}^\Lipg \leq \d (s_0) $ small enough.
Then the composition operator 
${\cal A} : u \mapsto {\cal A} u, \,  
({\cal A} u)(\ph,x) := u(\ph +  \alpha(\vphi)\omega, x)$
satisfies,  for any $ s \geq s_0 + 1$, 
\be\label{pr-comp1a}
\| {\cal A} u \|_{s}^\Lipg \lesssim_{s} \| u \|_{s+1}^\Lipg 
+ \| \alpha \|_{s}^\Lipg \| u \|_{s_0+2}^\Lipg \, .
\ee
The function $ \breve \alpha $, obtained by 
solving  $ \vartheta = \vphi +  \alpha(\vphi)\omega $ 
for $\vphi$,
 $ \vphi = \vartheta + \breve \alpha( \vartheta )\omega $,  
satisfies 
\be\label{p1-diffeo-inva}
\| \breve \alpha  \|_{s}^\Lipg \lesssim_{s}  \| \alpha \|_{s+1}^\Lipg \, , \quad \forall s \geq s_0 \,.
\ee

\end{lemma}
\begin{proof}
Item $(i)$ follows from (2.72) in \cite{Berti-Montalto} 
and $(ii)$-$(iii)$ follow from 
\cite[Lemma 2.30]{Berti-Montalto}. 
\end{proof}

If  $ \om $ is diophantine, namely 
$$
|\omega \cdot \ell| \geq \frac{\gamma}{|\ell|^{\tau}} \, ,  \quad  \forall \ell \in \Z^{\Splus} \setminus  \{ 0 \}  \, ,
$$
the equation $\ompaph v = u$, where $u(\ph,x)$ has zero average with respect to $ \vphi $, 
has the periodic solution 
$$
(\om \cdot \pa_\vphi )^{-1} u = \sum_{j \in \Z, \ell \in \Z^{\Splus} \setminus \{0\}} 
\frac{ u_{\ell, j} }{\ii \om \cdot \ell }e^{\ii (\ell \cdot \vphi + 2 \pi j x )} \, ,
$$
and it satisfies the estimate 
(cf. e.g. \cite[Lemma 2.2]{BKM})
\be \label{Diophantine-1}
\| (\om \cdot \pa_\vphi )^{-1} u \|_{s}^\Lipg
\leq C \g^{-1}\| u \|_{s+ 2 \tau + 1}^\Lipg \,.  
\ee
We also record Moser's tame estimate for the nonlinear composition operator
$$
u(\vphi, x) \mapsto {\mathtt f}(u)(\vphi, x) := f(\vphi, x, u(\vphi, x)) \, . 
$$ 
Since the variables $\vphi$ and $x$ play the same role, 
we state it  for the Sobolev space  $ H^s (\T^d ) $,
(cf. e.g. \cite[Lemma 2.31]{Berti-Montalto}). 

\begin{lemma}{\bf (Composition operator)} \label{Moser norme pesate}
Let $ f \in {\cal C}^{\infty}(\T^d \times \R^n, \C )$. 
If $v(\cdot;\omega) \in H^s(\T^d, \R^n )$, $\om \in {  \mathtt \Omega} $,  is a family of 
Sobolev functions
satisfying $\| v \|_{s_0(d)}^\Lipg \leq 1 $ where 
$ s_0(d) > d/2  $, 
 then, for any $ s \geq s_0(d) $, 
\begin{equation} \label{0811.10}
\| {\mathtt f}(v) \|_{s}^\Lipg \leq C(s,  f ) ( 1 + \| v \|_{s}^\Lipg) \,.
\end{equation} 
Moreover, if  $  f(\vphi, x, 0) = 0 $, then 
$ \| {\mathtt f}(v) \|_{s}^\Lipg \leq C(s,  f )  \| v \|_{s}^\Lipg $. 
\end{lemma}

\noindent
{\bf Linear operators.} 
Throughout the paper we consider $ \vphi $-dependent families of  linear operators
 $ A : \T^{\Splus} \to {\cal L}( L^2(\T_1, \C))  $, 
$ \vphi \mapsto A(\vphi) $, 
acting on complex valued functions $ u(x) $ of the space variable $ x $.
We also regard $ A $ as an operator (which for simplicity we denote by $A $ as well)
that acts on functions $ u(\varphi , x) $ of space-time, i.e.
as an element
in ${\cal L}(L^2(\T^{\Splus} \times \T_1, \C ) )$ defined by
\begin{equation}\label{def azione toplitz u vphi x}
A [u] (\varphi , x) \equiv ( A u) (\varphi , x) := (A(\varphi) u(\varphi, \cdot ))(x) \, .  
\end{equation}
We say that the operator $ A $ is {\it real} if 
it maps real valued functions into real valued functions.

When $u$ in \eqref{def azione toplitz u vphi x} is
 expanded in its Fourier series, 
\be\label{u-Fourier-expa}
u(\vphi, x ) =   \sum_{j \in \Z} u_{j} (\vphi) e^{2 \pi \ii j x } =
\sum_{j \in \Z, \ell \in \Z^{\Splus}}  u_{\ell,j}  e^{\ii (\ell \cdot \vphi + 2 \pi j x)}\,,
\ee
one obtains
\begin{equation}\label{matrice operatori Toplitz}
(A u) (\vphi, x) = \sum_{j , j' \in \Z} A_j^{j'}(\vphi) u_{j'}(\vphi) e^{\ii 2 \pi j x} = 
 \sum_{j \in \Z, \ell \in \Z^{\Splus}} \, \sum_{j' \in \Z, \ell' \in \Z^{\Splus}} A_j^{j'}(\ell - \ell') u_{\ell', j'} e^{\ii (\ell \cdot \vphi + 2 \pi j x)} \, . 
\ee
We shall identify an operator $ A $ with the matrix $ \big( A^{j'}_j (\ell- \ell') \big)_{j, j' \in \Z, \ell, \ell' \in \Z^{\Splus} } $.  

\begin{definition}\label{def:maj} 
Given a linear operator $ A $ as in \eqref{matrice operatori Toplitz} we define the following operators:
\begin{enumerate}
\item $ | A | $  \textsc{(majorant operator)}  
whose matrix elements are $  | A_j^{j'}(\ell - \ell')| $.  
\item $ \Pi_N A $, $ N \in \N_+  $ \textsc{ (smoothed operator)}  
whose matrix elements are
\be\label{proiettore-oper}
 (\Pi_N A)^{j'}_j (\ell- \ell') := 
 \begin{cases}
 A^{j'}_j (\ell- \ell') \quad {\rm if} \quad   \langle \ell - \ell'  \rangle \leq N \\
 0  \qquad \qquad \quad {\rm  otherwise} \, .
 \end{cases} 
\ee
\item $ \langle \pa_{\vphi} \rangle^{b}  A $, $ b \in \R $,  
whose matrix elements are  $  \langle \ell - \ell'  \rangle^b A_j^{j'}(\ell - \ell') $.
\item $\partial_{\vphi_m} A(\vphi) 
= [\partial_{\vphi_m}, A] $
\textsc{(differentiated operator)} whose matrix elements are $\ii (\ell_m - \ell_m') A_{j}^{j'}(\ell - \ell')$. 
\end{enumerate}
\end{definition}

\begin{definition} {\bf (Hamiltonian and  symplectic operators)}
\label{definition Hamiltonian operators}
(i) A $ \vphi $-dependent family of linear operators 
$ X(\vphi) $, $ \vphi \in \T^{\Splus} $,   densily defined 
in $ L_0^{2} (\T_1) $, 
 is \textsc{Hamiltonian} if 
$ X(\vphi ) = \partial_x G( \vphi ) $ for some  
real linear operator $ G(\vphi ) $ which is 
self-adjoint with respect to the $ L^2-$inner product.  
We also say that $ \om \cdot \partial_{\vphi} - \partial_x G( \vphi ) $ is Hamiltonian. \\
(ii) A $\vphi $-dependent family of linear operators 
$ A (\vphi ) : L_0^2 (\T_1) \to L_0^2 (\T_1) $, 
$ \forall \vphi \in \T^{\Splus} $, is {\sc symplectic} if   
$$
{\cal W}_{L^2_0}(A (\vphi) u, A (\vphi)  v) =
{\cal W}_{L^2_0} (u, v) \, , \quad \forall u,v \in L_0^2 (\T_1 )\, , 
$$
where the symplectic 2-form $ {\cal W}_{L^2_0} $ is defined in \eqref{KdV symplectic}. 
\end{definition}

Under a $ \vphi $-dependent family of symplectic transformations
 $ \Phi(\vphi)  $, $ \vphi \in \T^{\Splus} $, 
the linear  Hamiltonian operator
$ \om \cdot \pa_{\vphi} -  \partial_x G(\vphi) $ 
transforms into another Hamiltonian one.

\begin{lemma}\label{lem:PR} A family of  operators $ R (\vphi) $, $ \vphi \in \T^{\Splus} $,  expanded as
 $ R (\vphi)  = \sum_{\ell \in \Z^{\Splus}} R (\ell) e^{\ii \ell \cdot \vphi } $,  is \\
(i)
{\sc self-adjoint} if and only if 
$ \overline{ R_j^{j'}( \ell ) } = R_{j'}^j(- \ell ) $; \quad \\
(ii) {\sc real}  if and only if 
$\overline{R^j_{j'}(\ell)} = R^{-j}_{-j'}(-\ell) \, ;  $\\
$(iii)$ {\sc Real and self-adjoint} if and only if $R_j^{j'}(\ell) = R_{- j'}^{- j}(\ell)$.  
\end{lemma} 

\begin{lemma}\label{lem:229}
Let $X : H^{s + 3}_0(\T_1) \to H^s_0(\T_1)$ be a linear Hamiltonian vector field of the form 
\be\label{expX}
X = \sum_{k = 0}^2 a_{3 - k}(x) \partial_x^{3 - k} + {\rm bounded \ operator}
\ee
where $a_{3 - k} \in {\cal C}^\infty(\T_1)$. Then $a_2 = 2(a_3)_x$.
\end{lemma}
\begin{proof}
Since $X$ is a linear Hamiltonian vector field it has the form 
$ X = \partial_x {\cal A} $ where ${\cal A} $  
is a densely defined operator on $L^2_0(\T_1)$ satisfying 
 ${\cal A} = {\cal A}^\top$.
Therefore, using \eqref{expX}, 
\begin{align*}
& {\cal A} = \pa_x^{-1} X = a_3 (x) \pa_{xx} + 
\big( - (a_3)_x + a_2 \big) \pa_x + \ldots \\
& {\cal A}^\top = - X^\top \pa_{x}^{-1} = 
a_3 (x) \pa_{xx} + \big( 3 (a_3)_x  - a_2 \big) \pa_x + \ldots \, .
\end{align*}
The identity $ {\cal A} = {\cal A}^\top $ implies that  $a_2 = 2 (a_3)_x$. 
\end{proof}

\subsection{Pseudo-differential operators}\label{sec:pseudo}

In this section we recall properties of pseudo-differential operators on the torus used in this paper,  
following \cite{Berti-Montalto}. Note however that $ x \in \T_1  $ and not in $ \R / (2 \pi \Z) $.

\begin{definition} \label{def:Ps2}
We say that $a: \T_1 \times \R \to \C$
is a symbol of order $m \in \R$ if,  for any $ \a, \b \in \N $, 
\be\label{symbol-pseudo2}
\big| \pa_x^\a \pa_\xi^\b a (x,\xi ) \big| \leq C_{\a,\b} \langle \xi \rangle^{m - \b} \, , \quad \forall (x, \xi) \in \T_1 \times \R \, .
\ee
The set of such symbols is denoted by $S^m$.
Given $a \in S^m,$  we denote by $A$ the operator,
which maps a one periodic function  $u(x) =  \sum_{j \in \Z}u_j e^{\ii j x}$ to
$$
A[u](x) \equiv (Au) (x) := {\mathop \sum}_{j \in \Z} a(x,j) u_j e^{\ii jx}.
$$
The operator $A$ is referred to as the {\sc pseudo-differential operator ($\Psi {\rm DO}$)} of order $m$, associated to the symbol $a$, 
and is also denoted by $Op(a)$ or $a(x, D)$  where $D = \frac{1}{\ii} \partial_x$.
Furthermore we denote by $ OPS^m $ the set of pseudo-differential operators $a(x, D)$ with $a(x, \xi) \in S^m$ 
and set $ OPS^{-\infty} := \cap_{m \in \R} OPS^{m} $.
\end{definition}

When the symbol $a$ is independent of $ \xi $, the operator $ A = {\rm Op} (a) $ is 
the multiplication operator by the function $ a(x)$, 
i.e.,\  $ A : u (x) \mapsto a ( x) u(x )$ and 
we also write $a $ for $ A$. 
More generally, we consider symbols $a(\vphi, x, \xi; \omega)$,
depending in addition on the variable $\varphi \in \T^{\Splus}$ 
and the parameter $\omega$, where $a$ is  $ {\cal C}^\infty $ 
in  $\varphi$ and 
Lipschitz continuous with respect to $\omega$. 
By a slight abuse of notation, we denote the class of such symbols
of order $m$ also by $S^m$.  
Alternatively, we denote $A$ by
$A(\vphi)$ or ${\rm Op} (a(\vphi, \cdot))$. 

Given an even cut off function $\chi_0 \in {\cal C}^\infty(\R, \R)$,
satisfying
\begin{equation}\label{def chi 0}
\begin{aligned}
&   0 \leq \chi_0 \leq 1 \, , \qquad \quad 
\chi_0(\xi) = 0 \, , \quad \forall  |\xi| <  \frac12 \, , 
 \qquad \quad \chi_0(\xi) = 1 \, , \quad \forall |\xi| \geq \frac23 \, , 
\end{aligned}
\end{equation}
we define, for any $m \in \Z,$ 
$\partial_x^m = {\rm Op}(\chi_0(\xi) (\ii 2 \pi \xi)^{m})$, so that 
\be\label{pax-k}
\partial_x^{m}[e^{\ii 2 \pi j x}] = (\ii 2 \pi j)^m e^{\ii 2 \pi j x}\,, \ \  
j \in \Z \setminus \{ 0 \} \, , \quad 
\partial_x^{ m}[1] = 0\, .
\ee   
Note that 
$\partial^0_x [u] (x) = u(x) - u_0$, 
hence $\partial_x^{0}$ is not the identity operator.

Now we recall the norm of a symbol $ a (\vphi, x, \xi; \omega) $
in $S^m$, introduced in \cite[Definition 2.11]{Berti-Montalto}, 
which controls the regularity in $ (\vphi, x)$  and the decay in $ \xi $  of $a$ and its derivatives 
$ \pa_\xi^\b a \in S^{m - \b}$, $ 0 \leq \b \leq \a $, in the Sobolev norm $ \| \ \|_s $.  
By a slight abuse of terminology,  
we refer to it as     
the norm of the corresponding  pseudo-differential operator. 
Unlike \cite{Berti-Montalto} we consider 
the difference quotient instead of the derivative
with respect to $\omega$, and write
$ |  \ |_{m,s,\a}^{1,\gamma}$ instead of $ |  \ |_{m,s,\a}^{\Lipg} $.  

\begin{definition}\label{def:pseudo-norm}
Let $ A(\omega) := a(\vphi, x, D; \omega ) \in OPS^m $ 
be a family of pseudo-differential operators with symbols $ a(\vphi, x, \xi; \om) \in S^m $, $ m \in \R $.
For $ \g \in (0,1) $, $ \a \in \N $, $ s \geq 0 $, we define  the {\sc weighted $\Psi$do norm} 
of $A$ as
$$
| A |_{m, s, \alpha}^{\Lipg} :=   
\sup_{\omega \in { \mathtt \Omega}} |A(\omega ) |_{m, s, \alpha}  + \gamma  \sup_{\begin{subarray}{c}
\omega_1, \omega_2 \in {  \mathtt\Omega} \\
\omega_1 \neq \omega_2 
\end{subarray}}\frac{|A(\omega_1) - A(\omega_2)|_{m, s, \alpha}}{|\omega_1 - \omega_2|}
$$
where 
$ |A(\omega)  |_{m, s, \a} := 
\max_{0 \leq \beta  \leq \a} \, \sup_{\xi \in \R} \|  \partial_\xi^\beta 
a(\cdot, \cdot, \xi; \omega )  \|_{s} \langle \xi \rangle^{-m + \beta} $. 
\end{definition}
\noindent
Note that for any $s \leq s'$, $\a \leq  \a'$, and $m \leq m',$
\be\label{norm-increa}
| \cdot  |_{m, s, \a}^\Lipg \leq \, | \cdot  |_{m, s', \a}^{\Lipg} \, , \qquad  
|  \cdot   |_{m, s, \a}^\Lipg  \leq \, 
|  \cdot   |_{m, s, \a'}^\Lipg \ ,  \qquad  
\  |  \cdot   |_{m', s, \a}^\Lipg 
\leq |  \cdot  |_{m, s, \a}^\Lipg \, . 
\ee
For a Fourier multiplier $   g( D; \omega )  $ with symbol $ g \in S^m $, one has 
\be\label{Norm Fourier multiplier}
|  {\rm Op}(g)  |_{m, s, \a}^\Lipg 
= |  {\rm Op}(g)  |_{m, 0, \a}^\Lipg
\leq C(m, \a, g) \, , \quad \forall s \geq 0 \, , 
\ee
and, for a function $a(\vphi, x; \omega)$, 
\begin{equation}\label{norma pseudo moltiplicazione}
|{\rm Op}(a)|_{0, s, \alpha}^\Lipg = |{\rm Op}(a)|_{0, s, 0}^\Lipg \lesssim \| a \|_s^\Lipg\,. 
\end{equation}
{\bf Composition.}
If $ A = a(\varphi, x, D; \omega) \in OPS^{m} $, 
$ B = b(\varphi, x, D; \omega) \in  OPS^{m'} $   
then the composition
$ A B := A \circ B$ is a pseudo-differential operator with a symbol 
$ \sigma_{AB} (\varphi, x, \xi; \omega) $
 in $S^{m+m'}$ which, for any $N \geq 0$,  admits
the asymptotic expansion    
\be\label{expansion symbol}
\s_{AB} (\varphi, x, \xi; \omega) = 
\sum_{\b =0}^{N} \frac{1}{\ii^\b \b !}  
\pa_\xi^\b a (\varphi, x, \xi; \omega) \, 
\pa_x^\b b(\varphi, x, \xi; \omega) 
+ r_N (\varphi, x, \xi; \omega)
\ee
with remainder $ r_N \in S^{m + m' - N -1}$.
We record the following tame estimate for the composition of 
two pseudo-differential operators, proved in \cite[Lemma 2.13]{Berti-Montalto}.

\begin{lemma} \label{lemma stime Ck parametri} 
{\bf (Composition)} 
Let $ A = a(\vphi, x, D; \omega ) $, $ B = b(\vphi, x, D; \omega) $ be pseudo-differential operators
with symbols $ a ( \vphi, x, \xi; \omega) \in S^m $, $ b ( \vphi, x, \xi ; \omega) \in S^{m'} $, $ m , m' \in \R $. Then $ A(\omega) \circ B(\omega)$
is the pseudo-differential operator of order $m + m'$,
associated to the symbol $\sigma_{AB}(\varphi, x, \xi; \omega)$
which satisfies,  for any $ \a \in \N $, $ s \geq s_0 $, 
\begin{equation}\label{estimate composition parameters}
| A B |_{m + m', s, \alpha}^\Lipg 
\lesssim_{m,  \alpha} C(s) | A |_{m, s, \alpha}^\Lipg 
| B |_{m', s_0 + \alpha + |m|, \alpha}^\Lipg  
+ C(s_0) | A  |_{m, s_0, \alpha}^\Lipg  
| B |_{m', s + \alpha + |m|, \alpha}^\Lipg \, . 
\end{equation}
Moreover, for any integer $ N \geq 1  $,  
the remainder $ R_N := {\rm Op}(r_N) $ in \eqref{expansion symbol} satisfies
\be
\begin{aligned}
| R_N |_{m+ m'- N -1, s, \alpha}^\Lipg 
\lesssim_{m, N,  \alpha} &
C(s) | A |_{m, s, N + 1 + \alpha}^\Lipg 
| B  |_{m', s_0 + 2(N+1) + |m| + \alpha, \alpha }^\Lipg  
\\
+ \, &  C(s_0)| A |_{m, s_0 , N + 1 + \alpha}^\Lipg
| B  |_{m', s + 2(N+1) + |m| + \alpha, \alpha }^\Lipg.
\label{stima RN derivate xi parametri} 
\end{aligned}
\ee
\end{lemma}

By \eqref{expansion symbol} the  commutator $[A, B]$
of two pseudo-differential operators $ A = a(x, D) 
\in OPS^{m} $ and $ B = b (x, D)  \in OPS^{m'} $ 
is a pseudo-differential operator of order $m +m' -1$,
and Lemma \ref{lemma stime Ck parametri} then leads to the following 
lemma,  
cf. \cite[Lemma 2.15]{Berti-Montalto}.
\begin{lemma}{\bf (Commutator)} \label{lemma tame norma commutatore}
If $ A = a(\vphi, x, D; \omega) \in OPS^m $ and  $ B = b(\vphi, x, D; \omega) \in OPS^{m'} $,  
$ m , m' \in \R $, then  
the commutator $ [A, B] := AB - B A$ is 
the pseudo-differential operator of order $m+m'-1$
associated to the symbol 
$\sigma_{AB} (\varphi, x, \xi; \omega) 
- \sigma_{BA}(\varphi, x, \xi; \omega) 
\in S^{m+m'-1}$ which for any $ \a \in \N $ and $ s \geq s_0 $
satisfies
\be
\begin{aligned}
| [A,B]  |_{m+m' -1, s, \alpha}^\Lipg & 
\lesssim_{m,m',   \alpha} 
C(s) | A |_{m, s + 2 + |m' | + \alpha, \alpha + 1}^\Lipg 
| B  |_{m', s_0 + 2 + |m| + \alpha , \alpha + 1}^\Lipg  \\
& \qquad \quad  
+ C(s_0) | A  
|_{m, s_0 + 2 + |m' | + \alpha, \alpha + 1}^\Lipg 
| B |_{m', s + 2 + |m| + \alpha , \alpha + 1}^\Lipg \, .  
\label{stima commutator parte astratta}
\end{aligned}
\ee
\end{lemma}

In the case of operators 
of the special form $a \partial_x^m $, 
Lemma \ref{lemma stime Ck parametri} 
and Lemma \ref{lemma tame norma commutatore}
simplify as follows:
\begin{lemma}\label{composizione simboli omogenei}{\bf (Composition and commutator of homogeneous symbols)} 
Let $ A = a \partial_x^m$, $B = b \partial_x^{m'}$ where
$m, m' \in \Z $ and
 $a(\varphi, x ; \omega) $, $ b(\varphi, x ; \omega) $ are
$ {\cal C}^\infty-$smooth functions with respect to
 $(\varphi , x)$ and   Lipschitz with respect to $\omega \in  \mathtt\Omega$.
Then there exist combinatorial constants $K_{n, m} \in \R $, 
$0 \le n \le N $,  with  $K_{0, m} = 1$ and $K_{1, m} = m $
so that the following holds:  

\noindent
$(i)$ For any $N \in \N$, the composition $A \circ B$
is in $OPS^{m+m'}$ and
admits the asymptotic expansion 
$$
A \circ B = \sum_{n = 0}^{N} K_{n, m} \, a\,  (\partial_x^n b) \partial_x^{m + m' - n} + {\cal R}_N(a, b)
$$
where  the remainder ${\cal R}_N(a, b) $ is in $ OPS^{m + m' - N - 1}$. Furthermore there is a constant $\sigma_N(m) > 0$ so that, for any $s \geq s_0$, $\alpha \in \N$, 
$$
|{\cal R}_N(a, b)|_{m + m'- N - 1, s, \alpha}^\Lipg \lesssim_{m, m', s, N, \alpha} \| a\|_{s + \sigma_N(m) }^\Lipg \| b \|_{s_0 + \sigma_N(m)}^\Lipg + \| a \|_{s_0 + \sigma_N(m)}^\Lipg \| b \|_{s + \sigma_N(m)}^\Lipg\,. 
$$
$(ii)$ For any $N \in \N_+$, the commutator $[A, B]$ is in 
$OPS^{m + m' - 1}$ and admits the asymptotic expansion 
$$
[A,  B] = \sum_{n = 1}^N  \big( K_{n,  m}a (\partial_x^n b) -  K_{ n, m'}(\partial_x^n a) b \big) \partial_x^{m + m' - n} + {\cal Q}_N(a, b)
$$
where the remainder ${\cal Q}_N(a, b)$ is in 
$OPS^{m + m' - N - 1}$. Furthermore, there 
is a constant $\sigma_N(m, m') > 0 $ so that,  
for any $s \geq s_0$, $\alpha \in \N$, 
$$
|{\cal Q}_N(a, b)|_{m + m'- N - 1, s, \alpha}^\Lipg \lesssim_{m, m', s, N, \alpha} \| a\|_{s + \sigma_N(m, m') }^\Lipg \| b \|_{s_0 + \sigma_N(m, m')}^\Lipg + \| a \|_{s_0 + \sigma_N(m, m')}^\Lipg \| b \|_{s + \sigma_N(m, m')}^\Lipg\,. 
$$
\end{lemma}
\begin{proof} See formula \eqref{expansion symbol} 
and  Lemma \ref{lemma stime Ck parametri}. 
\end{proof}

We finally give the following result on the 
exponential of a pseudo-differential operator of order $0$.
\begin{lemma} {\bf (Exponential map)} \label{Neumann pseudo diff}
If $ A := {\rm Op}(a(\vphi, x, \xi; \omega ))$ 
is in $ OPS^{0} $,   
then $\sum_{k \ge 0} \frac{1}{k!} \sigma_{A^k}(\varphi, x , \xi; \omega)$ is a symbol of order $0$ and hence the corresponding pseudo-differential operator, denoted by
$\Phi = \exp(A)$, is in $OPS^0$, and 
for any $s \geq s_0$, $\alpha \in \N $,
there is a constant 
$C(s, \alpha) > 0$ so that 
\be\label{exp-map}
|\Phi - {\rm Id} |_{0, s, \alpha}^\Lipg \leq  |A|_{0, s + \alpha, \alpha}^\Lipg 
{\rm exp} \big( C(s, \alpha) 
|A|_{0, s_0 + \alpha, \alpha}^\Lipg \big) \, .
\ee 
\end{lemma}

\begin{proof}
Iterating  \eqref{estimate composition parameters}, 
for any $s \geq s_0$, $\alpha \in \N $,  there 
is a constant $C(s, \alpha) > 0$ such that
\begin{equation}\label{composizione iterata}
|A^k|_{0, s, \alpha}^\Lipg \leq C(s, \alpha)^{k - 1} (|A|_{0, s_0 + \alpha, \alpha}^\Lipg)^{k - 1} |A|_{0, s + \alpha, \alpha}^\Lipg\, , \quad \forall k \geq 1 \, .
\end{equation}
Therefore 
$$
\begin{aligned}
|\Phi - {\rm Id}|_{0, s, \alpha}^\Lipg & \leq  \sum_{k \geq 1}
\frac{1}{k!} |A^k|_{0, s, \alpha}^\Lipg  \stackrel{\eqref{composizione iterata}}{\leq} |A|_{0, s + \alpha, \alpha}^\Lipg\sum_{k \geq 1} \frac{1}{k!} 
C(s, \alpha)^{k - 1} (|A|_{0, s_0 + \alpha, \alpha}^\Lipg)^{k - 1}  \\
& \leq  |A|_{0, s + \alpha, \alpha}^\Lipg {\rm exp}\big( C(s, \alpha) |A|_{0, s_0 + \alpha, \alpha}^\Lipg \big) 
\end{aligned} \, .
$$
This shows that
$\sum_{k \ge 0} \frac{1}{k!} 
\sigma_{A^k}(\varphi, x , \xi; \omega)$ 
is a symbol in $S^0$ and that the estimate \eqref{exp-map} holds. 
\end{proof}

\subsection{$ \Lipg $-tame and modulo-tame operators}

In this section we recall  the notion and the main properties of 
$ \Lipg $-$\s$-tame and $ \Lipg $-modulo-tame operators.
We refer to \cite[Section 2.2]{Berti-Montalto} where this notion was introduced, 
with the only difference that here we consider 
difference quotients instead of first order derivatives with respect to the parameter $\omega$.

\begin{definition} { \bf ($ \Lipg $-$ \s $-tame)}\label{def:Ck0}
Let $\s \ge 0$.
A linear operator  $ A := A(\omega)  $ as in \eqref{def azione toplitz u vphi x} is 
$ \Lipg $-$ \s $-tame  if there exist  $S > s_1 \geq s_0$ and 
a non-decreasing function $[s_1, S] \to [0, + \infty)$, $s \mapsto {\mathfrak M}_A(s)$,
so that, for any $ s_1 \leq s \leq S $ and $u \in H^{s+\s} $,  
\be\label{CK0-sigma-tame}
\sup_{\omega \in  {  \mathtt \Omega}} 
\|  A(\omega) u \|_s +  \gamma 
\sup_{\omega_1, \omega_2 \in {  \mathtt \Omega} \atop \omega_1 \neq \omega_2} \Big\| \frac{A(\omega_1) - A(\omega_2)}{|\omega_1- \omega_2|} u \Big\|_{s} 
\leq  {\mathfrak M}_A(s_1) \| u \|_{s+\s} + 
{\mathfrak M}_A (s) \| u  \|_{s_1+\s} \,. 
\ee
When $ \sigma $ is zero, 
we simply write $ \Lipg $-tame instead of $ \Lipg $-$ 0 $-tame.
We say that $ {\mathfrak M}_{A}(s) $ is a {\sc tame constant} of the operator $ A $.
Note that $ {\mathfrak M}_A(s)$ is not uniquely determined 
and that it may also depend on 
the ``loss of derivatives" $\s$. 
We will not indicate this dependence. 
\end{definition}

Representing the operator $ A $ by its matrix elements 
$ \big(A_j^{j'} (\ell - \ell') \big)_{\ell, \ell' \in \Z^{\Splus}, j, j' \in \Z} $ as in \eqref{matrice operatori Toplitz}, we have, for all
$ j' \in \Z $, $ \ell' \in \Z^{\Splus} $, 
for all $ \omega_1, \omega_2 \in {  \mathtt\Omega} $, $\omega_1 \neq \omega_2 $, 
\be\label{tame-coeff}
\begin{aligned}
& {\mathop \sum}_{\ell , j} 
\langle \ell, j \rangle^{2 s_1} \Big( \big| A_j^{j'}(\ell - \ell') \big|^2 +
\gamma^{2}  \Big| \frac{\Delta_\omega A_j^{j'}(\ell - \ell')}{|\omega_1 - \omega_2|} \Big|^2 \Big) 
\lesssim \big({\mathfrak M}_A(s_1) \big)^2  \langle \ell', j' \rangle^{2 (s_1+\s)} 	
\end{aligned}
\ee
where we recall that 
$ \Delta_\omega f = f(\omega_1)- f (\omega_2) $.  

\begin{lemma}\label{composizione operatori tame AB} {\bf (Composition)}
Let $ A, B $ be, respectively, $ \Lipg $-$\sigma_A$-tame and 
$ \Lipg $-$\sigma_B$-tame operators with tame 
constants
$ {\mathfrak M}_A (s) $ and $ {\mathfrak M}_B (s) $. 
Then the composition 
$ A \circ B $ is $ \Lipg $-$(\sigma_A + \sigma_B)$-tame with a tame constant satisfying 
$$
 {\mathfrak M}_{A B} (s) \lesssim  {\mathfrak M}_{A}(s) 
 {\mathfrak M}_{B} (s_1 + \sigma_A) + {\mathfrak M}_{A} (s_1) 
{\mathfrak M}_{B} (s + \sigma_A) \,.
$$
\end{lemma}

\begin{proof}
See \cite[Lemma 2.20]{Berti-Montalto}. 
\end{proof}

We now discuss the action of a $ \Lipg $-$ \s $-tame operator  $ A(\omega) $ on a family of  
Sobolev functions  $ u (\omega) \in H^s $.

\begin{lemma}\label{lemma operatore e funzioni dipendenti da parametro}
{\bf (Action on $ H^s $)}
Let $ A := A(\omega) $ be a $ \Lipg $-$ \s $-tame operator
with tame constant ${\mathfrak M}_A(s)$. 
Then, for any family of Sobolev functions 
$ u := u(\omega) \in H^{s+\s} $, Lipschitz  with respect to $ \omega $, one has
$$
\| A u \|_s^\Lipg \lesssim  {\mathfrak M}_A(s_1) \| u \|_{s + \sigma}^\Lipg 
+ {\mathfrak M}_A(s) \| u \|_{s_1 + \sigma}^\Lipg  \,.
$$
\end{lemma}

\begin{proof}
See \cite[Lemma 2.22]{Berti-Montalto}. 
\end{proof}

Pseudo-differential operators are tame operators. We shall use in particular the following lemma.

\begin{lemma}\label{lemma: action Sobolev}
Let $ a( \vphi, x, \xi; \omega) \in S^0 $  be a family of symbols
that are Lipschitz with respect to  $\omega $. 
If $ A = a( \vphi, x, D; \omega)$ satisfies
$ | A |_{0, s, 0}^\Lipg < + \infty $, $ s \geq s_0 $, then  $ A $ is $ \Lipg $-tame with a tame constant satisfying 
\begin{equation}\label{interpolazione parametri operatore funzioni}
{\mathfrak M}_A(s) \leq C(s)  | A |_{0, s, 0}^\Lipg\,.
\end{equation}
As a consequence 
\begin{equation}\label{interpolazione parametri operatore funzioni (2)}
\| A u \|_s^\Lipg \leq C(s_0) | A |_{0, s_0, 0}^\Lipg \| u \|_{s}^\Lipg + 
C(s) | A |_{0, s, 0}^\Lipg \| u \|_{s_0}^\Lipg\,.
\end{equation}
\end{lemma}

\begin{proof}
See \cite[Lemma 2.21]{Berti-Montalto} for the proof of \eqref{interpolazione parametri operatore funzioni}.
The estimate 
\eqref{interpolazione parametri operatore funzioni (2)}
then follows from Lemma \ref{lemma operatore e funzioni dipendenti da parametro}.
\end{proof}

In the KAM reducibility scheme of Section  \ref{sec: reducibility}, 
we need to consider $ \Lipg $-tame operators $A$ which satisfy
a stronger condition, referred to
$ \Lipg $-modulo-tame operators. 

\begin{definition}\label{def:op-tame} 
{ \bf ($ \Lipg $-modulo-tame)} Let $S > s_1 \geq s_0$. A  linear operator $ A := A(\omega) $ as in \eqref{def azione toplitz u vphi x} is  
$ \Lipg $-modulo-tame  if there exists a non-decreasing function $[s_1, S] \to [0, + \infty)$, $s \mapsto {\mathfrak M}_{A}^\sharp (s)$, such that  
the majorant operators $  |  A(\omega)  | $ 
(see Definition \ref{def:maj})
satisfy, for any $ s_1 \leq s \leq S $ and $ u \in H^{s} $,  
\be\label{CK0-tame}
\sup_{\omega \in {  \mathtt \Omega}}
\| \, | A(\omega)  | u \|_s + 
\gamma
 \sup_{\omega_1, \omega_2 \in {  \mathtt \Omega} \atop \omega_1 \neq \omega_2} \Big\| \, \frac{| A(\omega_1)  -  A(\omega_2) |}{|\omega_1- \omega_2|} u \Big\|_s 
\leq  
{\mathfrak M}_{A}^\sharp (s_1) \| u \|_{s} +
{\mathfrak M}_{A}^\sharp (s) \| u \|_{s_1} \,.
\ee
The constant $ {\mathfrak M}_A^\sharp (s) $ is called a {\sc modulo-tame constant} of the operator $ A $. 
\end{definition}

If $ A $, $B$ are $ \Lipg $-modulo-tame operators, with
$ | A_j^{j'} (\ell) | \leq | B_j^{j'} (\ell) | $, then $ {\mathfrak M}^\sharp_{A}(s) \leq {\mathfrak M}^\sharp_{B}(s) $. 

\begin{lemma}\label{A versus |A|}
An operator $ A $ that is $ \Lipg $-modulo-tame
with modulo-tame constant $ {\mathfrak M}_A^\sharp (s)$ is also 
$ \Lipg $-tame and
$ {\mathfrak M}_A^\sharp (s) $ is a tame constant for $A$. 
\end{lemma}

\begin{proof}
See \cite[Lemma 2.24]{Berti-Montalto}. 
\end{proof}

The class of operators which are $ \Lipg  $-modulo-tame is closed under
sum and composition.

\begin{lemma} \label{interpolazione moduli parametri} {\bf (Sum and composition)}
Let $ A, B $ be $ \Lipg  $-modulo-tame operators with modulo-tame constants respectively 
$ {\mathfrak M}_A^\sharp(s) $ and $ {\mathfrak M}_B^\sharp(s) $. Then 
$ A+ B $ is $ \Lipg  $-modulo-tame with a modulo-tame constant satisfying 
\be\label{modulo-tame-A+B}
{\mathfrak M}_{A + B}^\sharp (s) \leq {\mathfrak M}_A^\sharp (s)  + {\mathfrak M}_B^\sharp (s)  \,.
\ee
The composed operator 
$  A  \circ B $ is $ \Lipg  $-modulo-tame with a modulo-tame constant satisfying 
\begin{equation}\label{modulo tame constant for composition}
 {\mathfrak M}_{A B}^\sharp (s) \leq  
 C \, \big( {\mathfrak M}_{A}^\sharp(s) 
 {\mathfrak M}_{B}^\sharp (s_1) + {\mathfrak M}_{A}^\sharp (s_1) 
{\mathfrak M}_{B}^\sharp (s) \big)
\end{equation}
where $C \ge 1$ is a constant.
Assume in addition that $ \langle \partial_{\vphi} \rangle^{\mathtt b} A $, 
$ \langle \partial_{\vphi} \rangle^{\mathtt b}  B $ (see Definition \ref{def:maj}) are $ \Lipg $-modulo-tame with a modulo-tame constants, 
respectively, $ {\mathfrak M}_{\langle \partial_{\vphi} \rangle^{\mathtt b} A}^\sharp (s) $ and  
 $ {\mathfrak M}_{\langle \partial_{\vphi} \rangle^{\mathtt b} B}^\sharp (s) $. 
Then $ \langle \partial_{\vphi} \rangle^{\mathtt b} (A  B) $ is $ \Lipg $-modulo-tame 
with a modulo-tame constant satisfying, for some  $  C( {\mathtt b} ) \geq 1 $, 
\be
\begin{aligned}\label{K cal A cal B}
{\mathfrak M}_{\langle \partial_{\vphi} \rangle^{\mathtt b} (A  B)}^\sharp (s) & \leq
 C({\mathtt b})\Big( 
{\mathfrak M}_{\langle \partial_{\vphi} \rangle^{\mathtt b} A}^\sharp (s) 
{\mathfrak M}_{B}^\sharp (s_1) + 
{\mathfrak M}_{\langle \partial_{\vphi} \rangle^{\mathtt b} A }^\sharp (s_1) 
{\mathfrak M}_{B}^\sharp (s)  \\ 
& \qquad \qquad \qquad \quad + {\mathfrak M}_{A}^\sharp (s) {\mathfrak M}_{ \langle\partial_{\vphi} \rangle^{\mathtt b} B}^\sharp (s_1) 
+ {\mathfrak M}_{A}^\sharp (s_1) {\mathfrak M}_{ \langle \partial_{\vphi} \rangle^{\mathtt b} B}^\sharp (s)\Big) \, .  
\end{aligned}
\ee
\end{lemma}
\begin{proof}
\noindent
 See \cite[Lemma 2.25]{Berti-Montalto}. 
\end{proof}

Iterating \eqref{modulo tame constant for composition}-\eqref{K cal A cal B} 
we obtain that, for any $n \geq 2$, 
\begin{equation}\label{M Psi n}
{\mathfrak M}_{A^n}^\sharp (s) \leq 
\big( 2 C{\mathfrak M}_{A}^\sharp (s_1) \big)^{n - 1} 
{\mathfrak M}_{A}^\sharp(s)\, , 
\end{equation}
and
\begin{equation}\label{K Psi n}
\begin{aligned}
{\mathfrak M}_{\langle \pa_\vphi \rangle^{ \mathtt b} A^n}^\sharp (s) & \leq  (4C(\mathtt b) C)^{n - 1} \Big( 
{\mathfrak M}^\sharp_{\langle \partial_\vphi \rangle^{\mathtt b} A}(s) \big[ {\mathfrak M}^\sharp_A(s_1) \big]^{n - 1} 
 + {\mathfrak M}^\sharp_{\langle \partial_\vphi \rangle^{\mathtt b} A}(s_1) {\mathfrak M}_A^\sharp(s) 
 \big[ {\mathfrak M}_A^\sharp(s_1) \big]^{n - 2} \Big)\,.
 \end{aligned}
\end{equation}
As an application of \eqref{M Psi n}-\eqref{K Psi n}  we obtain the following 
\begin{lemma}{\bf (Exponential map)}\label{serie di neumann per maggiorantia}
Let $ A $ and 
$ \langle \partial_{\vphi} \rangle^{\mathtt b}  A $ be $ \Lipg $-modulo-tame operators and 
assume that 
${\mathfrak M}_{A}^\sharp : [s_1, S] \to [0, + \infty)$ 
is a modulo-tame constant  satisfying
\begin{equation}\label{piccolezza neumann tamea}
  {\mathfrak M}_{A}^\sharp (s_1)  \leq 1 \, .
\end{equation}
Then the operators $ \Phi^{\pm 1} := {\rm exp}(\pm A) $, 
$ \Phi^{\pm 1} - {\rm Id} $ and $\langle \partial_\vphi \rangle^{\mathtt b}(\Phi^{\pm 1} - {\rm Id})$
 are $\Lipg $-modulo-tame with modulo-tame constants satisfying, for any $s_1 \leq s \leq S 
 $, 
\be \label{exp-MT}
\begin{aligned}
{\mathfrak M}_{ \Phi^{\pm 1} - {\rm Id}}^\sharp(s) & \lesssim  {\mathfrak M}^\sharp_{A}(s)\,, \\
 {\mathfrak M}_{\langle \partial_\vphi \rangle^{\mathtt b} (\Phi^{\pm 1} - {\rm Id})}^\sharp(s) &  \lesssim_{\mathtt b}  {\mathfrak M }_{\langle \partial_\vphi \rangle^{\mathtt b } A}^\sharp(s) + {\mathfrak M}_A^\sharp(s)  {\mathfrak M }_{\langle \partial_\vphi \rangle^{\mathtt b } A}^\sharp(s_1) \,.  \\
\end{aligned}
\ee
\end{lemma}

\begin{proof}
In view of the identity $\Phi^{\pm 1} - {\rm Id} = \sum_{n \geq 1} 
\frac{(\pm A)^n}{n !}$ and the 
assumption \eqref{piccolezza neumann tamea}
the claimed estimates follow by 
\eqref{M Psi n}-\eqref{K Psi n}. 
\end{proof}

\begin{lemma} \label{lemma:smoothing-tame} {\bf (Smoothing)}
Suppose that $ \langle \pa_{\vphi} \rangle^{\mathtt b} A $, $ {\mathtt b} \geq  0 $, is $ \Lipg  $-modulo-tame. Then 
the operator $ \Pi_N^\bot A $ (see Definition \ref{def:maj}) is $ \Lipg  $-modulo-tame with a modulo-tame constant satisfying 
\be\label{proprieta tame proiettori moduli}
{\mathfrak M}_{\Pi_N^\bot A}^\sharp (s) \leq N^{- {\mathtt b} }{\mathfrak M}_{ \langle \pa_{\vphi} \rangle^{\mathtt b} A}^\sharp (s) \, ,
\quad
{\mathfrak M}_{\Pi_N^\bot A}^\sharp (s) \leq  {\mathfrak M}_{ A}^\sharp (s) \, . 
\ee
\end{lemma}

\begin{proof}
See \cite[Lemma 2.27]{Berti-Montalto}. 
\end{proof}
\begin{lemma}\label{Lemma op proiettori}
Let $a_1(\cdot; \omega) , a_2(\cdot; \omega) $ be 
functions in $ {\cal C}^\infty(\T^{\mathbb S_+} 
\times \T_1, \C)$ and $\omega \in   \mathtt \Omega$. Consider the linear operator ${\cal R}$
defined by ${\cal R} h := a_1 \cdot ( a_2, h )_{L^2_x}$, for any
$h \in L^2_x $. Then for any 
$\lambda \in \N^{\mathbb S_+}$ and $n_1, n_2 \geq 0$, the operator $\langle D \rangle^{n_1} \partial_\vphi^\lambda{\cal R} \langle D \rangle^{n_2}$ is $\Lipg$-tame with a tame constant satisfying,
for some 
$\sigma \equiv \sigma(n_1, n_2, \lambda) > 0 $, 
$$
{\mathfrak M}_{\langle D \rangle^{n_1} \partial_\vphi^\lambda{\cal R} \langle D \rangle^{n_2}}(s) \lesssim_{S, n_1, n_2, \lambda} \,
( {\rm max}_{i = 1, 2} \| a_i \|_{s+ \sigma} ) \cdot 
( {\rm max}_{i = 1, 2} \| a_i \|_{s_0+ \sigma} ) \, .
$$ 
\end{lemma}
\begin{proof}
For any $n_1, n_2 \geq 0$, $\lambda \in \N^{\mathbb S_+}$, $h \in L^2_x$, one has 
$$
\begin{aligned}
\langle D \rangle^{n_1} \partial_\vphi^\lambda{\cal R} \langle D \rangle^{n_2} h &= \sum_{\lambda_1 + \lambda_2 = \lambda} c_{\lambda_1, \lambda_2} \langle D \rangle^{n_1}[\partial_\vphi^{\lambda_1} a_1] \, 
\big( \langle D \rangle^{n_2}[\partial_\vphi^{\lambda_2} a_2]\,,\, h \big)_{L^2_x}
\end{aligned}
$$
where we used that the operator $\langle D \rangle$ is symmetric. The lemma 
then follows by \eqref{p1-pr}. 
\end{proof}

\subsection{Tame estimates}\label{sec:2.4}
In this section we record various tame estimates for compositions of functions and 
operators with a torus embedding 
$\breve \io : \T^{\Splus} \to \mathcal E_s$ 
of the form (cf. \eqref{EsEs})
$$
\breve \io (\vphi) = (\vphi,0,0) + \io (\vphi)  \, , \quad \io (\vphi) = ( \Theta (\vphi), y(\vphi), w(\vphi) ) \, , 
$$
with norm $ \|  \iota  \|_s^\Lipg := \| \Theta \|_{H^s_\vphi}^\Lipg +  \| y  \|_{H^s_\vphi}^\Lipg 
+  \| w \|_s^\Lipg  $. 
We shall use that the Sobolev norm in \eqref{unified norm} is equivalent to 
\be\label{eq:norms}
\| \ \|_s = \| \ \|_{H^s_{\vphi,x}} \sim_s \| \ \|_{H^s_\vphi L^2_x} +  \| \ \|_{L^2_\vphi H^s_x} \,  
\ee 
and the interpolation estimate 
(which is a consequence of Young's inequality)
\be\label{int-young}
\| w \|_{H^s_\vphi H^\s_x} \leq 
\| w \|_{H^{s+\s}_\vphi L^2_x } + \| w \|_{L^{2}_\vphi H^{s+\s}_x }  \lesssim_{s,\s} \| w \|_{s + \sigma}\, .
\ee
Given a Banach space $ X $ with norm $ \| \ \|_X $, we consider
the space ${\cal C}^s(\T^{\Splus}, X)$,  $ s \in \N  $,  
of $ {\cal C}^s-$smooth maps $f: \T^{\Splus} \to X $ equipped with the norm 
\be\label{norm-Cs}
\|f \|_{{\cal C}^s_\vphi X} := \sum_{0 \leq |\alpha| \leq s} \| \partial_\vphi^\alpha f \|^{\sup}_X\, ,
\qquad  \| \partial_\vphi^\alpha f \|^{\sup}_X := \sup_{\vphi \in \T^{\Splus}} \| \partial_\vphi^\alpha f(\vphi) \|_X \, .
\ee
By the Sobolev embedding $ \| f \|_{{\cal C}^s_\vphi X} \lesssim_{s_1}  
\| f \|_{H^{s + s_1}_\vphi X} $ for $ s_1 > | {\mathbb S_+}|  $,
whereas if $X$ is a Hilbert space, the latter estimate is valid for
$ s_1 > | {\mathbb S_+}|  / 2 $. 
On  the scale of Banach spaces ${\cal C}^s(\T^{\mathbb S_+}, X)$
the following interpolation inequalities hold: for any $ 0 \leq k \leq s $, 
\begin{equation}\label{interpolazione astratta}
\| f \|_{{\cal C}^k_\vphi X} \lesssim_s \| f \|_{{\cal C}^0_\vphi X}^{1 - \frac{k}{s}} \| f \|_{{\cal C}^s_\vphi X}^{\frac{k}{s}}\,.  
\end{equation}
Recall that $ {\cal E}_s,  E_s  $ are defined in  \eqref{EsEs} 
and $ {\cal V}^s(\delta) $ in \eqref{Vns}.
Let $\Omega$ be an open bounded subset of $\R^{\mathbb S_+} $.
\begin{lemma}\label{lem:tame1}
Let $ \s > 0 $ and assume that, for any $ s \geq 0 $,
the map 
$ a : ({\cal V}^{ \s}(\delta) \cap {\cal E}_{s+\s})  \times \Omega \to H^s (\T_1)  $
is  $ {\cal C}^\infty $ 
with respect to ${\mathfrak x} = (\theta, y, w) $, $ {\cal C}^1 $ with respect to $ \omega $,  
and satisfies for any 
$\mathfrak x \in {\cal V}^{\sigma}(\delta) \cap {\mathcal E}_{s+\s }$, 
$ \a \in \N^{\Splus} $ with $ |\a | \leq 1 $, and $ l \geq 1 $,
the tame estimates 
\begin{equation}\label{ipotesi tame derivate per composizione x}
\begin{aligned}
& \| \partial_\om^\alpha a (\mathfrak x; \om)\|_{H^s_x}  \lesssim_s 1 + \| w \|_{H^{s + \sigma}_x}\,,  \\
& 
\| d^l \partial_\om^\alpha  a ( \mathfrak x; \om)
[\widehat{\mathfrak x}_1, \ldots, \widehat{\mathfrak x}_l]\|_{H^s_x}   \lesssim_{s, l, \alpha} \sum_{j = 1}^l \Big( \| \widehat{\mathfrak x}_j\|_{E_{s + \sigma}} \prod_{n \neq j}  \| \widehat{\mathfrak x}_n\|_{E_{\sigma}} \Big) + \| w \|_{H^{s + \sigma}_x} \prod_{j = 1}^l \|\widehat{\mathfrak x}_j \|_{E_{\sigma}}\, .  
\end{aligned}
\end{equation} 
Then for any $\breve \io  $ with 
$\| \io \|_{s_0 +  \sigma}^\Lipg  \leq \delta $,
the following tame estimates hold for any $ s \geq 0 $:

\noindent
$(i)$  
\begin{equation}\label{stime tame derivate per composizione x}
\begin{aligned}
& \| a(\breve \io)\|_s^{\Lipg} \lesssim_s 1 + \| \io \|_{s + s_0 +  \sigma}^{\Lipg} \, ,  \\
& \| d a (\breve \io)[\widehat\io_1] \|_{s}^{\Lipg}  \lesssim_s \| \widehat\io_1\|_{s + s_0 +
 \sigma}^{\Lipg} + \| \io \|_{s + s_0 + \sigma }^{\Lipg} 
\| \widehat\io_1\|_{s_0 +  \sigma}^{\Lipg} \, ,  \\
& \|d^2 a (\breve \io)[\widehat{\io}_1, \widehat{\io}_2]  \|_{s}^{\Lipg} 
 \lesssim_s \| \widehat{\io}_1\|_{s + s_0 +  \sigma}^{\Lipg} 
 \| \widehat{\io}_2\|_{s_0 + \sigma } + \| \widehat{\io}_1\|_{s_0 +  \sigma}^{\Lipg} 
 \| \widehat{\io}_2\|_{s + s_0 +  \sigma}^{\Lipg}  \\
& \qquad \qquad \qquad \qquad  \quad  + \|\io \|_{s+ s_0 +  \sigma}^{\Lipg}
\| \widehat{\io}_1\|_{s_0 +  \sigma}^{\Lipg} \| \widehat{\io}_2\|_{s_0 +  \sigma}^{\Lipg} \,. 
\end{aligned}
\end{equation} 

\noindent
$(ii)$ If in addition $a(\theta, 0, 0; \om) = 0$, then 
$\| a (\breve \io) \|_{s}^{\Lipg} \lesssim_s 
\| \io \|_{s + s_0 +  \sigma}^{\Lipg} $.

\noindent
$(iii)$ If in addition $a(\theta, 0, 0; \om ) =0,$ 
$ \partial_y a(\theta, 0, 0; \om) = 0$, and
$\partial_w a(\theta, 0, 0; \om) = 0$, then 
$$
\begin{aligned}
& \| a(\breve \io)\|_s^{\Lipg} \lesssim_s  \| \io \|_{s + s_0 +  \sigma}^{\Lipg} 
\| \io \|_{s_0 +  \sigma}^{\Lipg} \,, \\
& \| d a(\breve \io)[\widehat \io]\|_s^{\Lipg} \lesssim_s 
\| \io \|_{s_0 +  \sigma}^{\Lipg} 
\| \widehat \io\|_{s + s_0 +  \sigma}^{\Lipg} + 
\| \io \|_{s + s_0 +  \sigma}^{\Lipg} 
\| \widehat \io\|_{s_0 +  \sigma}^{\Lipg} \,. 
\end{aligned}
$$
\end{lemma}

\begin{proof}
{\sc $(i)$}
It suffices to prove the estimates 
in \eqref{stime tame derivate per composizione x} for 
$ \| d^2 a(\breve \io)[\widehat \io_1, \widehat \io_2]\|_s $ and 
$ \|  d^2 a(\breve \io)[\widehat \io_1, \widehat \io_2] \|_{s}^{\rm lip} $  
since the ones 
for $a(\breve \io)$ and $d a(\breve \io)$ then follow by Taylor expansions. 
By the hypothesis \eqref{ipotesi tame derivate per composizione x} with $ l = 2 $, $ \a = 0 $, 
we have, for any $\vphi \in \T^{\mathbb S_+}$, $ s \geq 0 $,  
\begin{equation}\label{stima Hs x d 2 a A}
\begin{aligned}
\| d^2 a(\breve \io(\vphi))[\widehat \io_1(\vphi), \widehat \io_2(\vphi)]\|_{H^s_x} & \lesssim_s \| \widehat \io_1(\vphi)\|_{E_{s + \sigma}} \| \widehat \io_2(\vphi)\|_{E_{\sigma }} + \| \widehat \io_1(\vphi)\|_{E_{\sigma }} \| \widehat \io_2(\vphi)\|_{E_{s + \sigma }}  \\
& \quad + \| \io(\vphi) \|_{E_{s + \sigma}}\| \widehat \io_1(\vphi)\|_{E_{ \sigma }} \| \widehat \io_2(\vphi)\|_{E_{\sigma }}\,.
\end{aligned}
\end{equation}
Squaring the expressions on the left and right hand side of \eqref{stima Hs x d 2 a A} 
and then integrating them with respect to $\vphi $, one concludes, using \eqref{eq:norms},  \eqref{int-young}, 
and the Sobolev embedding \eqref{SoboX}, that
\begin{equation}\label{stima Hs x d 2 a B}
\begin{aligned}
\| d^2 a(\breve \io)[\widehat \io_1, \widehat \io_2]\|_{L^2_\vphi H^s_x} & \lesssim_s \| \widehat \io_1\|_{s + \sigma} \| \widehat \io_2\|_{s_0 + \sigma} + \| \widehat \io_1\|_{s_0 + \sigma} \| \widehat \io_2\|_{s + \sigma}  
+ \| \io \|_{s + \sigma }\| \widehat \io_1\|_{s_0 + \sigma} \| \widehat \io_2\|_{s_0 + \sigma}\,.
\end{aligned}
\end{equation}
In order to estimate $\| d^2 a(\breve \io)[\widehat \io_1, \widehat \io_2] \|_{H^s_\vphi L^2_x}$, we estimate 
$\| d^2 a(\breve \io)
[\widehat \io_1, \widehat \io_2] \|_{{\cal C}^s_\vphi L^2_x}$. We claim that 
\begin{equation}\label{stima C s vphi operatore composizione d2 a}
\begin{aligned}
\|  d^2 a(\breve \io)[\widehat \io_1, \widehat \io_2]\|_{{\cal C}^s_\vphi L^2_x} & \lesssim_s  
 \| \widehat \io_1\|_{s_0 +  \sigma} \| \widehat \io_2\|_{s + s_0 + \sigma } + \| \widehat \io_1\|_{s + s_0 +  \sigma} \| \widehat \io_2\|_{ s_0 +  \sigma} + \| \io \|_{s + s_0 +  \sigma} 
 \| \widehat \io_1 \|_{s_0 +  \sigma} \| \widehat \io_2\|_{s_0 + \sigma} 
\end{aligned}
\end{equation}
so that the estimate 
for 
$ \| d^2 a(\breve \io) [\widehat \io_1, \widehat \io_2]  \|_s $  
stated in \eqref{stime tame derivate per composizione x} follows by 
\eqref{stima Hs x d 2 a B}, 
\eqref{stima C s vphi operatore composizione d2 a},
and \eqref{eq:norms}.
The bound for 
$ \|  d^2 a(\breve \io)[\widehat \io_1, \widehat \io_2] \|_{s}^{\rm lip} $
is obtained in the same fashion.
\\[1mm]
{\it Proof of \eqref{stima C s vphi operatore composizione d2 a}.} 
By the Leibnitz rule, for any $\beta \in 
\N^{\mathbb S_+} $,  $0 \leq |\beta | \leq s$,
\begin{align}\label{romolo e remo 0}
\partial_\vphi^\beta \Big( d^2 a(\breve \io(\vphi))[\widehat \io_1(\vphi), \widehat \io_2(\vphi)] \Big) & = \sum_{\beta_1 + \beta_2 + \beta_3 = \beta} c_{\beta_1, \beta_2, \beta_3} \partial_\vphi^{\beta_1}(d^2 a(\breve \io(\vphi)))
\big[ \partial_\vphi^{\beta_2} \widehat \io_1(\vphi), \partial_\vphi^{\beta_3} \widehat \io_2(\vphi) \big]  
\end{align}
where $c_{\beta_1, \beta_2, \beta_3}$ are combinatorial constants. 
Each term in the latter sum is estimated individually.
For $ 1 \leq |\b_1| \leq s $ we have 
$$
\begin{aligned}
& \partial_\vphi^{\beta_1} (d^2 a(\breve \io (\vphi) ))
\big[ 
\partial_\vphi^{\beta_2} \widehat \io_1(\vphi), \partial_\vphi^{\beta_3} \widehat\io_2(\vphi)
\big] = \\
&  \sum_{\begin{subarray}{c}
1 \leq m \leq |\beta_1| \\ 
\alpha_1 + \cdots + \alpha_m = \beta_1
\end{subarray}} c_{\a_1, \cdots , \a_m}  d^{m + 2} a(\breve \io(\vphi))
\big[ 
\partial_\vphi^{\alpha_1} \breve \io(\vphi), \cdots, \partial_\vphi^{\alpha_m} \breve \io(\vphi), \partial_\vphi^{\beta_2} \widehat \io_1(\vphi), \partial_\vphi^{\beta_3} \widehat \io_2(\vphi)
\big] 
\end{aligned}
$$
for suitable combinatorial constants $c_{\a_1, \cdots , \a_m} $. 
Then,  by \eqref{ipotesi tame derivate per composizione x} with $ l = m + 2 $, $ \a = 0 $, 
we have the bound
\begin{align}\label{der-beta1}
& \| \partial_\vphi^{\beta_1} (d^2 a(\breve \io))[\partial_\vphi^{\beta_2} \widehat \io_1, \partial_\vphi^{\beta_3} \widehat \io_2]  \|_{{\cal C}^0_\vphi L^2_x} \lesssim_\beta \\ 
& \sum_{\begin{subarray}{c}
1 \leq m \leq |\beta_1| \\
\alpha_1 + \cdots + \alpha_m = \beta_1
\end{subarray}} (1 + \| \io\|_{{\cal C}^{|\alpha_1|}_\vphi E_{ \sigma }}) \cdots (1 +  \|  \io\|_{{\cal C}^{|\alpha_m|}_\vphi E_{ \sigma}}) \| \widehat \io_1\|_{{\cal C}^{|\beta_2|}_\vphi E_{ \sigma }} \| \widehat \io_2\|_{{\cal C}^{|\beta_3|}_\vphi E_{ \sigma}}\,. \nonumber
\end{align}
Arguing as in the proof of the formula (75) in \cite{BKM}, for any 
$ j = 1, \ldots, m $,  we have  
$$
(1 +  \|  \io\|_{{\cal C}^{|\alpha_j|}_\vphi E_{ \sigma }}) \lesssim_\beta (1 +  \|  \io\|_{{\cal C}^{0}_\vphi E_{ \sigma}})^{1 - \frac{|\alpha_j|}{|\beta|}}(1 +  \|  \io\|_{{\cal C}^{|\beta|}_\vphi E_{ \sigma }})^{\frac{|\alpha_j|}{|\beta|}} \, ,
$$ 
and, using 
the interpolation estimate \eqref{interpolazione astratta},  we get 
\begin{align}
& (1 + \|   \io \|_{{\cal C}^{|\a_1|}_\vphi E_{ \sigma }})  \cdots ( 1 + \|  \io \|_{{\cal C}^{|\alpha_m|}_\vphi  E_{  \sigma }})  \| \widehat \io_1\|_{{\cal C}^{|\beta_2|}_\vphi  
E_{ \sigma }}   \| \widehat \io_2\|_{{\cal C}^{|\beta_3|}_\vphi  E_{ \sigma }} \label{Q1s} \\ 
& \lesssim_s \|\widehat \io_1 \|^{1 - \frac{|\beta_2|}{|\beta|}}_{{\cal C}^0_\vphi 
E_{ \sigma}} \| \widehat \io_1\|^{\frac{|\beta_2|}{|\beta|}}_{{\cal C}^{|\beta|}_\vphi 
E_{ \sigma}} \|\widehat \io_2 \|^{1 - \frac{|\beta_3|}{|\beta|}}_{{\cal C}^0_\vphi E_{ \sigma}} \| \widehat \io_2\|^{\frac{|\beta_3|}{|\beta|}}_{{\cal C}^{|\beta|}_\vphi E_{ \sigma}}
\prod_{j = 1}^m
( 1 + \|   \io \|_{{\cal C}^{0}_\vphi E_{ \sigma} })^{1 - \frac{|\a_j|}{|\beta |}}  ( 1 + 
\|   \io \|_{{\cal C}^{|\beta|}_\vphi E_{ \sigma}})^{\frac{|\a_j|}{|\beta |}} 
\nonumber \\
& \lesssim_s \|\widehat \io_1 \|^{1 - \frac{|\beta_2|}{|\beta|}}_{{\cal C}^0_\vphi 
E_{ \sigma}} \| \widehat \io_1\|^{\frac{|\beta_2|}{|\beta|}}_{{\cal C}^{|\beta|}_\vphi 
E_{ \sigma}} \|\widehat \io_2 \|^{1 - \frac{|\beta_3|}{|\beta|}}_{{\cal C}^0_\vphi E_{ \sigma}} \| \widehat \io_2\|^{\frac{|\beta_3|}{|\beta|}}_{{\cal C}^{|\beta|}_\vphi E_{ \sigma}}
( 1 + \|   \io \|_{{\cal C}^{0}_\vphi E_{ \sigma} })^{m - \sum_{j = 1}^m\frac{|\a_j|}{|\beta |}}  ( 1 + 
\|   \io \|_{{\cal C}^{|\beta|}_\vphi E_{ \sigma}})^{\sum_{j = 1}^m\frac{|\a_j|}{|\beta |}} \, . 
\nonumber 
\end{align}
By  \eqref{SoboX}, \eqref{int-young}, 
$(1 + \| \io \|_{{\cal C}^0_\vphi E_\sigma})^{m - 1} \stackrel{}{\lesssim} (1 + \| \io \|_{s_0 + \sigma})^{m - 1} \lesssim (1+ \delta)^{m-1}$  and $\frac{\sum_{j =1}^m |\alpha_j|}{|\beta|} = \frac{|\beta_1|}{|\beta|} = 1 - \frac{|\beta_2|}{|\beta|} - \frac{|\beta_3|}{|\beta|} $, 
so that  
\begin{align}
& 
\label{barbatroll}  \eqref{Q1s} 
\lesssim_s   \|\widehat \io_1 \|^{ \frac{|\beta_1| + |\beta_3|}{|\beta|}}_{{\cal C}^0_\vphi E_{ \sigma}} \| \widehat \io_1\|^{\frac{|\beta_2|}{|\beta|}}_{{\cal C}^{s}_\vphi E_{\sigma}} \|\widehat \io_2 \|^{\frac{|\beta_1| + |\beta_2|}{|\beta|}}_{{\cal C}^0_\vphi E_{ \sigma}} \| \widehat \io_2\|^{\frac{|\beta_3|}{|\beta|}}_{{\cal C}^{s}_\vphi E_{ \sigma}}
( 1 + \|   \io \|_{{\cal C}^{0}_\vphi E_{ \sigma}})^{\frac{|\beta_2| + |\beta_3|}{|\beta|}}  ( 1 + \|   \io \|_{{\cal C}^{s}_\vphi E_{ \sigma}})^{\frac{|\beta_1|}{|\beta |}}  \nonumber\\
& \lesssim_s \Big( \|\widehat \io_1 \|_{{\cal C}^0_\vphi E_{ \sigma}} \|\widehat \io_2 \|_{{\cal C}^0_\vphi E_{ \sigma}} ( 1 + \|   \io \|_{{\cal C}^{s}_\vphi E_{ \sigma}}) \Big)^{\frac{|\beta_1|}{|\beta|}} \Big( \| \widehat \io_1\|_{{\cal C}^{s}_\vphi E_{ \sigma}} \|\widehat \io_2 \|_{{\cal C}^0_\vphi E_{ \sigma}} ( 1 + \|   \io \|_{{\cal C}^{0}_\vphi E_{ \sigma}}) \Big)^{\frac{|\beta_2|}{|\beta|}}  \nonumber\\
& \quad \times \Big( \|\widehat \io_1 \|_{{\cal C}^0_\vphi E_{\sigma }} \|\widehat \io_2 \|_{{\cal C}^s_\vphi E_{ \sigma}} ( 1 + \|   \io \|_{{\cal C}^{0}_\vphi E_{ \sigma}}) \Big)^{\frac{|\beta_3|}{|\beta|}} \nonumber
\end{align}
and, by the iterated Young inequality with exponents $|\beta|/ |\beta_1|$, $|\beta|/ |\beta_2|$, $|\beta|/|\beta_3|$, we conclude that \eqref{Q1s} is bounded by 
\begin{align*} 
 & \|\widehat \io_1 \|_{{\cal C}^0_\vphi E_{ \sigma}} \|\widehat \io_2 \|_{{\cal C}^0_\vphi E_{ \sigma}} ( 1 + \|   \io \|_{{\cal C}^{s}_\vphi E_{ \sigma}}) + \| \widehat \io_1\|_{{\cal C}^{s}_\vphi E_{ \sigma}} \|\widehat \io_2 \|_{{\cal C}^0_\vphi E_{ \sigma}} ( 1 + \|   \io \|_{{\cal C}^{0}_\vphi E_{ \sigma}}) 
 + \|\widehat \io_1 \|_{{\cal C}^0_\vphi E_{ \sigma}} \|\widehat \io_2 \|_{{\cal C}^s_\vphi E_{ \sigma}} ( 1 + \|   \io \|_{{\cal C}^{0}_\vphi E_{ \sigma}}) \\
 & \stackrel{\eqref{SoboX}, \eqref{int-young}} {\lesssim_s}
  \| \io \|_{s + s_0 +  \sigma} \| \widehat \io_1 \|_{s_0 +  \sigma} \| \widehat \io_2\|_{s_0 + \sigma} + 
  + \| \widehat \io_1\|_{s + s_0 + \sigma} \| 
  \widehat \io_2\|_{ s_0 +  \sigma}  
  + \| \widehat \io_1\|_{s_0 +  \sigma} \| 
  \widehat \io_2\|_{s + s_0 + \sigma }
 \,. 
\end{align*}
Note that \eqref{der-beta1} satisfies the same type of bound as \eqref{Q1s}. The term in \eqref{romolo e remo 0} with 
$ \b_1 = 0 $ is estimated in the same way and thus 
\eqref{stima C s vphi operatore composizione d2 a} is proved. 
\\[1mm]
 {\sc Proof $(ii)$-$(iii)$.}   Let $\vphi \mapsto \breve \io (\vphi) = (\theta(\vphi), y(\vphi), w(\vphi))$ be a torus embedding. If $a(\theta, 0, 0) = 0$, we write 
 $$
 a(\breve \io) = \int_0^1 d a(\breve \io_t)[\widehat \io]\, d t \, , \quad 
 \breve \io_t = (1 -t) (\theta(\vphi), 0, 0) + t \breve \io(\vphi) \, ,  \ 
 \widehat \io := (0, y(\vphi), w(\vphi)) \, ,  
 $$
and,  if $a(\theta, 0, 0), \partial_y a(\theta, 0, 0 ), \partial_w a(\theta, 0, 0)$ vanish, we write 
 $$
 a(\breve \io) = \int_0^1 (1 - t) d^2 a(\breve \io_t)[\widehat \io, \widehat \io]\, d t\,. 
 $$
 Items ($ii$)-($iii$)  follow by  item $(i)$, noting that $\| \widehat \io \|_s^\Lipg = \| (0, y(\cdot), w(\cdot))\|_s^\Lipg \lesssim \| \io\|_s^\Lipg $ for any $s \geq 0$. 
\end{proof}
Given $M \in \N$, we define the constant 
\begin{equation}\label{def sM}
\sM := {\rm max}\{ s_0, M + 1 \}\,. 
\end{equation}
\begin{lemma}\label{lem:tame2}
Assume that, for any $ M \geq 0 $, there is $ \s_M \geq 0 $ so that:  
\begin{itemize}
\item{\bf Assumption A.}
For any $ s \ge 
0 $,  the map 
$$ 
{\cal R } : 
({\cal V}^{\sigma_M}(\delta) \cap {\mathcal E}_{s+\s_M}) \times 
 \Omega \to  \mathcal B (H^s (\T_1) , H^{s + M + 1} (\T_1) )
$$ 
is $ {\cal C}^\infty $ with respect to $ \mathfrak x  $,
$ {\cal C}^1 $ with respect to $ \omega $ and,  
for any $\mathfrak x  \in 
{\cal V}^{\sigma_M}(\delta) \cap {\mathcal E}_{s+\s_M} $, 
$ \a \in \N^{\Splus} $ with $ |\a| \leq 1 $,   
$$
 \| \pa_\om^{\alpha} {\cal R} (\mathfrak x; \omega)[\widehat w]\|_{H_x^{s + M + 1}} \lesssim_{s, M}  \| \widehat w \|_{H^s_x} +  \| w \|_{H^{s + \sigma_M}_x}  \| \widehat w\|_{L^2_x}   
  \, , 
$$
and, for any $ l \geq 1 $,  
$ \| d^l \pa_\om^{\alpha}  {\cal R} (\mathfrak x; \omega)[\widehat w][\widehat{\mathfrak x}_1, \ldots, \widehat{\mathfrak x}_l]\|_{H_x^{s + M + 1}} $ is bounded by 
$$
 \lesssim_{s, M, l}  
 \| \widehat w \|_{H^s_x}  \prod_{j=1}^l  \| \widehat{\mathfrak x}_j \|_{E_{\sigma_M}}  +
  \| \widehat w \|_{L^2_x} 
\Big( \| w \|_{H_x^{s+ \s_M}} 
 \prod_{j = 1}^l \|\widehat{\mathfrak x}_j \|_{E_{\sigma_M}} +
\sum_{j = 1}^l \big( \| \widehat{\mathfrak x}_j\|_{E_{s + \sigma_M}} \prod_{n \neq j}  \| \widehat{\mathfrak x}_n\|_{E_{\sigma_M}}\big) \Big)
 \,. 
$$
\item{\bf Assumption B.}
For any $ - M - 1 \leq s \leq 0 $,  the map 
$$ 
{\cal R } : 
{\cal V}^{\sigma_M}(\delta)  \times \Omega 
\to  \mathcal B (H^s (\T_1) , H^{s + M + 1} (\T_1) )
$$ 
is $ {\cal C}^\infty $ w.r to $ \mathfrak x  $,
$ {\cal C}^1 $ with respect to $ \omega $ and,  
for any $\mathfrak x  \in {\cal V}^{\sigma_M}(\delta)$,
$ \a \in \N^{\Splus} $ with $ |\a| \leq 1 $, and $ l \geq 1 $,   
\begin{align*} 
&  \| \pa_\om^{\alpha} {\cal R} (\mathfrak x; \omega)[\widehat w]\|_{H_x^{s + M + 1}} \lesssim_{s, M}  \| \widehat w \|_{H^s_x} 
  \, , \\ 
& 
\| d^l \pa_\om^{\alpha}  {\cal R} (\mathfrak x; \omega)[\widehat w][\widehat{\mathfrak x}_1, \ldots, \widehat{\mathfrak x}_l]\|_{H_x^{s + M + 1}}  \lesssim_{s, M, l}  
\| \widehat w \|_{H^{s}_x} \prod_{j = 1}^l \|\widehat{\frak x}_j \|_{E_{\sigma_M}}\,. 
\end{align*}
\end{itemize}
Then for any $S \ge \sM$ and $ \lambda \in \N^{\mathbb S_+} $, 
there is a constant $\sigma_M(\lambda ) > 0 $, 
so that  for any $  \breve \io(\vphi) = (\vphi, 0, 0) + \io(\vphi)  $ with
$\| \io \|_{s_0 + \sigma_M(\lambda)}^\Lipg \leq \delta $ and  
any $ n_1, n_2 \in \N $ satisfying $ n_1 + n_2 \leq M + 1$, the following holds: 

\noindent
$(i)$ The  operator  $ 
\langle D \rangle^{n_1} \pa_\vphi^\lambda ({\cal R} \circ \breve \io) \langle D \rangle^{n_2} 
$ is $ \Lipg $-tame  with a tame constant satisfying, for any $\sM \leq s \leq S $,
$$
{\mathfrak M}_{\langle D \rangle^{n_1} \pa_{{\vphi}}^\lambda ({\cal R} \circ \breve \io) \langle D \rangle^{n_2}}(s) 
\lesssim_{S, M, \lambda } 1 +  \| \iota \|_{s+ \sigma_M(\lambda)}^\Lipg \,. 
$$
\noindent
$(ii)$ 
The  operator  $ 
\langle D \rangle^{n_1} \pa_{\vphi}^\lambda 
(d {\cal R}( \breve \io)[\widehat \io]) \langle D \rangle^{n_2} 
$ is $ \Lipg $-tame  with a tame constant satisfying, 
for any $	\sM  \leq s \leq S $, 
$$
{\mathfrak M}_{\langle D \rangle^{n_1} \pa_{\vphi}^\lambda (d {\cal R}( \breve \io)[\widehat \io]) \langle D \rangle^{n_2}}(s) 
\lesssim_{S, M, \lambda} \| \widehat \io \|_{s + \sigma_M(\lambda)}^\Lipg +  
\| \iota \|_{s+ \sigma_M(\lambda)}^\Lipg \| \widehat \io\|_{s_0 + \sigma_M(\lambda)}^\Lipg \,. 
$$
\noindent
$(iii)$ If in addition ${\cal R}(\theta, 0, 0; \om) = 0$, then 
the  operator  $ 
\langle D \rangle^{n_1} \pa_{\vphi}^\lambda ({\cal R} \circ \breve \io) \langle D \rangle^{n_2} 
$ is  $ \Lipg $-tame with a tame constant satisfying, for any $ \sM \leq s \leq S $, 
$$
{\mathfrak M}_{\langle D \rangle^{n_1} \pa_{\vphi}^\lambda ({\cal R} \circ \breve \io) \langle D \rangle^{n_2}}(s) 
\lesssim_{S, M, \lambda}   \| \iota \|_{s+ \sigma_M(\lambda)}^\Lipg  \,.
$$
\end{lemma}   
   
 \begin{proof}
 Since item ($i$) and ($ii$) can be proved in a similar way,
 we only prove ($ii$).  
 For any given $n_1, n_2 \in \N$ with $n_1 + n_2 \leq M + 1$, set 
 ${\cal Q} := \langle D \rangle^{n_1} {\cal R} \langle D \rangle^{n_2} $. 
 Assumption A implies that  for any $s \geq M + 1$ and
  any ${\frak x} \in {\cal V}^{\sigma_M}(\delta) 
  \cap {\cal E}_{s + \sigma_M} $,  
  the operator ${\cal Q}(\frak x) $ 
 is in  ${\cal B}(H^s_x)$ and for any 
  $\widehat{\frak x}_1, \ldots, \widehat{\frak x}_l \in E_{s + \sigma_M}$
  with $l \geq 1$, and
 $\widehat w \in H^s_x $, 
 \be\label{Hy2Q}
 \begin{aligned}
 & \|  {\cal Q} (\mathfrak x)[\widehat w]\|_{H_x^{s}} 
  \lesssim_{s, M}  \| \widehat w \|_{H^s_x} +  \| w \|_{H^{s + \sigma_M}_x}  
 \| \widehat w\|_{H^{ M + 1}_x}  \,, \\
 & \|  d^l \big({\cal Q}(\frak x)[\widehat w] \big)[\widehat{\frak x}_1, \ldots, \widehat{\frak x}_l]\|_{H^s_x} \lesssim_{s, M,  l} \| \widehat w\|_{H^s_x} \prod_{j = 1}^l \|\widehat{\frak x}_j \|_{E_{\sigma_M}} \\
 & \quad + 
\| \widehat w\|_{H^{M + 1}_x}  \Big(\| \frak x\|_{E_{s + \sigma_M }}\prod_{j = 1}^l \|\widehat{\frak x}_j \|_{E_{\sigma_M}} + \sum_{j = 1}^l \| \widehat{\frak x}_j\|_{E_{s + \sigma_M }} \prod_{n \neq j} \| \widehat{\frak x}_n \|_{E_{\sigma_M}}\Big) \,.
 \end{aligned}
 \ee
Furthermore  Assumption B implies that, for any 
${\frak x} \in {\cal V}^{\sigma_M}(\delta)$,  the operator 
${\cal Q} (\frak x ) $ is in $ {\cal B}(L^2_x)$ and for any
 $\widehat{\frak x}_1, \ldots, \widehat{\frak x}_l \in E_{ \sigma_M} $,
$l \geq 1$, 
 \be\label{Hy2Qa}
 \begin{aligned}
 & \|  {\cal Q} (\mathfrak x)\|_{{\cal B}(L^2_x)} 
  \lesssim_{M}  1 \,, \qquad  \|  d^l {\cal Q}(\frak x)[\widehat{\frak x}_1, \ldots, \widehat{\frak x}_l]\|_{{\cal B}(L^2_x)} \lesssim_{M,  l} \prod_{j = 1}^l \|\widehat{\frak x}_j \|_{E_{\sigma_M}}\,.  \\
 \end{aligned}
 \ee 
One computes by Leibniz's rule
\begin{equation}\label{formula partial vphi lambda cal Q}
\partial_{\vphi}^\lambda \big(  d {\cal Q}(\breve \io(\vphi))[\widehat \io(\vphi)] \big) = \sum_{\begin{subarray}{c}
0 \leq k \leq |\lambda| \\
\lambda_1 + \ldots + \lambda_{k + 1} = \lambda \\
\end{subarray}} c_{\lambda_1, \ldots, \lambda_{k + 1}} d^{k + 1} {\cal Q}(\breve \io(\vphi))[\partial_{\vphi}^{\lambda_1} \breve \io(\vphi), \ldots, \partial_{\vphi}^{\lambda_k} \breve \io(\vphi), \partial_{\vphi}^{\lambda_{k + 1}} \widehat \io(\vphi)] 
\end{equation}
where $c_{\lambda_1, \ldots, \lambda_{k + 1}} $ are combinatorial constants. 
\\[1mm]
{\sc Estimate of $\| \partial_{\vphi}^\lambda \big(  d {\cal Q}(\breve \io(\vphi))[\widehat \io(\vphi)] \big)[\widehat w] \|_{L^2_\vphi H^s_x}$.}
By  \eqref{Hy2Q}, we have, for $ s \geq M + 1 $,   
\begin{align}\label{melograno 0}
& \| d^{k + 1} {\cal Q}(\breve \io(\vphi))[\partial_{\vphi}^{\lambda_1} \breve \io(\vphi), \ldots, \partial_{\vphi}^{\lambda_k} \breve \io(\vphi), \partial_{\vphi}^{\lambda_{k + 1}} \widehat \io(\vphi)][\widehat w(\vphi)]  \|_{H^s_x}  \\
& \lesssim_{s, M, k} 
\| \widehat w(\vphi) \|_{H^s_x} 
\| \partial_{\vphi}^{\lambda_{k + 1}} \widehat \io(\vphi)\|_{E_{ \s_M}} \prod_{n = 1}^k  \| \partial_{\vphi}^{\lambda_n} \breve \io(\vphi)\|_{E_{\s_M}} \nonumber \\
&    + \| \widehat w(\vphi) \|_{H^{M + 1}_x} \Big( 
\| \io(\vphi) \|_{E_{s + \s_M}} \| \partial_{\vphi}^{\lambda_{k + 1}} \widehat \io(\vphi)\|_{E_{ \s_M}} \prod_{n = 1}^k  \| \partial_{\vphi}^{\lambda_n} \breve \io(\vphi)\|_{E_{\s_M}} \nonumber \\ 
& + \sum_{j = 1}^k \| \partial_{\vphi}^{\lambda_j} \breve \io(\vphi)\|_{E_{s + \s_M}} \big(\prod_{n \neq j}  \| \partial_{\vphi}^{\lambda_n} \breve \io(\vphi)\|_{E_{\s_M}}\big) \|\partial_{\vphi}^{\lambda_{k + 1}} \widehat \io(\vphi) \|_{E_{\s_M}}  
+  \| \partial_{\vphi}^{\lambda_{k + 1}} \widehat \io(\vphi)\|_{E_{s + \s_M}} \prod_{n = 1}^k  \| \partial_{\vphi}^{\lambda_n} \breve \io(\vphi)\|_{E_{\s_M}}  \Big)  \,. \nonumber 
\end{align}
Note that by the Sobolev embedding and  \eqref{int-young},
for any $ s \geq 0$, $\mu \in \N^{\mathbb S_+}$, 
 \begin{equation}\label{melograno - 1}
 \| \partial_{\vphi}^\mu \breve \io(\vphi) \|_{E_{s}} 
 \lesssim 1 + \| \partial_{\vphi}^\mu \io \|_{C^0_\vphi E_{s}} \lesssim 1 + \| \io \|_{s + s_0 + |\mu|} \, .
\end{equation} 
Hence \eqref{formula partial vphi lambda cal Q}-\eqref{melograno 0} and
$\| \cdot \|_{L^2_\vphi H^s_x} \lesssim \| \cdot \|_s $  
imply that for any $\breve \io$ with
$ \| \io \|_{s_0 + \sigma_M(\lambda)}^\Lipg \leq \delta  $  
and any $ s \geq M + 1 $,   
 \begin{equation}\label{melograno 1}
\begin{aligned}
& \| \partial_{\vphi}^\lambda \big(  d {\cal Q}(\breve \io(\vphi))[\widehat \io(\vphi)] \big)[\widehat w(\vphi)]  \|_{L^2_\vphi H^s_x}  \\
& \lesssim_{s,  M, \lambda} 
\| \widehat w\|_{s } \| \widehat \io\|_{s_0 + \sigma_M(\lambda)}  + 
 \| \widehat w\|_{M + 1}  
 \big( \| \io \|_{s + \sigma_M(\lambda)}
 \| \widehat \io\|_{s_0 + \sigma_M(\lambda)} +
  \| \widehat \io\|_{s + \sigma_M(\lambda)} \big) 
\end{aligned}
\end{equation}
for some constant $\sigma_M(\lambda) > 0$. 
\\[1mm]
{\sc Estimate of $\| \partial_{\vphi}^\lambda \big(  d {\cal Q}(\breve \io(\vphi))[\widehat \io(\vphi)] \big)\|_{H^s_\vphi {\cal B}(L^2_x)}$.} 
For any $s \in \N$,  $\beta \in \N^{\mathbb S_+}$, $|\beta| \leq s$, we need to estimate 
$ \| \partial_\vphi^{\beta + \lambda}  \big(  d {\cal Q}(\breve \io(\vphi))[\widehat \io(\vphi)] \big)  \|_{L^2_\vphi {\cal B}(L^2_x)}$. As in \eqref{formula partial vphi lambda cal Q} we have 
\begin{equation}\label{melograno 3}
\begin{aligned}
&  \partial_{\vphi}^{\beta + \lambda} \big(  d {\cal Q}(\breve \io(\vphi))[\widehat \io(\vphi)]   \big)  
 =  \sum_{\begin{subarray}{c}
0 \leq k \leq |\beta| + |\lambda| \\
\alpha_1 + \ldots + \alpha_{k + 1} = \beta + \lambda \\
\end{subarray}} c_{\alpha_1, \ldots, \alpha_{k + 1}} d^{k + 1} {\cal Q}(\breve \io(\vphi))[\partial_{\vphi}^{\alpha_1} \breve \io(\vphi), \ldots, \partial_{\vphi}^{\alpha_k} \breve \io(\vphi), \partial_{\vphi}^{\alpha_{k + 1}} \widehat \io(\vphi)] 
\end{aligned}
\end{equation}
where $c_{\alpha_1, \ldots, \alpha_{k + 1}}  $ 
are combinatorial constants. By \eqref{Hy2Qa} and  \eqref{melograno - 1} one obtains that 
\begin{equation}\label{melograno 4}
\| d^{k + 1} {\cal Q}(\breve \io(\vphi))[\partial_{\vphi}^{\alpha_1} \breve \io(\vphi), \ldots, \partial_{\vphi}^{\alpha_k} \breve \io(\vphi), \partial_{\vphi}^{\alpha_{k + 1}} \widehat \io(\vphi)] \|_{L^2_\vphi {\cal B}(L^2_x)}  \lesssim_{\beta,\lambda} 
 \prod_{j = 1}^{k } (1 + \| \io\|_{|\alpha_j |+ \eta_M }) \| \widehat \io\|_{|\alpha_{k + 1}| + \eta_M} 
\end{equation}
for some $\eta_M > 0$. 
Using the interpolation inequality \eqref{2202.3},  and 
arguing as in the proof of the formula (75) in \cite{BKM}, 
we have, for any $\breve \io$ with 
$\| \io\|_{\eta_M} \leq 1$ and  any $ j = 1, \ldots, k $, 
$$
\begin{aligned}
& 1 + \| \io\|_{|\alpha_j| + \eta_M } \lesssim (1 + \|\io \|_{\eta_M})^{1 - \frac{|\alpha_j|}{|\beta +\lambda|}} (1 + \| \io\|_{|\beta + \lambda| + \eta_M})^{\frac{|\alpha_j|}{|\beta + \lambda|}} \stackrel{\| \io\|_{\eta_M} \leq 1}{\lesssim} (1 + \| \io\|_{|\beta + \lambda| + \eta_M})^{\frac{|\alpha_j|}{|\beta + \lambda|}} \,, \\
& \| \widehat \io\|_{|\alpha_{k + 1}| + \eta_M} \lesssim \| \widehat \io\|_{\eta_M}^{1 - \frac{|\alpha_{k + 1}|}{|\beta + \lambda|}} \| \widehat \io \|_{|\beta + \lambda| + \eta_M}^{\frac{|\alpha_{k + 1}|}{|\beta + \lambda|}} \, . 
\end{aligned}
$$
Then  by  \eqref{melograno 4} and since
$  \sum_{j = 1}^k |\alpha_j|  + |\alpha_{k + 1}|  =  |\beta + \lambda|  $,
it follows that
\begin{align}
& \| d^{k + 1} {\cal Q}(\breve \io(\vphi))[\partial_{\vphi}^{\alpha_1} \breve \io(\vphi), \ldots, \partial_{\vphi}^{\alpha_k} \breve \io(\vphi), \partial_{\vphi}^{\alpha_{k + 1}} \widehat \io(\vphi)] \|_{L^2_\vphi {\cal B}(L^2_x)} \nonumber \\
& \lesssim_{s,\lambda} (1 + \| \io\|_{|\beta + \lambda|+ \eta_M})^{\frac{\sum_{j = 1}^k |\alpha_j|}{|\beta + \lambda|}} 
\| \widehat \io\|_{\eta_M}^{1 - \frac{|\alpha_{k + 1}|}{|\beta + \lambda|}} 
\| \widehat \io \|_{|\beta + \lambda| + \eta_M}^{\frac{|\alpha_{k + 1}|}{|\beta + \lambda|}}  
\nonumber  \\
& \lesssim_{s,\lambda} \Big((1 + \| \io\|_{|\beta + \lambda|+ \eta_M}) \| \widehat \io\|_{\eta_M} \Big)^{\frac{\sum_{j = 1}^k |\alpha_j| }{|\beta | + |\lambda|}}   \| \widehat \io \|_{|\beta + \lambda| + \eta_M}^{\frac{|\alpha_{k + 1}|}{|\beta + \lambda|}} \nonumber \\
& \lesssim_{s,\lambda} \| \widehat \io \|_{|\beta + \lambda| + \eta_M} + \| \io\|_{|\beta + \lambda|+ \eta_M} \| \widehat \io\|_{\eta_M} \label{melograno 8}
\end{align}
where for the latter inequality we used  Young's inequality with exponents $\frac{|\beta +\lambda|}{\sum_{j = 1}^k |\alpha_j|  }, \frac{|\beta +\lambda|}{ |\alpha_{k + 1}|}$.
Combining    \eqref{melograno 3} and
 \eqref{melograno 8} we  obtain
\begin{equation}\label{melograno 9}
\begin{aligned}
& \| \partial_\vphi^\lambda (d {\cal Q}(\breve \io)[\widehat \io])  \|_{H^s_\vphi {\cal B}(L^2_x)} \lesssim_{s, M, \lambda} \| \widehat \io \|_{s + |\lambda| + \eta_M} + \| \io\|_{s + |\lambda|+ \eta_M} \| \widehat \io\|_{\eta_M}\,.  
\end{aligned}
\end{equation}
\noindent
{\sc Estimate of $\| \partial_\vphi^\lambda (d {\cal Q}(\breve \io)[\widehat \io])[\widehat w] \|_{H^s_\vphi L^2_x}$. } 
 Using that 
$$
\Big( \sum_{\ell \in \Z^{\Splus}} 
\| \widehat A (\ell ) \|_{  {\cal B}(L^2_x) }^2 
\langle \ell \rangle^{2s} \Big)^{1/2} \lesssim_{s_0} 
\| A \|_{H^{s+s_0}_\vphi  {\cal B}(L^2_x) }  
$$
one deduces from \cite[Lemma 2.12]{BKM}
that
 for any $\breve \io$ with  $\| \io \|_{2 s_0 + |\lambda| + \eta_M} \leq 1$ and any $s \geq s_0 $,  
\begin{align} \label{melograno 10}
 \| \partial_\vphi^\lambda (d {\cal Q}(\breve \io)[\widehat \io])[\widehat w] \|_{H^s_\vphi L^2_x} & \lesssim_s \| \partial_\vphi^\lambda (d {\cal Q}(\breve \io)[\widehat \io]) \|_{H^{2 s_0}_\vphi {\cal B}(L^2_x)} \| \widehat w \|_{H^{s}_\vphi L^2_x} + \|\partial_\vphi^\lambda (d {\cal Q}(\breve \io)[\widehat \io]) \|_{H^{s + s_0}_\vphi {\cal B}(L^2_x)} \| \widehat w\|_{H^{s_0}_\vphi L^2_x} \\
& \stackrel{\eqref{melograno 9}}{\lesssim_{s, M}}
\| \widehat w\|_s  \| \widehat \io\|_{2 s_0 + |\lambda| + \eta_M}  + 
\| \widehat w\|_{s_0} \big( 
\| \widehat \io \|_{s +  s_0 + |\lambda| + \eta_M}  + \| \io \|_{s + s_0 + |\lambda| + \eta_M} \| \widehat \io \|_{2 s_0 + |\lambda| + \eta_M} \big) \, . \nonumber
\end{align}
Increasing the constant $\sigma_M(\lambda)$ 
in \eqref{melograno 1} if needed, 
one infers from the estimates \eqref{melograno 1}, \eqref{melograno 10}
that for any  $ s \geq \sM = {\rm max}\{ s_0, M + 1\}$, 
$\partial_\vphi^\lambda (d {\cal Q}(\breve \io)[\widehat \io])$ satisfies  
\begin{equation}\label{melograno 100}
\| \partial_\vphi^\lambda (d {\cal Q}(\breve \io)[\widehat \io]) [\widehat w] \|_s 
 \lesssim_{s, M, \lambda} \| \widehat w\|_s  \| \widehat \io\|_{ s_0 + \sigma_M(\lambda)}  + 
\| \widehat w\|_{\sM} \big( 
\| \widehat \io \|_{s +  \sigma_M(\lambda)}  + \| \io \|_{s + \sigma_M(\lambda)} \| \widehat \io \|_{s_0 + \sigma_M(\lambda)} \big) \, .
\end{equation}
Furthermore, arguing similarly, one can show that for any $\omega_1, \omega_2 \in \Omega$, $\omega_1 \neq \omega_2$,  the operator $\partial_\vphi^\lambda \Delta_{\omega}(d {\cal Q}(\breve \io)[\widehat \io]) $ satisfies the estimate, for any $ s \geq \sM $ 
\begin{equation}\label{melograno 101}
\begin{aligned}
\gamma \frac{\| \partial_\vphi^\lambda \Delta_{\omega}(d {\cal Q}(\breve \io)[\widehat \io]) [\widehat w] \|_s}{|\omega_1 - \omega_2|} & \lesssim_{s, M, \lambda} \| \widehat w\|_s  \| \widehat \io\|_{ s_0 + \sigma_M(\lambda)}^\Lipg  
 + \| \widehat w\|_{\sM} \big( \| \widehat \io \|_{s +  \sigma_M(\lambda)}^\Lipg  + \| \io \|_{s + \sigma_M(\lambda)}^\Lipg \| \widehat \io \|_{s_0 + \sigma_M(\lambda)}^\Lipg \big) \,.
\end{aligned}
\end{equation} 
It then follows from \eqref{melograno 100} and \eqref{melograno 101}
that there exists a tame constant 
${\mathfrak M}_{\partial_\vphi^\lambda (d {\cal Q}
(\breve \io)[\widehat \io]) }(s)$ for $\partial_\vphi^\lambda (d {\cal Q}
(\breve \io)[\widehat \io]) $
satisfying the estimate stated in item ($ii$). 

\noindent
{\sc Proof of $(iii)$.} Since ${\cal R}(\theta, 0, 0) = 0$, we can write 
$$
{\cal R}(\breve \io) = \int_0^1 d {\cal R}(\breve \io_t)[\widehat \io]\, d t \, , \quad 
 \breve \io_t = (1-t) (\theta(\vphi), 0, 0) + t \breve \io(\vphi) \, ,
 \quad  \widehat \io(\vphi) := (0, y(\vphi), w(\vphi)) \, . 
$$ 
Since $\| \widehat \io \|_s \lesssim \| \io \|_s$ for any $s \geq 0$, item ($iii$) is thus a direct consequence of $(ii)$. 
 \end{proof} 
 
 \subsection{Egorov type theorems}\label{sezione astratta egorov}
 
 The main purpose of this section is to investigate 
 operators obtained by conjugating a 
 pseudo-differential  operator of the form $ a(\vphi, x) \pa_x^m $, $ m \in \Z $, 
 by the flow map of a transport equation. These results are used
 in Section \ref{elimination x dependence highest order}.

Let $ \Phi (\tau_0, \tau, \vphi) $  denote the flow of 
  the transport equation
\be\label{flow0}
\partial_\tau  \Phi (\tau_0, \tau, \vphi)  = B(\tau, \vphi) \Phi(\tau_0, \tau, \vphi) \, , \quad 
 \Phi(\tau_0, \tau_0, \vphi) = {\rm Id}   \, , 
\ee
 where 
 \begin{equation}\label{definizione b egorov astratto}
\begin{aligned}
B(\tau, \vphi) & := \Pi_{\bot}   \big( b (\tau, \vphi, x) \partial_x + b_x(\tau, \vphi, x) \big) 
\, , \quad 
b\equiv  b (\tau, \vphi, x)  := \frac{  \beta (\vphi, x)  }{ 1 + \tau  \beta_x (\vphi, x)  } \, , 
\end{aligned}
\end{equation}
and the real valued function $ \beta (\vphi, x) \equiv \beta (\vphi, x; \omega) $ is 
$ {\cal C}^\infty$ with respect to the variables
$ (\varphi, x) $ and  Lipschitz with respect to the parameter
$\om \in { \Omega} $.
For brevity  we set $\Phi(\tau, \vphi) := \Phi(0, \tau, \vphi)$
  and $\Phi(\vphi):= \Phi(0, 1, \vphi)$. Note that
  $\Phi(\vphi)^{- 1} = \Phi(1, 0 , \vphi)$ 
  and that  
\begin{equation}\label{proprieta elementare flusso}
\Phi(\tau_0, \tau, \vphi) = \Phi(\tau, \vphi) \circ \Phi(\tau_0, \vphi)^{- 1}\,. 
\end{equation} 
By standard hyperbolic estimates, equation \eqref{flow0} is well-posed. The flow $\Phi(\tau_0, \tau, \vphi) $ has
the following properties. 
 
\begin{lemma}{\bf (Transport flow)} \label{proposition 2.40 unificata}
Let $ \lambda_0 \in \N$, $S > s_0$. 
For any $\lambda\in \N$ with $\lambda \leq \lambda_0 $, 
 $n_1, n_2 \in \R$ with $n_1 + n_2 = - \lambda - 1$, and $s \geq s_0$, 
there exist constants 
$ \sigma(\lambda_0, n_1, n_2) > 0$, $\delta \equiv \delta(S, \lambda_0, n_1, n_2) \in (0, 1)$ such that,  if
\begin{equation}\label{piccolezza a partial vphi beta k D beta k}
\| \beta \|_{s_0 + \sigma(\lambda_0, n_1, n_2)}^\Lipg \leq \delta\,, 
\end{equation}
then for any $m \in \mathbb S_+$, $ \langle D \rangle^{n_1} \partial_{\vphi_m}^\lambda\Phi(\tau_0, \tau, \vphi) \langle D \rangle^{ n_2} $ is a $\Lipg$-tame operator with a tame constant satisfying 
\begin{equation}\label{stima flusso 1 astratto}
\mathfrak M_{\langle D \rangle^{ n_1} \partial_{\vphi_m}^\lambda\Phi(\tau_0, \tau, \vphi) \langle D \rangle^{ n_2}}(s) \lesssim_{S, \lambda_0, n_1, n_2} 1 + \| \beta\|_{s + \sigma(\lambda_0, n_1, n_2)}^\Lipg\,, \qquad \forall s_0 \leq s \leq S\,, \quad \forall \tau_0, \tau \in [0, 1]\,.   
\end{equation}
In addition, if $n_1 + n_2 = - \lambda - 2$, then $\langle D \rangle^{ n_1}  \partial_{\vphi_m}^\lambda ( \Phi(\tau_0, \tau, \vphi) - {\rm Id} ) \langle D \rangle^{n_2}$ is 
 $\Lipg$-tame  with a tame constant satisfying 
\begin{equation}\label{stima flusso astratto 2}
{\mathfrak M}_{\langle D \rangle^{ n_1}  \partial_{\vphi_m}^\lambda(\Phi(\tau_0, \tau, \vphi) - {\rm Id}) \langle D \rangle^{n_2}}(s) \lesssim_{S, \lambda_0, n_1, n_2} \| \beta\|_{s + \sigma(\lambda_0, n_1, n_2)}^\Lipg \, , \qquad \forall s_0 \leq s \leq S \, , 
\quad \forall \tau_0, \tau \in [0, 1]\,. 
\end{equation}
Furthermore, let $s_0 < s_1 < S$, $n_1, n_2 \in \R$, $\lambda_0 \in \N$, $\lambda \leq \lambda_0$ with $n_1 + n_2  =- \lambda - 1$, $m \in \mathbb S_+$. If $\beta_1$ and $\beta_2$ satisfy $\| \beta_i \|_{s_1 + \sigma(n_1, n_2)} \leq \delta$ for some $\sigma(n_1, n_2) > 0 $,  and $\delta \in (0, 1)$ small enough, then 
\begin{equation}\label{stima flusso astratto 2b}
\| \langle D \rangle^{n_1} \partial_{\vphi_m}^\lambda \Delta_{12} \Phi(\tau_0, \tau, \vphi) \langle D \rangle^{n_2} \|_{{\cal B}(H^{s_1})} \lesssim_{s_1, \lambda_0, n_1, n_2} \| \Delta_{12} \beta\|_{s_1 + \sigma(n_1, n_2)}\,, \quad \tau_0, \tau \in [0, 1] \, , 
\end{equation}
where $\Delta_{12} \beta:= \beta_2 - \beta_1$ and 
$\Delta_{12} \Phi(\tau_0, \tau, \vphi):= 
\Phi(\tau_0, \tau, \vphi; \beta_2) - \Phi(\tau_0, \tau, \vphi; \beta_1) $. 
\end{lemma}

\begin{proof}
The proof of  \eqref{stima flusso 1 astratto} is similar to the one of Propositions A.7, A.10 and A.11 in \cite{Berti-Montalto}.
In comparison to the latter results the main difference is
that the vector field 
\eqref{definizione b egorov astratto} is of order $1$, whereas
the vector field considered in \cite{Berti-Montalto} is of order $\frac12$. 
Using \eqref{stima flusso 1 astratto} we now prove  \eqref{stima flusso astratto 2}. By \eqref{flow0}, one has that 
$$
\Phi(\tau_0, \tau, \vphi) - {\rm Id} = \int_{\tau_0}^\tau B(t, \vphi) \Phi(\tau_0, t, \vphi)\, d t\,.
$$
 Then, for any $\lambda \in \N$ with $\lambda \leq \lambda_0$
 and any $n_1, n_2 \in \R$ with $n_1 + n_2 =- \lambda  - 2$,  one has by Leibniz' rule
$$
\begin{aligned}
& \langle D \rangle^{n_1} \partial_{\vphi_m}^\lambda(\Phi(\tau_0, \tau, \vphi) - {\rm Id}) \langle D \rangle^{n_2} \\
&  = \sum_{\lambda_1 + \lambda_2 = \lambda} c_{\lambda_1, \lambda_2} \int_{\tau_0}^\tau \big( \langle D \rangle^{n_1}\partial_{\vphi_m}^{\lambda_1}B(t, \vphi) \langle D \rangle^{ n_2 + \lambda_2 + 1} \big) \big(  \langle D \rangle^{- n_2  - \lambda_2 - 1} \partial_{\vphi_m}^{\lambda_2}\Phi(\tau_0, t, \vphi) \langle D \rangle^{n_2} \big)\, d t \\
& = \sum_{\lambda_1 + \lambda_2 = \lambda} c_{\lambda_1, \lambda_2} \int_{\tau_0}^\tau \big( \langle D \rangle^{n_1}\partial_{\vphi_m}^{\lambda_1}B(t, \vphi) \langle D \rangle^{- 1 - n_1 - \lambda_1} \big) \big(  \langle D \rangle^{- n_2  - \lambda_2 - 1} \partial_{\vphi_m}^{\lambda_2}\Phi(\tau_0, t, \vphi) \langle D \rangle^{n_2} \big)\, d t
\end{aligned}
$$
where $c_{\lambda_1, \lambda_2}$ are combinatorial constants and we used that
$ n_2 + \lambda_2 + 1 =  -1 - n_1 - \lambda_1 $.
Recalling the definition \eqref{definizione b egorov astratto}
of  $B$, using Lemmata \ref{lemma stime Ck parametri}, \ref{lemma: action Sobolev}, \ref{lemma flusso ODE}-$(i)$, and  \eqref{stima flusso 1 astratto}, one has that for any 
$s \geq s_0$,
\begin{equation}\label{stima campo B nel lemma}
\begin{aligned}
\mathfrak M_{\langle D \rangle^{n_1} \partial_{\vphi_m}^{\lambda_1}B \langle D \rangle^{- 1 - n_1 - \lambda_1}}(s) & \lesssim_s |\langle D \rangle^{n_1} B \langle D \rangle^{- 1 - n_1 - \lambda_1}|_{0, s + \lambda_1, 0}^\Lipg \lesssim_{s, \lambda_1,  n_1} \| \beta\|_{s + \sigma(\lambda_1, n_1)}^\Lipg\,, \\
\mathfrak M_{ \langle D \rangle^{- 1- n_2  - \lambda_2} \partial_{\vphi_m}^{\lambda_2}\Phi(\tau_0, t, \vphi) \langle D \rangle^{n_2}}(s)& \lesssim_{s, \lambda_2,  n_1, n_2} 1 + \|\beta \|_{s + \sigma(\lambda_2, n_1, n_2)}^\Lipg \, . 
\end{aligned}
\end{equation}
Then \eqref{stima flusso astratto 2} 
follows by \eqref{stima campo B nel lemma}, 
Lemma \ref{composizione operatori tame AB} and 
 \eqref{piccolezza a partial vphi beta k D beta k}. The estimate \eqref{stima flusso astratto 2b} follows by similar arguments. 
\end{proof}

For what follows we need to study
the solutions of the characteristic ODE  
$\partial_\tau x = - b(\tau,\vphi, x)$
associated to the transport operator defined in 
\eqref{definizione b egorov astratto}.
\begin{lemma} {\bf (Characteristic flow)}
The characteristic flow $\gamma^{\tau_0, \tau}(\vphi, x) $ 
defined by 
\begin{equation}\label{flusso caratteristiche egorov 0}
\partial_\tau \gamma^{\tau_0, \tau}(\vphi, x) = - 
b\big(\tau, \vphi, \gamma^{\tau_0, \tau}(\vphi, x) \big) \, , \quad \gamma^{\tau_0, \tau_0}(\vphi, x) = x\, , 
\end{equation}
is given by 
\begin{equation}\label{flow-ex}
\gamma^{\tau_0, \tau} (\varphi, x) = x + \tau_0 \beta (\vphi, x) +
\breve \beta (\tau, \varphi, x + \tau_0 \beta (\vphi, x)) \, ,  
\end{equation}
where
$ y \mapsto  y + \breve \beta (\tau, \vphi, y) $ is the inverse diffeomorphism of
$ x \mapsto x + \tau \beta (\vphi, x) $. 
\end{lemma}

\begin{proof}
A direct computation proves that $ \gamma^{0,\tau} (y) = y + \breve \beta (\tau, \vphi, y) $ 
and therefore $ \gamma^{\tau,0} (x) = x + \tau \beta ( \vphi, x) $. By the composition rule
of the flow  $ \gamma^{\tau_0, \tau} = \gamma^{0, \tau} \circ \gamma^{\tau_0, 0} $  we deduce \eqref{flow-ex}. 
\end{proof}

\begin{lemma}\label{lemma flusso ODE}
There are $ \sigma, \d  > 0 $ such that,
if  $ \| \beta \|_{s_0 + \sigma}^\Lipg \leq \delta $, then 
 
\noindent
$(i)$ $\| b \|_s^\Lipg \lesssim_s \| \beta\|_{s + \sigma}^\Lipg $ for any $s \geq s_0$. 

\noindent
$(ii)$ For any $\tau_0, \tau \in [0, 1]$, $s \geq s_0 $, we have  
$\|  \gamma^{\tau_0, \tau}(\vphi, x) - x \|_s^\Lipg 
\lesssim_s \| \beta\|_{s +\sigma}^\Lipg$. 

\noindent
$(iii)$ Let $s_1 > s_0$ and assume that 
$\| \beta_j \|_{s_1 + \s} \le \delta,$ $j= 1,2$. 
	Then $ \Delta_{12} b := b(\cdot ; \beta_2) - b(\cdot; \beta_1)$ and 
        $\Delta_{12} \gamma^{\tau_0, \tau}:=  \gamma^{\tau_0, \tau}(\cdot; \beta_2) -  \gamma^{\tau_0, \tau}(\cdot; \beta_1)$  
can be estimated in terms of  $\Delta_{12} \beta:= \beta_2 - \beta_1$ as
$$
\| \Delta_{12} b\|_{s_1} \lesssim_{s_1} \| \Delta_{12} \beta\|_{s_1 + \sigma}
\, , \qquad
\| \Delta_{12} \gamma^{\tau_0, \tau} \|_{s_1} \lesssim_{s_1} \| \Delta_{12} \beta\|_{s_1 + \sigma}\, . 
$$
\end{lemma}
\begin{proof}
Item ($i$) follows by the definition of $b$ in \eqref{definizione b egorov astratto}
and Lemma \ref{Moser norme pesate}.
Item ($ii$) follows by \eqref{flow-ex} and 
Lemma \ref{lemma:LS norms}. Item $(iii)$ follows by similar arguments. 
\end{proof}

Now we prove the following Egorov type theorem, saying that the operator, obtained by conjugating
$  a(\vphi, x) \partial_x^m $, $m \in \Z$, 
with the time one flow $\Phi(\varphi) = \Phi(0, 1, \vphi)  $ of the transport equation \eqref{flow0}, remains a
pseudo-differential operator with a homogenous asymptotic expansion.  

\begin{proposition}\label{proposizione astratta egorov}
{\bf (Egorov)}
Let $N, \lambda_0 \in \N$, $S > s_0$ 
and assume that $\beta(\cdot; \omega), a(\cdot; \omega)$ are in ${\cal C}^\infty(\T^{\mathbb S_+} \times \T_1)$ and Lipschitz continuous with respect 
to $\omega \in  \Omega $. 
Then there exist constants $\sigma_N(\lambda_0),$ $\sigma_N > 0$, 
$\delta(S, N, \lambda_0) \in (0, 1)$, and $C_0 > 0$ such that, if 
\begin{equation}\label{smallness egorov astratto}
\| \beta\|_{s_0 + \sigma_N(\lambda_0)}^\Lipg \leq \delta \, , \quad \| a  \|_{s_0 + \sigma_N(\lambda_0)}^\Lipg \leq C_0 \, , 
\end{equation}
then the conjugated operator 
$$
{\cal P}(\vphi) := \Phi(\vphi) {\cal P}_0 (\vphi)\Phi(\vphi)^{- 1} \, , \qquad 
{\cal P}_0 :=  a(\vphi, x; \om) \partial_x^m\, , \quad m \in \Z\, ,
$$ 
is a pseudo-differential operator of order $m$ with an expansion of the form 
\be\label{expPN}
{\cal P}(\vphi) = \sum_{i = 0}^N p_{m - i} (\vphi, x; \om ) \partial_x^{m - i} + {\cal R}_N(\vphi)
\ee
with the following properties:
\begin{enumerate}
\item
The principal symbol $p_m$ of $\cal P$ is given by  
\begin{equation}\label{ordine principale esplicito egorov}
p_m(\vphi, x; \om) =  \Big( [1 + \breve \beta_y (\vphi, y; \om )]^m a(\vphi, y; \om) \Big)|_{y = x +  \beta(\vphi, x; \om)}  
\end{equation}
where 
$ y \mapsto y + \breve \beta(\vphi, y; \om ) $  denotes the inverse 
diffeomorphism of $x \mapsto x + \beta(\vphi, x; \om ) $.
\item 
For any $s \geq s_0$ and $i = 1, \ldots, N $,   
\begin{equation}\label{stima-fun}
\| p_m - a\|_s^\Lipg \, , \ \| p_{m - i}\|_s^\Lipg \lesssim_{s, N} \| \beta\|_{s + \sigma_N}^\Lipg + \| a \|_{s + \sigma_N}^\Lipg \| \beta\|_{s_0 + \sigma_N}^\Lipg\, .
\end{equation}
\item 
For any $\lambda \in \N$ with $\lambda \leq \lambda_0$, 
 $n_1, n_2 \in \N$  with $n_1 + n_2 + \lambda_0 \leq N - 1 - m  $, $k \in \mathbb S_+$,  the pseudo-differential operator $\langle D \rangle^{n_1}\partial_{\vphi_k}^\lambda {\cal R}_N(\vphi) \langle D \rangle^{n_2}$ is 
$\Lipg$-tame with a tame constant satisfying, for any $s_0 \leq s \leq S $,   
\begin{equation}\label{stima resto Egorov teo astratto}
{\mathfrak M}_{\langle D \rangle^{n_1}\partial_{\vphi_k}^\lambda {\cal R}_N(\vphi) \langle D \rangle^{n_2}}(s) \lesssim_{S, N, \lambda_0} \| \beta\|_{s + \sigma_N(\lambda_0)}^\Lipg + \| a \|_{s + \sigma_N(\lambda_0)}^\Lipg \| \beta\|_{s_0 + \sigma_N(\lambda_0)}^\Lipg\,. 
\end{equation}
\item Let $s_0 < s_1 $ 
and assume that  $\| \beta_j \|_{s_1 + \sigma_N(\lambda_0)} \leq \delta,$ 
$\| a_j\|_{s_1 + \sigma_N(\lambda_0)} \leq C_0$, $j = 1,2$. Then 
$$
\begin{aligned}
\| \Delta_{12} p_{m - i} \|_{s_1} \lesssim_{s_1, N} \| \Delta_{12} a\|_{s_1 + \sigma_N} + \| \Delta_{12} \beta\|_{s_1 + \sigma_N}, \quad i = 0, \ldots, N \, , 
\end{aligned}
$$
and, for any $\lambda \leq \lambda_0$,  $n_1, n_2 \in \N$ with $n_1 + n_2 + \lambda_0 \leq N - 1 - m$, and  $k \in \mathbb S_+$, 
$$
\| \langle D \rangle^{n_1}\partial_{\vphi_k}^\lambda \Delta_{12} {\cal R}_N(\vphi) \langle D \rangle^{n_2} \|_{{\cal B}(H^{s_1})} \lesssim_{s_1, N, n_1, n_2} \| \Delta_{12} a\|_{s_1 + \sigma_N(\lambda_0)} + \| \Delta_{12} \beta\|_{s_1 + \sigma_N(\lambda_0)}
$$
where we refer to Lemma \ref{proposition 2.40 unificata}
for the meaning of  $\Delta_{12}$.
\end{enumerate}
\end{proposition}

\begin{proof}
The orthogonal projector $ \Pi_\bot $ is a Fourier multiplier of order $0$,
$ \Pi_\bot = {\rm Op} ( \chi_\bot (\xi) ) $, 
where $\chi_\bot$ is a $ {\cal C}^\infty(\R, \R)$ cut-off function which is equal to $ 1 $ on a 
neighborhood of $ \Sbot $ and vanishes 
in a  neighborhood of $ {\mathbb S} \cup \{0\} $. Then 
we decompose  the operator 
$B(\tau, \vphi) = \Pi_\bot (b(\tau, \vphi, x) \pa_x + b_x (\tau, \vphi, x))$ as
\begin{equation}\label{decomposizione A}
\begin{aligned}
& B(\tau, \vphi)  
= B_1(\tau, \vphi) + B_\infty(\tau, \vphi)\,,  \\
& B_1(\tau, \vphi) :=  
b(\tau, \vphi, x) \pa_x + b_x (\tau, \vphi, x)\,, \quad 
B_\infty(\tau, \vphi) := {\rm Op}(b_\infty(\tau, \vphi, x, \xi)) \in OPS^{- \infty}
\end{aligned}
\end{equation}
where  for some $\sigma > 0$,  $B_\infty$ satisfies, 
for any $s, m  \geq 0$ and $\alpha \in \N $,  the estimate 
\begin{equation}\label{stima a infty}
|B_\infty|_{- m, s, \alpha}^\Lipg \lesssim_{m, s, \alpha} \| \beta\|_{s + \sigma}^\Lipg \, .
\end{equation}
The conjugated operator 
${\cal P}(\tau, \vphi) := \Phi(\tau, \vphi) {\cal P}_0(\vphi) \Phi(\tau, \vphi)^{- 1}$ solves 
the Heisenberg equation 
\begin{equation}\label{Heisenberg sezione Egorov}
\partial_\tau {\cal P}(\tau, \vphi) =  [B(\tau, \vphi), {\cal P}(\tau, \vphi)]\, ,  \quad 
{\cal P}(0, \vphi) =  {\cal P}_0(\vphi) = 
a(\varphi, x; \omega) \partial_x^m  \, .
\end{equation}
We look for an approximate solution of \eqref{Heisenberg sezione Egorov}  of the form 
\begin{equation}\label{definizione cal PN tau}
{\cal P}_N(\tau, \vphi) := \sum_{i = 0}^N p_{m - i}(\tau, \vphi, x)  \partial_x^{m - i}  
\end{equation}
for suitable functions $ p_{m - i}(\tau, \vphi, x)$ to be determined. 
By \eqref{decomposizione A}   
\begin{equation}\label{sardine 0}
[B(\tau, \vphi), {\cal P}_N(\tau, \vphi)] = [B_1(\tau, \vphi), {\cal P}_N(\tau, \vphi)] + [B_\infty(\tau, \vphi), {\cal P}_N(\tau, \vphi)] 
\end{equation}
where  $[ B_\infty(\tau, \vphi), {\cal P}_N(\tau, \vphi)] $ is in $ OPS^{- \infty}$, and 
$$
[B_1(\tau, \vphi), {\cal P}_N(\tau, \vphi)] = \sum_{i = 0}^N 
\big[ b \partial_x + b_x  \,,\, p_{m - i} \partial_x^{m - i} \big]\,.
$$
By  Lemma \ref{composizione simboli omogenei}, one has for any $ i = 0, \ldots, N$, 
$$
\big[ b  \partial_x + b_x, p_{m - i} \partial_x^{m - i} \big] = 
\big( b (p_{m - i})_x - (m - i) b_x p_{m - i} \big) \partial_x^{m - i} + \sum_{j = 1}^{N - i} g_{j}(b, p_{m -i}) \partial_x^{m - i - j} + {\cal R}_N(b, p_{m - i}) \,
$$
where the functions $g_j(b, p_{m - i}) := g_j(b, p_{m - i})(\tau, \vphi, x)$, 
$j = 0, \ldots, N - i$, and the remainders ${\cal R}_N(b, p_{m - i})$ 
can be estimated as follows:
there exists $\sigma_N := \sigma_N(m) > 0$ so that for any $s \geq s_0$,
 (cf. Lemma \ref{lemma flusso ODE}-$(i)$)
\begin{equation}\label{gj a p m - i formula}
\| g_j (b, p_{m - i}) \|_s^\Lipg \lesssim_{m, N, s}  \| \beta \|_{s + \sigma_N}^\Lipg \| p_{m - i}\|_{s_0 + \sigma_N}^\Lipg +  
\| \beta \|_{s_0 + \sigma_N}^\Lipg \| p_{m - i}\|_{s + \sigma_N}^\Lipg \, ,  
\end{equation} 
 and for any $s \geq s_0$ and $\alpha \in \N $ 
 (cf. Lemma \ref{composizione simboli omogenei}-($ii$))
\begin{equation}\label{stima cal RN b p m - i}
|{\cal R}_N(b, p_{m - i})|_{m - N - 1, s, \alpha}^\Lipg \lesssim_{m, N, s, \alpha} \| \beta \|_{s + \sigma_N}^\Lipg \| p_{m - i}\|_{s_0 + \sigma_N}^\Lipg + \| \beta \|_{s_0 + \sigma_N}^\Lipg \| p_{m - i}\|_{s + \sigma_N}^\Lipg\,. 
\end{equation}
Adding up the expansions for $\big[ b  \partial_x + b_x, p_{m - i} \partial_x^{m - i} \big]$, $0 \le i \le N$, yields
\begin{align}
\big[ B_1(\tau, \vphi), {\cal P}_N(\tau, \vphi) \big] 
& = \sum_{i = 0}^N \big( b (p_{m - i})_x - (m - i) b_x p_{m - i} \big)\partial_x^{m - i} + 
\sum_{i = 0}^N \sum_{j = 1}^{N - i}  g_{j}(b, p_{m -i}) \partial_x^{m - i - j} + 
\sum_{i = 0}^N {\cal R}_N(b, p_{m - i}) \nonumber\\
& =  \sum_{i = 0}^N 
\big( b (p_{m - i})_x - (m - i) b_x p_{m - i} \big) \partial_x^{m - i} + \sum_{k = 1}^N \sum_{j = 1}^k g_j (b, p_{m - k + j}) \partial_x^{m - k } + \sum_{i = 0}^N {\cal R}_N(b, p_{m - i}) \nonumber\\ 
& = \big( b (p_m)_x - m b_x p_m \big)  \partial_x^m + \sum_{i = 1}^N 
\big( b (p_{m - i})_x - (m - i) b_x p_{m - i} + 
\widetilde g_{ i} \big)  \partial_x^{m - i} + {\cal Q}_N \label{calcolo commutatore cal PN tau}
\end{align}
where,  for any $i = 1, \ldots, N$, 
$ \widetilde g_{ i} := \sum_{j= 1}^i g_j(b, p_{m - i + j}) $ and 
$ {\cal Q}_N := \sum_{i = 0}^N {\cal R}_N(b, p_{m - i}) 
\in OPS^{m - N - 1} $. 
Defining for any $s \geq 0$, 
\begin{equation}\label{M s p m - i}
\mathtt M_{< i}(s)  := {\rm max}\{ \| p_{m - k}\|_s^\Lipg,  k = 0, \ldots, i - 1  \} \, ,
\ \mathtt M(s)  := {\rm max} \{ \| p_{m - i}\|_s^\Lipg,  i = 0, \ldots, N \} \, ,
\end{equation}
we deduce from \eqref{gj a p m - i formula} and \eqref{stima cal RN b p m - i} 
that
for any $s \geq s_0$, $\alpha \in \N$, $i = 0, \ldots, N$,  
\begin{equation}\label{stima g tilde i cal QN egorov}
\begin{aligned}
& \| \widetilde g_i\|_s^\Lipg \lesssim_{s, N} \mathtt M_{< i}(s + \sigma_N) \| \beta \|_{s_0 + \sigma_N}^\Lipg + \mathtt M_{< i}(s_0 + \sigma_N) \| \beta \|_{s + \sigma_N}^\Lipg  \\
&  |{\cal Q}_N|_{m - N - 1, s, \alpha}^\Lipg \lesssim_{s, N} \mathtt M(s + \sigma_N) \| \beta \|_{s_0 + \sigma_N}^\Lipg + \mathtt M(s_0 + \sigma_N) \| \beta \|_{s + \sigma_N}^\Lipg\, .  
\end{aligned}
\end{equation}
By \eqref{definizione cal PN tau}, \eqref{sardine 0}, and 
 \eqref{calcolo commutatore cal PN tau}   
the operator ${\cal P}_N(\tau, \vphi)$ solves the {\it approximated} Heisenberg equation 
$$
\partial_\tau {\cal P}_N(\tau, \vphi) = [B(\tau, \vphi), {\cal P}_N(\tau, \vphi)] + OPS^{m - N - 1} \, , 
$$
if the functions $p_{m - i}$ solve the transport equations  
\begin{equation}\label{acqua bella 0}
\begin{aligned}
& \partial_\tau p_m = b  (p_m)_x - m b_x p_m\,, \\
& \partial_\tau p_{m - i} = b  (p_{m - i})_x - (m - i) b_x p_{m - i} + 
\widetilde g_{ i} \, , \quad i = 1, \ldots, N\,. 
\end{aligned}
\end{equation}
Note that, since $\widetilde g_i$ only 
depends on $p_{m - i + 1}, \ldots, p_{m }$, 
we can solve \eqref{acqua bella 0} inductively. 
\\[1mm]
{\sc Determination of $p_m$.}  We 
solve the first equation in \eqref{acqua bella 0},
$$
\partial_\tau p_m(\tau, \vphi, x) = b(\tau, \vphi, x) \partial_x p_m (\tau, \vphi, x) - m b_x(\tau, \vphi,  x) p_m(\tau, \vphi,  x) \, , \quad p_m(0, \vphi,  x) = a(\vphi, x)\,. 
$$
By  the  method of characteristics we deduce that 
\begin{equation}\label{ninja 0}
p_m(\tau, \vphi, \gamma^{0, \tau}(\vphi, x)) = {\rm exp}\Big( - m \int_0^\tau b_x(t, \vphi,  \gamma^{0, t}(\vphi, x))\, d t \Big) a(\vphi, x)
\end{equation}
where $\gamma^{0, \tau}(\vphi, x)$ is given by \eqref{flow-ex}.
Differentiating the equation \eqref{flusso caratteristiche egorov 0} with respect to the initial datum $x$, we get 
$$
\partial_\tau (\partial_x \gamma^{\tau_0, \tau}(x)) = - b_x (\tau, \vphi, \gamma^{\tau_0, \tau}(x)) \partial_x \gamma^{\tau_0, \tau}(x)\,, \quad \partial_x \gamma^{\tau_0, \tau_0}(x) = 1 \, , 
$$
implying that 
\begin{equation}\label{ninja 1}
\partial_x \gamma^{\tau_0, \tau}(\vphi, x)= {\rm exp}\Big(- \int_{\tau_0}^\tau b_x(t, \vphi,  \gamma^{\tau_0, t}(\vphi, x))\, d t \Big)\,.
\end{equation}
From \eqref{ninja 0} and \eqref{ninja 1} we infer that
\begin{equation}\label{ninja 2}
p_m(\tau,\vphi,  y) = \Big( [\partial_x \gamma^{0, \tau}(\vphi, x)]^m a(\vphi, x) \Big)|_{x = \gamma^{\tau, 0}(\vphi, y)}\,. 
\end{equation}
Evaluating the latter identity at $\tau = 1$ and using \eqref{flow-ex}, we obtain 
 \eqref{ordine principale esplicito egorov}. 
\\[1mm]
{\sc Inductive determination of $p_{m - i}$.} 
For $i = 1, \ldots, N $, we solve the inhomogeneous transport equation, 
$$
\partial_\tau p_{m - i} = b \partial_x p_{m - i} - (m - i) 
b_x p_{m - i} + \widetilde g_{ i} \, , \qquad p_{m - i}(0, \vphi,  x) = 0\,. 
$$
By the  method of characteristics one has 
\begin{equation}\label{acqua bella 4}
p_{m - i}(\tau , \vphi,  y) =  \int_0^\tau {\rm exp}\Big(- (m - i) \int_t^\tau b_x(s,\vphi,  \gamma^{\tau , s}(\vphi, y))\, d s  \Big) 
\widetilde g_{ i}(t ,\vphi,  \gamma^{\tau , t}(\vphi, y))\, d t\,.
\end{equation}
The functions  $ p_{m - i}(\vphi, y) $ in the expansion
 \eqref{expPN} are
then given by $ p_{m - i}(\vphi, y) := p_{m -i}(1,\vphi,  y)$.

\begin{lemma}\label{lem:claim}
There are $\sigma_N^{(N)} > \sigma_N^{(N - 1)} > \ldots > \sigma_N^{(0)} > 0$ such that, for any $i \in \{ 1, \ldots, N \}$, $ \tau \in [0,1] $, $s \geq s_0 $, 
\begin{equation}\label{stima p m - i}
\begin{aligned}
\| p_m(\tau, \cdot) - a\|_s^\Lipg & \lesssim_s \| \beta\|_{s + \sigma_N^{(0)}}^\Lipg +    \| a \|_{s + \sigma_N^{(0)}}^\Lipg \| \beta\|_{s_0 + \sigma_N^{(0)}}^\Lipg    \,, \\
\| p_{m - i}(\tau, \cdot)\|_s^\Lipg & \lesssim_s \| \beta \|_{s + \sigma_N^{(i)}}^\Lipg  +  \| a \|_{s + \sigma_N^{(i)}}^\Lipg \| \beta \|_{s_0 + \sigma_N^{(i)}} \,. 
\end{aligned}
\end{equation}
\end{lemma}

\begin{proof} 
We argue by induction. First we prove the claimed estimate for 
$p_m - a$ with $p_m$ given by \eqref{ninja 2}.
Recall that 
$\gamma^{0, \tau}(\vphi, x) = x + \breve \beta (\tau, \vphi, x) $
and $\gamma^{\tau, 0}(\vphi, y) = y + \tau \beta ( \vphi, y)$ 
(cf. \eqref{flow-ex}). Since
$a( \varphi, y + \tau \beta(\varphi, y) )  - a(\varphi, y) 
= \int_0^\tau a_x(\varphi, y + t \beta(\varphi, y)) \beta(\varphi, y) dt$,
the claimed estimate for $p_m$ then follows by
Lemmata \ref{lemma:LS norms}, \ref{lemma flusso ODE} and 
assumption \eqref{smallness egorov astratto}.
Now assume that for any $k \in \{ 1, \ldots, i - 1 \}$, $1 \le i \le N$,
the function $p_{m - k}$, given by  \eqref{acqua bella 4}, satisfies the estimates  \eqref{stima p m - i}. 
The ones for $p_{m - i}$ then follow  by   
Lemmata \ref{lemma:LS norms}, \ref{Moser norme pesate}, \ref{lemma flusso ODE},  \eqref{stima g tilde i cal QN egorov}, 
\eqref{M s p m - i}, and \eqref{smallness egorov astratto}.  
 \end{proof}

Lemma \ref{lem:claim} proves \eqref{stima-fun}. Furthermore, 
in view of the definition  
\eqref{definizione cal PN tau} of ${\cal P}_N(\tau, \vphi)$, 
it follows from  \eqref{stima p m - i}, 
 Lemma \ref{lemma stime Ck parametri}, \eqref{norma pseudo moltiplicazione}
  and \eqref{Norm Fourier multiplier}  that
   for any $s \geq s_0$, $\alpha \in \N $,  
\begin{equation}\label{stima simbolo cal PN}
|{\cal P}_N(\tau, \vphi)  |_{m, s, \alpha}^\Lipg \lesssim_{m, s, N, \alpha} \| a \|_{s }^\Lipg +  \| \beta\|_{s + \sigma_N^{(N)}}^\Lipg + \| a \|_{s + \sigma_N^{(N)}}^\Lipg \| \beta\|_{s_0 + \sigma_N^{(N)}}^\Lipg   \,. 
\end{equation}
By \eqref{sardine 0}, \eqref{calcolo commutatore cal PN tau},
and \eqref{acqua bella 0}
we deduce that ${\cal P}_N(\tau, \vphi)$ solves 
\begin{equation}\label{equazione risolta da cal PN tau}
\begin{aligned}
& \partial_\tau {\cal P}_N(\tau, \vphi) = [B(\tau, \vphi), {\cal P}_N(\tau, \vphi)] - {\cal Q}_N^{(1)}(\tau, \vphi)\,,  \quad {\cal P}_N(0, \vphi) = a \pa_x^m \, , 
\\
& 
{\cal Q}_N^{(1)}(\tau, \vphi) := {\cal Q}_N(\tau, \vphi) + [B_\infty(\tau, \vphi), {\cal P}_N(\tau, \vphi)] \in  OPS^{m - N - 1} \,. 
\end{aligned}
\end{equation}
We now estimate the difference between ${\cal P}_N(\tau)$ and ${\cal P}(\tau)$. 
\begin{lemma}
The operator $ {\cal R}_N(\tau, \vphi) := {\cal P}(\tau, \vphi) 
- {\cal P}_N(\tau, \vphi)$ is given by
\begin{equation}\label{espansione finalissima cal P}
 {\cal R}_N(\tau, \vphi)   = \int_0^\tau \Phi(\eta, \tau, \vphi) {\cal Q}_N^{(1)}(\eta, \vphi) \Phi(\tau, \eta, \vphi)\, d \eta\,.
\end{equation}
\end{lemma}

\begin{proof}
One writes  
\begin{equation}\label{cal PN - cal P tau}
{\cal P}_N(\tau, \vphi) - {\cal P}(\tau, \vphi) = {\cal V}_N(\tau, \vphi) \Phi(\tau, \vphi)^{- 1}\,, \quad {\cal V}_N(\tau, \vphi) := {\cal P}_N(\tau, \vphi) \Phi(\tau, \vphi) - \Phi(\tau, \vphi){\cal P}_0(\vphi) \, , 
\end{equation}
and a  direct calculation shows that  ${\cal V}_N(\tau)$ solves 
$$
\partial_\tau {\cal V}_N (\tau, \vphi) = B(\tau, \vphi){\cal V}_N(\tau, \vphi) - {\cal Q}_N^{(1)}(\tau, \vphi) \Phi(\tau, \vphi)\,,  \quad  {\cal V}_N(0, \vphi )  = 0\,. 
$$
Hence, by variation of the constants, 
$ {\cal V}_N(\tau, \vphi ) = - \int_0^\tau \Phi(\tau, \vphi) \Phi(\eta, \vphi)^{- 1} {\cal Q}_N^{(1)}(\eta, \vphi) \Phi(\eta, \vphi)\, d \eta $
and,  by \eqref{cal PN - cal P tau} and  \eqref{proprieta elementare flusso}, we deduce \eqref{espansione finalissima cal P}.
\end{proof}

\noindent
Next we prove the estimate \eqref{stima resto Egorov teo astratto} 
of Proposition \ref{proposizione astratta egorov}  of $ {\cal R}_N(\tau, \vphi) $, given by \eqref{espansione finalissima cal P}.
First we estimate ${\cal Q}_N^{(1)} \in OPS^{m - N - 1}$,
defined in \eqref{equazione risolta da cal PN tau}.
 The estimate of ${\cal Q}_N$, obtained from
\eqref{stima g tilde i cal QN egorov}, \eqref{M s p m - i},
\eqref{stima p m - i} ,
and the one of $[B_\infty(\tau, \vphi), {\cal P}_N(\tau, \vphi)],$
obtained from \eqref{stima a infty}, \eqref{stima simbolo cal PN},
Lemma \ref{lemma tame norma commutatore},
yield that there exists a constant $\aleph_N > 0 $ so that 
for any $s \geq s_0$, $\alpha \in \N $, 
\begin{equation}\label{stima finale cal QN}
|{\cal Q}_N^{(1)}(\eta, \vphi) |_{m - N - 1, s, \alpha}^\Lipg \lesssim_{m, s, \alpha, N} \| \beta \|_{s + \aleph_N}^\Lipg +   \| a \|_{s + \aleph_N}^\Lipg \| \beta \|_{s_0 + \aleph_N}^\Lipg \, .
\end{equation} 
Let $\lambda_0, n_1, n_2 \in \N$ with $\lambda \leq \lambda_0$ and $n_1 + n_2 + \lambda_0 + m \leq N - 1$, $k \in \mathbb S_+$. 
In view of the definition \eqref{espansione finalissima cal P} of
$ {\cal R}_N(\tau, \vphi) $,
the claimed estimate of  
$\langle D \rangle^{n_1}\partial_{\vphi_k}^\lambda{\cal R}_N(\tau, \vphi)\langle D \rangle^{n_2}$ will follow from corresponding ones of  
$ \langle D \rangle^{n_1}\partial_{\vphi_k}^{\lambda_1}\Phi(\eta, \tau, \vphi) \partial_{\vphi_k}^{\lambda_2}{\cal Q}_N^{(1)}(\eta, \vphi)  \partial_{\vphi_k}^{\lambda_3}\Phi(\tau, \eta, \vphi) \langle D \rangle^{n_2}$
($\tau, \eta \in [0, 1]$ and $\lambda_1 + \lambda_2 + \lambda_3 = \lambda$) 
which we write as  
$$  \Big( \! \langle D \rangle^{n_1}\partial_{\vphi_k}^{\lambda_1}\Phi(\eta, \tau, \vphi) \langle D \rangle^{- n_1 - \lambda_1 - 1 } \! \Big) \! \Big(\! \langle D \rangle^{n_1 +  \lambda_1 + 1} \partial_{\vphi_k}^{\lambda_2}{\cal Q}_N^{(1)}(\eta, \vphi) \langle D \rangle^{n_2 +  \lambda_3 + 1} \! \Big) \! \Big(  \! \langle D \rangle^{- n_2 -   \lambda_3 - 1 } \partial_{\vphi_k}^{\lambda_3}\Phi(\tau, \eta, \vphi) \langle D \rangle^{n_2}  \Big). 
$$
Then, we use Lemma \ref{proposition 2.40 unificata} to estimate the tame constants of the operators $\langle D \rangle^{n_1}\partial_{\vphi_k}^{\lambda_1}\Phi(\eta, \tau, \vphi) \langle D \rangle^{- n_1 - \lambda_1 - 1 }$, $\langle D \rangle^{- n_2 -  \lambda_3 - 1 } \partial_{\vphi_k}^{\lambda_3}\Phi(\tau, \eta, \vphi) \langle D \rangle^{n_2} $, the estimates \eqref{stima finale cal QN}, \eqref{Norm Fourier multiplier} and Lemmata \ref{lemma stime Ck parametri}, \ref{lemma: action Sobolev} to estimate the tame constant of
 $\langle D \rangle^{n_1 +  \lambda_1 + 1} \partial_{\vphi_k}^{\lambda_2}{\cal Q}_N^{(1)}(\eta, \vphi) \langle D \rangle^{n_2 +  \lambda_3 + 1}$ and Lemma \ref{composizione operatori tame AB} together with  
 the assumption \eqref{smallness egorov astratto}, to estimate the tame constant of the composition. The  bound \eqref{stima resto Egorov teo astratto} is finally proved. 
 
 Item $4$ of Proposition \ref{proposizione astratta egorov} 
 can be shown by similar arguments. This completes the proof of the latter.
\end{proof}

In the sequel we also need to study the operator obtained by conjugating
$ \omega \cdot \partial_\vphi$ 
with the time one flow $\Phi(\varphi) = \Phi(0, 1, \vphi)  $ of the transport equation \eqref{flow0}.
Here we analyze the operator 
$\Phi(\vphi )\circ  \omega \cdot \partial_\vphi (\Phi(\vphi)^{- 1})$,
which turns out to be a pseudo-differential operator of order one  
with an expansion in decreasing symbols.  
\begin{proposition}\label{proposizione astratta egorov 2}
{\bf (Conjugation of $\omega \cdot \partial_\vphi   $)}
Let $N, \lambda_0 \in \N$, $S > s_0$ and assume that $\beta(\cdot; \omega)$
is in ${\cal C}^\infty(\T^{\mathbb S_+} \times \T_1)$ and Lipschitz continuous with respect to $\omega \in  \Omega $.
Then there exist constants $\sigma_N(\lambda_0), \sigma_N > 0$, $\delta(S, N, \lambda_0) \in (0, 1)$, $C_0 > 0$ so that, if 
\begin{equation}\label{smallness egorov astratto 2}
\| \beta\|_{s_0 + \sigma_N(\lambda_0)}^\Lipg \leq \delta \, , 
\end{equation}
then ${\cal P}(\vphi) := \Phi(\vphi) \circ 
\omega \cdot \partial_\vphi ( \Phi(\vphi)^{- 1})$ 
is a pseudo-differential operator of order 1 with an expansion of the form
$$
{\cal P}(\vphi) = \sum_{i = 0}^N p_{1 - i} (\vphi, x ; \om) \partial_x^{1 - i} + {\cal R}_N(\vphi)
$$
with the following properties:
\begin{enumerate} 
\item For any
$i = 0, \ldots, N$ and $s \geq s_0$, 
$ \| p_{1 - i}\|_s^\Lipg \lesssim_{s, N} \| \beta\|_{s + \sigma_N}^\Lipg $. 
\item 
For any $\lambda \in \N$ with $\lambda \leq \lambda_0$, for any $n_1, n_2 \in \N$  with $n_1 + n_2 + \lambda_0 \leq N - 2$, 
and for any $k \in \mathbb S_+ $,  
the pseudo-differential operator $\langle D \rangle^{n_1}\partial_{\vphi_k}^\lambda {\cal R}_N(\vphi) \langle D \rangle^{n_2}$ is 
$\Lipg$-tame  with a tame constant satisfying, for any $s_0 \leq s \leq S $,   
$$
{\mathfrak M}_{\langle D \rangle^{n_1}\partial_{\vphi_k}^\lambda {\cal R}_N(\vphi) \langle D \rangle^{n_2}}(s) \lesssim_{S, N, \lambda_0} \| \beta\|_{s + \sigma_N(\lambda_0)}^\Lipg\,. 
$$
\item Let $s_0 < s_1 < S$ and assume that  
$\| \beta_i \|_{s_1 + \sigma_N(\lambda_0)} \leq \delta$, $i = 1,2$. Then 
$$
\begin{aligned}
\| \Delta_{12} p_{1 - i} \|_{s_1} \lesssim_{s_1, N} \| \Delta_{12} \beta\|_{s_1 + \sigma_N}, \quad i = 0, \ldots, N \, , 
\end{aligned}
$$
and, for any $\lambda \leq \lambda_0$, $n_1, n_2 \in \N$ with $n_1 + n_2 + \lambda_0 \leq N - 2$, and $k \in \mathbb S_+$
$$
\| \langle D \rangle^{n_1}\partial_{\vphi_k}^\lambda \Delta_{12} {\cal R}_N(\vphi) \langle D \rangle^{n_2} \|_{{\cal B}(H^{s_1})} \lesssim_{s_1, N, n_1, n_2}  \| \Delta_{12} \beta\|_{s_1 + \sigma_N(\lambda_0)}
$$
where we refer to Lemma \ref{proposition 2.40 unificata} for 
the meaning of  $\Delta_{12}$.
\end{enumerate}
\end{proposition}
\begin{proof}
The operator 
$\Psi(\tau, \vphi)  := \Phi(\tau, \vphi) \circ 
\omega \cdot \partial_\vphi ( \Phi(\tau, \vphi)^{- 1}) $ solves the 
inhomogeneous Heisenberg equation 
$$
\partial_\tau \Psi(\tau, \vphi) = [B(\tau, \vphi), \Psi(\tau, \vphi))] - \omega \cdot \partial_\vphi ( B(\tau, \vphi)) \, ,  \quad \Psi(0, \vphi) = 0\,. 
$$
The latter equation can be solved in a similar way as \eqref{Heisenberg sezione Egorov} by looking for approximate solutions of the form of
a pseudo-differential operator of order $1$, admitting an expansion 
in homogeneous components (cf. \eqref{definizione cal PN tau}). The 
proof then proceeds in the same way as the one for Proposition \ref{proposizione astratta egorov} and hence is omitted. 
\end{proof}

We finish this section by the following application of Proposition \ref{proposizione astratta egorov} to Fourier multipliers.
\begin{lemma}\label{Fourier multiplier}
Let  $N, \lambda_0 \in \N $, $S > s_0$ and assume that
${\cal Q} $ is a Lipschitz family of Fourier multipliers with an expansion 
of the form
\begin{equation}\label{real madrid 1}
{\cal Q} = \sum_{n = 0}^N c_{m - n}(\omega) \partial_x^{m - n} + {\cal Q}_N
(\om) \,, \quad {\cal Q}_N (\om) \in {\cal B}(H^s, H^{s + N + 1 - m}), \quad \forall s \geq 0\,.
\end{equation}
Then there exist
 $\sigma_N(\lambda_0)$, $\sigma_N > 0$, and $\delta(S, N, \lambda_0) \in (0, 1)$ so that,  if 
\begin{equation}\label{piccolezza corollario Fourier multiplier}
\| \beta \|_{s_0 + \sigma_N(\lambda_0)}^\Lipg \leq \delta(S, N, \lambda_0) \, ,
\end{equation}  
then 
$ \Phi(\vphi) {\cal Q} \Phi(\vphi)^{- 1}$ is an operator of the form 
${\cal Q} + {\cal Q}_\Phi (\vphi) + {\cal R}_N(\vphi) $
with the following properties: 
\begin{enumerate}
\item 
${\cal Q}_\Phi (\vphi) = \sum_{n = 0}^N 
\alpha_{m - n}(\vphi, x; \omega) \partial_x^{m - n}$ 
where for any $s \geq s_0$, 
\begin{equation}\label{stima cal Q Phi lemma astratto}
\| \alpha_{m - n} \|_s^\Lipg \lesssim_{s, N} 
\| \beta\|_{s + \sigma_N}^\Lipg\, , \quad n = 0, \ldots, N \, . 
\end{equation}
\item 
For any $\lambda \in \N$ with $\lambda \leq \lambda_0$, $n_1, n_2 \in \N$ with $n_1 + n_2 + \lambda_0 \leq N - m - 2$, and
$k \in \mathbb S_+$,
the operator $\langle D \rangle^{n_1}\partial_{\vphi_k}^\lambda{\cal R}_N \langle D \rangle^{n_2}$ is $\Lipg$-tame with a tame constant satisfying 
\begin{equation}\label{stima resto corollario egorov}
{\mathfrak M}_{\langle D \rangle^{n_1}\partial_{\vphi_k}^\lambda{\cal R}_N \langle D \rangle^{n_2}}(s) \lesssim_{S, N, \lambda_0} 
\| \beta\|_{s + \sigma_N(\lambda_0)}^\Lipg\,, 
\quad \forall s_0 \leq s \leq S\,. 
\end{equation}
\item Let $s_0 < s_1 < S$
 and assume that  $\| \beta_i \|_{s_1 + \sigma_N(\lambda_0)} \leq \delta$, $i = 1,2$. Then 
$$
\begin{aligned}
\| \Delta_{12} \alpha_{m - n} \|_{s_1} \lesssim_{s_1, N} \| \Delta_{12} \beta\|_{s_1 + \sigma_N}, \quad n = 0, \ldots, N \, , 
\end{aligned}
$$
and, for any $\lambda \leq \lambda_0$, $n_1, n_2 \in \N$ with 
$n_1 + n_2 + \lambda_0 \leq N - m - 2$, and $k \in \mathbb S_+$, 
$$
\| \langle D \rangle^{n_1}\partial_{\vphi_k}^\lambda \Delta_{12} {\cal R}_N(\vphi) \langle D \rangle^{n_2} \|_{{\cal B}(H^{s_1})} \lesssim_{s_1, N, n_1, n_2}  \| \Delta_{12} \beta\|_{s_1 + \sigma_N(\lambda_0)}
$$
where we refer to Lemma \ref{proposition 2.40 unificata} for 
the meaning of  $\Delta_{12}$.
\end{enumerate}
\end{lemma}
\begin{proof}
Applying Proposition \ref{proposizione astratta egorov} to 
 $\Phi(\vphi) \partial_x^{m - n} \Phi(\vphi)^{- 1}$ for $n = 0, \ldots, N$, 
 we get 
$$
\begin{aligned}
\Phi(\vphi) \Big( \sum_{n = 0}^N c_{m -n} (\om) \partial_x^{m - n} \Big) \Phi(\vphi)^{- 1} &  = \sum_{n = 0}^N c_{m - n}(\om) \partial_x^{m - n} + {\cal Q}_\Phi (\vphi) 
+ {\cal R}_N^{(1)}(\vphi) 
\end{aligned}
$$
where $ {\cal Q}_\Phi  (\vphi) = \sum_{n = 0}^N \alpha_{m - n}(\vphi, x; \omega) \partial_x^{m - n} $ with $  \alpha_{m - n} $ satisfying \eqref{stima cal Q Phi lemma astratto}  
and the remainder ${\cal R}_N^{(1)}(\vphi)$ satisfying
 \eqref{stima resto corollario egorov}.  Next we write
$ \Phi(\vphi){\cal Q}_N \Phi(\vphi)^{- 1}  = {\cal Q}_N + {\cal R}_N^{(2)}(\vphi) $
where
$$
\begin{aligned}
{\cal R}_N^{(2)}(\vphi) & : = \big( \Phi(\vphi) - {\rm Id} \big) {\cal Q}_N 
\Phi(\vphi)^{- 1} + {\cal Q}_N  \big(\Phi(\vphi)^{- 1} - {\rm Id} \big)\,. 
\end{aligned}
$$
We then argue as in the proof of the estimate of the remainder ${\cal R}_N(\tau, \vphi)$ in Proposition \ref{proposizione astratta egorov}.
Using Lemma \ref{proposition 2.40 unificata} and 
the assumption that ${\cal Q}_N$ is a Fourier multiplier 
in ${\cal B}(H^s, H^{s + N + 1 - m})$ 
 we get that  ${\cal R}_N^{(2)}(\vphi)$ satisfies \eqref{stima resto corollario egorov},
and 
 ${\cal R}_N(\vphi) = {\cal R}_N^{(1)}(\vphi) + {\cal R}_N^{(2)}(\vphi)$
 satisfies \eqref{stima resto corollario egorov} as well. 
  Item $3$ follows by similar arguments. 
\end{proof} 

\section{Integrable features of KdV}

According to \cite{KP}, the KdV equation \eqref{kdv} 
on the torus is an integrable PDE in the strongest  possible sense, 
meaning that it admits global analytic Birkhoff coordinates. 
We endow the sequence spaces $ h^s $
with the standard Poisson bracket defined by 
$ \{ z_n, z_{k} \} = \ii 2 \pi n \, \delta_{k,-n} $ for any $  n, k \in \Z $. 

\begin{theorem} \label{Birkhoff coordinates for KdV}
{\em (Birkhoff coordinates, \cite{KP})} 
There exists a real analytic diffeomorphism  
$\Psi^{kdv} : h^0_0 \to H^0_0 (\T_1)$ 
 so that the following holds: \\ 
(i) for any $s \in \Z_{\geq 0}$, $\Psi^{kdv}(h^s_0) \subseteq H^s_0 (\T_1)$ 
and $\Psi^{kdv}: h^s_0 \to H^s_0 (\T_1) $ is a real analytic 
symplectic diffeomorphism. \\
(ii) ${H}^{kdv} \circ \Psi^{kdv}: h^1_0 \to \R$ is a real analytic function of the actions $I_k := \frac{1}{2\pi k} z_k z_{- k}$, $k \ge 1$. The KdV Hamiltonian, 
viewed as a function of the actions $(I_k)_{k \ge 1}$, is denoted by $ \mathcal H^{kdv}_o$.\\
(iii) $\Psi ^{kdv}(0) = 0$ and  the differential $d_0\Psi^{kdv}$ of $\Psi^{kdv}$ at $0$ is the inverse Fourier transform ${\cal F}^{- 1}$.
\end{theorem}
By Theorem \ref{Birkhoff coordinates for KdV}, 
the KdV equation, expressed in the Birkhoff coordinates $(z_n)_{n \ne 0}$, reads
$$
\partial_t z_n = \ii \omega_n^{kdv}((I_k)_{k \ge 1}) z_n \, , \,\, \forall n \in \Z \setminus \{0\} \, , \quad 
\omega_{\pm m}^{kdv}((I_k)_{k \ge 1}) : =  \pm \partial_{I_m} {\mathcal H}_o^{kdv}((I_k)_{k \ge 1}) \, , 
\,\, \forall m \ge 1\, ,
$$
and its solutions are given by $z(t) := (z_n)_{n \ne 0}$ where
$$
z_{n}(t) = z_{n}(0) 
{\rm exp}\big(\ii \omega_{n}^{kdv}((I_k^{(0)})_{k \ge 1}) \, t \big)  \, , 
\,\,\, \forall 
n \in \Z \setminus \{0\}  \,, \quad  
I_k^{(0)} :
=  \frac{1}{2\pi k} z_k(0) z_{- k}(0)\,, \,\,\, \forall {k \ge 1}\,. 
$$
Let us consider a finite set 
$\Splus \subset \N_+ := \{1, 2, \ldots \} $ and define 
$$
 \S := \Splus  \cup (- \Splus ) \, , \quad  
\Sbot_+ := \N_+ \setminus \Splus \, , \quad 
 \S^\bot := \S^{\bot}_+ \cup (- \Sbot_+) \subset \Z \setminus \{0\} \, . 
 $$
In Birkhoff coordinates, a $ \Splus-$gap solution of the KdV equation, also referred to as $\S-$gap solution, 
is a solution of the form 
\begin{equation}\label{definition finite gap solutions}
z_n(t) = 
{\rm exp}\big(\ii \omega^{kdv}_n( \ac , 0) t \big)  z_n(0) \, , \,\,
 z_n(0) \ne  0 \, , \,\, \forall n \in \S\,, \quad z_n(t) = 0 \, , \,\, \forall  n \in \Sbot \, ,
\end{equation}
where $ \ac : = ( I_k^{(0)})_{k \in \Splus} \in  \R_{> 0}^{\Splus} $ and, 
by a slight abuse of notation, we write 
\be\label{fre-kdv}
\omega^{kdv}_n( I ,   (I_k)_{k \in \Splus^\bot}) := \omega^{kdv}_n(   (I_k)_{k \geq 1} ) 
  \, , \quad  I := (I_k)_{k \in \Splus } \in \R_{> 0}^{\Splus} \, . 
\ee
Such solutions are quasi-periodic in time 
with frequency vector (cf. \eqref{frequency finite gap})
$ \omega^{kdv} (\ac) = \big( \omega_n^{kdv}(\ac, 0) \big)_{n \in \Splus} \in \R^{\Splus} $, 
parametrized by $ \nu \in \R^{\Splus}_{> 0} $. The map $ \ac \mapsto 
\omega^{kdv} (\ac) $ is a local 
analytic diffeomorphism, see Remark \ref{rem:diffeo}.
When written in action-angle coordinates 
(cf. \eqref{finite gap in coordinate originarie}),   
$$
\theta := (\theta_n)_{n \in \Splus} \in \T^{\Splus} \, , \quad
I = (I_n)_{n \in \Splus} \in \R_{> 0}^{\Splus} \, , 
\quad
z_n = \sqrt{2\pi n I_n} e^{- \ii \theta_n} \, , \ n \in \Splus \, , 
$$
instead of the complex Birkhoff coordinates $z_n$, 
the $ \S-$gap solution \eqref{definition finite gap solutions} reads 
$$
 \theta(t) = \theta^{(0)} - \omega^{kdv} (\ac) t \, , 
 \qquad  I(t) = \ac \, , \qquad  z_n(t) = 0 \, , \quad  \forall n \in  \Sbot \, . 
$$

\subsection{Normal form coordinates for the KdV equation}\label{sec canonical coordinates}

In this section we rephrase Theorem 1.1 in \cite{Kappeler-Montalto-pseudo} adapted to our purposes and prove some corollaries. 

We consider an open bounded 
set $\Xi \subset \R^{\mathbb S_+}_{> 0}$ 
so that \eqref{azioni come parametri 1} holds for some   $\delta > 0 $.
Recall that $ {\cal V}^s(\delta) \subset \mathcal E_s$,  
$ {\cal V}(\delta) = {\cal V}^0(\delta)$ are defined in \eqref{Vns}
and that we denote by $\mathfrak x = (\theta, y, w) $ its elements. 
The space $ {\cal V}(\delta) \cap {\cal E}_s $ is endowed with the symplectic form 
 \begin{equation}\label{2form}
{\cal W} := \Big( {\mathop \sum}_{j \in \Splus} d y_j \wedge d \theta_j \Big)  \oplus 
{\cal W}_\bot  
\ee
where $ {\cal W}_\bot $ is the restriction to  $ L^2_\bot(\T_1) $ of the symplectic form 
$ {\cal W}_{L^2_0} $ defined  in \eqref{KdV symplectic}. 
The Poisson structure ${\cal J}$ corresponding to $\mathcal W$, defined by the identity
$ \{ F, G \} = {\cal W}(X_F, X_G) =  \big\langle \nabla F \,,\, {\cal J} \nabla G \big\rangle $,   
is the unbounded operator 
\begin{equation}\label{Poisson struture}
\mathcal J : E_s \to E_s \, , \quad 
(\widehat \theta,  \widehat y,  \widehat w) \mapsto 
(- \widehat y,   \widehat \theta,  \partial_x \widehat w) \,  
\end{equation}
where $ \langle \ , \ \rangle $ is the inner product \eqref{bi-form}.
\begin{theorem}\label{modified Birkhoff map}{\bf (Normal KdV coordinates with pseudo-differential expansion, 
\cite{Kappeler-Montalto-pseudo}).}
Let $\Splus \subseteq \N$ be finite, 
$\Xi$ an open bounded subset of $\R_{> 0}^{\Splus}$
so that \eqref{azioni come parametri 1} holds, for some 
$\delta > 0 $. Then, for $\delta > 0$ sufficiently small, there exists a canonical 
$ {\cal C}^\infty$ family of diffeomorphisms 
$\Psi_\ac : {\cal V}(\delta) \to \Psi_\ac ({\cal V}(\delta)) \subseteq L^2_0 (\T_1)\,, \,  (\theta, y, w)\mapsto q$, $ \ac \in \Xi $,  
with the property that $\Psi_\ac $ satisfies 
$$
\Psi_\ac(\theta, y , 0 ) = \Psi^{kdv}(\theta, \ac + y, 0) \, , \quad \forall (\theta, y, 0 ) \in {\cal V}(\delta)\,, \quad \forall \ac \in \Xi \, , 
$$
and is compatible  with the scale of Sobolev spaces $H^s_0(\T_1), s \in 
\N $, 
in the sense that 
$\Psi_\ac  \big( {\cal V}(\delta)\cap {\mathcal E}_s \big) \subseteq H^s_0(\T_1)$ and $\Psi_\ac  : {\cal V}(\delta)\cap {\mathcal E}_s \to H^s_0(\T_1)$ 
is a $ {\cal C}^\infty-$diffeomorphism onto its image, 
so that the following holds:
\begin{description}
\item[({\bf AE1})] For any integer $ M \ge 1$,  $\ac \in \Xi$, $\mathfrak x = (\theta, y , w) \in \mathcal V(\delta)$,  $\Psi_\ac (\mathfrak x)$ admits an asymptotic expansion  of the form  
\be\label{Psi-espl}
\Psi_\ac( \theta, y , w ) = \Psi^{kdv}(\theta, \ac + y, 0) +  w + 
\sum_{k = 1}^M a_{-k }^\Psi(\mathfrak x; \ac ) \, \partial_x^{- k} w
+ {\cal R}_{M }^\Psi(\mathfrak x; \ac ) 
\ee
where  ${\cal R}_{M}^\Psi(\theta, y, 0; \ac ) = 0$ and,  
for any $s \in \N$  and $1 \le k \le M $,  the functions
$$
 {\cal V}(\delta) \times  \Xi  \to  H^s(\T_1),\,  (\mathfrak x, \ac)  \mapsto   a_{-k }^\Psi (\mathfrak x; \ac ) \, ,  \quad
 ({\cal V}(\delta) \cap {\mathcal E}_s) \times \Xi   \to H^{s + M +1}(\T_1),\, 
(\mathfrak x,\ac ) \mapsto   \mathcal R_{M }^\Psi(\mathfrak x; \ac ) \, ,
$$
are ${\cal C}^\infty$. 
\item[({\bf AE2})]
For any $\mathfrak x  \in {\cal V}^1(\delta)$, $\ac \in \Xi$, the transpose $d \Psi_\ac(\mathfrak x)^\top $ 
of the differential
$d \Psi_\ac(\mathfrak x) : E_1 \to H^1_0 (\T_1) $ is a bounded linear operator $d \Psi_\ac (\mathfrak x)^\top : H^1_0(\T_1)  \to E_1$,
and, 
for any $\widehat q \in H^1_0 (\T_1) $
and integer $M \ge 1$,  
$ d \Psi_\ac (\mathfrak x)^\top [\widehat q] $ 
admits an expansion of the form
\be\label{depsit}
d \Psi_\ac (\mathfrak x)^\top[\widehat q] = 
\Big(   0, 0,  {\Pi}_\bot \widehat q  + \Pi_\bot \sum_{k = 1}^M  a_{-k}^{d \Psi^\top}(\mathfrak x; \ac )  \partial_x^{- k}\widehat q \, 
+  \Pi_\bot \sum_{k = 1}^M  (\partial_x^{- k} w) \,
\mathcal A_{-k}^{ d \Psi^\top}(\mathfrak x; \ac )[\widehat q]   \Big)
   + {\cal R}_{M}^{d \Psi^\top}(\mathfrak x; \ac )[\widehat q]
\ee
where, for any $ s \ge 1 $ and $1 \le  k \le M$, 
$$
\begin{aligned}
& 
 {\cal V}^1(\delta) \times \Xi  \to H^s (\T_1) \,, \,\, (\mathfrak x, \ac) \mapsto  a_{-k}^{d\Psi^\top} (\mathfrak x; \ac  ) 
\,, \\
&  {\cal V}^1(\delta) \times \Xi   \to  {\cal B}(H^1_0 (\T_1), H^s(\T_1))\,, \, \,   (\mathfrak x, \ac) \mapsto \mathcal A_{-k}^{d\Psi^\top} (\mathfrak x; \ac )\,,  \\
 &  ({\cal V}^1 (\delta) \cap {\mathcal E}_s)  \times \Xi \to {\cal B}(H^s_0 (\T_1), E_{s +M +1}) \, , \, \,  (\mathfrak x, \ac ) \mapsto   {\cal R}_M^{d\Psi^\top}(\mathfrak x; \ac ) \, , 
 \end{aligned}
$$
are $ {\mathcal C}^\infty $. Furthermore, 
\begin{equation}\label{coefficiente - 1 d Psi Psi top}
a_{- 1}^{d \Psi^\top}(\frak x; \nu) = - a_{- 1}^\Psi(\frak x; \nu)\,. 
\end{equation}
 \item[({\bf AE3})] For any $\ac \in \Xi$, the Hamiltonian 
 $ {\cal H}^{kdv} ( \cdot \, ; \ac ) := H^{kdv} \circ \Psi_\ac : {\cal V}^1(\delta) \to \R$ is in normal form up to order three, meaning that 
\be\label{Hkdv}
{\cal H}^{kdv} (\theta, y, w; \ac ) =  \omega^{kdv}(\ac) \cdot y  + \frac{1}{2} 
\big( \Omega^{kdv}(D ;\ac) w, w \big)_{L^2_x}  + \frac12 \Omega_{\mathbb S_+}^{kdv}(\ac) [y] \cdot y+  {\cal R}^{kdv} (\theta, y , w; \ac )
\ee
where $\omega^{kdv}(\ac) = (\omega_n^{kdv}(\ac))_{n \in \mathbb S_+},$
\be\label{Omega-normal}
\begin{aligned}
&
\Omega^{kdv}(D ; \ac) w: =  \sum_{n \in \Sbot} \Omega_n^{kdv}(\ac ) w_n e^{\ii 2\pi n x}\,, \qquad
 \Omega_{\mathbb S_+}^{kdv}(\ac) 
 := (\partial_{I_j} \omega^{kdv}_k (\ac))_{j, k \in \mathbb S_+}\,,\\
  & \Omega_n^{kdv}(\ac) := \frac{1}{2\pi n} \omega_n^{kdv}(\ac , 0)\,, \quad \forall n \in \Sbot, \qquad  
  w =  \sum_{n \in \Sbot}  w_n e^{\ii 2\pi n x}
  \end{aligned}
\ee
and 
$ {\cal R}^{kdv} :   {\cal V}^1(\delta) \times \Xi \to \R $ is a ${\cal C}^\infty$ map satisfying 
\be\label{Rkdv-cubic}
{\cal R}^{kdv}(\theta, y, w; \ac ) = O\big(( \| y \| + \| w \|_{H^1_x})^3  \big) \, , 
\ee
and has the property that, for any $s \ge 1$, its $L^2-$gradient
$$
 ({\cal V}^1(\delta) \cap {\mathcal E}_s)  \times \Xi  \to E_{s}, \, (\mathfrak x, \ac ) \mapsto \nabla {\cal R}^{kdv}(\mathfrak x; \ac ) = 
\big( \nabla_{\theta} {\cal R}^{kdv}(\mathfrak x; \ac), \nabla_{y} {\cal R}^{kdv}(\mathfrak x; \ac), \nabla_{w} {\cal R}^{kdv}(\mathfrak x; \ac) \big)
$$
is a $ {\mathcal C}^\infty $ map as well. As a consequence
\be\label{Rkdv-cubic-der}
\nabla {\cal R}^{kdv}(\theta, 0, 0; \ac) = 0 \, , \ 
d_\bot \nabla {\cal R}^{kdv}(\theta, 0, 0; \ac) = 0 \, , \
\partial_y \nabla {\cal R}^{kdv}(\theta, 0, 0; \ac) = 0 \, . 
\ee
 \item[(Est1)] For any $\ac \in \Xi$, $\alpha \in 
  \N^{\mathbb S_+}$, $\mathfrak x \in {\cal V} (\delta)$, $1 \le k \le M$, $\widehat{\mathfrak x}_1, \ldots, \widehat{\mathfrak x}_l \in E_0$, $s \in \N$,
  $$
  \begin{aligned}
  & \| \partial_\ac^\alpha a_{- k}^\Psi(\mathfrak x; \ac) \|_{H^s_x} \lesssim_{s, k, \alpha} \, 1 
  \,,  \qquad
    \| d^l \partial_\ac^\alpha a_{- k}^\Psi(\mathfrak x; \ac)[\widehat{\mathfrak x}_1, \ldots, \widehat{\mathfrak x}_l]\|_{H^s_x}\lesssim_{s, k, l, \alpha} \, 
    \prod_{j = 1}^l \| \widehat{\mathfrak x}_j\|_{E_0}\,. 
  \end{aligned}
  $$
  Similarly, for any $\ac \in \Xi$, $\alpha \in \N^{\mathbb S_+}$,
  $\mathfrak x  \in {\cal V} (\delta) \cap {\mathcal E}_s$,  
 $\widehat{\mathfrak x}_1, \ldots, \widehat{\mathfrak x}_l \in E_s$,  
  $s \in \N$,
  $$
  \begin{aligned}
  & \| \partial_\ac^\alpha {\cal R}_M^\Psi(\mathfrak x; \ac)\|_{H^{s + M + 1}_x} \lesssim_{s, M, \alpha}  \| w  \|_{H^s_x }\,,  \quad \\
  & \| d^l \partial_\ac^\alpha{\cal R}_M^\Psi(\mathfrak x; \ac)[\widehat{\mathfrak x}_1, \ldots, \widehat{\mathfrak x}_l]\|_{H^{s + M + 1}_x} 
   \lesssim_{s, M, l, \alpha} \sum_{j = 1}^l 
   \Big( \| \widehat{\mathfrak x}_j\|_{E_s} \prod_{i \neq j} \| \widehat{\mathfrak x}_i\|_{E_0} \Big) + \| w  \|_{H^s_x} \prod_{j = 1}^l \| \widehat{\mathfrak x}_j\|_{E_0}\,.
  \end{aligned}
  $$
 \item[(Est2)] For any $\ac \in \Xi$, $\alpha \in \N^{\mathbb S_+}$,
  $\mathfrak x \in {\cal V}^1(\delta)$, $1 \le k \le M$, 
   $\widehat{\mathfrak x}_1, \ldots, \widehat{\mathfrak x}_l \in E_1$, $s \ge 1$,
  $$
  \begin{aligned}
  &  \| \partial_\ac^\alpha a_{-k}^{{d \Psi}^\top}(\mathfrak x; \ac)\|_{H^s_x} \lesssim_{s, k, \alpha} 1  \,, \qquad
   \| d^l \partial_\ac^\alpha a_{-k}^{{d \Psi}^\top}(\mathfrak x; \ac )[\widehat{\mathfrak x}_1, \ldots, \widehat{\mathfrak x}_l]\|_{H^s_x} 
  \lesssim_{s, k, l, \alpha}  \prod_{j = 1}^l \| \widehat{\mathfrak x}_j\|_{E_1} \, , \\
  &  \| \partial_\ac^\alpha \mathcal A_{-k}^{{d \Psi}^\top}(\mathfrak x; \ac )\|_{{\cal B}(H_0^1, H^s_x)} \lesssim_{s, k, \alpha} \,   1  \,, \quad
   \| d^l \partial_\ac^\alpha \mathcal A_{-k}^{{d \Psi}^\top}(\mathfrak x ; \ac)[\widehat{\mathfrak x}_1, \ldots, \widehat{\mathfrak x}_l]\|_{{\cal B}(H_0^1, H^s_x)} 
  \lesssim_{s, k, l, \alpha}  \prod_{j = 1}^l \| \widehat{\mathfrak x}_j\|_{E_1}\,. 
  \end{aligned}
  $$
  Similarly, for any $\ac \in \Xi$, $\alpha \in \N^{\mathbb S_+}$,
  $\mathfrak x \in {\cal V}^1(\delta) \cap {\mathcal E}_s $, 
 $\widehat{\mathfrak x}_1, \ldots, \widehat{\mathfrak x}_l \in E_s$, 
  $\widehat q \in H^s_0$, $s \ge 1$,
  $$
  \begin{aligned}
  & \| \partial_\ac^\alpha {\cal R}_M^{d\Psi^\top} (\mathfrak x; \ac ) [\widehat q]\|_{E_{s + M + 1}} \lesssim_{s, M, \alpha} \| \widehat q\|_{H^s_x}+ \| w \|_{H^s_x} \| \widehat q\|_{H^1_x}\,,  \\
  & \| d^l \big( \partial_\ac^\alpha {\cal R}_M^{d \Psi^\top} (\mathfrak x; \ac) [\widehat q] \big)[\widehat{\mathfrak x}_1, \ldots, \widehat{\mathfrak x}_l]\|_{E_{s + M + 1}} 
  \lesssim_{s, M, l, \alpha} \| \widehat q\|_{H^s_x} \prod_{j = 1}^l \| \widehat{\mathfrak x}_j\|_{E_1} + \| \widehat q\|_{H^1_x} \sum_{j = 1}^l \Big( \| \widehat{\mathfrak x}_j\|_{E_s} \prod_{i \neq j} \| \widehat{\mathfrak x}_i\|_{E_1} \Big)  \\
  & \qquad \qquad\qquad\qquad\qquad\qquad\qquad\qquad \qquad 
   \quad +   \| \widehat q\|_{H^1_x} \| w\|_{H^s_x} \prod_{j = 1}^l \| \widehat{\mathfrak x}_j \|_{E_1}\,.
    \end{aligned}
  $$
  \end{description}
\end{theorem}

We now apply Theorem \ref{modified Birkhoff map} to prove new results concerning the extensions of $d \Psi_\nu(\frak x)^\top$ and $ d \Psi_\nu(\frak x) $ to Sobolev spaces of negative order. We refer to the paragraph after 
\eqref{EsEs} for the definitions of $\mathcal E_s$,
$E_s$ for negative $s$.

\begin{corollary}\label{corollary transpose negative sobolev}
{\bf (Extension of $d \Psi_\nu(\frak x)^\top$ and its asymptotic expansion)} 
Let $M \geq 1$. 
There exists $\sigma_M > 0 $  so that for any  $\frak x  \in  {\cal V}^{\s_M}(\delta) $ and $\ac \in \Xi$, 
 the operator $d \Psi_\nu(\frak x)^\top$ 
extends to a bounded linear operator 
$d \Psi_\nu(\frak x)^\top:  H^{- M - 1}_0(\T_1) \to  E_{- M - 1}$ and  
for any $ \widehat q \in H^{- M - 1}_0(\T_1) $, 
$d \Psi_\nu(\frak x)^\top [\widehat q ]$
admits an expansion of the form 
\begin{equation}\label{negativi 3}
d \Psi_\nu(\frak x)^\top [\widehat q ] = \Big( 0,0, \Pi_\bot \widehat q + \Pi_\bot \sum_{k = 1}^M a_{- k}^{ext}(\frak x; \nu; d \Psi^\top) \partial_x^{- k} \widehat q\Big) + {\cal R}_M^{ext}(\frak x; \nu; d \Psi^\top)[\widehat q]
\end{equation}
with the following properties:  
\\[1mm]
(i) For any $s \geq 0 $,  the maps
$$
\begin{aligned}
& {\cal V}^{\s_M}(\delta) \times \Xi \to H^s(\T_1) \, , \quad (\frak x, \nu) \mapsto a_{- k}^{ext}(\frak x; \nu; d \Psi^\top)\,, 
\qquad 1 \le k \le  M \, ,
\end{aligned}
$$
are ${\cal C}^\infty$. They satisfy 
$ a_{- 1}^{ext}(\frak x; \nu; d \Psi^\top) = 
a_{-1}^{d \Psi^\top}(\mathfrak x;  \ac )  $ 
(cf. Theorem \ref{modified Birkhoff map}-${\bf (AE2)}$) and  
 for any $\alpha \in \N^{\mathbb S_+}$, $\widehat{\frak x}_1, \ldots, \widehat{\frak x}_l \in E_{\s_M}$, 
 and  $(\frak x, \nu) \in {\cal V}^{\s_M}(\delta) \times \Xi $,
\begin{equation}\label{negativi 5}
\begin{aligned}
& \|\partial_\nu^\alpha a_{- k}^{ext}(\frak x; \nu; d \Psi^\top) \|_{H^s_x} \lesssim_{s, M, \alpha} 1\,, \\
& \|\partial_\nu^\alpha d^l a_{- k}^{ext}(\frak x; \nu; d \Psi^\top)[\widehat{\frak x}_1, \ldots, \widehat{\frak x}_l] \|_{H^s_x} \lesssim_{s, M, l, \alpha} \prod_{j = 1}^l \| \widehat{\frak x}_j \|_{E_{\s_M}}\, .
\end{aligned}
\end{equation}
$(ii)$ For any $- 1 \leq s \leq M+ 1$, the map 
$$
{\cal R}_M^{ext}(\cdot ; \cdot; d \Psi^\top) : {\cal V}^{\sigma_M}(\delta) \times \Xi \to {\cal B}(H_0^{- s}(\T_1), E_{M + 1 - s})
$$
is ${\cal C}^\infty$ and satisfies for any $\alpha \in \N^{\mathbb S_+}$, $\widehat{\frak x}_1, \ldots, \widehat{\frak x}_l \in E_{\s_M}$, 
$\widehat q \in H^{- s}_0(\T_1)$, 
and  $(\frak x, \nu) \in {\cal V}^{\s_M}(\delta) \times \Xi $,  
\begin{equation}\label{negativi 7}
\begin{aligned}
& \|\partial_\nu^\alpha {\cal R}_M^{ext}(\frak x; \nu; d \Psi^\top)[\widehat q] \|_{E_{M + 1- s}} \lesssim_{ M, \alpha } \| \widehat q\|_{H_x^{- s}}\,, \\
& \|\partial_\nu^\alpha d^l {\cal R}_M^{ext}(\frak x; \nu; d \Psi^\top)[\widehat{\frak x}_1, \ldots, \widehat{\frak x}_l][\widehat q] \|_{E_{M + 1 - s}} \lesssim_{s, M, l, \alpha} \| \widehat q\|_{H^{- s}_x} \prod_{j = 1}^l \| \widehat{\frak x}_j \|_{E_{\s_M}}\,. \\
\end{aligned}
\end{equation}
 $(iii)$ 
For any $s \geq 1$,   the map 
$$
{\cal R}_M^{ext}(\cdot ; \cdot; d \Psi^\top ) : 
\big({\cal V}^{\sigma_M}(\delta) \cap {\cal E}_{s + \sigma_M}\big) \times \Xi \to {\cal B}(H_0^{ s}(\T_1), E_{s + M + 1} )
$$
is ${\cal C}^\infty$ and satisfies 
for any $\alpha \in \N^{\mathbb S_+}$, 
$ \widehat{\frak x}_1, \ldots, \widehat{\frak x}_l \in E_{s + \sigma_M}$, 
$\widehat q \in H^{s}_0(\T_1)$, and  
$(\frak x, \nu) \in \big({\cal V}^{\sigma_M}(\delta) 
\cap {\cal E}_{s + \sigma_M}\big) \times \Xi$,
\begin{equation}\label{negativi 15a}
\begin{aligned}
& \|\partial_\nu^\alpha {\cal R}_M^{ext}(\frak x; \nu; d \Psi^\top)[\widehat q] \|_{E_{M + 1+ s}} \lesssim_{ s, M, \alpha} \| \widehat q\|_{H_x^{s}} + \| \frak{x}\|_{s + \sigma_M} \| \widehat q\|_{H^1_x} \,, \\
& \|\partial_\nu^\alpha d^l {\cal R}_M^{ext}(\frak x; \nu; d \Psi^\top ) [\widehat{\frak x}_1, \ldots, \widehat{\frak x}_l]
[\widehat q]  \|_{E_{M + 1+ s}} \lesssim_{ s, M,  l, \alpha } \| \widehat q\|_{H_x^s} \prod_{j = 1}^l \| \widehat{\frak x}_j\|_{E_{\sigma_M}}  \\
& \qquad  \qquad  +  \| \widehat q\|_{H^1_x} \Big( \sum_{j = 1}^l \| \widehat{\frak x}_j\|_{E_{s +  \sigma_M}} \prod_{i \neq j} \| \widehat{\frak x}_i\|_{E_{\sigma_M}} + \| \frak x \|_{E_{s + \sigma_M}} \prod_{j = 1}^l \| \widehat{\frak x}_j \|_{E_{\sigma_M}} \Big)\,. 
\end{aligned}
\end{equation}
\end{corollary}

\begin{proof}
By Theorem \ref{modified Birkhoff map}, 
for any $(\frak x, \nu) \in {\cal V}(\delta) \times \Xi $,  
the differential $d \Psi_\nu(\frak x) : E_0 \to L^2_0(\T_1)$ is bounded and, 
for any $  M \geq 1 $, differentiating \eqref{Psi-espl}, 
$d \Psi_\nu(\frak x) [\widehat{\frak x}]$ admits the expansion 
for any $\widehat{\frak x}= (\widehat \theta, \widehat y, \widehat w) \in E_0$ of the form 
\begin{align}
& d \Psi_\nu(\frak x) [\widehat{\frak x}] = \widehat w + 
\sum_{k = 1}^M a_{- k}^\Psi(\frak x; \nu) \partial_x^{- k} \widehat w + {\cal R}_M^{(1)}(\frak x; \nu)[\widehat{\frak x}]\,,  \label{pr-exp} \\
& {\cal R}^{(1)}_M (\frak x; \nu)[\widehat{\frak x}] := 
\sum_{k = 1}^M (\partial_x^{- k} w)  d a_{- k}^\Psi(\frak x; \nu)[\widehat{\frak x}] + d {\cal R}_M^\Psi(\frak x; \nu)[\widehat{\frak x}]
+ d_{\theta,y} \Psi^{kdv}(\theta, \nu + y, 0) [\widehat \theta, \widehat y] \, . \nonumber
\end{align}
For $\s_M  \geq M $, 
the map ${\cal R}_M^{(1)} : {\cal V}^{\s_M} (\delta) \times \Xi \to
 {\cal B}(E_0, H^{M + 1}(\T_1))$ is ${\cal C}^\infty$ and satisfies, 
by Theorem \ref{modified Birkhoff map}-${\bf(Est1)}$, 
for any $ \a \in \N^{\Splus} $, $ l \geq 1 $, 
\begin{equation}\label{negativi 1}
\begin{aligned}
& \| \partial_\nu^\alpha {\cal R}_M^{(1)}(\frak x; \nu)[\widehat{\frak x}] \|_{H^{M + 1}_x} \lesssim_{M, \alpha} \| \widehat{\frak x}\|_{E_0}\,, \\
& \| \partial_\nu^\alpha d^l {\cal R}_M^{(1)}(\frak x; \nu)[\widehat{\frak x}_1, \ldots, \widehat{\frak x}_l] [\widehat{\frak x}] \|_{H^{M + 1}_x} \lesssim_{M, l, \alpha} \| \widehat{\frak x}\|_{E_0} \prod_{j = 1}^l \| \widehat{\frak x}_j\|_{E_{\s_M}}\,. 
\end{aligned}
\end{equation}
Now consider the transpose operator $d \Psi_\nu(\frak x)^\top : L^2_0(\T_1) \to E_0$. 
By \eqref{pr-exp}, for any $\widehat q \in L^2_0(\T_1) $, 
one has
\begin{equation}\label{negativi 2}
d \Psi_\nu(\frak x)^\top [\widehat q ] = \Big(0,0,  \Pi_\bot \widehat q + \Pi_\bot \sum_{k = 1}^M (- 1)^k \partial_x^{- k}\big( a_{-k}^\Psi(\frak x; \nu) \ \widehat q \big) \Big) + {\cal R}_M^{(1)}(\frak x; \nu)^\top [\widehat q]\,.
\end{equation}
Since each function $ a_{-k}^\Psi(\frak x; \nu) $ is ${\cal C}^\infty$ and 
${\cal R}^{(1)}_M (\frak x; \nu)^\top : H^{- M - 1} (\T_1) \to E_0$ is  bounded, the right hand side of \eqref{negativi 2} 
defines a linear operator in 
${\cal B} ( H^{- M - 1}_0(\T_1), E_{- M - 1}) $, which we also
denote by $d \Psi_\nu(\frak x)^\top$. 
By \eqref{composizione simboli omogenei}, the expansion \eqref{negativi 2} yields one of the form \eqref{negativi 3} where by \eqref{negativi 1} and 
Theorem \ref{modified Birkhoff map}-${\bf (Est1)}$,  the remainder 
$ {\cal R}_M^{ext}(\frak x; \nu ; d \Psi^\top)$ satisfies
 for any $\alpha \in \N^{\mathbb S_+}$, $\widehat{\frak x}_1, \ldots, \widehat{\frak x}_l \in E_{\s_M}$, 
 and $\widehat q \in H^{-M  - 1}_0(\T_1)$
\begin{equation}\label{pallone 0}
\begin{aligned}
& \|\partial_\nu^\alpha {\cal R}_M^{ext}(\frak x; \nu; d \Psi^\top)[\widehat q] \|_{E_{0}} \lesssim_{ M, \alpha } \| \widehat q\|_{H_x^{- M - 1}}\,, \\
& \|\partial_\nu^\alpha d^l {\cal R}_M^{ext}(\frak x; \nu; d \Psi^\top)[\widehat{\frak x}_1, \ldots, \widehat{\frak x}_l][\widehat q] \|_{E_{0}} \lesssim_{M, l, \alpha} \| \widehat q\|_{H^{- M - 1}_x} \prod_{j = 1}^l \| \widehat{\frak x}_j \|_{E_{\s_M}}\,. 
\end{aligned}
\end{equation}
The restriction of the operator $d \Psi_\nu(\frak x)^\top:  H^{- M - 1}_0(\T_1) \to  E_{- M - 1}$ to $ H_0^1(\T_1) $ coincides 
with \eqref{depsit} and,  
by the uniqueness of an expansion of this form, 
$$
\begin{aligned}
& a_{- k}^{ext}(\frak x; \nu; d \Psi^\top) = a_{-k}^{d \Psi^\top}(\mathfrak x;  \ac ) \, , \quad k = 1, \ldots, M\,, 
\\
&  {\cal R}_M^{ext}(\frak x; \nu ; d \Psi^\top)[\widehat q]= \sum_{k = 1}^M  (\partial_x^{- k} w )
\mathcal A_{-k}^{ d \Psi^\top}(\mathfrak x;  \ac )[\widehat q]  \,  + 
    {\cal R}_{M}^{d \Psi^\top}(\mathfrak x;  \ac )[\widehat q] \, , \quad 
    \forall \widehat q \in H_0^1(\T_1) \, . 
    \end{aligned}
$$ 
The claimed estimates \eqref{negativi 5} and \eqref{negativi 15a}
 then follow by Theorem \ref{modified Birkhoff map}-${\bf (Est2)}$. 
 In particular we have, for any $\alpha \in \N^{\mathbb S_+}$, $\widehat{\frak x}_1, \ldots, \widehat{\frak x}_l \in E_{\s_M}$, $\widehat q \in H^{1}_0(\T_1)$, 
\begin{equation}\label{negativi 6}
\begin{aligned}
& \|\partial_\nu^\alpha {\cal R}_M^{ext}(\frak x; \nu; d \Psi^\top)[\widehat q] \|_{E_{M + 2}} \lesssim_{ M, \alpha } \| \widehat q\|_{H_x^{1}}\,, \\
& \|\partial_\nu^\alpha d^l {\cal R}_M^{ext}(\frak x; \nu; d \Psi^\top)[\widehat{\frak x}_1, \ldots, \widehat{\frak x}_l][\widehat q] \|_{E_{M + 2}} \lesssim_{M, l, \alpha} \| \widehat q\|_{H^{1}_x} \prod_{j = 1}^l \| \widehat{\frak x}_j \|_{E_{\s_M}}\,. \\
\end{aligned}
\end{equation}
Finally 
the estimates \eqref{negativi 7} 
follow by interpolation between \eqref{pallone 0} and \eqref{negativi 6}.
\end{proof}

\begin{corollary}\label{corollary Birkhoff negative sobolev}
{\bf (Extension of $d_\bot \Psi_\nu(\frak x)$ and its asymptotic expansion)}
Let $M \geq 1$. 
There exists $\sigma_M > 0 $  so that for any  $\frak x  \in  {\cal V}^{\s_M}(\delta) $ and $\ac \in \Xi$,  
the operator $d_\bot \Psi_\nu(\frak x)$  extends to a bounded linear operator, 
 $d_\bot \Psi_\nu(\frak x): H_\bot^{- M - 2}(\T_1) 
 \to  H^{- M - 2}_0(\T_1)$, and 
for any  $\widehat w \in H_\bot^{- M - 2}(\T_1) $, 
$d_\bot \Psi_\nu(\frak x)[\widehat w]$ admits an expansion 
\begin{equation}\label{negativi 11}
d_\bot \Psi_\nu(\frak x)[\widehat w] = \widehat w +  
 \sum_{k = 1}^M a_{- k}^{ext}(\frak x; \nu; d_\bot \Psi) \partial_x^{- k} \widehat w + {\cal R}_M^{ext}(\frak x; \nu; d_\bot \Psi)[\widehat w]
\end{equation}
with the following properties:
\\[1mm]
(i) For any $s \geq 0 $,  the maps
$$
\begin{aligned}
& {\cal V}^{\s_M}(\delta) \times \Xi \to H^s(\T_1) \, , \quad (\frak x, \nu) \mapsto a_{- k}^{ext}(\frak x; \nu; d_\bot \Psi)\,, 
\qquad 1 \le k \le  M \, ,
\end{aligned}
$$
are ${\cal C}^\infty$. They satisfy
$ a_{- 1}^{ext}(\frak x; \nu; d_\bot \Psi) = 
a_{-1}^\Psi(\mathfrak x; \ac ) $ 
(cf. Theorem \ref{modified Birkhoff map}-${\bf (AE1)}$) and 
 for any $\alpha \in \N^{\mathbb S_+}$, $\widehat{\frak x}_1, \ldots, \widehat{\frak x}_l \in E_{\s_M}$,  and
 $(\frak x, \nu) \in {\cal V}^{\s_M}(\delta) \times \Xi $,
\begin{equation}\label{negativi 13}
\begin{aligned}
& \|\partial_\nu^\alpha a_{- k}^{ext}(\frak x; \nu; d_\bot \Psi) \|_{H^s_x} \lesssim_{s, M, \alpha} 1\,, \\
& \|\partial_\nu^\alpha d^l a_{- k}^{ext}(\frak x; \nu; d_\bot \Psi)[\widehat{\frak x}_1, \ldots, \widehat{\frak x}_l] \|_{H^s_x} \lesssim_{s, M, l, \alpha} \prod_{j = 1}^l \| \widehat{\frak x}_j \|_{E_{\s_M}}\, .
\end{aligned}
\end{equation}
$(ii)$ For any $0 \leq s \leq M + 2 $, the map 
$$
{\cal R}_M^{ext}(\cdot , \cdot; d_\bot \Psi) : {\cal V}^{\sigma_M}(\delta) \times \Xi \to {\cal B}(H_\bot^{- s}(\T_1), H^{M + 1 - s}(\T_1))
$$
is ${\cal C}^\infty$ and satisfies, 
for any $\alpha \in \N^{\mathbb S_+}$, 
$\widehat{\frak x}_1, \ldots, \widehat{\frak x}_l \in E_{\s_M}$, $\widehat w \in H^{- s}_\bot(\T_1)$, and
$(\frak x, \nu) \in {\cal V}^{\s_M}(\delta) \times \Xi $,
\begin{equation}\label{negativi 15}
\begin{aligned}
& \|\partial_\nu^\alpha 
 {\cal R}_M^{ext}(\frak x; \nu; d_\bot \Psi)[\widehat w] \|_{H^{M + 1- s}_x} \lesssim_{ M, \alpha } \| \widehat w\|_{H_x^{- s}}\,, \\
& \|\partial_\nu^\alpha d^l {\cal R}_M^{ext}(\frak x; \nu; d_\bot \Psi)[\widehat{\frak x}_1, \ldots, \widehat{\frak x}_l][\widehat w] \|_{H^{M + 1 - s}_x} \lesssim_{s, M, l, \alpha} \| \widehat w\|_{H^{- s}_x} \prod_{j = 1}^l \| \widehat{\frak x}_j \|_{E_{\s_M}}\,. \\
\end{aligned}
\end{equation}
$(iii)$ For any $s \geq 0$,  
the map 
$$
{\cal R}_M^{ext}(\cdot , \cdot; d_\bot \Psi) : 
\big( {\cal V}^{\sigma_M}(\delta) \cap {\cal E}_{s + \sigma_M} \big)
 \times \Xi \to {\cal B}(H_\bot^{ s}(\T_1), H^{M + 1 + s}(\T_1))
$$
is ${\cal C}^\infty$ and satisfies for any $\alpha \in \N^{\mathbb S_+}$, $\widehat{\frak x}_1, \ldots, \widehat{\frak x}_l \in E_{s + \sigma_M}$, 
$\widehat w \in H^{s}_\bot(\T_1)$, and
$(\frak x, \nu) \in 
\big( {\cal V}^{\sigma_M}(\delta) \cap {\cal E}_{s + \sigma_M} \big) \times \Xi $,
\begin{equation}\label{negativi 16}
\begin{aligned}
& \|\partial_\nu^\alpha  {\cal R}_M^{ext}(\frak x; \nu; d_\bot \Psi)[\widehat w] 
\|_{H^{M + 1+ s}_x} \lesssim_{ s, M,  \alpha} 
\| \widehat w\|_{H_x^s} + \| \frak x\|_{E_{s + \sigma_M}} \| \widehat w\|_{L^2_x} \, , \\
& \|\partial_\nu^\alpha d^l \big( {\cal R}_M^{ext}(\frak x; \nu; d_\bot \Psi)[\widehat w] \big)[\widehat{\frak x}_1, \ldots, \widehat{\frak x}_l] \|_{H^{M + 1+ s}_x} \lesssim_{ s, M,  l, \alpha } \| \widehat w\|_{H_x^s} \prod_{j = 1}^l \| \widehat{\frak x}_j\|_{E_{\sigma_M}}  \\
& \quad +  \| \widehat w\|_{L^2_x} \Big( \sum_{j = 1}^l \| \widehat{\frak x}_j\|_{E_{s + \sigma_M}} \prod_{i \neq j} \| \widehat{\frak x}_i\|_{E_{\sigma_M}} + \| \frak x \|_{E_{s + \sigma_M}} \prod_{j = 1}^l \| \widehat{\frak x}_j \|_{E_{\sigma_M}} \Big)\,. 
\end{aligned}
\end{equation}
\end{corollary}

\begin{proof}
By Theorem \ref{modified Birkhoff map}-${\bf (AE2)}$, 
for any $(\frak x, \nu) \in {\cal V}^1 (\delta) \times \Xi $,  the operator 
$d_\bot \Psi_\nu(\frak x)^\top : H_0^1(\T_1) \to H^1_\bot(\T_1)$ is bounded and for any 
$ M \geq 1 $ and $  \widehat q \in H_0^1(\T_1) $,   
$d_\bot \Psi_\nu(\frak x)^\top [\widehat{q} \, ]$ 
admits the expansion of the form
\begin{equation}\label{negativi 8}
\begin{aligned}
& d_\bot \Psi_\nu(\frak x)^\top [\widehat{q} \, ] = \Pi_\bot \widehat q + \Pi_\bot \sum_{k = 1}^M a_{- k}^{d \Psi^\top}(\frak x; \nu) \partial_x^{- k} \widehat q + {\cal R}_M^{(2)}(\frak x; \nu)[\widehat q \, ] \, ,  \\
& {\cal R}^{(2)}_M (\frak x; \nu)[ \widehat q \,  ] 
:= \Pi_\bot\sum_{k = 1}^M  (\partial_x^{- k} w) \, {\cal A}_{- k}^{d \Psi^\top}(\frak x; \nu)[\widehat q] +   {\cal R}_M^{d \Psi^\top}(\frak x; \nu)[\widehat q \, ] \, .
\end{aligned}
\end{equation}
For $ \s_M \geq M + 1$, the map ${\cal R}_M^{(2)} : {\cal V}^{\s_M}(\delta) \times \Xi \to {\cal B}(H_0^1(\T_1), H_\bot^{M + 2}(\T_1))$ is ${\cal C}^\infty$ and   
by Theorem \ref{modified Birkhoff map}-${\bf(Est2)}$,
satisfies for any $ \a \in \N^{\Splus} $
and $\widehat{\frak x}_1, \ldots, \widehat{\frak x}_l \in E_{\s_M}$
\begin{equation}\label{negativi 9}
\begin{aligned}
& \| \partial_\nu^\alpha {\cal R}_M^{(2)}(\frak x; \nu)[\widehat q] \|_{H^{M + 2}_x} \lesssim_{M, \alpha} \| \widehat q\|_{H^1_x}\,, \\
& \| \partial_\nu^\alpha d^l {\cal R}_M^{(2)}(\frak x; \nu)[\widehat{\frak x}_1, \ldots, \widehat{\frak x}_l][\widehat q] \|_{H^{M + 2}_x} \lesssim_{M, l, \alpha} \| \widehat q\|_{H^1_x} \prod_{j = 1}^l \| \widehat{\frak x}_j\|_{E_{\s_M}}\,. 
\end{aligned}
\end{equation}
Now consider the transpose operator $\big(d_\bot \Psi_\nu(\frak x)^\top \big)^\top : H^{- 1}_\bot(\T_1) \to H_0^{- 1}(\T_1)$. It defines an extension
of  $ d_\bot \Psi_\nu(\frak x) $ to $H^{- 1}_\bot(\T_1)$,
which we denote again by $d_\bot \Psi_\nu(\frak x)$. 
By \eqref{negativi 8}, for any $\widehat w \in H_\bot^{- 1}(\T_1)$, one has
\begin{equation}\label{negativi 10}
d_\bot \Psi_\nu(\frak x)[\widehat w] = \widehat w +
 \sum_{k = 1}^M (- 1)^k \partial_x^{- k}\big( a_{-k}^{d \Psi^\top}(\frak x; \nu) \widehat w \big)  + 
{\cal R}_M^{(2)}(\frak x; \nu)^\top [\widehat w]\,.
\end{equation}
Since each function $ a_{-k}^{d \Psi^\top}(\frak x; \nu) $ is 
${\cal C}^\infty$ and 
the operator ${\cal R}_M^{(2)}(\frak x; \nu)^\top : H^{- M - 2}_\bot(\T_1) \to H_0^{- 1}(\T_1)$ is bounded, 
the right hand side of \eqref{negativi 10}
defines a linear operator in 
${\cal B} ( H^{- M - 2}_0(\T_1), E_{- M - 2}) $, which we also
denote by $d \Psi_\nu(\frak x)$. 
By \eqref{composizione simboli omogenei}, the expansion 
\eqref{negativi 10} yields one of the form \eqref{negativi 11} 
where by \eqref{negativi 9} and 
Theorem \ref{modified Birkhoff map}-${\bf (Est2)}$,  the remainder 
$ {\cal R}_M^{ext}(\frak x; \nu ; d \Psi^\top)$ satisfies
 for any $\alpha \in \N^{\mathbb S_+}$, $\widehat{\frak x}_1, \ldots, \widehat{\frak x}_l \in E_{\s_M}$, 
 and $\widehat w \in H^{-M  - 2}_0(\T_1)$
\begin{equation}\label{pallone 1}
\begin{aligned}
& \|\partial_\nu^\alpha 
 {\cal R}_M^{ext}(\frak x; \nu; d_\bot \Psi)[\widehat w] \|_{H^{- 1}_x} \lesssim_{ M, \alpha } \| \widehat w\|_{H_x^{- M - 2}}\,, \\
& \|\partial_\nu^\alpha d^l {\cal R}_M^{ext}(\frak x; \nu; d_\bot \Psi)[\widehat{\frak x}_1, \ldots, \widehat{\frak x}_l][\widehat w] \|_{H^{- 1}_x} \lesssim_{M, l, \alpha} \| \widehat w\|_{H_x^{- M - 2}} \prod_{j = 1}^l \| \widehat{\frak x}_j \|_{E_{\s_M}}\,. \\
\end{aligned}
\end{equation}
The restriction of the expansion \eqref{negativi 10} to $L^2_\bot (\T_1)$  coincides with the one of
$ d_\bot \Psi_\nu (\frak x)[\widehat w] $, obtained  by
differentiating  \eqref{Psi-espl} (see \eqref{pr-exp}). 
It then follows from the uniqueness of an expansion of this form that 
$$
\begin{aligned}
&a_{- k}^{ext}(\frak x; \nu; d_\bot \Psi) = a_{-k }^\Psi(\mathfrak x; \ac ) \, , \quad k = 1, \ldots, M\,, \\
&  {\cal R}_M^{ext}(\mathfrak x; \ac; d_\bot \Psi)[\widehat w]= \sum_{k = 1}^M 
(\partial_x^{- k} w) d_\bot a_{-k }^\Psi(\mathfrak x; \ac )[\widehat w] 
+ d_\bot {\cal R}_{M }^\Psi(\mathfrak x; \ac ) [\widehat w] , \quad \forall \widehat w \in L^2_\bot(\T_1) \, . 
\end{aligned}
$$
The claimed estimates \eqref{negativi 13} and 
 \eqref{negativi 16} thus follow by 
Theorem \ref{modified Birkhoff map}-${\bf (Est1)}$. In particular,  
for any $\alpha \in \N^{\mathbb S_+}$, $\widehat{\frak x}_1, \ldots, \widehat{\frak x}_l \in E_{\s_M}$, and $\widehat w \in L^2_\bot(\T_1)$,  
\begin{equation}\label{negativi 14}
\begin{aligned}
& \|\partial_\nu^\alpha {\cal R}_M^{ext}(\frak x; \nu; d_\bot \Psi)[\widehat w] \|_{H^{M + 1}_x} \lesssim_{ M, \alpha } \| \widehat w\|_{L^2_x}\,, \\
& \|\partial_\nu^\alpha d^l {\cal R}_M^{ext}(\frak x; \nu; d_\bot \Psi)[\widehat{\frak x}_1, \ldots, \widehat{\frak x}_l][\widehat w] \|_{H^{M + 1}_x} \lesssim_{M, l, \alpha} \| \widehat w\|_{L^2_x} \prod_{j = 1}^l \| \widehat{\frak x}_j \|_{E_{\s_M}}\,. 
\end{aligned}
\end{equation}
The claimed estimates \eqref{negativi 15} are then obtained by interpolating between \eqref{pallone 1} and  \eqref{negativi 14}. 
\end{proof}

\subsection{Expansions of linearized Hamiltonian vector fields}\label{espansione linearized}

For any Hamiltonian of the form
$ P(u) = \int_{\T_1} f(x, u, u_x)\, d x$ with a $ {\cal C}^\infty  $-smooth density 
\be\label{densityf}
f : \T_1 \times \R \times \R \mapsto \R \, , \ 
(x, \zeta_0, \zeta_1 ) \mapsto f (x, \zeta_0, \zeta_1 ) 	\, , 
\ee
define
\begin{equation}\label{definizione perturbazione quasilin}
{\cal P}  := P \circ \Psi_\ac \, , \quad 
{\cal P} (\theta, y, w; \ac):= P (\Psi_\ac (\theta, y, w))   
\end{equation}
where $\Psi_\ac$ is the coordinate transformation of Theorem \ref{modified Birkhoff map}. 
As a first result, we provide an expansion 
of the linearized Hamiltonian vector field 
$ \partial_x d_\bot  \nabla_w \mathcal P $. 
 \begin{lemma}\label{differential nabla perturbation-true}
 {\bf (Expansion of $ \partial_x d_\bot \nabla_w {\cal P} $)}
 Let $ P(u) = \int_{\T_1} f(x, u, u_x)\, d x$ with $f \in {\cal C}^{\infty}(\T_1 \times \R \times \R)$. 
For any $ M \in \N $ there is  $ \sigma_M > 0 $  so that 
for any  $\mathfrak x  \in {\cal V}^{\s_M}(\delta)$
and $\ac \in \Xi$, the operator 
$ \partial_x d_\bot \nabla_w {\cal P}(\mathfrak x; \ac )$ admits an expansion of the form
\begin{equation}\label{pseudodifferential expansion of d bot nabla P} 
\partial_x d_\bot \nabla_w {\cal P}(\mathfrak x; \ac)[\cdot ] =
{\Pi}_\bot  \sum_{k = 0}^{M+3}  a_{3 -k}(\mathfrak x; \ac ; \partial_x d_\bot \nabla_w {\cal P}) \, \partial_x^{3 - k}[\cdot  ]  \,    +  {\cal R}_M(\mathfrak x; \ac ; \partial_x d_\bot \nabla_w {\cal P})[\cdot ]  
\end{equation}
with the following properties:
\begin{enumerate}
\item \label{pr1} For any $s \geq 0 $, the maps
$$
( {\cal V}^{\s_M}(\delta)
\cap {\mathcal E}_{s+\s_M} )  \times \Xi  \to H^s(\T_1) \, ,  \quad
(\mathfrak x ; \ac) \mapsto  a_{3 - k}( \mathfrak x; \ac ; \partial_x d_\bot \nabla_w {\cal P} ) \, , \qquad 0 \le k \le  M + 3 \, ,
$$
are  $ {\cal C}^\infty $, and satisfy for any 
$ \alpha \in \N^{\mathbb S_+}$,
$\widehat {\mathfrak x}_1, \ldots, \widehat { \mathfrak x}_l \in E_{s + \sigma_M}$, and $(\mathfrak x, \ac)  \in  \big({\cal V}^{\s_M}(\delta) \cap {\mathcal E}_{s+\s_M}\big) \times \Xi$, 
 \begin{align} \label{le1:es1}
 & \!\!\!\!\!\! \!\!\!\!\!\!  \| \partial_\ac^\alpha a_{3 - k}( \mathfrak x; \ac ; \partial_x d_\bot \nabla_w {\cal P}) \|_{H^s_x} \lesssim_{s, M, \alpha} 1 + \| w \|_{H^{s + \sigma_M}_x}\,, \\
 & \!\!\!\!\!\! \!\!\!\!\!\! 
 \| \partial_\ac^\alpha d^l a_{3 - k}( \mathfrak x; \ac; \partial_x d_\bot \nabla_w {\cal P})[\widehat{\mathfrak x}_1, \ldots, \widehat{\mathfrak x}_l]\|_{H^s_x}   \lesssim_{s, M, l, \alpha} \sum_{j = 1}^l \big( \| \widehat{\mathfrak x}_j\|_{E_{s + \sigma_M}} \prod_{n \neq j}  \| \widehat{\mathfrak x}_n\|_{E_{\sigma_M}}\big) + \| w \|_{H^{s + \sigma_M}_x} \prod_{j = 1}^l \|\widehat{\mathfrak x}_j \|_{E_{\sigma_M}}.  \nonumber 
 \end{align}

\item \label{pr3} 
For any $0 \leq s \leq M + 1$,   the map 
  $$
{\cal V}^{\sigma_M}(\delta)  \times \Xi \to {\cal B}(H^{-s}(\T_1), 
H_\bot^{M + 1 - s}(\T_1)) \, , \quad (\frak x, \nu) \mapsto {\cal R}_M(\mathfrak x; \ac ; \partial_x d_\bot \nabla_w {\cal P}) \, , 
  $$
  is ${\cal C}^\infty$ and satisfies for any 
    $\alpha \in \N^{\mathbb S_+}$, 
  $\widehat{\frak x}_1, \ldots, \widehat{\frak x}_l \in E_{ \sigma_M}$,
  $(\frak x, \ac) \in {\cal V}^{\sigma_M}(\delta) \times \Xi$, 
  and $\widehat w \in H^{-s}_\bot (\T_1)$,
\be\label{le1:es3}
\begin{aligned}
 &  \| \partial_\nu^\alpha {\cal R}_M(\mathfrak x; \ac ; \partial_x d_\bot \nabla_w {\cal P})[\widehat w]\|_{H^{ M + 1 - s}_x} \lesssim_{s, M, \alpha} \| \widehat w\|_{H^{-s}_x} \,,  \\
 & \| \partial_\nu^\alpha d^l \big( {\cal R}_M(\mathfrak x; \ac ; \partial_x d_\bot \nabla_w {\cal P})[\widehat w] \big)[\widehat{\frak x}_1, \ldots, \widehat{\frak x}_l]\|_{H^{M + 1 -s}_x} \lesssim_{s, M, l, \alpha} \| \widehat w\|_{H^{-s}_x} \prod_{j = 1}^l 
 \| \widehat{\frak x}_j \|_{E_{\sigma_M}} \,. \\
   \end{aligned}
\ee
 \item \label{pr2}
For any $ s \geq 0 $, the map 
 $$
 ( {\cal V}^{\sigma_M}(\delta) \cap {\cal E}_{ s + \sigma_M} ) \times \Xi \to {\cal B}(H^s(\T_1), H_\bot^{s + M + 1}(\T_1)) \, , \quad (\frak x, \nu) \mapsto{\cal R}_M(\mathfrak x; \ac ; \partial_x d_\bot \nabla_w {\cal P}) \, , 
  $$
  is ${\cal C}^\infty$ and satisfies for any 
 $\alpha \in \N^{\mathbb S_+}$,
 $\widehat{\frak x}_1, \ldots, \widehat{\frak x}_l \in E_{s + \sigma_M} $,   
  $(\frak x, \ac) \in 
  ({\cal V}^{\sigma_M}(\delta) \cap {\cal E}_{ s + \sigma_M})
  \times \Xi$, and  $\widehat w \in H^s_\bot (\T_1)$,
\be\label{le1:es2}
  \begin{aligned}
 &  \| \partial_\nu^\alpha {\cal R}_M(\mathfrak x; \ac ; \partial_x d_\bot \nabla_w {\cal P})[\widehat w]\|_{H^{s + M + 1}_x} \lesssim_{s, M, \alpha} \| \widehat w\|_{H^s_x} + \| w\|_{H^{s + \sigma_M}_x} \| \widehat w\|_{L^2_x}\,,  \\
 & \| \partial_\nu^\alpha d^l \big( {\cal R}_M(\mathfrak x; \ac ; \partial_x d_\bot \nabla_w {\cal P})[\widehat w] \big)[\widehat{\frak x}_1, \ldots, \widehat{\frak x}_l]\|_{H^{s + M + 1}_x} \lesssim_{s, M, l, \alpha} \| \widehat w\|_{H^s_x} \prod_{j = 1}^l \| \widehat{\frak x}_j \|_{E_{\sigma_M}}  \\
 & \quad + \| \widehat w\|_{L^2_x} \big( \| w\|_{H^{s + \sigma_M}_x} \prod_{j = 1}^l \| \widehat{\frak x}_j\|_{E_{\sigma_M}} + \sum_{j = 1}^l \| \widehat{\frak x}_j \|_{E_{s + \sigma_M}} \prod_{i \neq j} \| \widehat{\frak x}_i \|_{E_{\sigma_M}} \big)  \,.
  \end{aligned}
\ee
\end{enumerate}
\end{lemma}
 
 \begin{remark}\label{rem:3.5}
The coefficient $a_3$ in 
\eqref{pseudodifferential expansion of d bot nabla P}
can be computed as
$ a_3 (\mathfrak x; \ac ; \partial_x d_\bot \nabla_w {\cal P} ) 
= - (\pa_{\zeta_1}^2 f)(x, u, u_x) \big|_{u = \Psi_\ac(\mathfrak x)} $.    
\end{remark}

\begin{proof}  Differentiating \eqref{definizione perturbazione quasilin} we have that
\be\label{gradcalP0}
\nabla {\mathcal P} ( \mathfrak x ; \ac) = (d \Psi_\ac ( \mathfrak x))^\top \big[ \nabla P (\Psi_\ac (\mathfrak x ) ) \big] \, , 
\ee
where,  recalling \eqref{densityf},  
\be\label{nablaP}
\nabla P ( u ) = 
\Pi_0^\bot \big[ 
(\pa_{\zeta_0} f) (x, u, u_x )-  
\big( (\pa_{\zeta_1} f ) (x, u, u_x) \big)_x \big]  
\ee
and $ \Pi_0^\bot $ is the $L^2$-orthogonal projector of
$L^2(\T_1)$ onto $L^2_0(\T_1)$. 
By \eqref{gradcalP0}, the $w-$component 
$\nabla_w {\mathcal P}( \mathfrak x; \ac )$ of 
$\nabla {\mathcal P}( \mathfrak x; \ac )$ equals
$  (d_\bot\Psi_\ac ( \mathfrak x))^\top 
\big[ \nabla P (\Psi_\ac (\mathfrak x ) ) \big] $.   
Differentiating it  with respect to $w$ in direction $\widehat w$ then yields
\be\label{dgradcalP}
d_\bot \nabla_w {\mathcal P} ( \mathfrak x; \ac ) [\widehat w] 
= (d_\bot \Psi_\ac ( \mathfrak x))^\top \big[ d \nabla P (\Psi_\ac (\mathfrak x ) ) \big[ d_\bot \Psi_\ac(\mathfrak x) [  \widehat w ]\big] \big]  + \big( d_\bot  (d_\bot \Psi_\ac ( \mathfrak x))^\top  [\widehat w]\big)  
\big[ \nabla P (\Psi_\ac (\mathfrak x ) ) \big]\, .
\ee
{\em Analysis  of the first term on the right hand side of \eqref{dgradcalP}:} 
Evaluating the differential $d \nabla P(u)$ of \eqref{nablaP}
at $u = \Psi_\ac(\mathfrak x)$, one gets
\begin{equation}\label{dp6}
\begin{aligned}
& d (\nabla P) (\Psi_\ac(\mathfrak x))[h]   = \Pi_0^\bot \big( b_2(\mathfrak x; \ac) \partial_x^2 h +b_1(\mathfrak x; \ac) \partial_x h + b_0(\mathfrak x ; \ac) h \big)  \\ 
& b_2(\mathfrak x; \nu)  := - \pa_{\zeta_1}^2  f (x, u, u_x)\Big|_{u = \Psi_\ac(\mathfrak x)}\,, \qquad b_1(\mathfrak x; \ac) := (b_2(\mathfrak x ; \ac))_x \, , \\
& b_0(\mathfrak x; \ac) :=  \big( (\pa_{\zeta_0}^2 f)(x, u, u_x)  - 
  \big(  (\pa^2_{\zeta_0\zeta_1} f) (x, u, u_x) \big)_x \big) \big|_{u = \Psi_\ac(\mathfrak x)}\,.
  \end{aligned}
  \end{equation}
  By Lemma \ref{Moser norme pesate} and 
  Theorem \ref{modified Birkhoff map} one infers that
  for any $s \geq 0$, the maps
  $$
  ({\cal V}^{3}(\delta ) \cap {\mathcal E}_{s+3}) \times \Xi \to H^s_x \, , \quad (\mathfrak x; \ac) \mapsto b_i(\mathfrak x; \ac) \, ,
  \quad i = 0,1,2  \, , 
  $$
  are ${\cal C}^\infty$ and satisfy for any 
  $\alpha \in \N^{\mathbb S_+}$, 
  $\widehat{\frak x}_1, \ldots, \widehat{\frak x}_l \in E_{s + 3} $,
  and  
  $(\mathfrak x, \ac) \in 
  \big({\cal V}^{3} (\delta )  \cap {\mathcal E}_{s+3}\big) 
  \times \Xi$,
  \begin{equation}\label{stima bi nabla P}
  \begin{aligned}
 &  \| \partial_\nu^\alpha b_i(\mathfrak x; \ac)\|_{H^s_x} \lesssim_{s, \alpha} 1 + \| w \|_{H^{s + 3}_x}\,, \\
 & \| \partial_\nu^\alpha d^l b_i(\frak x; \nu)[\widehat{\frak x}_1, \ldots, \widehat{\frak x}_l] \|_{H^s_x} \lesssim_{s, l, \alpha} \sum_{j = 1}^l \| \widehat{\frak x}_j\|_{E_{s + 3}} \prod_{i \neq j} \| \widehat{\frak x}_i\|_{E_{3}} + \| w \|_{H^{s + 3}_x} \prod_{j = 1}^l \| \widehat{\frak x}_j \|_{E_3}\,. 
  \end{aligned} 
  \end{equation}  
  By Corollary \ref{corollary transpose negative sobolev}
  (expansion of $ (d_\bot \Psi_\ac)^\top $), 
  Corollary \ref{corollary Birkhoff negative sobolev}
  (expansion of $ d_\bot \Psi_\ac $),  
  \eqref{stima bi nabla P} (estimates of $b_i$), 
\eqref{dp6} (formula for $d (\nabla P) (\Psi_\ac(\mathfrak x))$),
  and 
Lemma \ref{composizione simboli omogenei} (composition), one obtains
the expansion 
  \begin{equation}\label{espansione d bot Psi nabla sec pezzo-0}
 \partial_x  (d_\bot \Psi_\ac ( \mathfrak x))^\top 
  \big[ d \nabla P (\Psi_\ac (\mathfrak x ) ) 
  \big[ d_\bot \Psi_\ac(\mathfrak x) [  \cdot ] \big] \big]
  = \Pi_\bot \sum_{k = 0}^{M + 3} a_{3 - k}^{(1)}(\mathfrak x; \ac) \partial_x^{3 - k}  + R_1(\mathfrak x; \ac)
  \end{equation}
  where  $ a_{3}^{(1)}(\frak x; \nu) =  b_2(\mathfrak x; \nu)  $, 
  the functions $a_{3 - k}^{(1)}(\frak x; \nu)$, $k = 0, \ldots, M+ 3 $, 
  and the remainder $R_1(\mathfrak x; \ac)$ satisfy the claimed properties
  \ref{pr1}-\ref{pr2} of the lemma, in particular  \eqref{le1:es1}-\eqref{le1:es2}.
 \\[1mm]  
  {\em Analysis  of the second term on the right hand side of \eqref{dgradcalP}:}  
  Since $d \Psi_\nu(\frak x)$ is symplectic,
 $d \Psi_\nu(\frak x)^\top = {\cal J}^{- 1} d \Psi_\nu(\frak x)^{- 1} \partial_x $  
 where $ {\cal J} $ is the Poisson operator defined in \eqref{Poisson struture}, implying that for any $\widehat w \in H^1_\bot(\T_1) $, 
$$
\begin{aligned}
d_\bot \big(d \Psi_\nu(\frak x)^\top \big) [\widehat w] & = - {\cal J}^{- 1} d \Psi_\nu(\frak x)^{- 1} \big(d_\bot d \Psi_\nu(\frak x)[\widehat w] \big) d \Psi_\nu(\frak x)^{- 1} \partial_x  \\
& = - d \Psi_\nu(\frak x)^\top \partial_x^{- 1} d \big(d_\bot \Psi_\nu(\frak x)[\widehat w] \big)\big[ {\cal J} d \Psi_\nu(\frak x)^\top \cdot \big]\,. 
\end{aligned}
$$
By this identity we get 
\begin{equation}\label{Ucraina 1}
\partial_x \big( d_\bot  (d_\bot \Psi_\ac ( \mathfrak x))^\top  [\cdot]\big)  
\big[ \nabla P (\Psi_\ac (\mathfrak x ) ) \big]  = - \partial_x d \Psi_\nu(\frak x)^\top \partial_x^{- 1} d \big(d_\bot \Psi_\nu(\frak x)[\cdot] \big)\big[ {\cal J} d \Psi_\nu(\frak x)^\top \nabla P (\Psi_\ac (\mathfrak x ) )  \big]\,. 
\end{equation}
Arguing as for the first term on the right hand side of \eqref{dgradcalP}
(cf. \eqref{espansione d bot Psi nabla sec pezzo-0}) one gets an expansion of the form
  \begin{equation}\label{espansione d bot Psi nabla sec pezzo a}
 \partial_x \big( d_\bot  (d_\bot \Psi_\ac ( \mathfrak x))^\top  [\cdot]\big)  
\big[ \nabla P (\Psi_\ac (\mathfrak x ) ) \big]
  = \Pi_\bot \sum_{k = 3}^{M + 3} a_{3 - k}^{(2)}(\mathfrak x; \ac) \partial_x^{3 - k} + R_2(\mathfrak x; \ac)
  \end{equation}
  where the functions $a_{3 - k}^{(2)}(\frak x; \nu)$, $k = 3, \ldots, M+ 3 $, and the remainder $R_2(\mathfrak x; \ac)$ satisfy the 
  claimed properties
  \ref{pr1}-\ref{pr2} of the lemma, in particular  \eqref{le1:es1}-\eqref{le1:es2}.\\[1mm] 
{\em Conclusion:} By \eqref{dgradcalP} and the above analysis of the 
expansions \eqref{espansione d bot Psi nabla sec pezzo-0}
and \eqref{espansione d bot Psi nabla sec pezzo a},
the lemma and Remark \ref{rem:3.5} follow. 
 \end{proof} 

As a second result of this section 
we derive an expansion for the linearized Hamiltonian vector field 
$ \partial_x d_\bot  \nabla_w {\mathcal  H}^{kdv} $
where ${\mathcal H}^{kdv}(\cdot; \ac)= H^{kdv} \circ \Psi_\ac$
(cf. Theorem \ref{modified Birkhoff map}-({\bf AE3})).
\begin{lemma}\label{differential nabla kdv remainder-true}
{\bf (Expansion of $\partial_x d_\bot \nabla_w {\cal H}^{kdv} $)}
For any $M \in \N$  there is $ \sigma_M \ge M+1$ so that, for any 
$(\mathfrak x, \ac)  \in \mathcal V^{\sigma_M } (\delta)\times \Xi $, 
the operator  
$\partial_x d_\bot \nabla_w {\cal H}^{kdv}(\mathfrak x; \ac)$ admits an expansion of the form
\begin{equation}\label{pseudodifferential expansion of remainder mathcal R_Nkdv} 
\begin{aligned}
& \partial_x d_\bot \nabla_w {\cal H}^{kdv}(\mathfrak x; \ac)[\cdot] = \partial_x \Omega^{kdv}(D ;\ac) [\cdot] + 
 \partial_x d_\bot \nabla_w {\cal R}^{kdv}(\mathfrak x; \ac)[\cdot]\,, \\
& \partial_x d_\bot \nabla_w {\cal R}^{kdv}(\mathfrak x; \ac)[ \cdot ] =
 \, {\Pi}_\bot  \sum_{k = 0}^{M + 1}  a_{1  -k}(\mathfrak x ; \ac ; \partial_x d_\bot \nabla_w {\cal R}^{kdv}) \, \partial_x^{1 - k} [\cdot ] +  
{\cal R}_M (\mathfrak x; \ac ;  \partial_x d_\bot \nabla_w {\cal R}^{kdv})[\cdot] \, , 
\end{aligned}
\end{equation}
with the following properties:
\begin{enumerate}
\item For any $s \geq 0 $, the maps
$$
\begin{aligned}
& 
( {\cal V}^{\sigma_M}(\delta) \cap {\mathcal E}_{s + \s_M} ) \times \Xi \to H^s (\T_1),  \,\,
(\mathfrak x , \ac) \mapsto  
a_{1 -k}(\mathfrak x ; \ac ; \partial_x d_\bot \nabla_w {\cal R}^{kdv})  \,, \qquad 0 \le  k \le M + 1 \, ,
\end{aligned}
$$
are $ {\cal C}^\infty $ and satisfy for any 
$\alpha \in \N^{\mathbb S_+}$,
$\widehat{\mathfrak x}_1, \ldots, \widehat{\mathfrak x}_l \in 
E_{s + \sigma_M}$, and
$(\mathfrak x, \ac)  \in 
({\cal V}^{\sigma_M}(\delta) \cap {\mathcal E}_{s + \s_M}) \times \Xi$,
\begin{equation}\label{stima a - k partial x d nabla R kdv}
\begin{aligned}
&  \|\partial_\ac^\alpha  a_{1 - k}(\mathfrak x; \ac ; \partial_x d_\bot \nabla_w {\cal R}^{kdv})\|_{H^s_x} \lesssim_{s, k, \alpha} \| y \| + \| w \|_{H_x^{s + \sigma_M}}\,, \quad \\
&   \| d^l \partial_\ac^\alpha a_{1 - k}(\mathfrak x; \ac ; \partial_x d_\bot \nabla_w {\cal R}^{kdv} )[\widehat{\mathfrak x}_1, \ldots, \widehat{\mathfrak x}_l]\|_{H^s_x} 
\lesssim_{s, k, l, \alpha} \sum_{j = 1}^l \big( \| \widehat{\mathfrak x}_j\|_{E_{s + \sigma_M}} \prod_{n \neq j} \| \widehat{\mathfrak x}_n\|_{E_{\sigma_M}}\big)  \\
& \qquad  \qquad \qquad \qquad \qquad \qquad \qquad \qquad \qquad \qquad +
 ( \| y \| + \| w \|_{H_x^{s + \sigma_M}}) \prod_{j = 1}^l \| \widehat{\mathfrak x}_j\|_{E_{\sigma_M}}\,. 
 \end{aligned}
\end{equation}
\item For any $ 0 \leq s \leq M+1$, the map 
$$
{\cal R}_M(\cdot ; \cdot ; \pa_x d_\bot \nabla_w {\cal R}^{kdv}) : 
{\cal V}^{\sigma_M}(\delta)  
 \times \Xi \to  
 \mathcal B (H_\bot^{-s }(\T_1) ,H_\bot^{M + 1 - s}(\T_1)) 
$$
 is ${\cal C}^\infty$ and satisfies for any 
 $\alpha \in \N^{\mathbb S_+}$,
 $\widehat{\frak x}_1, \ldots, \widehat{\frak x}_l \in E_{ \sigma_M}$,
 $(\frak x, \ac) \in  {\cal V}^{\sigma_M}(\delta) \times \Xi$, and
 $\widehat w \in H^{-s}_\bot(\T_1)$, 
\begin{align}\label{OpR2kdva}
&  \|  \partial_\ac^\alpha{\cal R}_M (\mathfrak x; \ac ;  \pa_x d_\bot \nabla {\cal R}^{kdv})[\widehat w]\|_{H^{M + 1 -s}_x} 
\lesssim_{s, M, \a} 
( \| y \| +  \| w \|_{H^{ \sigma_M}_x}) \| \widehat w\|_{H^{-s}_x}   \,,  \\
& \label{OpR3kdva}
\| d^l \partial_\ac^\alpha {\cal R}_M (\mathfrak x; \ac ;  \pa_x d_\bot \nabla {\cal R}^{kdv}) [\widehat w][\widehat{\mathfrak x}_1, \ldots, \widehat{\mathfrak x}_l]\|_{H^{ M + 1 - s}_x}  \lesssim_{s, M, l,\a}  
\| \widehat w \|_{H^{-s}_x} \prod_{j = 1}^l \| \widehat{\frak x}_j \|_{E_{\sigma_M}} \,. 
\end{align}
\item 
For any $s \geq 0$, the map 
$$
\begin{aligned}
& {\cal R}_M( \cdot ; \cdot ; \pa_x d_\bot \nabla_w {\cal R}^{kdv}) : 
( {\cal V}^{\sigma_M}(\delta) \cap {\mathcal E}_{s + \s_M} ) 
 \times \Xi \to  \mathcal B (H_\bot^{s }(\T_1) ,H_\bot^{s + M + 1}(\T_1)) \, , 
 \end{aligned}
 $$
is ${\cal C}^\infty$ and satisfies for any 
$\alpha \in \N^{\mathbb S_+}$,
$\widehat{\frak x}_1, \ldots, \widehat{\frak x}_l \in E_{s + \sigma_M}$, 
$(\frak x, \ac) \in ({\cal E}_{s + \sigma_M} \cap {\cal V}^{\sigma_M}(\delta)) \times \Xi$, and $\widehat w \in H^s_\bot(\T_1)$, 

\be\label{OpR2kdv}
\begin{aligned}
&  \|  \partial_\ac^\alpha{\cal R}_M (\mathfrak x; \ac ;  \pa_x d_\bot \nabla {\cal R}^{kdv})[\widehat w]\|_{H^{s + M + 1}_x} \\
& \lesssim_{s, M, \a} 
( \| y \| +  \| w \|_{H^{s + \sigma_M}_x}) \| \widehat w\|_{L^2_x} +  ( \| y \| + \| w\|_{H^{\sigma_M}_x})  \| \widehat w \|_{H^{s}_x}  \,, 
\end{aligned}
\ee
\be\label{OpR3kdv}
\begin{aligned}
& \| d^l \partial_\ac^\alpha {\cal R}_M (\mathfrak x; \ac ;  \pa_x d_\bot \nabla {\cal R}^{kdv}) [\widehat w][\widehat{\mathfrak x}_1, \ldots, \widehat{\mathfrak x}_l]\|_{H^{s + M + 1}_x}  \lesssim_{s, M, l,\a}  
\| \widehat w \|_{H^s_x}  \prod_{j=1}^l  \| \widehat{\mathfrak x}_j \|_{E_{\sigma_M}}    \\
& 
\qquad + \| \widehat w \|_{L^2_x} \sum_{j = 1}^l 
\Big( \| \widehat{\mathfrak x}_j\|_{E_{s + \sigma_M}} \prod_{n \neq j}  \| \widehat{\mathfrak x}_n\|_{E_{\sigma_M}} \Big)  
+   
\| \widehat w \|_{L^2_x}   \| w \|_{H^{s+  \s_M}_x}  \prod_{j = 1}^l \|\widehat{\mathfrak x}_j \|_{E_{\sigma_M}} \,. 
\end{aligned}
\ee
\end{enumerate}
\end{lemma}

\begin{proof} 
Differentiating 
${\mathcal H}^{kdv}(\frak x; \ac) = 
H^{kdv}( \Psi_\ac(\frak x))$, we get
\be\label{gradcalPkdv}
\nabla_w {\mathcal H}^{kdv} ( \frak x; \nu  ) = 
(d_\bot \Psi_\nu ( \frak x))^\top 
\big[ \nabla H^{kdv} (\Psi_\nu (\frak x ) ) \big] 
\ee
where, recalling \eqref{hamiltoniana kdv0},  
\be\label{nablaPkdv}
\nabla H^{kdv} ( u ) = 
\Pi_0^\bot ( 3 u^2 - u_{xx}  \big) 
\ee
and $\Pi_0^\bot $ is the $L^2$-orthogonal projector onto $L^2_0(\T_1)$.
Differentiating \eqref{gradcalPkdv}
with respect to $w$ in direction $\widehat w$ we get 
\be\label{dgradcalPkdv}
\begin{aligned}
& d_\bot \nabla_w {\mathcal H}^{kdv} ( \frak x; \nu  ) [\widehat w ] = \\
& (d_\bot  \Psi_\nu ( \frak x))^\top  
\big[ 
d \nabla H^{kdv} (\Psi_\nu (\frak x ) ) 
[d_\bot \Psi_\nu(\frak x)[\widehat w ]]\big] 
+  \big( d_\bot  (d_\bot \Psi_\nu ( \frak x))^\top  [\widehat w ]\big) 
 \big[ \nabla H^{kdv} (\Psi_\nu (\frak x ) ) \big] \, .
 \end{aligned}
\ee
On the other hand, by \eqref{Hkdv}
$$ 
d_\bot \nabla_w {\mathcal H}^{kdv} ( \frak x ; \nu)  = 
\Omega^{kdv} (D ; \nu)  +  d_\bot \nabla_w {\mathcal R}^{kdv} ( \frak x ; \nu) 
$$
and by \eqref{Rkdv-cubic-der}
 $ d_\bot \nabla_w {\mathcal R}^{kdv} ( \theta, 0, 0 ; \nu) = 0 $, 
 implying that
\be\label{diffRdkv}
\begin{aligned}
& d_\bot \nabla_w {\mathcal H}^{kdv} (  \theta, 0, 0; \nu ) = \Omega^{kdv}(D; \nu)\,, \\
& d_\bot \nabla_w {\mathcal R}^{kdv} ( \frak x; \nu )  = 
 d_\bot \nabla_w {\mathcal H}^{kdv} ( \theta, y, w; \nu)   
- d_\bot \nabla_w {\mathcal H}^{kdv} (  \theta, 0, 0; \nu )    \, .
\end{aligned}
\ee
In order to obtain the expansion 
\eqref{pseudodifferential expansion of remainder mathcal R_Nkdv} it  thus suffices to 
expand 
$d_\bot \nabla_w {\mathcal H}^{kdv} ( \theta, y, w; \nu) ) [\widehat w ]$ and then subtract from it the expansion of
$d_\bot \nabla_w {\mathcal H}^{kdv} ( \theta, 0, 0; \nu) ) [\widehat w ]$. 
We analyze separately the two terms in  \eqref{dgradcalPkdv}. 
\\[1mm]
{\em Analysis  of the first term on the right hand side of \eqref{dgradcalPkdv}:}  
Evaluating the differential $d \nabla H^{kdv}(u)$
at $u = \Psi_\ac(\mathfrak x)$, one gets 
\begin{equation}\label{dp6c}
\begin{aligned}
& d (\nabla H^{kdv}) (\Psi_\ac(\mathfrak x))[h]  = 
\Pi_0^\bot \big(- \partial_x^2 h +  b_0(\frak x; \nu) h  \big) \, , \quad  b_0(\mathfrak x; \ac) 
:= 6 \Psi_\nu(\frak x)     \, .
  \end{aligned}
  \end{equation}
By Theorem \ref{modified Birkhoff map}-({\bf AE1}) and the estimates 
({\bf Est1}),  the function  $ b_0(\mathfrak x; \ac) $ satisfies,  for any $s \geq 0$, 
   \begin{equation}\label{stima bi nabla P-kdv}
  \begin{aligned}
 &  \| \partial_\nu^\alpha b_0 (\mathfrak x; \ac)\|_{H^s_x} \lesssim_{s, \alpha} 
 1 + \| w \|_{H^{s + 1}_x}\,, \\
 & \| \partial_\nu^\alpha d^l b_0(\frak x; \nu)[\widehat{\frak x}_1, \ldots, \widehat{\frak x}_l] \|_{H^s_x} \lesssim_{s, l, \alpha} \sum_{j = 1}^l \| \widehat{\frak x}_j\|_{E_{s + 1}} \prod_{i \neq j} \| \widehat{\frak x}_i\|_{E_{1}} + \| w \|_{H^{s + 1}_x} \prod_{j = 1}^l \| \widehat{\frak x}_j \|_{E_1}\,. 
  \end{aligned} 
  \end{equation}  
 By Corollary \ref{corollary transpose negative sobolev}
  (expansion of $ (d_\bot \Psi_\ac)^\top $), 
  Corollary \ref{corollary Birkhoff negative sobolev}
  (expansion of $ d_\bot \Psi_\ac $), 
  \eqref{stima bi nabla P-kdv} (estimates of $b_0$),  
\eqref{dp6c} 
\big(formula for $d (\nabla H^{kdv}) (\Psi_\ac(\mathfrak x))$\big),
 and 
Lemma \ref{composizione simboli omogenei} (composition), one obtains
the expansion    
  \begin{equation}\label{espansione d bot Psi nabla sec pezzo}
  \begin{aligned}
&  \partial_x  (d_\bot \Psi_\ac ( \mathfrak x))^\top 
  \big[ d \nabla H^{kdv} (\Psi_\ac (\mathfrak x ) ) 
  \big[ d_\bot \Psi_\ac(\mathfrak x) [  \cdot ] \big] \big] \\
 &  = \Pi_\bot \big( - \partial_{x}^3 - (a_{- 1}^\Psi(\frak x; \nu) 
 + a_{- 1}^{d \Psi^\top}(\frak x; \nu)) \partial_x^2 + 
 \sum_{k = 0}^{M + 1} a_{1 - k}^{(1)}(\mathfrak x; \ac) 
 \partial_x^{1 - k} \, \big)   + R_1(\mathfrak x; \ac) \\
  & \stackrel{\eqref{coefficiente - 1 d Psi Psi top}}{=} \Pi_\bot \big( - \partial_{x}^3  + \sum_{k = 0}^{M + 1} a_{1 - k}^{(1)}(\mathfrak x; \ac) \partial_x^{1 - k} \big)  + R_1(\mathfrak x; \ac)
  \end{aligned}
  \end{equation}
  where the functions $a_{1 - k}^{(1)}(\frak x; \nu)$, $k = 0, \ldots, M+ 1$ and the remainder $R_1(\mathfrak x; \ac)$  satisfy the properties
  \ref{pr1}--\ref{pr2} 
  stated in Lemma \ref{differential nabla perturbation-true}, in particular  \eqref{le1:es1}-\eqref{le1:es2}.
 \\[1mm]
 {\em Analysis  of the second term on the right hand side of \eqref{dgradcalPkdv}:}  
By \eqref{Ucraina 1} one has
$$
\partial_x \big( d_\bot  (d_\bot \Psi_\ac ( \mathfrak x))^\top  [\cdot]\big)  
\big[ \nabla H^{kdv} (\Psi_\ac (\mathfrak x ) ) \big]  = - \partial_x d \Psi_\nu(\frak x)^\top \partial_x^{- 1} d \big(d_\bot \Psi_\nu(\frak x)[\cdot] \big)\big[ {\cal J} d \Psi_\nu(\frak x)^\top \nabla H^{kdv} (\Psi_\nu (\frak x ) )  \big]\,. 
$$
Arguing as for the first term on the right hand side of 
\eqref{dgradcalPkdv}
 one obtains an expansion of the form
  \begin{equation}\label{espansione d bot Psi nabla sec pezzo ac}
 \partial_x \big( d_\bot  (d_\bot \Psi_\ac ( \mathfrak x))^\top  [\cdot]\big)  
\big[ \nabla H^{kdv} (\Psi_\ac (\mathfrak x ) ) \big]
  = \Pi_\bot \sum_{k = 0}^{M + 1} a_{1 - k}^{(2)}(\mathfrak x; \ac) \partial_x^{1 - k} + R_2(\mathfrak x; \ac)
  \end{equation}
  where $a_{1}^{(2)}(\frak x; \nu)  = 0 $ (cf. \eqref{negativi 3}) and 
  where the functions 
  $a_{1 - k}^{(2)}(\frak x; \nu)$, $k = 1, \ldots, M+ 1$ 
  and the remainder $R_2(\mathfrak x; \ac)$ satisfy the properties
  \ref{pr1}-\ref{pr2} of Lemma \ref{differential nabla perturbation-true}, in particular  \eqref{le1:es1}-\eqref{le1:es2}.
  \\ [1mm]
{\em Conclusion:} Combining \eqref{dgradcalPkdv}, 
\eqref{diffRdkv},
\eqref{espansione d bot Psi nabla sec pezzo},
and \eqref{espansione d bot Psi nabla sec pezzo ac} one obtains
the claimed expansion 
\eqref{pseudodifferential expansion of remainder mathcal R_Nkdv} with 
$$
\begin{aligned}
& a_{1 - k}(\mathfrak x; \ac ; \partial_x d_\bot \nabla_w {\cal R}^{kdv}) := a_{1 - k}^{(1)}(\frak x; \nu) - a_{1 - k}^{(1)}(\theta, 0, 0; \nu) + a_{1 - k}^{(2)}(\frak x; \nu) - a_{1 - k}^{(2)}(\theta, 0, 0; \nu) \\
& {\cal R}_M (\mathfrak x; \ac ;  \pa_x d_\bot \nabla {\cal R}^{kdv}) := R_1(\frak x ; \nu) - R_1(\theta, 0, 0; \nu) + R_2(\frak x; \nu) - R_2(\theta, 0, 0; \nu) \,. 
\end{aligned}
$$ 
Since
$ a_{1 - k}^{(1)}(\frak x; \nu) $, $R_1(\mathfrak x; \ac)$, and 
$ a_{1 - k}^{(2)}(\frak x; \nu) $, $R_2 (\mathfrak x; \ac)$ satisfy 
properties
  \ref{pr1}-\ref{pr2} of Lemma \ref{differential nabla perturbation-true}, in particular  \eqref{le1:es1}-\eqref{le1:es2}, 
the claimed estimates 
\eqref{stima a - k partial x d nabla R kdv}-\eqref{OpR3kdv} then follow 
by the mean value theorem. 
\end{proof}

\subsection{Frequencies of KdV}\label{sec:kdvfre}

In this section we record  properties of the KdV frequencies $ \omega^{kdv}_n $
used in this paper.
In Section \ref{linearizzato siti normali} we need to 
analyze $\pa_x \Omega^{kdv}(D ;I)$. Recall that by \eqref{Omega-normal},
$ \Omega^{kdv}( D; I)$ is defined  for 
$I \in \Xi \subset \R^{\mathbb S_+}_{> 0}$. Actually, it is defined on
all of $\R^{\mathbb S_+}_{> 0}$ (cf. \eqref{fre-kdv}) and according to 
\cite[Lemma 4.1]{Kappeler-Montalto-pseudo} $\pa_x \Omega^{kdv}(D; I)$ 
can be written as
\be\label{deco:Q-1}
\pa_x \Omega^{kdv}(D; I) = - \pa_x^3 + Q_{-1}^{kdv} (D; I )   
\ee
where
$ Q_{-1}^{kdv} (D; I) $ is a family of Fourier multiplier 
operators of order $ - 1 $
with an expansion in homogeneous components up to any order.  
\begin{lemma}\label{lem:anal}
For any $M \in \N$  
and $I \in \R_{>0}^{\Splus}$, $\, Q_{-1}^{kdv} (D; I )$ admits an expansion of the form
\be\label{expQ-1}
Q_{-1}^{kdv} (D; I )  
=   \Omega_{-1}^{kdv} (D; I)  + {\cal R}_M (D; I; Q_{-1}^{kdv})\,,
\quad \Omega_{-1}^{kdv} (\xi; I) = \sum_{k= 1}^{ M}
 a_{-k} (I; \Omega_{-1}^{kdv})  \chi_0(\xi) ( \ii 2 \pi \xi)^{-k}  \, ,   
\ee
where the functions 
$  a_{-k} (I; \Omega_{-1}^{kdv}) $ are real analytic and bounded on compact subsets  of $\R_{>0}^{\Splus}$,
$ a_{-k} (I;  \Omega_{-1}^{kdv})$ vanishes identially for $k$ even, and
${\cal R}_M (D; I ; Q_{-1}^{kdv})$  is a Fourier multiplier operator with multipliers 
\be\label{RM-1}
{\cal R}_M (n; I; Q_{-1}^{kdv}) = \frac{{\cal R}_{M}^{\om_n}(I)}{(2\pi n)^{M+1}}\,, \qquad
{\cal R}_M (- n; I; Q_{-1}^{kdv}) = - {\cal R}_M (n; I; Q_{-1}^{kdv})\,, \quad \forall n \in \Splus^\bot \, , 
\ee
where the functions $ I \mapsto {\cal R}_{M}^{\om_n}(I)$ are 
real analytic  and satisfy, for any $j \in \Splus,$ $\beta \in \N $, 
$$
\sup_{n \in \Sbot} |{\cal R}_{M}^{\om_n}(I)| \leq C_M \, , \quad  \,
\sup_{n \in \Sbot} | \pa_{I_j}^{\beta} {\cal R}_{ M}^{\om_n}(I)| \leq C_{M,\beta} \, ,
$$
uniformly on compact subsets  of $\R_{>0}^{\Splus}$. 
\end{lemma}

\begin{proof}
The result follows by  \cite[Lemma C.7]{Kappeler-Montalto-pseudo}.
\end{proof}
In Section \ref{sec: reducibility}, we shall use the following asymptotics 
of the KdV frequencies
\begin{equation}\label{Proposition 2.30}
\omega^{kdv}_n (I,0) - (2 \pi n)^3 = O( n^{-1} ) \, , \quad 
n \, \pa_{I} \omega^{kdv}_n (I,0)  = O( 1 ) \,,  
\end{equation}
uniformly on compact sets of actions $ I \in \R_{>0}^{\Splus}$.

\begin{lemma}(\cite[Proposition 15.5]{KP}) \label{Proposition 2.3}
{\bf (Non-degeneracy of KdV frequencies)} For any finite subset $ \Splus \subset {\mathbb N} $
the following holds on $\R^{\S_+}_{>0}$:

(i) The map $ I \mapsto {\rm det}\big( (\partial _{I_k}  \omega _n^{kdv}(I, 0))_{k,n \in \Splus} \big) $ is real analytic and does not vanish identically.

(ii) For any $\ell \in {\mathbb Z}^{\Splus} and \ j, k \in \S^\bot $ with $ (\ell, j,k) \neq (0,j,j) $, the following functions are real
analytic and do not vanish  identically, 
\be\label{Melnikov-1-2}
\sum _{n\in {\Splus}} \ell _n \omega^{kdv}_n + \omega^{kdv}_j \not= 0\,, 
\qquad \sum _{n \in {\Splus}} \ell _n \omega^{kdv}_n + \omega^{kdv}_j - \omega^{kdv}_k \not= 0\,.
\ee
\end{lemma}
\begin{remark}\label{rem:diffeo}
It was shown in \cite{BK} that for any $I \in \R^{\S_+}_{>0}$, ${\rm det} \big( 
(\partial _{I_k} \omega _n^{kdv}(I, 0)) _{k,n \in \Splus} \big) \ne 0.$
\end{remark}

\section{Nash-Moser theorem}
 
 In the symplectic variables
 $  (\theta, y,  w) \in {\cal V}(\d) \cap {\mathcal E}_s $ 
  defined by Theorem \ref{modified Birkhoff map}, 
with symplectic $ 2 $-form given by \eqref{2form},
the Hamiltonian equation \eqref{hamiltonian PDE} reads
 \be \label{HSpsi}
 \pa_t \theta = - \nabla_y {\mathcal H}_\e \, , \qquad  \pa_t y = 
 \nabla_{\theta} {\mathcal H}_\e \, , \qquad
 \pa_t w = \pa_x \nabla_w {\mathcal H}_\e \, , 
 \ee
 where $ {\mathcal H}_\e:= H_\e\circ \Psi_\ac $ and 
 $ H_\e $ given by \eqref{hamiltoniana kdv}. More explicitly,
 \begin{equation}\label{HepNM}
 \begin{aligned}
&  {\mathcal H}_\e (\theta, y, w; \ac ) = {\mathcal H}^{kdv} (\theta, y, w; 
\ac ) + \e {\mathcal P}(\theta, y, w; \ac ) \,,  \\ 
  & {\mathcal H}^{kdv} =  H^{kdv} \circ \Psi_\ac\,, \quad  {\mathcal P} = P \circ \Psi_\ac\, , \quad \ac \in \Xi \, , 
  \end{aligned}
 \end{equation}
where $ {\mathcal H}^{kdv} (\theta, y, w; \ac ) $ has the normal form expansion  \eqref{Hkdv}. We denote by $ X_{ {\mathcal H}_\e }$ the 
Hamiltonian vector field associated to ${\cal H}_\e$. 
For $ \e = 0 $, the Hamiltonian system \eqref{HSpsi} possesses, 
for any value of the parameter $\ac \in \Xi $,   the  invariant torus
$ \T^{\Splus} \times \{ 0 \}  \times \{0 \} $, 
 filled by quasi-periodic finite gap solutions of the KdV equation
  with frequency vector 
 $   \omega^{kdv} (\ac) := (\omega^{kdv}_n (\ac,0))_{n \in \Splus}  $
 introduced in \eqref{frequency finite gap}.
  
By our choice of $\Xi,$ the map  
$ - \omega^{kdv} : \Xi \to   \Omega := - \omega^{kdv} (\Xi)$ 
 is a real analytic diffeomorphism. In 
the sequel, we consider $ \ac $ as a function of the parameter $ \om \in  \Omega $, namely
\be\label{mu-kdv-om}
 \ac \equiv \ac(\omega) : = (\om^{kdv})^{-1} ( - \omega )\ .
\ee
 For simplicity 
we often will not record the dependence of the Hamiltonian  
$ {\cal H}_\e $ on $ \ac = (\om^{kdv})^{-1} ( - \omega ) $. 

Consider  the set of diophantine frequencies in $\Omega$,   
\begin{equation}\label{diofantei in omega}
\mathtt{DC}(\gamma, \tau) := \Big\{ \omega \in \Omega : |\omega \cdot \ell| \geq \frac{\gamma}{\langle \ell \rangle^\tau}, \quad \forall \ell \in \Z^{\mathbb S_+} \setminus \{ 0 \} \Big\}\,. 
\end{equation}
For any torus embedding 
$ \T^{\Splus} \to {\cal V}(\delta ) \cap {\mathcal E}_s   $, 
$  \vphi \mapsto ( \theta (\vphi), y(\vphi), w(\vphi) ) $, 
close to the identity, 
consider its lift 
\be\label{tor-emb}
\breve{\io} : \R^{\Splus} \to \R^{\Splus} \times \R^{\Splus} \times H^s_\bot(\T_1) \, , \quad 
\breve \io(\vphi)  = (\vphi, 0, 0)  + \iota(\vphi) \, , 
\ee
where 
$ \iota(\vphi)  = ( {\Theta} (\ph), y(\ph), w(\ph)) $, with 
$ \Theta(\ph) := \teta (\vphi) - \vphi $, 
is $ (2 \pi \Z)^{\Splus}$ periodic. 

We look for a torus embedding $\breve{\io}  $ such  that
$\mathcal F_\omega ( \iota, \zeta) = 0$ where
\be\label{operatorF}
\mathcal F_\omega ( \iota, \zeta) :=
\left(
\begin{aligned}
& \om \cdot \pa_{\vphi} \theta (\vphi) + (\nabla_y {\mathcal H}_\e) (\breve{\io} (\vphi)    ) \\
& \om \cdot \pa_{\vphi} y (\vphi) - (\nabla_\theta {\mathcal H}_\e) (\breve{\io} (\vphi)    ) 
- \zeta \\
& \om \cdot \pa_{\vphi} w (\vphi) - \pa_x (\nabla_w {\mathcal H}_\e) (\breve{\io} (\vphi)  ) \\
\end{aligned}
\right).
\ee
The additional variable $ \zeta \in \R^{\Splus} $ 
is introduced in order to control the average of the $y$-component 
of the linearized Hamiltonian equations -- see Section \ref{sezione approximate inverse}. 
Actually any invariant torus 
for $ X_{{\cal H}_{\e, \zeta}} = X_{{\cal H}_\e} + (0, \zeta, 0 )$
with modified Hamiltonian 
\be\label{def:Hep}
{\cal H}_{\e, \zeta} (\teta, y, w)  
:=  {\cal H}_\e (\teta, y, w) + \zeta \cdot \theta\,,\quad \zeta \in \R^{\Splus} \, , 
\ee
is  invariant  for $ X_{{\cal H}_\e} $, see \eqref{zeta Z inequality}. 
Note that $ {\cal H}_{\e, \zeta}  $ is not periodic in 
$ \theta $, but that its Hamiltonian vector field is.
The Lipschitz Sobolev norm of the periodic part 
$\iota(\vphi) = (\Theta(\vphi), y(\vphi), w(\vphi)) $ 
of the embedded torus \eqref{tor-emb}
is 
$$ 
\|  \iota  \|_s^\Lipg :=   \| \Theta \|_s^\Lipg 
+  \| y  \|_s^\Lipg  +  \| w \|_s^\Lipg 
$$
where  $  \| w  \|_s^\Lipg$ is the Lipschitz Sobolev norm
 introduced in \eqref{def norm Lip Stein uniform} and
 \begin{equation}\label{def H^s_varphi}
 \| \Theta \|_s^\Lipg  \equiv \| \Theta \|_{H^s_\vphi}^\Lipg := 
 \| \Theta  \|_{H^s (\T^{\Splus}, \R^{\Splus})}^\Lipg \, , \qquad 
\| y  \|_s^\Lipg \equiv  \| y  \|_{H^s_\vphi}^\Lipg
 := \| y  \|_{H^s (\T^{\Splus}, \R^{\Splus})}^\Lipg \, . 
 \end{equation}

\begin{theorem}\label{main theorem}
{\bf (Nash-Moser)}
There exist $ \bar s > (|\Splus| + 1) / 2 $ and $ \varepsilon _0 > 0 $  so that 
for any $0 <\varepsilon  \leq \varepsilon _0$, 
there is a measurable subset  $  \Omega_\varepsilon 
\subseteq \Omega$ satisfying 
\be\label{measure estimate Omega in Theorem 4.1}
 \lim_{\e\to 0} \, \frac{{\rm meas}( \Omega_\e)}
 {{\rm meas}( \Omega )} = 1
\ee
and  for any $ \omega \in  \Omega_\e $, there exists a torus embedding
$\breve \io_\omega $ as in \eqref{tor-emb}  which satisfies the estimate 
$$
\| \breve \io_\omega - (\vphi, 0,0) \|_{\bar s}^\Lipg = O( \e \gamma^{-2}) \ , \qquad  \gamma = 
\e^{\mathfrak a} \, , 
\ 0 <  {\mathfrak a} \ll 1 \, ,
$$
and solves
$$ 
\om \cdot \partial_\vphi \breve \io_\omega(\vphi) - X_{{\cal H}_\e}( \breve \io_\omega(\vphi)) = 0 \, .
$$
As a consequence  the embedded torus 
$\breve \io_\omega (\T^{\Splus}) $ is invariant for the Hamiltonian vector field $ X_{{\mathcal H}_\e (\cdot; \ac)} $  
with $ \ac = (\om^{kdv})^{-1}( - \om ) $,  
and  it is filled by quasi-periodic solutions of \eqref{HSpsi} 
with frequency vector $ \om \in \Omega_\e $.
 Furthermore, the quasi-periodic solution
$\breve \io_\omega(\omega t) = \omega t + \io_\omega(\omega t)$ is
linearly  stable.  
\end{theorem}

Theorem \ref{main theorem} is proved in Section \ref{sec:NM}. 
The main issue concerns the construction of an approximate 
right inverse of  the linearized operator 
$ d_{\iota, \zeta} \mathcal F_{\omega} ( \iota, \zeta) $ at an approximate solution. 
This construction is carried out in Sections \ref{sezione approximate inverse},
\ref{linearizzato siti normali} and  \ref{sec: reducibility}.  

\smallskip

Along the proof we shall use 
the following tame estimates of the Hamiltonian vector field $X_{{\cal H}_\e}$ with respect to the norm $\| \cdot \|_s^\Lipg$. Recalling the expansion \eqref{Hkdv} provided in Theorem \ref{modified Birkhoff map}, and the definition of ${\cal P}$ in \eqref{definizione perturbazione quasilin}, we 
decompose the Hamiltonian ${\cal H}_\e $ defined in 
\eqref{HepNM} as
\begin{equation}\label{splitting ham forma normale + perturbazione}
\begin{aligned}
& {\cal H}_\e = {\cal N} + {\cal P}_\e \quad {\rm where}   \\
& {\cal N}(y, w;\ac) := \omega^{kdv}(\ac) \cdot y + \frac12 
\Omega_{\mathbb S_+}^{kdv}(\ac) [y] \cdot y + \frac12 \big( \Omega^{kdv}(D ;\ac)w\,,\,w \big)_{L^2_x}  \, , 
\quad  {\cal P}_\e := {\cal R}^{kdv} + \e {\cal P}\,. 
\end{aligned}
\end{equation}
\begin{lemma}\label{stime tame campo ham per NM AI}
There exists $\sigma_1 = \sigma_1(\mathbb S_+) > 0$ so that for any
$s \ge 0,$ any torus embedding 
$\breve \io$ of the form \eqref{tor-emb}
 with 
$\| \io \|_{s_0 + \sigma_1}^\Lipg \leq \delta$, and any maps
 $\widehat \io, \widehat \io_1, \widehat \io_2 :
 \T^{\Splus} \to  E_s$, 
 the following tame estimates hold: 
$$
\begin{aligned}
\| X_{{\cal P}_\e} (\breve \io) \|_s^\Lipg & \lesssim_s  \e (1 + \| \io\|_{s + \sigma_1}^\Lipg) + \| \io \|_{s_0 + \sigma_1}^\Lipg  \| \io \|_{s + \sigma_1}^\Lipg  \,, \\
\| d X_{{\cal P}_\e}(\breve \io)[\widehat \io] \|_s^\Lipg & \lesssim_s  (\e + \| \io \|_{s_0 + \sigma_1}^\Lipg) \| \widehat \io\|_{s + \sigma_1}^\Lipg +  \| \io \|_{s + \sigma_1}^\Lipg  \| \widehat \io \|_{s_0 + \sigma_1}^\Lipg\,,  \\
\| d^2 X_{{\cal H}_\e}(\breve \io)[\widehat \io_1, \widehat \io_2] \|_s^\Lipg & \lesssim_s \| \widehat \io_1\|_{s + \sigma_1}^\Lipg \| \widehat \io_2\|_{s_0 + \sigma_1}^\Lipg +  \| \widehat \io_1\|_{s_0 + \sigma_1}^\Lipg \| \widehat \io_2\|_{s + \sigma_1}^\Lipg + \| \io \|_{s + \sigma_1}^\Lipg  (\| \widehat \io_1\|_{s_0 + \sigma_1}^\Lipg \| \widehat \io_2\|_{s_0 + \sigma_1}^\Lipg \, . 
\end{aligned}
$$
\end{lemma}

\begin{proof}
Note that $X_{{\cal P}_\e} = \e X_{\cal P} + X_{{\cal R}^{kdv}}$
and $d^2 X_{{\cal H}_\e} = d^2 X_{\mathcal N} + d^2 X_{{\cal P}_\e}$.
The claimed estimates then follow from 
estimates of
$\e X_{\cal P}$, obtained from 
Lemmata \ref{differential nabla perturbation-true}, \ref{lem:tame1}, \ref{lem:tame2},
and from estimates of 
$X_{{\cal R}^{kdv}}$ obtained from 
Lemmata \ref{differential nabla kdv remainder-true}, \ref{lem:tame1}, \ref{lem:tame2}, and the mean value theorem.
\end{proof}

\section{Approximate inverse}\label{sezione approximate inverse} 
In order to implement a convergent Nash-Moser scheme that leads to a solution of 
$ \mF_\omega(\io, \zeta) = 0 $ (cf. \eqref{operatorF})
we construct an \emph{almost-approximate right inverse} (see Theorem \ref{thm:stima inverso approssimato})
of the linearized operator 
\be\label{diaF}
d_{\io , \zeta} {\cal F}_\omega(\io, \zeta )[\widehat \io \,, \widehat \zeta ] =
\Dom \widehat \io - d_\io X_{{\cal H}_{\e}} ( \breve \io ) [\widehat \io ] -  (0, \widehat \zeta, 0 )
\ee
where ${\cal H}_\e = {\cal N} + {\cal P}_\e $ is the Hamiltonian in 
\eqref{splitting ham forma normale + perturbazione}.
Note that 
the perturbation ${\cal P}_\e $ and the differential
$ d_{\io, \zeta} {\cal F}_\omega(\io, \zeta )$ 
are independent of $ \zeta $. 
In the sequel, we will often write 
$d_{\io, \zeta} {\cal F}_\omega(\io )$ instead of  
$d_{\io, \zeta} {\cal F}_\omega(\io, \zeta )$.

Since  the $\theta$, $y$, and $w$ components of
$ d_\io X_{{\cal H}_{\e}} ( \breve \io (\vphi) )[\widehat \io] $ 
are all coupled, inverting the linear operator 
$d_{\io , \zeta} {\cal F}_\omega(\io, \zeta )$ in \eqref{diaF}
is intricate. As a first step, we  implement the approach 
developed in \cite{BBM-auto}, \cite{BB13}, \cite{Berti-Montalto},  to approximately reduce $d_{\io , \zeta} {\cal F}_\omega(\io, \zeta )$
to a triangular form -- see \eqref{operatore inverso approssimato} below. 

Along this section we assume the following hypothesis, 
which is verified by the approximate solutions obtained at each step of the Nash-Moser 
Theorem \ref{iterazione-non-lineare}. 

\begin{itemize}
\item {\bf Ansatz.} 
{\em The map $\omega \mapsto \iota(\omega) :=  \breve \io(\ph; \om) - (\ph,0,0) $ 
is Lipschitz continuous with respect to $\omega \in \Omega $, 
and,  for $\gamma \in (0, 1)$, 
$ \mu_0 := \mu_0 (\t, \mathbb S_+) >  0 $
(with $\tau$ being specified later (cf. Section \ref{sec:NM}))
\begin{equation}\label{ansatz 0}
\| \io  \|_{\mu_0}^\Lipg \lesssim \e \gamma^{- 2} \,, \quad \| Z\|_{s_0}^\Lipg \lesssim \e \, , 
\end{equation}
where  $ Z  $ is the ``error function''  defined by
\begin{equation} \label{def Zetone}
Z(\vphi) :=  (Z_1, Z_2, Z_3) (\vphi) := {\cal F}_\omega(\io, \zeta) (\vphi) =
\om \cdot \pa_\vphi \breve \io (\vphi) - X_{{\cal H}_\e}(\breve \io (\vphi)) - (0, \zeta, 0) \, .
\end{equation}
}
\end{itemize}
We first noe that the $2$-form $ {\cal W} $ given in \eqref{2form} is 
$$
{\cal W} := \Big( {\mathop \sum}_{j \in \Splus} d y_j \wedge d \theta_j \Big)  \oplus 
{\cal W}_\bot = d \Lambda 
$$
where  $ \Lambda $ is the Liouville  $1$-form
\begin{equation}\label{Lambda 1 form}
\Lambda_{(\theta, y, w)}[\widehat \theta, \widehat y, \widehat w] := 
{\mathop \sum}_{j \in \Splus} y_j  \widehat \theta_j + 
\frac12 \big( \partial_x^{- 1} w\,,\,\widehat w \big)_{L^2_x}  \, . 
\end{equation}
Arguing as in \cite[Lemma 6.1]{BBM-auto}, one obtains
\begin{equation}\label{zeta Z inequality}
|\zeta|^\Lipg \lesssim \| Z \|_{s_0}^\Lipg\,. 
\end{equation}
 An invariant torus $ \breve \io $ with Diophantine flow 
is isotropic, 
meaning that the pull-back $ \breve \io^* \Lambda $
of the $ 1$-form $\Lambda $ is closed, 
or equivalently that the pull back $\breve \io^* {\cal W}$
satisfies
$ \breve \io^* {\cal W} =  \breve \io^* d \Lambda  = d \breve \io^* \Lambda = 0 $ (cf. \cite{BB13}). 
For an approximately invariant torus embedding $ \breve \io $, 
the 1-form 
\begin{equation}\label{coefficienti pull back di Lambda}
\breve \io^* \Lambda = {\mathop \sum}_{k \in \mathbb S_+} a_k (\vphi) d \vphi_k \,,\quad 
a_k(\vphi) :=  \big( [\pa_\ph \teta (\vphi)]^\top y (\vphi)  \big)_k 
+ \frac12 ( \partial_x^{- 1} w (\ph),
\partial_{\vphi_k} w(\ph) )_{L^2_x} \, , 
\end{equation} 
is only ``approximately closed", in the sense that 
\begin{equation} \label{def Akj} 
i_0^* {\cal W} = d \, i_0^* \Lambda = {\mathop\sum}_{\begin{subarray}{c}
k, j \in \mathbb S_+ \\
 k < j 
\end{subarray}} A_{k j}(\vphi) d \vphi_k \wedge d \vphi_j\,,\quad A_{k j} (\vphi) := 
\partial_{\vphi_k} a_j(\ph) - \partial_{\vphi_j} a_k(\ph) \, , 
\end{equation} 
is of order $O(Z)$. More precisely, the following lemma holds. 

\begin{lemma} \label{lemma:aprile}
Let  $ \om \in \mathtt {DC} (\gamma, \tau)$ (cf.
\eqref{diofantei in omega}). 
Then the coefficients $ A_{kj}  $ in \eqref{def Akj} satisfy 
\begin{equation}\label{stima A ij}
\| A_{k j} \|_s^\Lipg \lesssim_s \gamma^{-1}
\big(\| Z \|_{s+ \sigma}^\Lipg + \| Z \|_{s_0+ \sigma}^\Lipg \|  \io \|_{s+\sigma}^\Lipg \big)
\end{equation}
for some $\sigma = \sigma(\tau, \mathbb S_+) > 0$. 
\end{lemma}

\begin{proof}
The  $ A_{kj}  $  satisfy the identity 
$  \Dom A_{k j} 
= $ $ {\cal W}\big( \pa_\ph Z(\vphi) \underline{e}_k ,  \pa_\ph \breve \io(\vphi)  \underline{e}_j \big) 
+ $ $ {\cal W} \big(\pa_\ph \breve \io_0(\vphi) \underline{e}_k , \pa_\ph Z(\vphi) \underline{e}_j \big)  $
where  $\underline{e}_k$, $k \in \Splus$, denotes the
standard basis of $ \R^{\mathbb S_+}$ (cf.  \cite[Lemma 5]{BB13}). 
Then \eqref{stima A ij} follows by \eqref{ansatz 0} and  \eqref{Diophantine-1}. 
\end{proof}

As in \cite{BB13}, \cite{BBM-auto} we first modify the approximate torus $ \breve \io $ to obtain an isotropic torus $ \breve \io_\d $ which is 
still approximately invariant. Let  $ \Delta_\vphi := \sum_{k\in \mathbb S_+} \partial_{\vphi_k}^2 $.

\begin{lemma}\label{toro isotropico modificato} {\bf (Isotropic torus)} 
Let  $ \om \in \mathtt {DC} (\gamma, \tau)$.
The torus $ \breve \io_\delta(\vphi) := (\theta(\vphi), y_\delta(\vphi), w (\vphi) ) $ defined by 
\begin{equation}\label{y 0 - y delta}
y_\d(\vphi) :=  y(\vphi)  -  [\pa_\ph \theta(\vphi)]^{- \top}  \rho(\vphi) \, , \qquad 
\rho_j(\vphi) := \Delta_\vphi^{-1} {\mathop\sum}_{ k \in \mathbb S_+} \partial_{\vphi_k} A_{k j}(\vphi) \, ,  
\end{equation}
 is {\it isotropic}
and there is $ \s =  \s(\tau, \mathbb S_+) > 0 $ so that, for any $s \geq s_0$ 
\begin{align} \label{2015-2}
\| y_\delta - y \|_s^\Lipg & \lesssim_s  \| \io  \|_{s + \sigma}^\Lipg\\
\label{stima y - y delta}
\| y_\delta - y \|_s^\Lipg 
& \lesssim_s  \gamma^{-1} \big(\| Z \|_{s + \s}^\Lipg + 
\|  \io \|_{s + \s}^\Lipg \| Z \|_{s_0 + \s}^\Lipg  \big) \,,
\\
\label{stima toro modificato}
\| {\cal F}_\omega(\io_\delta , \zeta) \|_s^\Lipg
& \lesssim_s  \| Z \|_{s + \s}^\Lipg  +   \|  \io \|_{s + \s}^\Lipg \| Z \|_{s_0 + \s}^\Lipg \\
\label{derivata i delta}
\| d_\io \io_\d [ \widehat \io ] \|_s^\Lipg & \lesssim_s \| \widehat \io \|_{s + \s}^\Lipg +  
\| \io \|_{s + \s}^\Lipg \| \widehat \io  \|_{s_0}^\Lipg \, .
\end{align}
 \end{lemma}
\begin{remark}
In the sequel, $\omega$ will always be assumed to be in
$\mathtt {DC} (\gamma, \tau)$. Furthermore, 
 $ \s := \s( \tau, \mathbb S_+ ) $ will denote
different, possibly larger ``loss of derivatives"  constants. 
\end{remark}
\begin{proof}
The Lemma follows  as in \cite[Lemma 6.3]{BBM-auto}
by   Lemma \ref{stime tame campo ham per NM AI}, 
 \eqref{coefficienti pull back di Lambda}-\eqref{stima A ij} and 
 the ansatz \eqref{ansatz 0}. 
\end{proof}

In order to find an approximate inverse of the linearized operator $d_{\io, \zeta} {\cal F}_\omega(\io_\delta )$, 
we introduce  the symplectic diffeomorpshim 
$ G_\delta : (\phi, \eta , v) \mapsto (\theta, y , w)$ of the phase space $\T^{\mathbb S_+} \times \R^{\mathbb S_+} \times L^2_\bot(\T_1)$, defined by
\begin{equation}\label{trasformazione modificata simplettica}
\begin{pmatrix}
\theta \\
y \\
w
\end{pmatrix} := G_\delta \begin{pmatrix}
\phi \\
\eta \\
v
\end{pmatrix} := 
\begin{pmatrix}
\!\!\!\!\!\!\!\!\!\!\!\!\!\!\!\!\!\!\!\!\!\!\!\!\!\!\!\!\!\!\!\!\!
\!\!\!\!\!\!\!\!\!\!\!\!\!\!\!\!\!\!\!\!\!\!\!\!\!\!\!\!\!\!\!\!\!\!
\!\!\!\!\!\!\!\!\!\!\!\!\!\!\!\!\!\!\!\!\!\!\!\!\!\!\!\!\!\!\!\! \theta(\phi) \\
y_\delta (\phi) + [\pa_\phi \theta(\phi)]^{-T} \eta - \big[ (\pa_\teta \tilde{w}) (\theta(\phi)) \big]^\top \partial_x^{- 1} v \\
\!\!\!\!\!\!\!\!\!\!\!\!\!\!\!\!\!\!\!\!\!\!\!\!\!\!\!\!\!
\!\!\!\!\!\!\!\!\!\!\!\!\!\!\!\!\!\!\!\!\!\!\!\!\!\!\!\!\!
\!\!\!\!\!\!\!\!\!\!\!\!\!\!\!\!\!\!\!\!\!\!\!\!\!\!\!\!\! \! w(\phi) + v
\end{pmatrix} 
\end{equation}
where $ \tilde{w} := w \circ \theta^{-1} $. 
It is proved in \cite[Lemma 2]{BB13} that $ G_\delta $ is symplectic, since 
by Lemma \ref{toro isotropico modificato},
$ \breve \io_\d $ is an isotropic torus embedding.
In the new coordinates,  $ \breve \io_\delta $ is the trivial embedded torus
$ (\phi , \eta , v ) = (\phi , 0, 0 ) $
and the Hamiltonian 
vector field $ X_{{\cal H}_{\e, \zeta}} $ (with $  {\cal H}_{\e, \zeta}$ defined in \eqref{def:Hep}) is given by
\be\label{new-Hamilt-K}
X_{{\cal K}} = (d G_\d)^{-1} X_{{\cal H}_{\e, \zeta}} \circ G_\d \qquad {\rm where} \qquad{\cal K} \equiv {\cal K}_{\e, \zeta} := {\cal H}_{\e, \zeta} \circ G_\d  \, .
\ee
The Taylor expansion of $ {\cal K}$ in $\eta, v$ at the trivial torus $ (\phi , 0, 0 ) $ is of the form
\begin{align} 
{\cal K}(\phi, \eta , v, \zeta)
& =  \theta(\phi) \cdot \zeta + {\cal K}_{00}(\phi) + {\cal K}_{10}(\phi) \cdot \eta + ({\cal K}_{0 1}(\phi), v)_{L^2_x} + 
\frac12 {\cal K}_{2 0}(\phi) \eta \cdot \eta 
\nonumber \\ & 
\quad +  \big( {\cal K}_{11}(\phi) \eta , v \big)_{L^2_x } 
+ \frac12 \big({\cal K}_{02}(\phi) v , v \big)_{L^2_x} + {\cal K}_{\geq 3}(\phi, \eta, v)  
\label{KHG}
\end{align}
where $ {\cal K}_{\geq 3} $ collects the terms which are at least cubic in the variables $ (\eta, v )$, 
 ${\cal K}_{00}(\phi) \in \R $,  
${\cal K}_{10}(\phi) \in \R^{\mathbb S_+} $,  
${\cal K}_{01}(\phi) \in L^2_\bot(\T_1)$, 
${\cal K}_{20}(\phi) $ is a $|\mathbb S_+|\times |\mathbb S_+|$ real matrix, 
${\cal K}_{02}(\phi) : L^2_\bot(\T_1) \to L^2_\bot (\T_1)$ is a linear self-adjoint operator and 
${\cal K}_{11}(\phi) : \R^{\mathbb S_+} \to L^2_\bot(\T_1)$
is a linear operator of finite rank. 
At an exact solution of 
$ \mF_\omega(\io, \zeta) = 0 $ one has  $ Z  = 0 $ and
the coefficients in the Taylor expansion \eqref{KHG} satisfy 
$ {\cal K}_{00} = {\rm const} $, $ {\cal K}_{10} = - \omega $, 
${\cal K}_{01} = 0 $.

\begin{lemma} \label{coefficienti nuovi} 
There exists  $\s := \s( \tau, \mathbb S_+ ) $ so that
\begin{equation}\label{K 00 10 01}
\begin{aligned}
& \|  \partial_\phi {\cal K}_{00} \|_s^\Lipg 
+ \| {\cal K}_{10} + \om  \|_s^\Lipg +  \| {\cal K}_{0 1} \|_s^\Lipg 
\lesssim_s  \| Z \|_{s + \s}^\Lipg +  \| \io \|_{s + \s}^\Lipg  \| Z \|_{s_0 + \s}^\Lipg \, . \\
&   
  \|{\cal K}_{20} -  \Omega_{\mathbb S_+}^{kdv}(\ac)  \|_s^\Lipg \lesssim_s \e + \| \io\|_{s + \s}^\Lipg\, ,  \\ 
& \| {\cal K}_{11}  \eta  \|_s^\Lipg 
\lesssim_s \e \gamma^{- 2} \|  \eta \|_{s + \sigma}^\Lipg
+ \| \iota \|_{s + \sigma}^\Lipg  
\| \eta \|_{s_0 + \sigma}^\Lipg  \, ,  \\
&   \| {\cal K}_{11}^\top v \|_s^\Lipg
\lesssim_s \e \gamma^{- 2} \| v  \|_{s + \sigma}^\Lipg
+  \| \io \|_{s + \sigma}^\Lipg
\| v \|_{s_0 + \sigma}^\Lipg \, . 
\end{aligned}
\end{equation}
\end{lemma}

\begin{proof}
The lemma follows as in \cite{BB13}, \cite{BBM-auto},
by applying Lemma \ref{stime tame campo ham per NM AI} and \eqref{ansatz 0}, \eqref{2015-2}, \eqref{stima y - y delta}, \eqref{stima toro modificato}\,. 
\end{proof}

Denote by ${\rm Id}_\bot$ 
the identity operator on $L^2_\bot(\T_1)$.
The linear transformation  
$d G_\delta|_{(\vphi, 0, 0)} \equiv d G_\delta(\vphi, 0, 0)$
then reads
\begin{equation}\label{DGdelta}
d G_\delta|_{(\vphi, 0, 0)}
\begin{pmatrix}
\widehat \phi \, \\
\widehat \eta \\
\widehat v
\end{pmatrix} 
:= 
\begin{pmatrix}
\pa_\phi \theta(\vphi) & 0 & 0 \\
\pa_\phi y_\delta(\vphi) & [\pa_\phi \theta(\vphi)]^{-\top} & 
 - [(\pa_\theta \tilde{w})(\theta(\vphi))]^\top \partial_x^{- 1} \\
\pa_\phi w(\vphi) & 0 & {\rm Id}_\bot
\end{pmatrix}
\begin{pmatrix}
\widehat \phi \, \\
\widehat \eta \\
\widehat v
\end{pmatrix} \, .
\end{equation}
It approximately transforms
the linearized operator  $d_{\io, \zeta}{\cal F}_\om(\io_\delta )$ 
(see the proof of Theorem \ref{thm:stima inverso approssimato}) 
into the one  obtained when 
the Hamiltonian system with Hamiltonian 
$\mathcal K$ (cf. \eqref{new-Hamilt-K})
is linearized at $(\phi, \eta , v ) = (\vphi, 0, 0 )$,
differentiated also with respect to $ \zeta $, and when $ \partial_t$ is exchanged by $\Dom $, 
\begin{equation}\label{lin idelta}
\begin{pmatrix}
\widehat \phi  \\
\widehat \eta    \\ 
\widehat v \\
\widehat \zeta 
\end{pmatrix} \mapsto
\begin{pmatrix}
\Dom \widehat \phi + \partial_\phi {\cal K}_{10}(\vphi)[\widehat \phi \, ] + 
{\cal K}_{2 0}(\vphi)\widehat \eta + {\cal K}_{11}^\top (\vphi) \widehat v \\
 \Dom  \widehat \eta - \big(\partial_\phi \theta(\vphi) \big)^\top[\widehat \zeta] - \partial_\phi\big(\partial_\phi \theta(\vphi)^\top[\zeta]\big)[\widehat \phi] - \partial_{\phi\phi} {\cal K}_{00}(\vphi)[\widehat \phi] - 
[\partial_\phi {\cal K}_{10}(\vphi)]^\top \widehat \eta - 
[\partial_\phi  {\cal K}_{01}(\vphi)]^\top \widehat v   \\ 
\Dom  \widehat v - \partial_x 
\{ \partial_\phi {\cal K}_{01}(\vphi)[\widehat \phi] + {\cal K}_{11}(\vphi) \widehat \eta + {\cal K}_{02}(\vphi) \widehat v \}
\end{pmatrix} \! .  \hspace{-5pt}
\end{equation}
Using \eqref{ansatz 0} and \eqref{2015-2}, one shows 
as in \cite{BBM-auto} that the induced operator 
 $ \widehat \io := (\widehat \phi, \widehat \eta, \widehat v) \mapsto dG_\delta [\widehat \io]$ satisfies 
\begin{gather} \label{DG delta}
\|dG_\delta(\vphi,0,0) [\widehat \io] \|_s^\Lipg\,, 
\|dG_\delta(\vphi,0,0)^{-1} [\widehat \io] \|_s^\Lipg 
\lesssim_s \| \widehat \io \|_{s}^\Lipg +  \| \io \|_{s + \s}^\Lipg  \| \widehat \io \|_{s_0 }^\Lipg\,,
\\ 
\| d^2 G_\delta(\vphi,0,0)[\widehat \io_1, \widehat \io_2] \|_s^\Lipg 
\lesssim_s  \| \widehat \io_1\|_s^\Lipg  \| \widehat \io_2 \|_{s_0}^\Lipg 
+ \| \widehat \io_1\|_{s_0}^\Lipg  \| \widehat \io_2 \|_{s}^\Lipg 
+  \| \io   \|_{s + \s}^\Lipg \|\widehat \io_1 \|_{s_0}^\Lipg  \| \widehat \io_2\|_{s_0}^\Lipg \, .  \label{DG2 delta} 
\end{gather}
In order to construct an ``almost-approximate" inverse of \eqref{lin idelta} we need  that 
\be\label{Lomega def}
{\cal L}_\omega := \Pi_\bot \big(\Dom   - \partial_x {\cal K}_{02}(\vphi) \big)_{|L^2_\bot} 
\ee
is ``almost-invertible" up to remainders of size $O(N_{n - 1}^{- \mathtt a})$ (see precisely \eqref{stima R omega corsivo}) where 
 \be\label{NnKn}
 N_n := K_n^p \, , \quad \forall n \geq 0\,,  
 \ee
 and  
\be\label{definizione Kn}
K_n := K_0^{\chi^{n}} \, , \quad \chi := 3/ 2 \, , 
\ee
are the scales used in the nonlinear Nash-Moser iteration in Section \ref{sec:NM}. 
Based on results obtained in Sections \ref{linearizzato siti normali}-\ref{sec: reducibility}, 
the almost invertibility of ${\cal L}_\om$ is proved in Theorem \ref{inversione parziale cal L omega}, but here it is stated as an assumption to avoid the involved definition of the set $\Omega_o$.
Recall that $\mathtt{DC}(\gamma, \tau)$ is the set
of diophantine frequencies in $\Omega$ (cf. \eqref{diofantei in omega}).

\begin{itemize}
\item {\bf Almost-invertibility of ${\cal L}_\omega$.} 
{\em There exists a subset $ \Omega_o  
\subset \mathtt{DC}(\gamma, \tau)  $ such that, 
for all $ \omega \in  \Omega_o  $, the operator $ {\cal L}_\omega $ in \eqref{Lomega def} admits a decomposition
\be\label{inversion assumption}
{\cal L}_\omega  
= {\cal L}_\omega^< + {\cal R}_\omega + {\cal R}_\omega^\bot 
\ee
with the following properties: there exist constants 
$  K_0,N_0, \sigma, \tau_1,  \mu(\mathtt b), \mathtt a, p, \sM > 0$ so that for any $\sM \leq s \leq S$ and $\omega\in \Omega_o$  one has:

(i) The operators ${\cal R}_\omega$, ${\cal R}_\omega^\bot $ satisfy the estimates  
\begin{align}  \label{stima R omega corsivo}
\|{\cal R}_\omega h \|_s^\Lipg & 
\lesssim_S  \e \gamma^{- 2} N_{n - 1}^{- {\mathtt a}}\big( \|  h \|_{s + \sigma}^\Lipg +  N_0^{\tau_1} \gamma^{- 1}\| \iota \|_ {s +  \mu (\mathtt b) + \sigma }^\Lipg \| h \|_{\sM + \sigma}^\Lipg \big)\, , \\
\label{stima R omega bot corsivo bassa}
\| {\cal R}_\omega^\bot h \|_{\sM}^\Lipg & \lesssim_{S,b}
K_n^{- b} \big( \| h \|_{\sM + b 
+ \sigma}^\Lipg + 
N_0^{\tau_1} \gamma^{- 1}\| \iota \|_ {\sM + \mu (\mathtt b)  + \sigma + b }^\Lipg  \|  h \|_{\sM + \sigma}^\Lipg\big)\,,  
\qquad \forall b > 0\,, 
\end{align}
(ii) 
For any 
$ g \in H^{s+\sigma}_{\bot}(\T^{\mathbb S_+} \times \T_1) $, there 
is a solution $ h :=  ({\cal L}_\om^<)^{- 1} g  \in H^{s}_{\bot}(\T^{\mathbb S_+} \times \T_1) $ 
of the linear equation $ {\cal L}_\om^< h = g $, 
satisfying  the tame estimates
\begin{equation}\label{tame inverse}
\| ({\cal L}_\om^<)^{- 1} g \|_s^\Lipg \lesssim_S  \g^{-1} 
\big(  \| g \|_{s + \sigma}^\Lipg + 
N_0^{\tau_1} \gamma^{- 1} \| \io \|_{s + \mu({\mathtt b}) + \sigma}^\Lipg  \|g \|_{\sM + \sigma}^\Lipg  \big) \,.
\end{equation}
}
\end{itemize}
In order to find an almost-approximate inverse of the linear operator \eqref{lin idelta} 
and hence of $ d_{\io, \zeta} {\cal F}_\om (\io_\d) $,
it is sufficient to invert the operator
\begin{equation}\label{operatore inverso approssimato} 
{\mathbb D} [\widehat \phi, \widehat \eta, \widehat v, \widehat \zeta ] := 
  \begin{pmatrix}
\Dom \widehat \phi  + 
{\cal K}_{20}(\vphi) \widehat \eta  + {\cal K}_{11}(\vphi)^\top \widehat v\\
\Dom  \widehat \eta - \partial_\phi \theta(\vphi)^\top \widehat \zeta   \\
{\cal L}_\omega^{<} \widehat v  - \partial_x  {\cal K}_{11}(\vphi)\widehat \eta  
\end{pmatrix}
\end{equation}
obtained by neglecting in \eqref{lin idelta} 
the terms $ \partial_\phi {\cal K}_{10} $, $ \partial_{\phi \phi} {\cal K}_{00} $, 
$ \partial_\phi {\cal K}_{00} $,
$ \partial_\phi {\cal K}_{01} $, $\partial_\phi\big(\partial_\phi \theta(\vphi)^\top[\zeta]\big)$ 
and by replacing ${\cal L}_\omega $ by 
${\cal L}_\omega^< $ (cf. \eqref{inversion assumption}). Note that the remainder
${\cal L}_\omega - {\cal L}_\omega^< = {\cal R}_\omega + {\cal R}_\omega^\bot $ is small and that
by Lemma \ref{coefficienti nuovi} and \eqref{zeta Z inequality}, 
$\partial_\phi {\cal K}_{10} $, $ \partial_{\phi \phi} {\cal K}_{00}$, 
$ \partial_\phi {\cal K}_{00} $, $ \partial_\phi {\cal K}_{01} $ 
and
$\partial_\phi\big(\partial_\phi \theta(\vphi)^\top[\zeta]\big)$
are $O(Z)$.

We look for an inverse of ${\mathbb D}$ 
by solving the system 
\begin{equation}\label{operatore inverso approssimato proiettato}
{\mathbb D} [\widehat \phi, \widehat \eta, \widehat v, \widehat \zeta] 
= 
\begin{pmatrix}
g_1  \\
g_2  \\
g_3 
\end{pmatrix}\,. 
\end{equation}
We first consider the second equation in \eqref{operatore inverso approssimato proiettato}, 
$ \omega \cdot \partial_\vphi  \widehat \eta  = 
g_2  +  \partial_\phi \theta(\vphi)^\top \widehat \zeta $. Since $\partial_\vphi \theta(\vphi) = {\rm Id} + \partial_\vphi \Theta(\vphi)$, 
the average 
$\langle \partial_\vphi \theta^\top \rangle_\vphi
= \frac{1}{(2\pi)^{|\Splus|}} \int_{\T^{\Splus}} \partial_\vphi \theta^\top(\vphi) d \vphi$
equals the identity matrix $ {\rm Id} $ of $\R^{\Splus}$. We then define 
\begin{equation}\label{definizione zeta hat}
\widehat \zeta := - \langle g_2 \rangle _\vphi\
\end{equation}
so that $\langle g_2  +  
\partial_\phi \theta(\vphi)^\top \widehat \zeta \rangle_\vphi$
vanishes and choose
\begin{equation}\label{soleta}
\widehat \eta := \widehat \eta_0 + \widehat \eta_1, \quad \widehat \eta_1 :=   ( \omega \cdot \partial_\vphi )^{-1} \big(
g_2  +  \partial_\phi \theta(\vphi)^\top \widehat \zeta \, \big) 
\end{equation}
where the constant vector $\widehat \eta_0 \in \R^{\Splus} $ will be determined in order to control the average of the first equation in \eqref{operatore inverso approssimato proiettato}. 
Next we consider the third equation in \eqref{operatore inverso approssimato proiettato},
$ ({\cal L}_\om^<) \widehat v = g_3 + \partial_x  {\cal K}_{11}(\vphi) \widehat \eta $,
which, by assumption \eqref{tame inverse}
on the inveritibility of ${\cal L}_\om^<$, has the solution 
\begin{equation}\label{normalw}
\widehat v := ({\cal L}_\om^<)^{-1} \big( g_3 + \partial_x  {\cal K}_{11}(\vphi) \widehat \eta_1 \big) + ({\cal L}_\om^<)^{-1} \partial_x  {\cal K}_{11}(\vphi) \widehat \eta_0  \, .  
\end{equation}
Finally, we solve the first equation in \eqref{operatore inverso approssimato proiettato}.
After substituting the solutions $\widehat \zeta$,
$\widehat \eta$, defined in \eqref{soleta}, and 
$\widehat v$, defined in \eqref{normalw}, this equation becomes
\begin{equation}\label{equazione psi hat}
\omega \cdot \partial_\vphi \widehat \phi  = 
g_1 +  M_1\widehat \eta_0  + M_2 g_2 + M_3 g_3 + M_4 \widehat \zeta 
\end{equation}
where $M_j: \varphi \mapsto M_j(\vphi)$, $1 \le j \le 4$, are 
defined as
\begin{align}\label{M1}
& M_1(\vphi) :=  - {\cal K}_{2 0}(\vphi) - {\cal K}_{11}(\vphi)^\top ({\cal L}_\omega^{<})^{- 1} \partial_x {\cal K}_{11}(\vphi) \,, \\
& \label{cal M2}
M_2(\vphi) :=  M_1(\vphi)[\omega \cdot \partial_\vphi]^{-1} \, , \\
& M_3(\vphi) :=  - {\cal K}_{11} (\vphi)^\top ({\cal L}_\om^<)^{-1} \, , \\
& \label{cal M4}
M_4(\vphi) :=  M_2(\vphi) \partial_\phi \theta(\vphi)^\top\,. 
\end{align}
In order to solve equation \eqref{equazione psi hat} we have 
to choose $ \widehat \eta_0 $ such that the right hand side
of it  has zero average.  
By Lemma \ref{coefficienti nuovi}, by the ansatz \eqref{ansatz 0} and 
 \eqref{tame inverse}, the $\ph$-averaged matrix is
$ \langle M_1 \rangle_\vphi =  \Omega_{\mathbb S_+}^{kdv}(\ac) + O( \e \gamma^{-2  }) $. Since the matrix $\Omega_{\mathbb S_+}^{kdv}(\ac) = (\partial _{I_k} \omega _n^{kdv}(\ac)) _{k,n \in \Splus}$ is invertible (cf. Lemma \ref{Proposition 2.3}-$(i)$,
Remark \ref{rem:diffeo}),   
$ \langle M_1 \rangle_\vphi$ is invertible 
for $ \e \gamma^{- 2} $ small enough and
$\langle M_1 \rangle_\vphi^{-1} = \Omega_{\mathbb S_+}^{kdv}(\ac)^{- 1}
+ O(\e \gamma^{- 2})$. We then define 
\begin{equation}\label{sol alpha}
\widehat \eta_0  := - \langle M_1 \rangle_\vphi^{-1} 
\Big( \langle g_1 \rangle_\vphi + \langle M_2 g_2 \rangle_\vphi + \langle M_3 g_3 \rangle_\vphi + 
\langle M_4 \widehat \zeta \rangle_\vphi \Big) \, .
\end{equation}
With this choice of $ \widehat \eta_0$,
the equation \eqref{equazione psi hat} has the solution
\begin{equation}\label{sol psi}
\widehat \phi :=
(\omega \cdot \partial_\vphi )^{-1} \big( g_1 + 
M_1[\widehat \eta_0] + M_2 g_2 + M_3 g_3 + M_4 \widehat \zeta \big) \, . 
\end{equation}
Altogether we have obtained
a solution  $(\widehat \phi, \widehat \eta, \widehat v, \widehat \zeta)$ of the linear system \eqref{operatore inverso approssimato proiettato}. 

\begin{proposition}\label{prop: ai}
Assume the ansatz \eqref{ansatz 0} with $\mu_0 = \mu(\mathtt b) + \sigma$ and the  estimates  \eqref{tame inverse} hold. 
Then, for any $\omega \in \Omega_o$ and 
any $ g := (g_1, g_2, g_3) $ 
with $ g_1, g_2 \in H^{s+\sigma}(\T^{\mathbb S_+}, \R^{\Splus})$,
$ g_3 \in H^{s+\sigma}_{\bot}(\T^{\mathbb S_+} \times \T_1) $,
and $\sM \leq s \leq S $, 
the system \eqref{operatore inverso approssimato proiettato} has a solution 
$  (\widehat \phi, \widehat \eta, \widehat v, \widehat \zeta ) := {\mathbb D}^{-1} g $,
where $\widehat \phi$, $\widehat \eta$, $\widehat v$, 
$\widehat \zeta$ are defined in \eqref{definizione zeta hat}-\eqref{normalw},
\eqref{sol alpha}-\eqref{sol psi} and satisfy
\begin{equation} \label{stima T 0 b}
\| {\mathbb D}^{-1} g \|_s^\Lipg
\lesssim_S \gamma^{-2} \big( \| g \|_{s + \sigma }^\Lipg 
+ N_0^{\tau_1} \gamma^{- 1} \| \io  \|_{s + \mu(\mathtt b) + \sigma}^\Lipg
 \| g \|_{\sM + \sigma}^\Lipg  \big).
\end{equation}
\end{proposition}

\begin{proof}
The proposition follows by the definitions of 
$ \widehat \zeta$ (cf. \eqref{definizione zeta hat}),
$\widehat \eta_1$ (cf. \eqref{soleta}),
$\widehat v$ (cf. \eqref{normalw}),
$\widehat \eta_0$ (cf. \eqref{sol alpha}),
$\widehat \phi$ (cf. \eqref{sol psi}),
the definitions of $M_j$, $1 \le j \le 4$, in \eqref{M1}-\eqref{cal M4},  the estimates of Lemma \ref{coefficienti nuovi}, 
and the assumptions  \eqref{ansatz 0} and \eqref{tame inverse}.
\end{proof}

Let  
$  \widetilde{G}_\delta:  (\phi, \eta, v, \zeta) \mapsto
   \big( G_\delta (\phi, \eta, v), \, \zeta \big) $ and note that
    its differential   
$ d  \widetilde{G}_\delta(\phi, \eta, v, \zeta)$ 
is independent of $ \zeta$. In the sequel, we denote it by
$ d  \widetilde{G}_\delta(\phi, \eta, v)$ or
 $d  \widetilde{G}_\delta|_{ (\phi, \eta, v )} $.  
Finally we prove that the operator 
\begin{equation}\label{definizione T} 
{\bf T}_0 := {\bf T}_0(\io ) := 
d { \widetilde G}_\delta |_{(\vphi,0,0)} \circ {\mathbb D}^{-1} \circ \big(d G_\delta |_{(\vphi,0,0)} \big)^{-1}
\end{equation}
is an almost-approximate right  inverse for $d_{\io,\zeta} {\cal F}_\omega(\io  )$.
Let 
$ \| (\phi, \eta, v, \zeta) \|_s^\Lipg :=$ 
$  \max \{  \| (\phi, \eta, v) \|_s^\Lipg, $ $ | \zeta |^\Lipg  \} $.

\begin{theorem}  \label{thm:stima inverso approssimato}
{\bf (Almost-approximate inverse)}
Assume that  \eqref{inversion assumption}-\eqref{tame inverse}
hold (Almost-invertibility of ${\cal L}_\omega$, 
$\omega \in \Omega_0$).
Then there exists $  \sigma_2 :=  \sigma_2(\tau, \mathbb S_+ ) > 0 $ so that, 
if the ansatz \eqref{ansatz 0} holds with $\mu_0 \geq \sM +  \mu(\mathtt b) +
 \sigma_2 $, then for any $ \omega \in \Omega_o $
and any $ g := (g_1, g_2, g_3) $ with
$ g_1, g_2 \in H^{s+\sigma}(\T^{\mathbb S_+}, \R^{\Splus})$,
$ g_3 \in H^{s+\sigma}_{\bot}(\T^{\mathbb S_+} \times \T_1) $,
and $\sM \leq s \leq S $, 
$ {\bf T}_0(\io) g$, defined by \eqref{definizione T}, satisfies
\begin{equation}\label{stima inverso approssimato 1}
\| {\bf T}_0(\io) g \|_{s}^\Lipg 
\lesssim_S  \gamma^{-2}  \Big(\| g \|_{s + \sigma_2}^\Lipg 
+ N_0^{\tau_1} \gamma^{- 1}  \| \io  \|_{s + \mu({\mathtt b}) +  \sigma_2 }^\Lipg
\| g \|_{\sM + \sigma_2}^\Lipg  \Big)\, .
\end{equation}
 Moreover  ${\bf T}_0(\io) $ is an almost-approximate inverse of $d_{\io, \zeta} 
 {\cal F}_\omega(\io)$, namely 
\begin{equation}\label{splitting per approximate inverse}
d_{\io , \zeta} {\cal F}_\omega (\io) \circ {\bf T}_0(\io) 
- {\rm Id} = {\cal P} + {\cal P}_\omega  + {\cal P}_\omega^\bot
\end{equation}
where 
\begin{align}
\| {\cal P} g \|_{\sM}^\Lipg & \lesssim_S \g^{-3 } \| {\cal F}_\omega(\io, \zeta) \|_{\sM + \sigma_2}^\Lipg \Big(1   + N_0^{\tau_1} \gamma^{- 1} \| \io \|_{\sM + \mu({\mathtt b}) + \sigma_2}^\Lipg \Big)\| g \|_{\sM + \sigma_2}^\Lipg  \label{stima inverso approssimato 2} \\
\| {\cal P}_\omega g \|_{\sM}^\Lipg & 
\lesssim_S  \e \g^{- 4} N_{n - 1}^{- {\mathtt a}} \big( 1  + N_0^{\tau_1} \gamma^{- 1}\| \iota \|_{\sM  + \mu({\mathtt b}) + \sigma_2}^\Lipg \big)   \|  g \|_{\sM + \sigma_2}^\Lipg \,, \label{stima cal G omega} \\
\| {\cal P}_\omega^\bot g\|_{\sM}^\Lipg & \lesssim_{S, b} 
\gamma^{- 2} K_n^{- b } \big( \| g \|_{\sM + \sigma_2 + b }^\Lipg +
N_0^{\tau_1} \gamma^{- 1}\| \io \|_{\sM +  \mu({\mathtt b}) + \sigma_2  +b    }^\Lipg \big \| g \|_{\sM +\sigma_2}^\Lipg \big)\,,\,\,
\quad\forall b > 0 \, .    \label{stima cal G omega bot bassa} 
\end{align}
\end{theorem}

\begin{proof}
The bound \eqref{stima inverso approssimato 1} follows 
from the definition of ${\bf T}_0(\io)$ 
(cf.\eqref{definizione T}), the estimate of $\mathbb D^{-1}$
(cf. \eqref{stima T 0 b}),  and the estimates of 
$dG_\delta(\vphi,0,0)$ and of its inverse (cf. \eqref{DG delta}).
By formula \eqref{diaF}) for $d_{\io, \zeta} {\cal F}_\omega(\io)$ 
and since only the $y-$components of $ \breve \io_\d $ and $ \breve \io $ differ from each other (cf. \eqref{y 0 - y delta}),  
the difference ${\cal E}_0 := d_{\io, \zeta} {\cal F}_\omega(\io)   - d_{\io, \zeta} {\cal F}_\omega(\io_\delta) $
can be written as
\be \label{verona 0}
{\cal E}_0 [\widehat \io, \widehat \zeta] 
=  \int_0^1  \partial_y d_\io X_{{\cal H}_\e } (\theta, y_\d + s (y - y_\d), w) [y - y_\d, \widehat \io]  ds.
\ee
We introduce the projection 
$\Pi: (\widehat \io, \widehat \zeta )  \mapsto  \widehat \io $. 
Denote by $  {\mathtt u} := (\phi, \eta, v ) $  the symplectic coordinates defined by $ G_\d $ 
(cf. \eqref{trasformazione modificata simplettica}). 
Under the symplectic map $G_\delta $, the nonlinear operator ${\cal F}_\omega$ (cf. \eqref{operatorF}) is transformed into 
\be \label{trasfo imp}
{\cal F}_\omega(G_\delta(  {\mathtt u} (\vphi) ), \zeta ) 
= d G_\delta( {\mathtt u}  (\vphi) ) [  \omega \cdot \partial_\vphi {\mathtt u} (\vphi) - X_{{\cal K}} ( {\mathtt u} (\vphi), \zeta) ] 
\ee
where $ {\cal K}= {\cal H}_{\e, \zeta} \circ G_\delta $
(cf.  \eqref{new-Hamilt-K}). 
Differentiating  \eqref{trasfo imp} at the trivial torus 
$ {\mathtt u}_\delta (\vphi) = G_\delta^{-1}(\io_\delta) (\vphi) = (\ph, 0 , 0 ) $, 
we get
\begin{align} \label{verona 2}
d_{\io , \zeta} {\cal F}_\omega(\io_\delta ) 
=  & \,  d G_\delta( {\mathtt u}_\delta) 
\big( \Dom 
- d_{\mathtt u, \zeta} X_{{\cal K}}( {\mathtt u}_\delta, \zeta) 
\big) d \widetilde G_\d ( {\mathtt u}_\d)^{-1}
+ {\cal E}_1 \,,
\\
{\cal E}_1 
:=  \, & 
d^2 G_\delta( {\mathtt u}_\delta) \big[ d G_\delta( {\mathtt u}_\delta)^{-1} {\cal F}_\omega(\io_\delta, \zeta), \,  d G_\d({\mathtt u}_\d)^{-1} 
 \Pi [ \, \cdot \, ] \,  \big] \, .  \label{verona 2 0} 
\end{align}
In expanded form $ \Dom  - d_{\mathtt u, \zeta} X_{{\cal K}}( {\mathtt u}_\delta, \zeta) $ is provided by \eqref{lin idelta}.
Recalling the definition of $\mathbb{D}$ in 
\eqref{operatore inverso approssimato} and the discussion following it,
we decompose $\om  \cdot  \pa_\vphi 
- d_{\mathtt u, \zeta} X_{\cal K}( {\mathtt u}_\delta, \zeta) $ as
\begin{equation}\label{splitting linearizzato nuove coordinate}
\om  \cdot  \pa_\vphi 
- d_{\mathtt u, \zeta} X_{\cal K}( {\mathtt u}_\delta, \zeta) 
= \mathbb{D}  
+ R_Z   + {\mathbb R}_\omega+ 
{\mathbb R}_\omega^\bot  
\end{equation}
where  
$$
R_Z [  \widehat \phi, \widehat \eta, \widehat v, \widehat \zeta]
:= \begin{pmatrix}
  \partial_\phi {\cal K}_{10}(\vphi) [\widehat \phi ] \\
- \partial_{\phi \phi} {\cal K}_{00} (\vphi) [ \widehat \phi ] - \partial_\phi\big(\partial_\phi \theta(\vphi)^\top[\zeta]\big)[\widehat \phi] -
 [\partial_\phi {\cal K}_{10}(\vphi)]^\top \widehat \eta -
 [\partial_\phi {\cal K}_{01}(\vphi)]^\top \widehat v   \\
 - \partial_x  \big( \partial_{\phi} {\cal K}_{01}(\vphi)[ \widehat \phi ] \big)
 \end{pmatrix} 
$$
and 
$$
{\mathbb R}_\omega[\widehat \phi, \widehat y, \widehat w, \widehat \zeta] := \begin{pmatrix}
0 \\
0 \\
{\cal R}_\omega [\widehat w]
\end{pmatrix}\,,\qquad {\mathbb R}_\omega^\bot[\widehat \phi, \widehat y , \widehat w, \widehat \zeta] := \begin{pmatrix}
0 \\
0 \\
{\cal R}_\omega^\bot[\widehat w]
\end{pmatrix}\,.
$$
By \eqref{verona 0} and \eqref{verona 2}-\eqref{splitting linearizzato nuove coordinate} we get the decomposition
\begin{align} 
 d_{\io, \zeta} {\cal F}_\omega(\io ) 
& = d G_\delta({\mathtt u}_\delta) \circ {\mathbb D} 
\circ \big( d {\widetilde G}_\delta ({\mathtt u}_\delta)\big)^{-1} + {\cal E} + {\cal E}_\omega + {\cal E}_\omega^\bot  \label{E2}
 \end{align}
where 
\begin{equation}\label{cal E (n) omega}
{\cal E} := {\cal E}_0 + {\cal E}_1 + 
d G_\delta ( {\mathtt u}_\delta)  R_Z 
\big(d {\widetilde G}_\delta ({\mathtt u}_\delta)\big)^{-1}\,, 
\end{equation}
\begin{equation}\label{cal E omega bot}
 {\cal E}_\omega := d G_\delta ( {\mathtt u}_\delta)   {\mathbb R}_\omega  \big(d {\widetilde G}_\delta ({\mathtt u}_\delta)\big)^{-1}\,,
 \qquad
   {\cal E}_\omega^\bot :=   d G_\delta( {\mathtt u}_\delta)  {\mathbb R}_\omega^\bot  
\big( d {\widetilde G}_\delta ({\mathtt u}_\delta)\big)^{-1} \, . 
\end{equation}
Letting the operator $ {\bf T}_0 = {\bf T}_0(\io)$ 
(cf. \eqref{definizione T})
act from the right to both sides of the identity \eqref{E2} 
and recalling that $ {\mathtt u}_\delta (\vphi) = (\vphi, 0, 0 ) $, 
one obtains
\begin{align*}
&  \qquad \quad d_{\io, \zeta} {\cal F}_\omega(\io) \circ {\bf T}_0  - {\rm Id} 
={\cal P} + {\cal P}_\omega + {\cal P}_\omega^\bot\,, \qquad {\cal P} := {\cal E} \circ {\bf T}_0, \quad
 {\cal P}_\omega := \mE_\omega \circ {\bf T}_0 \, , \quad {\cal P}_\omega^\bot :=  \mE_\omega^\bot \circ {\bf T}_0 \, . 
\end{align*}
To derive the claimed estimate for $\cal P$
we first need to estimate $\mathcal E$.
By \eqref{ansatz 0}, \eqref{zeta Z inequality} (estimate for 
$\zeta$), 
\eqref{K 00 10 01} (estimates related to $\io_\delta$), 
\eqref{2015-2}--\eqref{stima toro modificato} (estimates of the components of $R_Z$), and
 \eqref{DG delta}-\eqref{DG2 delta}  (estimates of $dG_\delta(u_\delta)$ and its inverse)
 one infers that 
\begin{equation}\label{stima parte trascurata 2}
\| {\cal E} [\, \widehat \io, \widehat \zeta \, ] \|_s^\Lipg \lesssim_s \gamma^{- 1} \Big( 
 \| Z \|_{s_0 + \s}^\Lipg \| \widehat \io \|_{s + \s}^\Lipg +  
\| Z \|_{s + \s}^\Lipg \| \widehat \io \|_{s_0 + \s}^\Lipg + 
\| Z \|_{s_0 + \s}^\Lipg  \| \io \|_{s+\s}^\Lipg  \| \widehat \io \|_{s_0 + \s}^\Lipg \Big)  \, ,
\end{equation}
for some $\sigma > 0$, where $Z$ is the error function, $ Z = \mF_\omega(\io , \zeta)$ (cf. \eqref{def Zetone}). 
The claimed estimate \eqref{stima inverso approssimato 2}
for $\mathcal P$ then follows from \eqref{stima parte trascurata 2},
the estimate \eqref{stima inverso approssimato 1} of ${\bf T}_0$, 
and the ansatz \eqref{ansatz 0}. 
The claimed estimates \eqref{stima cal G omega}, 
\eqref{stima cal G omega bot bassa} for ${\cal P}_\omega$
and, respectively, ${\cal P}_\omega^\bot$
 follow by the assumed estimates 
 \eqref{stima R omega corsivo}-\eqref{stima R omega bot corsivo bassa} of 
 ${\cal R}_\omega$ and ${\cal R}_\omega^\bot$, 
 the estimate \eqref{stima inverso approssimato 1} of ${\bf T}_0$,  
 the estimate \eqref{DG delta} of $dG_\delta(u_\delta)$
 and its inverse,
 and the ansatz \eqref{ansatz 0}.  
 \end{proof}

The goal of Sections \ref{linearizzato siti normali} and  \ref{sec: reducibility} below is to prove that 
the Hamiltonian operator  $ {\mathcal L}_\omega $, defined in \eqref{Lomega def},
satisfies the almost-invertibility property stated in \eqref{inversion assumption}-\eqref{tame inverse}. 

\section{Reduction of $ {\mathcal L}_\omega $ up to order zero}
\label{linearizzato siti normali} 

The goal of this section is to reduce the Hamiltonian operator 
${\mathcal L}_\omega$, defined in \eqref{Lomega def},
to a differential operator of order three
with constant coefficients, up to order zero  -- see \eqref{Op:L4} below for a precise statement.
In the sequel, we consider torus embeddings
$\breve{\io}(\vphi) = (\vphi, 0, 0) + \io(\vphi)$ with 
$\io(\cdot) \equiv \io (\cdot \ ; \omega) $,  
$ \omega \in \mathtt{DC}(\g,\t) $ 
(cf. \eqref{diofantei in omega}, satisfying 
\begin{equation}\label{ansatz I delta}
\| \iota \|_{ \mu_0}^{\Lip(\g)} \lesssim \e \g^{-2}   \ \, , \qquad 
\e \gamma^{- 2} \leq \delta(S) 
\end{equation}
where $ \mu_0 := \mu_0(\tau, \Splus) > s_0$, $S > s_0$ 
are sufficiently large, $0 < \delta(S) < 1$ is sufficiently small, and $0 < \gamma < 1 $. 
The Sobolev index $S$ will be fixed in \eqref{costanti nash moser 2}.
In the course of the Nash-Moser iteration
we will verify that \eqref{ansatz I delta} is satisfied by 
each approximate solution -- see the bounds \eqref{ansatz induttivi nell'iterazione}. 
For a quantity $ g(\io) \equiv g(\breve{\io})$ such as an operator, a map, or a scalar function, depending on 
$\breve{\io}(\vphi) = (\vphi, 0, 0) + \io(\vphi)$, we denote for any
two such tori embeddings $\breve{\io}_1$, $\breve{\io}_2$ by 
$ \Delta_{12} g$ the difference
$$ \Delta_{12} g := g (\io_2) - g (\io_1) \, .
$$ 

\subsection{Expansion of $ {\mathcal L}_\omega $}

As a first step, we derive an expansion of the operator 
$ {\mathcal L}_\om = \Pi_\bot \big(\Dom   - \partial_x {\cal K}_{02}(\vphi) \big)_{|L^2_\bot}$, defined in \eqref{Lomega def}. 

\begin{lemma} \label{thm:Lin+FBR}
The Hamiltonian operator  
$  \partial_x {\mathcal K}_{02}(\vphi) $ acting on $ L^2_\bot(\T_1) $ is of the form
\begin{equation}\label{K 02}
\partial_x {\mathcal K}_{02}(\vphi) = \Pi_\bot \partial_x  ( d_\bot  \nabla_w 
{\mathcal H}_\e) ( \breve \io_\delta(\vphi)) + R(\vphi) 
\end{equation}
where $ {\mathcal H}_\e $ is the 
Hamiltonian defined in \eqref{HepNM} and the remainder $ R(\vphi) $ is given by 
\begin{equation}\label{forma buona resto}
R(\vphi)[h] = {\mathop \sum}_{j \in \Splus} \big(h\,,\,g_j \big)_{L^2_x} \chi_j\,, \quad \forall 
h \in L^2_\bot(\T_1) \, ,  
\end{equation}
with functions $ g_j, \chi_j 
\in H^s_\bot$,
$ j \in \Splus $, satisfying,  
for some $ \sigma:= \sigma(\tau, \Splus) > 0 $ and any $s \geq s_0 $
\begin{align}\label{stime gj chij}
\| g_j\|_s^{\Lipg} +\| \chi_j\|_s^{\Lipg} \lesssim_s \e + 
\| \io \|_{s + \s}^{\Lipg}\,. 
\end{align}
Let $s_1 \geq s_0$ and let $\breve{\io}_1, \breve{\io}_2$ be two tori satisfying \eqref{ansatz I delta} with $\mu_0 \geq s_1 + \sigma$. Then, for any $ j \in \Splus $,  
$$
\| \Delta_{12} g_j \|_{s_1} +\| \Delta_{12} \chi_j
\|_{s_1} \lesssim_{s_1} \| \io_2 - \io_1 \|_{s_1 + \s} \, . 
$$
\end{lemma}

\begin{proof}
The lemma follows as in \cite[Lemma 6.1]{Berti-Montalto},
using  Lemma \ref{stime tame campo ham per NM AI} and the ansatz \eqref{ansatz I delta}. 
\end{proof}

By Lemma \ref{thm:Lin+FBR} the linear Hamiltonian operator 
$ {\cal L}_\om $ has the form 
\be\label{representation Lom}
{\cal L}_\om = {\cal L}_\omega^{(0)} - R \, , 
\qquad {\cal L}_\omega^{(0)}  
:= \om \cdot \pa_{\vphi} 
 - \Pi_{\bot} \pa_x ( d_\bot \nabla_w {\mathcal H}_\e) (\breve{\io}_\delta (\vphi)  ) 
\ee
where here and in the sequel, we  write 
$\omega \cdot \partial_\vphi$ instead of  
$\Pi_\bot \ \omega \cdot \partial_\vphi |_{L^2_\bot}$
in order to simplify notation. 
We now prove that the Hamiltonian operator 
$ {\cal L}^{(0)}_\omega $, acting on $ L^2_\bot(\T_1) $,  
is a sum of a pseudo-differential operator of order three, 
a Fourier multiplier with $\varphi-$independent coefficients
and a small smoothing remainder. 
Since
${\mathcal H}_\e = {\mathcal H}^{kdv} + \e {\mathcal P}$
(cf. \eqref{HepNM}) and 
$ \partial_x d_\bot \nabla_w {\cal H}^{kdv} 
= \partial_x \Omega^{kdv} +  
\partial_x d_\bot \nabla_w {\cal R}^{kdv}$ (cf. \eqref{Hkdv})
we have
\begin{align}
{\cal L}^{(0)}_\omega 
& = 
\omega \cdot \pa_\vphi + \pa_x^3 - \Pi_{\bot} 
Q_{-1}^{kdv} (D;  \om) - 
\Pi_{\bot} \partial_x d_\bot \nabla_w {\cal R}^{kdv}(\breve \io_\delta ) -
\e \Pi_{\bot} \partial_x d_\bot \nabla_w {\cal P}(\breve \io_\delta ) \label{Lom0-1}
\end{align}
where we write $\pa_x^3$ instead of
$\pa_x^3|_{L^2_\bot}$ and where $Q_{-1}^{kdv} (D;  \om)$
is given by (cf.  \eqref{deco:Q-1})
\be\label{decokdv}
Q_{-1}^{kdv} (D;  \om) \equiv
Q_{-1}^{kdv} ( D;  \nu(\omega)) 
= \pa_x \Omega^{kdv}(D; \nu(\omega)) + \pa_x^3 \, , 
\ee 
with $\nu(\omega)$ defined in \eqref{mu-kdv-om}.    
The operator $Q_{-1}^{kdv} (D;  \om)$ is a Fourier multiplier
with $\varphi-$independent coefficients. It admits an expansion
as described in the following lemma. 
\begin{lemma}\label{lemma resto frequenze kdv}
For any $M \in \N$, 
\begin{equation}\label{espansione simboli Q kdv - 1}
Q_{-1}^{kdv} (D; \om) = \sum_{k = 1}^M c_{ - k}^{kdv}(\omega) \partial_x^{- k} + {\cal R}_M( Q_{-1}^{kdv}; \om)  
\end{equation}
where for any $1 \le k \le M$, the function
$\Omega \to \R , \, \omega \mapsto c_{ - k}^{kdv}(\omega)$ is Lipschitz  and where 
${\cal R}_M( Q_{-1}^{kdv}; \om): L^2_\bot(\T_1) \to L^2_\bot(\T_1)$ 
is a Lipschitz family of diagonal operators of order  $- M - 1$. 
Furthermore, for any $n_1, n_2 \in \N$, $n_1 + n_2 \leq  M + 1$, 
the operator 
$\langle D \rangle^{n_1} {\cal R}_M( Q_{-1}^{kdv}; \om) \langle D \rangle^{n_2} $ is 
$\Lipg$-tame  with a tame constant satisfying ${\mathfrak M}_{\langle D \rangle^{n_1} {\cal R}_M(  Q_{-1}^{kdv}; \om) \langle D \rangle^{n_2}}(s) \leq C(s, M)$ for any $s \geq s_0$ and $C(s, M) > 0$.  
\end{lemma}  

\begin{proof}
The claimed statements follow by Lemma \ref{lem:anal}.
\end{proof}

\begin{lemma}\label{Lemma di partenza riduzione}
For any $M \in \N$, the Hamiltonian operator 
$ {\mathcal L}^{(0)}_\omega $, acting on $ L^2_\bot(\T_1) $,
defined  in \eqref{representation Lom},
 admits an expansion of the form
\be\label{L0}
{\cal L}^{(0)}_\omega :=
\omega \cdot \pa_\vphi - \Pi_{\bot} \Big( a_3^{(0)} \pa_x^3 + 
2 (a_3^{(0)})_x \pa_x^2 + a_1^{(0)} \pa_x  + {\rm Op}( r_0^{(0)}) +
Q_{-1}^{kdv} (D ; \om) \Big)  + {\mathcal R_M^{(0)}}( \breve \io_\delta (\vphi); \omega)
\ee
where $ a_3^{(0)} := a_3^{(0)} (\vphi, x; \o) $, $ a_1^{(0)} := a_1^{(0)} (\vphi, x; \o) $ are real valued functions satisfying for any $s \ge s_0$ 
\be\label{a30-a10}
\|  a_3^{(0)}  + 1 \|_s^\Lipg \lesssim_{s, M} \e ( 1 + \| \iota \|_{s+ \sigma_M }^\Lipg )  \, , \quad
\|  a_1^{(0)}   \|_s^\Lipg \lesssim_{s, M} \e  + \| \iota \|_{s+\sigma_M}^\Lipg 
\ee
for some $\s_M > 0 $.
The pseudo-differential symbol  $ r_0^{(0)} := r_0^{(0)} (\vphi, x, \xi; \o) $ has an expansion in homogeneous components 
\be\label{sym:r0}
r_0^{(0)}(\vphi, x, \xi; \o) = \sum_{k=0}^{M} a_{-k}^{(0)}(\vphi, x; \o) (\ii 2 \pi \xi)^{-k} \chi_0(\xi) 
\ee
(with $\chi_0$ defined in \eqref{def chi 0}) 
where the coefficients $ a_{-k}^{(0)} 
:= a_{-k}^{(0)}(\vphi, x; \o) $ satisfy
\be\label{aM0} 
\sup_{k=0, \ldots, M} \| a_{-k}^{(0)} \|_s^\Lipg \lesssim_{s,  M} \e  + \| \iota \|_{s+\sigma_M}^\Lipg \, , \quad \forall s \geq s_0\,,
\ee
the  remainder is defined by
\be\label{R0:dyn}
 {\mathcal R_M^{(0)}}( \breve \io_\delta (\vphi); \omega ) := 
- {\cal R}_M (\breve \io_\delta (\vphi); \ac(\omega) ; \partial_x d_\bot \nabla_w {\cal R}^{kdv}) -  \e {\cal R}_M (\breve \io_\delta (\vphi); \ac(\omega) ; \partial_x d_\bot \nabla_w {\cal P}) 
\ee
and the latter two remainder terms are given by
 \eqref{pseudodifferential expansion of remainder mathcal R_Nkdv}
 and \eqref{pseudodifferential expansion of d bot nabla P}
with $\ac(\omega ) = (\omega^{kdv})^{- 1}(- \omega)$.

Let $s_1 \geq s_0$ and let $\breve{\io}_1, \breve{\io}_2$ be two tori satisfying \eqref{ansatz I delta} for $\mu_0 \geq s_1 + \sigma_M$. Then, for any $0 \le k \le M +1$,  
\begin{equation}\label{stima Delta 1 2 a k (0)}
\| \Delta_{12} a_3^{(0)}\|_{s_1} \lesssim_{s_1, M} \e \| \io_1 - \io_2 \|_{s_1 + \sigma_M},  \quad 
\|\Delta_{12} a_{1 - k}^{(0)}\|_{s_1}  \lesssim_{s_1, M} \|  \io_1 - \io_2 \|_{s_1 + \sigma_M}\,. 
\end{equation}
\end{lemma}

\begin{proof}
By the definition \eqref{Lom0-1} of ${\cal L}^{(0)}_\omega$, 
the expansion 
\eqref{pseudodifferential expansion of remainder mathcal R_Nkdv}
of $ \partial_x d_\bot \nabla_w {\cal R}^{kdv}$, 
the expansion
\eqref{pseudodifferential expansion of d bot nabla P}
of $\partial_x d_\bot \nabla_w {\cal P}$, 
and the formula for the coefficient of $\partial_x^2$, described 
in Lemma \ref{lem:229}, one obtains  \eqref{L0} 
with
$$
\begin{aligned}
a_3^{(0)}(\vphi, x; \o) & := - 1 +  \e a_3(\breve \io_\delta (\vphi); \ac(\omega) ; \partial_x d_\bot \nabla_w {\cal P})\,, \\
a_1^{(0)}(\vphi, x; \o)  & := a_1(\breve \io_\delta (\vphi); \ac(\omega) ; \partial_x d_\bot \nabla_w {\cal R}^{kdv}) +  \e a_1(\breve \io_\delta (\vphi); \ac(\omega) ; \partial_x d_\bot \nabla_w {\cal P})\,, \\
a_{-k}^{(0)}(\vphi, x; \o) & :=  a_{- k}(\breve \io_\delta (\vphi); \ac(\omega) ; \partial_x d_\bot \nabla_w {\cal R}^{kdv}) +  \e a_{- k}(\breve \io_\delta(\vphi); \ac(\omega) ; \partial_x d_\bot \nabla_w {\cal P}) \, , \quad k = 0, \ldots, M\,, 
 \end{aligned}
$$
and  $\ac(\omega ) = (\omega^{kdv})^{- 1}(- \omega)$. By 
Lemma \ref{differential nabla kdv remainder-true}-1, the functions $a_{1 - k}(\mathfrak x; \ac(\omega) ; \partial_x d_\bot \nabla_w {\cal R}^{kdv})$, $0 \le k \le  M + 1$,  satisfy the hypothesis of Lemma \ref{lem:tame1}-($ii$). 
In view of \eqref{2015-2} one then infers that
 for any $s \geq s_0$ 
$$
\| a_{1 - k}(\breve \io_\delta (\vphi); \ac(\omega) ; \partial_x d_\bot \nabla_w {\cal R}^{kdv})\|_s^\Lipg \lesssim_{s, M} \| \io\|_{s + \sigma_M}^\Lipg  
$$ 
for some  $\sigma_M > 0$. 
Similarly, by the first item of  
Lemma \ref{differential nabla perturbation-true},  
the functions
$a_{3 - k}(\breve \io_\delta (\vphi); \ac(\omega) ; \partial_x d_\bot \nabla_w {\cal P})$,  $0 \le k \le  M + 3$, satisfy the hypothesis
of Lemma \ref{lem:tame1}-($i$), implying that for any $s \geq s_0$, 
$$
\| a_{3 - k}(\breve \io_\delta (\vphi); \ac(\omega) ; \partial_x d_\bot \nabla_w {\cal P}) \|_s^\Lipg \lesssim_{s, M} 1 + \| \io\|_{s + \sigma_M}^\Lipg
$$ 
for some  $\sigma_M > 0$, proving \eqref{a30-a10}, \eqref{aM0}. 
The estimates \eqref{stima Delta 1 2 a k (0)} follow by similar arguments.
\end{proof}

We remark that in the finitely many steps of our reduction procedure,
described in this section,
 the {\it loss of derivatives} $\sigma_M = \sigma_M(\tau, {\mathbb S}_+) > 0$ might have to be increased, but the notation will not be changed. 

\subsection{Quasi-periodic reparametrization of time} 
\label{sec:RT}

We conjugate the operator $ {\mathcal L}_\omega $  
(cf.  \eqref{representation Lom}) 
by the change of variable induced by the quasi-periodic reparametrization of time  
$$
\vartheta = \vphi + \alpha^{(1)} (\vphi) \omega  
\qquad \text{or equivalently} \qquad \vphi = \vartheta  + \breve \alpha^{(1)} (\vartheta ) \omega
$$
where 
$ \alpha^{(1)}: \T^{\Splus} \to \R $, is a small, real valued  function  
chosen below (cf. \eqref{alphi(2)}). Denote by 
\begin{equation}\label{definition Phi (2)}
(\Phi^{(1)} h) (\vphi, x) := h( \vphi + \alpha^{(1)} (\vphi) \omega, x ) \, , \quad   
((\Phi^{(1)})^{-1} h) (\vartheta, x) := 
h(  \vartheta  + \breve \alpha^{(1)} (\vartheta ) \omega, x ) \, ,
\end{equation}
the induced diffeomorphisms on functions. The goal is to achieve
that the operator ${\mathcal L}_\omega^{(1)}$, defined in
\eqref{def:L2}, is of the form \eqref{Op:L2}, so that 
its highest order coefficient $a_3^{(1)}$ satisfies
\eqref{constant:sec}. The latter property will allow us in Section \ref{elimination x dependence highest order}
to conjugate ${\mathcal L}_\omega^{(1)}$  
to an operator with constant highest order 
coefficient (cf. \eqref{Op:L1}). 

Since by \eqref{a30-a10}, the coefficient $a_3^{(0)}$ satisfies  
$ a_3^{(0)} = -1 + O(\e)$, the expression
$(a_3^{(0)}(\vphi,x))^{\frac13}$ is well defined where 
$(x)^{\frac13}$ denotes the branch of the third root of 
$x \in (- \infty , 0)$, determined by $(-1)^{\frac13} = -1$. 
\begin{lemma}
Let $m_3$ be the constant 
\be\label{def:m3}
m_3(\omega)  := 
 \frac{1}{(2 \pi)^{|\Splus|}} \int_{\T^{\Splus}}   
\Big( \int_{\T_1} \frac{dx}{ (a_3^{(0)}(\vartheta,x; \o))^{\frac13}} \Big)^{-3}\, d \vartheta  \, ,
\ee 
and define, for  $ \omega \in \mathtt{DC}(\g,\t)$, the function
\be\label{alphi(2)}
\breve \alpha^{(1)} (\vartheta; \o) := 
(\om \cdot \pa_{\vphi} )^{-1} \Big[ \frac{1}{m_3}  
\Big( \int_{\T_1} \frac{dx}{ (a_3^{(0)}(\vartheta,x; \o))^{\frac13}}  \Big)^{-3}
 - 	1  \Big] \, .
\ee
Then for any $M \in \N$, there exists a constant $\s_M > 0$ so that the following holds:\\
(i) The constant $ m_3 $ satisfies
\be\label{m3Lip}
|m_3 + 1|^\Lipg \lesssim_M \e \,
\ee
and for any $s \geq s_0$, $\alpha^{(1)}, \breve \alpha^{(1)}$ 
satisfy 
\begin{equation}\label{stima m3}
\| \alpha^{(1)} \|_s^\Lipg, \| \breve \alpha^{(1) }\|_s^\Lipg 
\lesssim_{s, M} 
\e \gamma^{- 1}(1 + \| \io\|_{s + \sigma_M}^\Lipg)\,. 
\end{equation}
(ii) The Hamiltonian operator 
\be \label{def:L2} 
{\mathcal L}_\omega^{(1)} := \frac{1}{ \rho } \Phi^{(1)} {\mathcal L}_\omega 
\, (\Phi^{(1)})^{- 1} \, , 
\quad \rho(\vartheta) := \Phi^{(1)} ( 1 + \om \cdot \pa_{\vartheta} \breve \alpha^{(1)} ) = 1 + \Phi^{(1)} ( \om \cdot \pa_{\vartheta} \breve \alpha^{(1)} ) \, , 
\ee
admits an expansion of the form
\be \label{Op:L2}
{\mathcal L}_\omega^{(1)} =  
\omega \cdot \pa_\vartheta -   \Big(  a_3^{(1)}   \pa_x^3 +  
2 (a_3^{(1)})_x  \pa_x^2 
+ a_1^{(1)} \pa_x  + {\rm Op}( r_0^{(1)}) +  
Q_{-1}^{kdv} (D ;\om)\Big)  + {\mathcal R_M^{(1)}}
\ee
where the coefficients 
$ a_3^{(1)} :=   a_3^{(1)}(\vartheta, x; \om)  $, $ a_1^{(1)} := a_1^{(1)} (\vartheta, x; \om) $ 
are real valued and satisfy
\begin{equation}\label{stima a 1 (2)}
\|  a_3^{(1)} +1  \|_s^\Lipg \lesssim_{s, M} \e (1 + \| \io\|_{s + \sigma_M}) , \quad 
\|  a_1^{(1)}   \|_s^\Lipg  \lesssim_{s, M} \e  + \| \iota \|_{s+\sigma_M}^\Lipg \, , \quad \forall s \geq s_0   \,,
\end{equation}
and 
\be\label{constant:sec}
\int_{\T_1} \frac{dx}{ (a_3^{(1)}( \vartheta, x; \o))^{\frac13} }  =  m_3^{- \frac13}  \, , \quad
\forall  \vartheta \in \T^{\Splus} \,. 
\ee  
 The function  $ r_0^{(1)} \equiv r_0^{(1)} (\vartheta, x, \xi; \o)$
 is a pseudo-differential symbol in $S^0$ and admits an expansion of the form
 \begin{equation}\label{stima r 0 (1)}
 r_0^{(1)} (\vartheta, x, \xi; \o ) = \sum_{k = 0}^M 
 a_{- k}^{(1)}(\vartheta, x; \o) (\ii 2 \pi \xi)^{ - k} \chi_0(\xi)
 \end{equation}
 where for any $0 \le k \le  M $, $s \geq s_0$, 
 \be\label{est:r1a1}
\| a_{- k}^{(1)} \|_s^\Lipg  \lesssim_{s,  M} \e  + \| \iota \|_{s+\sigma_M}^\Lipg \, .
\ee
Furthermore, the function $\rho$ appearing in \eqref{def:L2} satisfies 
\begin{equation}\label{stima rho}
\| \rho - 1\|_s^\Lipg, \  \| \rho^{- 1} - 1\|_s^\Lipg \lesssim_{s, M} \e + \| \io\|_{s + \sigma_M}^\Lipg\,.
\end{equation}
Let $s_1 \geq s_0$ and let $\io_1, \io_2$ be two tori satisfying \eqref{ansatz I delta} with $\mu_0 \geq s_1 + \sigma_M$. Then 
\begin{equation}\label{stime delta 12 secondo step}
\begin{aligned}
&
|  \Delta_{12} 
m_3 |, \| \Delta_{12} \alpha^{(1)} \|_{s_1}, \| \Delta_{12}\breve \alpha^{(1)} \|_{s_1}, \| \Delta_{12} a_1^{(1)}   \|_{ s_1}\,,\, \| \Delta_{12} \rho^{\pm 1}\|_{s_1} \lesssim_{s_1, M} \| \io_1 - \io_2\|_{s_1 + \sigma_M} \, , \\
&  \|\Delta_{12} a_{- k}^{(1)} \|_{s_1} \lesssim_{s_1, M} \| \io_1 - \io_2\|_{s_1 + \sigma_M}\,, \quad \forall k= 0, \ldots, M\,. 
\end{aligned}
\end{equation} 
(iii) Let $S > \sM$ where $\sM$ is defined in \eqref{def sM}. Then the maps $(\Phi^{(1)})^{\pm 1}$ are $\Lipg$-1-tame operators with 
a tame constant satisfying 
\begin{equation}\label{stima tame Phi (2) enunciato}
\mathfrak M_{(\Phi^{(1)})^{\pm 1}}(s) \lesssim_{S, M} 1 + \| \io \|_{s + \sigma_M}^\Lipg, \quad \forall s_0 + 1 \leq s \leq S\,. 
\end{equation} 
For any given $\lambda_0 \in \N$ there exists a constant $\sigma_M(\lambda_0) > 0$ so that  for any $ m \in \Splus $,
$ \lambda,  n_1, n_2 \in \N $ with $ \lambda \leq \lambda_0 $ 
and $n_1 + n_2 +  \lambda_0 \leq M + 1 $,  the operator 
$ \pa_{\vphi_m}^\lambda \langle D \rangle^{n_1}{\mathcal R_M^{(1)}} \langle D \rangle^{n_2}$ is $\Lipg$-tame 
with a tame constant satisfying 
\be\label{SM2}
{\mathfrak M}_{\pa_{\vphi_m}^\lambda  \langle D \rangle^{n_1}{\mathcal R_M^{(1)}} \langle D \rangle^{n_2}} (s) \lesssim_{S, M} \e  + \| \iota \|_{s+\sigma_M(\lambda_0)}^\Lipg \,, \quad \forall \sM \leq s \leq S \, .
\ee
If in addition $s_1 \geq \sM$ and $\breve \io_1, \breve \io_2$ are two  tori satisfying \eqref{ansatz I delta}  with $\mu_0 \geq s_1 + \sigma_M(\lambda_0)$, then 
\begin{equation}\label{stima delta 12 RM (2)}
\| \pa_{\vphi_m}^\lambda  \langle D \rangle^{n_1} \Delta_{12} {\mathcal R_M^{(1)}} \langle D \rangle^{n_2} \|_{{\cal B}(H^{s_1})} \lesssim_{s_1, M, \lambda_0} \| \io_1 - \io_2\|_{s_1 + \sigma_M(\lambda_0)}\,. 
\end{equation}
\end{lemma}

\begin{proof}
Writing $\Pi_\bot$ as ${\rm Id} + (\Pi_\bot - {\rm Id} )$
the expression \eqref{L0} for ${\cal L}_\omega^{(0)}$ becomes 
$$
{\cal L}_\omega^{(0)} = 
\omega \cdot \pa_\vphi -  \Big( a_3^{(0)} \pa_x^3 + 
2 (a_3^{(0)})_x \pa_x^2 + a_1^{(0)} \pa_x  + {\rm Op}( r_0^{(0)}) +
Q_{-1}^{kdv} (D ; \om) \Big)  + 
{\cal R}_M^{(I)}( \breve \io_\delta (\vphi); \omega)
+ {\mathcal R}_M^{(0)}( \breve \io_\delta (\vphi); \omega) 
$$
where using that $({\rm Id} - \Pi_\bot)\partial_x^3 h = 0$
for any $h \in H^s_\bot$, the operator 
${\cal R}_M^{(I)} \equiv 
{\cal R}_M^{(I)}( \breve \io_\delta (\vphi); \omega)$ 
can be written as
\begin{equation}\label{def cal RM (I)}
{\cal R}_M^{(I)} =  ({\rm Id} - \Pi_\bot)\Big( (a_3^{(0)} + 1) \pa_x^3 + 
2 (a_3^{(0)})_x \pa_x^2 + a_1^{(0)} \pa_x  + {\rm Op}( r_0^{(0)}) \Big) \, .
\end{equation}
Since 
$({\rm Id} - \Pi_\bot) h = \sum_{j \in \mathbb S} 
\big(h, e^{- \ii 2\pi j x} \big)_{L^2_x} e^{\ii 2\pi j x} $ for any
$h$ in $L^2_x$, $ {\cal R}_M^{(I)} $
is a finite rank operator of the form \eqref{forma buona resto} 
with functions $ g_j, \chi_j \in H^s_\bot$ satisfying \eqref{stime gj chij}
(use \eqref{a30-a10}, \eqref{aM0}). 

The estimate \eqref{stima tame Phi (2) enunciato} follows by  
Lemma \ref{lemma:LS norms}-($iii$) and \eqref{stima m3}. 
Note  that 
$$
\Phi^{(1)} \circ \om \cdot \pa_\vphi \circ (\Phi^{(1)})^{-1} 
= \rho (\vartheta) \om \cdot \pa_\vartheta \, , 
\quad \rho := \Phi^{(1)} ( 1 + \om \cdot \pa_{\vphi} \breve \alpha^{(1)} ) \, , 
$$
and that any Fourier multiplier $ g(D) $ is left unchanged under conjugation, i.e. 
$ \Phi^{(1)} g(D) (\Phi^{(1)})^{-1} = g(D)$. 
Using \eqref{representation Lom} and \eqref{L0}, we obtain \eqref{Op:L2}  where
\be\label{a31}
a_3^{(1)} := \Phi^{(1)} \Big( \frac{a_3^{(0)}}{1 + \om \cdot \pa_\vphi 
\breve \a^{(1)}}  \Big) \, ,
\ee
$a_1^{(1)} := \frac{1}{\rho} \Phi^{(1)} (a_1^{(0)})$, 
$r_0^{(1)}$ is of the form
\eqref{stima r 0 (1)} with
$ a_{- k}^{(1)} := \frac{1}{\rho} \Phi^{(1)}(a_{-k}^{(0)}) $, 
and the remainder ${\cal R}_M^{(1)} $ is given by
\begin{equation}\label{def cal RM (1) maxtom}
{\cal R}_M^{(1)} =
\frac{1}{\rho} \Phi^{(1)} {\cal R}_M^{(I)} (\Phi^{(1)})^{- 1}
+ \frac{1}{\rho}  \Phi^{(1)} 
{\cal R}_M^{(0)}(\breve \io_\delta (\vphi)) (\Phi^{(1)})^{- 1}   -  
 \frac{1}{\rho}  \Phi^{(1)}  R (\vphi) (\Phi^{(1)})^{- 1}
\, . 
\end{equation} 
We choose $ \breve \a^{(1)} $ such that \eqref{constant:sec} holds, 
obtaining \eqref{def:m3}, \eqref{alphi(2)}. 
We now verify the estimates, stated in items ($i$), ($ii$).
Recall that we assume throughout that \eqref{ansatz I delta} holds.
The estimates \eqref{m3Lip}-\eqref{stima m3} follow by 
\eqref{def:m3}, \eqref{alphi(2)}, \eqref{a30-a10}, and
by using Lemma \ref{lemma:LS norms}-($iii$), 
Lemma \ref{Moser norme pesate}. 
The estimate \eqref{stima rho} on $\rho$ follows by the definition \eqref{def:L2}, \eqref{alphi(2)}, 
and by applying Lemma \ref{lemma:LS norms}-($iii$), 
Lemma \ref{Moser norme pesate}. Hence, by Lemma \ref{lemma:LS norms} and  the estimates 
\eqref{a30-a10}, \eqref{aM0}, and \eqref{stima rho},
we deduce \eqref{est:r1a1}. The estimates \eqref{stime delta 12 secondo step} are obtained by similar arguments. 
Let us now prove item ($iii$).
The estimate \eqref{stima tame Phi (2) enunciato} follows 
from Lemma \ref{lemma:LS norms}-($iii$).
Since $(\Phi^{(1)})^{\pm 1} $ commutes with every Fourier multiplier, we get 
\begin{equation}\label{RM 0 riparametrizzato}
 \frac{1}{\rho} \langle D \rangle^{n_1} \Phi^{(1)} 
{\cal R}_M^{(0)}(\breve \io_\delta (\vphi)) (\Phi^{(1)})^{- 1} \langle D \rangle^{n_2} 
= \frac{1}{\rho}   \langle D \rangle^{n_1} {\cal R}_M^{(0)}  
( \breve \io_{\d,\a} ( \vphi ))  \langle D \rangle^{n_2}
\end{equation} 
where 
$\breve \io_{\d,\a} ( \vphi )  
:= \breve \io_\delta ( \vphi + \a^{(1)} (\vphi) \omega)$. 
By Lemma \ref{lemma:LS norms}, \eqref{2015-2},
and \eqref{stima m3} one has
$\|  \io_{\d,\a}   \|^\Lipg_s \lesssim_s 
\|\io \|^\Lipg_{s+\sigma_M}$. 
Moreover, by \eqref{forma buona resto},  we have 
\begin{equation}\label{rho R Phi (1) Phi (1) inv}
\begin{aligned}
&  \frac{1}{\rho}  \Phi^{(1)}  R (\vphi) (\Phi^{(1)})^{- 1} h = 
{\mathop \sum}_{j \in \Splus} 
 \big( h \,,\, (\Phi^{(1)} g_j) \big)_{L^2_x}  \frac{1}{\rho}  (\Phi^{(1)} \chi_j) \,, \quad \forall 
h \in L^2_\bot \, , 
\end{aligned}
\end{equation}
and by \eqref{def cal RM (I)}, the conjugated operator 
$ \frac{1}{\rho}  \Phi^{(1)}  {\cal R}_M^{(I)}   (\Phi^{(1)})^{- 1} h $
has the same form. 
The estimates \eqref{SM2} 
then follow by \eqref{RM 0 riparametrizzato}, \eqref{R0:dyn},  
and Lemmata \ref{differential nabla perturbation-true}, \ref{differential nabla kdv remainder-true}, \ref{lem:tame2} to estimate the first term 
on the right hand side of 
 \eqref{def cal RM (1) maxtom} and by 
 \eqref{rho R Phi (1) Phi (1) inv}, 
 \eqref{stima tame Phi (2) enunciato}, \eqref{stime gj chij} and
 Lemma \ref{Lemma op proiettori}, to estimate the second and  third term in \eqref{def cal RM (1) maxtom}. The estimates 
 \eqref{stima delta 12 RM (2)} are proved by similar arguments. 
\end{proof}

\subsection{Elimination of the $ (\vphi, x) $-dependence of
the highest order coefficient}\label{elimination x dependence highest order}

The goal of this section is to remove the $ (\vphi, x) $-dependence of the coefficient 
$ a_3^{(1)}(\vphi, x) $ 
of the Hamiltonian operator $ {\mathcal L}^{(1)}_\omega $, given by \eqref{def:L2}-\eqref{Op:L2}, 
 where we rename $ \vartheta $ with $ \vphi$.  
Actually this step will at the same time also remove 
the coefficient of $\pa_x^2$. 
We achieve these goals  by  conjugating the operator $ {\mathcal L}^{(1)}_\omega $  
by the 
 flow
$  \Phi^{(2)} (\tau, \vphi)  $, acting on $ L^2_\bot(\T_1)$, defined by the transport equation
\be\label{flow1}
\partial_\tau  \Phi^{(2)} (\tau, \vphi)  = \Pi_{\bot}  \pa_x  \big( b^{(2)} (\tau, \vphi, x) \Phi^{(2)} (\tau, \vphi)  \big) \, , \quad 
 \Phi^{(2)} (0, \vphi ) = {\rm Id}_{\bot}   \, ,
\ee
for a real valued function  
$$ 
b^{(2)} \equiv  b^{(2)} (\tau, \vphi, x) := \frac{  \beta^{(2)} (\vphi, x)  }{ 1 + \tau  
\beta^{(2)}_x (\vphi, x)  } \, ,
$$ 
where $ \beta^{(2)} (\vphi, x) $ is a small,
real valued periodic function chosen in \eqref{definizione beta (1) breve} below. 
The flow $ \Phi^{(2)} (\tau, \vphi) $  is well defined for $0 \le \tau \le 1$ and satisfies the 
tame estimates provided in Lemma \ref{proposition 2.40 unificata}. 
Since the vector field  $ \Pi_{\bot}  \pa_x  \big(  b^{(2)} h  \big) $, $ h \in H^s_\bot(\T_1) $, is Hamiltonian
(it is generated by the 
Hamiltonian $ \frac12 \int_{\T_1} b^{(2)} h^2 \, dx $), 
each $ \Phi^{(2)} (\tau, \vphi) $, $0 \le \tau \le 1$, $\vphi \in \T^{\mathbb S_+}$ is a symplectic linear isomorphism of $ H^s_\bot(\T_1) $. 
Therefore the time one conjugated operator
\be\label{def:L1}
{\mathcal L}_\omega^{(2)} :=  \Phi^{(2)}  {\mathcal L}_\omega^{(1)} 
 \big(\Phi^{(2)}  \big)^{-1} \, , \quad \Phi^{(2)} := \Phi^{(2)}(1, \vphi) \, , 
\ee
is  a Hamiltonian operator acting on $ H^s_\bot(\T_1) $.

Given the $(\tau, \vphi)$-dependent family of diffeomorphisms of the torus  $ \T_1 $, 
$x \mapsto  y = x +  \tau \beta^{(2)} (\vphi, x) $, we denote the family of its inverses by
$y \mapsto x = y +   \breve \beta^{(2)} (\tau,  \vphi, y) $.

\begin{lemma}
Let  $ \breve \beta^{(2)} ( \vphi, y; \o) \equiv \breve \beta^{(2)}(1, \vphi, y; \o) $ be the real valued, periodic function 
\begin{equation}\label{definizione beta (1) breve}
\breve \beta^{(2)}( \vphi, y; \o) := 
\pa_y^{-1} \Big( \frac{  m_3^{1/3} }{ ( a_3^{(1)} (\vphi, y; \o) )^{1/3} } - 1  \Big) 
\end{equation}
(which is well defined by \eqref{constant:sec}) and let $M \in \N$.
Then there exists $\sigma_M > 0$ so that the following holds: \\
(i) For any $s \geq s_0$
\begin{equation}\label{stima beta a (3) 1}
\| \beta^{(2)}\|_s^\Lipg, \| \breve \beta^{(2)}\|_s^\Lipg
 \lesssim_{s, M} 
 \e \big( 1 + \| \io\|_{s + \sigma_M}^\Lipg \big) \, . 
\end{equation}
(ii) The Hamiltonian operator $ {\mathcal L}_\omega^{(2)}  $ in \eqref{def:L1} admits an expansion of the form 
\be\label{Op:L1}
{\mathcal L}_\omega^{(2)} =  
\omega \cdot \pa_\vphi -  \big(  m_3  \pa_x^3 
 + a_1^{(2)} \pa_x  + {\rm Op}( r_0^{(2)}) + 
Q_{-1}^{kdv} (D; \om) \big)  + {\mathcal R_M^{(2)}}
\ee
where $ a_1^{(2)} := a_1^{(2)} (\vphi, x; \om) $ 
is a real valued, periodic function, satisfying
\begin{equation}\label{stima a 1 (1)}
\|  a_1^{(2)}   \|_s^\Lipg \lesssim_{s, M} \e  + \| \iota \|_{s+\sigma_M}^\Lipg \, .
\end{equation}
 The pseudo-differential symbol  $ r_0^{(2)} \equiv r_0^{(2)} (\vphi, x, \xi; \o)$ is in $S^0 $ and
 satisfies, for any $ s \geq s_0 $, the estimate 
\begin{equation}\label{stima r 0 (2)}
| {\rm Op}( r_0^{(2)})   |_{0, s, 0}^\Lipg \lesssim_{s, M}  \e  + \| \iota \|_{s+\sigma_M}^\Lipg\,.   
\end{equation}
Let $s_1 \geq s_0$ and let $\breve \io_1, \breve \io_2$ be two tori  satisfying \eqref{ansatz I delta} for $\mu_0 \geq s_1 + \sigma_M $. Then, for any $k = 0,  \ldots, M $, 
\begin{equation}\label{stima Delta 1 2 a k (1)}
\| \Delta_{12} \beta^{(2)}\|_{s_1}, \| \Delta_{12} \breve \beta^{(2)}\|_{s_1}, 
 \| \Delta_{12} a_1^{(2)} \|_{s_1},\, |\Delta_{12} {\rm Op}( r_0^{(2)}) |_{0, s_1, 0}  \lesssim_{s_1, M} \| \io_1 - \io_2 \|_{s_1 + \sigma_M}\,. 
\end{equation}
(iii) Let $S> \sM$. Then the symplectic maps $(\Phi^{(2)})^{\pm 1}$ are $\Lipg$-1 tame operators with a tame constant satisfying 
\begin{equation}\label{stima tame Phi (1) enunciato}
\mathfrak M_{(\Phi^{(2)})^{\pm 1}}(s) 
\lesssim_{S, M} 1 + \| \io \|_{s + \sigma_M}^\Lipg, \quad \forall s_0 + 1 \leq s \leq S\,.
\end{equation}
Let $\lambda_0 \in \N$. Then there exists a constant $\sigma_M(\lambda_0) > 0$ such that,  for any 
$\lambda, n_1, n_2 \in \N$
with $ \lambda\leq \lambda_0 $ and $n_1 + n_2 + \lambda_0 \leq M - 1$,  
the operator  
$ \pa_{\vphi_m}^\lambda \langle D \rangle^{n_1}{\mathcal R_M^{(2)}} \langle D \rangle^{n_2}$, $ m  \in \Splus $, is $\Lipg$-tame  with a tame constant satisfying
\be\label{SM1}
{\mathfrak M}_{\pa_{\vphi_m}^\lambda \langle D \rangle^{n_1} {\mathcal R_M^{(2)}} \langle D \rangle^{n_2} } (s) \lesssim_{S, M, \lambda_0} \e  + \| \iota \|_{s+\sigma_M(\lambda_0)}^\Lipg \,, \quad \forall \sM \leq s \leq S \, .
\ee
Let $s_1 \geq \sM$ and $\io_1, \io_2$ be tori satisfying \eqref{ansatz I delta} with $\mu_0 \geq s_1 + \sigma_M(\lambda_0)$. Then  
\begin{equation}\label{Delta 12 RM (1)}
\| \pa_{\vphi_m}^\lambda \langle D \rangle^{n_1} \Delta_{12}{\mathcal R_M^{(2)}} \langle D \rangle^{n_2} \|_{{\cal B}(H^{s_1})} \lesssim_{s_1, M, \lambda_0 } \| \io_1 - \io_2\|_{s_1 + \sigma_M(\lambda_0)}\,. 
\end{equation}
\end{lemma}

\begin{proof}
The proof of this lemma uses the Egorov type 
results proved in Section \ref{sezione astratta egorov}.
According to \eqref{Op:L2}, \eqref{stima r 0 (1)}, the conjugated operator is given by 
\begin{align}\label{cal L omega (1) a}
& {\cal L}_\omega^{(2)}  = \Phi^{(2)} {\cal L}_\omega^{(1)} (\Phi^{(2)})^{- 1}  \\
& = \omega \cdot \partial_\vphi - \Phi^{(2)}  a_3^{(1)} \pa_x^3   (\Phi^{(2)})^{- 1} - 2  \Phi^{(2)}  (a_3^{(1)})_x \pa_x^2  (\Phi^{(2)})^{- 1} -  \Phi^{(2)}  a_1^{(1)} \pa_x (\Phi^{(2)})^{- 1}  \nonumber \\
& \ \  - \sum_{k=0}^{M} \Phi^{(2)} a_{-k}^{(1)} \partial_x^{-k} (\Phi^{(2)})^{- 1} - \Phi^{(2)} Q_{- 1}^{kdv}(D ; \omega)  (\Phi^{(2)})^{- 1} 
 +   \Phi^{(2)} {\cal R}_M^{(1)}  (\Phi^{(2)})^{- 1} + \Phi^{(2)}  \big( \om \cdot \pa_{\vphi}  \, \, ( \Phi^{(2)} )^{-1} \big)\,. \nonumber 
\end{align}
By \eqref{definizione beta (1) breve}, \eqref{m3Lip}, \eqref{stima a 1 (2)} and Lemmata \ref{lemma:LS norms}, \ref{Moser norme pesate}, the estimate \eqref{stima beta a (3) 1} follows. Using the ansatz \eqref{ansatz I delta} with $\mu_0 > 0$ large enough, the estimate \eqref{stima beta a (3) 1} implies that $\| \beta^{(2)}\|_{s_0 + \sigma_M(\lambda_0)}^\Lipg \lesssim _{M, \lambda_0} \e \gamma^{- 2}$, where the constant $\sigma_M(\lambda_0)$ is the constant appearing in the smallness conditions \eqref{smallness egorov astratto}, \eqref{smallness egorov astratto 2}, \eqref{piccolezza corollario Fourier multiplier}.  Now
we apply Proposition \ref{proposizione astratta egorov} to expand the terms 
$$
\Phi^{(2)}  a_3^{(1)} \pa_x^3   (\Phi^{(2)})^{- 1} \, , \quad 
2  \Phi^{(2)}  (a_3^{(1)})_x \pa_x^2  (\Phi^{(2)})^{- 1} \, , \quad
\Phi^{(2)} a_{1-k}^{(1)} \pa_x^{1-k} (\Phi^{(2)})^{- 1} \, ,
\, 0 \le k \le M+1\, ,
$$ 
Lemma \ref{Fourier multiplier} to expand the term 
$\Phi^{(2)} Q_{- 1}^{kdv}(D; \o)  (\Phi^{(2)})^{- 1}$, and Proposition \ref{proposizione astratta egorov 2} to expand  
$\Phi^{(2)}  \big( \om \cdot \pa_{\vphi}  \, \, ( \Phi^{(2)} )^{-1} \big) $.
Using also the estimates \eqref{a30-a10}, \eqref{aM0}, \eqref{stima beta a (3) 1} one deduces \eqref{stima a 1 (1)}, \eqref{stima r 0 (2)}. 
By the choice of $ \breve \beta^{(2)} $ in \eqref{definizione beta (1) breve} and Proposition \ref{proposizione astratta egorov}, the coefficient of the highest order term
of $\Phi^{(2)}  a_3^{(1)} \pa_x^3   (\Phi^{(2)})^{- 1}$
(and of ${\cal L}_\omega^{(2)}$) is given by
$$
\big( [1 + \breve \beta^{(2)}_y (\vphi, y)]^3 a_3^{(1)}(\vphi, y) \big)|_{y = x +  \beta^{(2)}(\vphi, x)}    = m_3  
$$ 
which  is constant in $ (\vphi, x)$ by \eqref{constant:sec}. 
Since $\Phi^{(2)}$ is symplectic, 
the operator ${\cal L}_\omega^{(2)}$ is Hamiltonian and
hence by Lemma \ref{lem:229}
the second order term equals $2  (m_3)_x \partial_{x}^2$ 
which vanishes
since $ m_3 $ is constant. The remainder $\Phi^{(2)} {\cal R}_M^{(1)}  (\Phi^{(2)})^{- 1} $ can be estimated by arguing as at the end of the proof of Proposition \ref{proposizione astratta egorov} (estimate of ${\cal R}_N(\tau, \vphi)$), using 
Lemma \ref{proposition 2.40 unificata} to estimate $\Phi^{(2)}$, $(\Phi^{(2)})^{- 1}$, the estimate \eqref{SM2} for  ${\cal R}_M^{(1)}$, the estimate \eqref{stima beta a (3) 1} of $\beta^{(2)}$,
$\breve \beta^{(2)}$,
and the ansatz \eqref{ansatz I delta} with $\mu_0$ large enough. The estimates \eqref{stima tame Phi (1) enunciato} follow by  
\eqref{stima flusso 1 astratto} and  
\eqref{stima beta a (3) 1}. The estimates \eqref{stima Delta 1 2 a k (1)}, \eqref{Delta 12 RM (1)} are derived by similar arguments. 
\end{proof}

\subsection{Elimination of the $ x $-dependence of the first order 
coefficient}

The goal of this section is to remove the $ x $-dependence of the coefficient $  a_1^{(2)} (\vphi,x) $ 
of the Hamiltonian operator $ {\mathcal L}^{(2)}_\omega $ in \eqref{def:L1}, \eqref{Op:L1}. 
We   conjugate the operator $ {\mathcal L}^{(2)}_\omega $ 
by the change of variable induced by the flow $   \Phi^{(3)}(\tau, \vphi) $, 
acting on $ L^2_\bot (\T_1 ) $, defined by 
\begin{equation}\label{Phi (3) tau}
\partial_\tau  \Phi^{(3)} (\tau, \vphi)  = \Pi_{\bot} \big( b^{(3)} (\vphi, x) \pa_{x}^{-1} \Phi^{(3)} (\tau, \vphi)  \big) \, , \quad 
 \Phi^{(3)} (0) = {\rm Id}_{\bot}   \, ,
\end{equation}
where $b^{(3)} (\vphi, x) $ is a small, real valued, periodic  function  
chosen in \eqref{alphi 0} below. 
Since the vector field  $ \Pi_{\bot}    \big(  b^{(3)} \pa_x^{-1} h  \big) $, $ h \in H^s_\bot(\T_1) $, is Hamiltonian
(it is generated by the 
Hamiltonian $ \frac12 \int_{\T_1} b^{(3)} (\pa_x^{-1} h)^2 \, dx $), 
each  $ \Phi^{(3)} (\tau, \vphi) $ 
is a symplectic linear isomorphism of $ H^s_\bot $ for any $0 \le \tau \le 1$ and $\vphi \in \T^{\Splus}$, and 
the time one conjugated operator
\be\label{def:L3}
{\mathcal L}_\omega^{(3)} :=  \Phi^{(3)}   {\mathcal L}_\omega^{(2)}    \big(\Phi^{(3)}  \big)^{-1}  
\, , \quad \Phi^{(3)} :=  \Phi^{(3)}(1) \, , 
\ee
is Hamiltonian. 

\begin{lemma}
Let  $  b^{(3)} (\vphi, x; \o)  $ be the real valued periodic function 
\be\label{alphi 0}
b^{(3)}(\vphi, x; \o) :=  \frac{1}{3 m_3} \pa_x^{-1} \Big( a_1^{(2)}(\vphi, x; \o) - 
\langle a_1^{(2)} \rangle_x (\vphi ; \o)   \Big)  \, , 
\quad \langle a_1^{(2)} \rangle_x (\vphi; \o )   
:= \int_{\T_1} a_1^{(2)} (\vphi, x; \o)\,d x
\ee
and let $M \in \N$. Then there exists $\sigma_M > 0$ with the following properties: \\
(i) For any $s \geq s_0 $,
\be\label{stima b (3)}
\| b^{(3)}\|_s^\Lipg \lesssim_{s, M} \e + \| \io\|_{s + \sigma_M}^\Lipg
\ee
and the symplectic maps $(\Phi^{(3)})^{\pm 1}$ are 
 $\Lipg$-tame  and satisfy
\begin{equation}\label{stima Phi (3) tau tame nel lemma}
\mathfrak M_{(\Phi^{(3)})^{\pm 1}}(s) \lesssim_{s, M} 1 + \| \io\|_{s + \sigma_M}^\Lipg\,. 
\end{equation}
(ii) The Hamiltonian operator in \eqref{def:L3} admits an expansion of the form
\be\label{Op:L3}
{\mathcal L}_\omega^{(3)} =  
\omega \cdot \pa_\vphi -  \big(  m_3   \pa_x^3 
 + a_1^{(3)}(\vphi)  \pa_x  + {\rm Op}( r_0^{(3)}) + Q_{-1}^{kdv} (D; \om) \big) 
 + {\mathcal R_M^{(3)}}
\ee
where the real valued, periodic function $ a_1^{(3)} (\vphi; \o) := \langle a_1^{(2)} \rangle_x(\vphi; \o)$ 
satisfies 
\begin{equation}\label{stima a 1 (3)}
\|  a_1^{(3)}   \|_s^\Lipg \lesssim_{s, M} \e  + \| \iota \|_{s+\sigma_M}^\Lipg \, ,  
\end{equation}
and $ r_0^{(3)} := r_0^{(3)} (\vphi, x, \xi ; \o )$
is a pseudo-differential symbol in $S^0 $ 
satisfying for any $s \geq s_0$, 
\begin{equation}\label{stima Op r 0 3 statement}
|{\rm Op}(r_0^{(3)})|_{0, s, 0}^\Lipg \lesssim_{s, M} \e + \| \iota \|_{s+\sigma_M}^\Lipg\,.
\end{equation}
Let $s_1 \geq s_0$ and let 
$\breve \io_1, \breve \io_2$ be two tori satisfying \eqref{ansatz I delta} with $\mu_0 \geq s_1 + \sigma_M$. Then 
\begin{equation}\label{stime delta 12 terzo step}
\begin{aligned}
& \| \Delta_{12} b^{(3)} \|_{s_1} \, , \| \Delta_{12} a_1^{(3)}\|_{s_1}  \lesssim_{s_1, M} \| \io_1 - \io_2\|_{s_1 + \sigma_M} \, , \quad   |\Delta_{12} {\rm Op}( r_0^{(3)}) |_{0, s_1, 0} \lesssim_{s_1, M} \| \io_1 - \io_2\|_{s_1 + \sigma_M}\,. 
\end{aligned}
\end{equation} 
(iii) Let $S > \sM$, $ \lambda_0 \in \N $. Then there exists a constant $\sigma_M(\lambda_0) > 0$ so that for any 
$ m \in \Splus$ and 
$\lambda, n_1, n_2 \in \N$ with
$ \lambda \leq \lambda_0 $ and  
$n_1 + n_2 + \lambda_0\leq M  - 1$, the operator 
$ \langle D \rangle^{n_1}\pa_{\vphi_m}^\lambda {\mathcal R_M}^{(3)} \langle D \rangle^{n_2}$,  is 
$\Lipg$-tame with tame constants satisfying
\be\label{SM3}
{\mathfrak M}_{\pa_{\vphi_m}^\lambda  \langle D \rangle^{n_1}{\mathcal R_M^{(3)}} \langle D \rangle^{n_2}} (s) \lesssim_{S, M, \lambda_0} \e  + \| \iota \|_{s+\sigma_M(\lambda_0)}^\Lipg
\, , \quad \forall \sM  \leq s \leq S \, .
\ee
Let $s_1 \geq \sM$ and let 
$\breve \io_1, \breve \io_2$ be tori satisfying \eqref{ansatz I delta} with $\mu_0 \geq s_1 + \sigma_M(\lambda_0)$. Then  
\begin{equation}\label{stima delta 12 RM (3)}
\| \pa_{\vphi_m}^\lambda  \langle D \rangle^{n_1} \Delta_{12} {\mathcal R_M^{(3)}} \langle D \rangle^{n_2} \|_{{\cal B}(H^{s_1})} \lesssim_{s_1, M, \lambda_0} \| \io_1 - \io_2\|_{s_1 + \sigma_M(\lambda_0)}\,. 
\end{equation}
\end{lemma}

\begin{proof}
The estimate \eqref{stima b (3)} follows by the definition \eqref{alphi 0} and 
\eqref{stima a 1 (1)}, \eqref{m3Lip}.
We now provide estimates for the flow 
$$
\Phi^{(3)}(\tau) = {\rm exp}\big(\tau \Pi_\bot  b^{(3)} (\vphi, x; \o) \pa_{x}^{-1}  \big) \, , 
\quad \forall \tau \in [- 1 , 1] \, . 
$$ 
By   \eqref{norm-increa}, Lemma \ref{lemma stime Ck parametri}, and \eqref{stima b (3)}, one infers that 
for any $s \geq s_0$, $|\Pi_\bot b^{(3)} \partial_x^{- 1}|_{- 1, s, 0}^\Lipg \lesssim_{s, M} \e + \| \io\|_{s + \sigma_M}^\Lipg \, . $
Therefore, by Lemma \ref{Neumann pseudo diff},  there exists $\sigma_M > 0$ such that, if \eqref{ansatz I delta} holds with $\mu_0 \geq \sigma_M$, then, for any $s \geq s_0$,
\be \label{stima Phi (3) tau}
\sup_{\tau \in [- 1, 1]} |\Phi^{(3)}(\tau) - {\rm Id}|_{0, s, 0}^\Lipg \lesssim_{s} \e  + \| \io\|_{s + \sigma_M}^\Lipg\,. 
\ee
The latter estimate, together with Lemma \ref{lemma: action Sobolev}, imply \eqref{stima Phi (3) tau tame nel lemma}. 

By \eqref{Op:L1} and using Lemma \ref{lemma resto frequenze kdv} for the operator $Q_{- 1}^{kdv}(D; \omega)$, one has that 
$$
\Phi^{(3)} {\cal L}_\omega^{(2)} (\Phi^{(3)})^{- 1}  =  \omega \cdot \partial_\vphi -  \Phi^{(3)} \big( m_3   \pa_x^3 
 + a_1^{(2)} \pa_x   \big) (\Phi^{(3)})^{- 1} - 
 Q^{kdv}_{- 1}(D ; \ \omega) +  {\cal R}_0^{(I)} + {\cal R}_M^{(3)}
 $$
where 
\begin{equation}\label{coniugazione Phi (3) nel lemma}
\begin{aligned}
{\cal R}_0^{(I)}  & := 
- \Phi^{(3)}   {\rm Op}( r_0^{(2)}) (\Phi^{(3)})^{- 1} + 
\Phi^{(3)} \big( \omega \cdot \partial_\vphi (\Phi^{(3)})^{- 1} \big) 
 - (\Phi^{(3)} - {\rm Id}_\bot)\Pi_\bot \Big( \sum_{k = 1}^M 
 c_{ - k}^{kdv}(\omega) \partial_x^{- k}  \Big) (\Phi^{(3)})^{- 1}  \\
& \quad - \Pi_\bot \Big( \sum_{k = 1}^{M } c_{ - k}^{kdv}(\omega) \partial_x^{- k}  \Big) \Big( (\Phi^{(3)})^{- 1} - {\rm Id}_\bot \Big) \,, \\
{\cal R}_M^{(3)} & := \Phi^{(3)} {\cal R}_M^{(2)} (\Phi^{(3)})^{- 1} 
- (\Phi^{(3)} - {\rm Id}_\bot) {\cal R}_{M }( \omega , Q_{-1}^{kdv}) (\Phi^{(3)})^{- 1}  
 - {\cal R}_{M }( \omega , Q_{-1}^{kdv}) \big( (\Phi^{(3)})^{- 1} - {\rm Id}_\bot \big)\,. 
\end{aligned}
\end{equation}
Note that ${\cal R}_0^{(I)}$ is a pseudo-differential operator in $OPS^0$ (cf. Lemma \ref{Neumann pseudo diff}).
Moreover, by a Lie expansion, recalling \eqref{Phi (3) tau}, one has 
$$
\begin{aligned}
\Phi^{(3)} \big( m_3   \pa_x^3 
 + a_1^{(2)} \pa_x   \big) (\Phi^{(3)})^{- 1} & = m_3   \pa_x^3 
 + a_1^{(2)} \pa_x + [ \Pi_\bot b^{(3)} \partial_x^{- 1}, \, m_3   \pa_x^3 
 + a_1^{(2)} \pa_x]   \\
 & \ \  + \int_0^1 (1 - \tau) \Phi^{(3)}(\tau) \Big[\Pi_\bot b^{(3)} \partial_x^{- 1},
  \Big[\Pi_\bot b^{(3)} \partial_x^{- 1}, m_3   \pa_x^3 
 + a_1^{(2)} \pa_x\Big] \Big] \Phi^{(3)}(\tau)^{- 1}\, d \tau \\
 & = m_3   \pa_x^3 
 + \big( a_1^{(2)} - 3 m_3 b^{(3)}_x\big) \pa_x 
 + {\cal R}_0^{(II)}\, ,
 \end{aligned}
$$
\begin{equation}\label{def cal R 0 (II)}
\begin{aligned}
{\cal R}_0^{(II)} & := - 3 m_3 b^{(3)}_{xx} - m_3 b^{(3)}_{xxx} \partial_x^{- 1} + [ \Pi_\bot b^{(3)} \partial_x^{- 1}
 , a_1^{(2)} \pa_x]  + [ (\Pi_\bot - {\rm Id}) b^{(3)} \partial_x^{- 1}, m_3   \pa_x^3] \\
 & \quad + \int_0^1 (1 - \tau) \Phi^{(3)}(\tau) \Big[\Pi_\bot b^{(3)} \partial_x^{- 1}\,,\, \Big[\Pi_\bot b^{(3)} \partial_x^{- 1}, m_3   \pa_x^3 
 + a_1^{(2)} \pa_x\Big] \Big] \Phi^{(3)}(\tau)^{- 1}\, d \tau \in OPS^0\, . \\
\end{aligned}
\end{equation}
Note that ${\cal R}_0^{(II)}$ is a pseudo-differential operator in $OPS^0$ (cf. Lemma \ref{Neumann pseudo diff}).
Hence, \eqref{coniugazione Phi (3) nel lemma}-\eqref{def cal R 0 (II)} and the choice of $b^{(3)}$ in \eqref{alphi 0} lead to the expansion \eqref{Op:L3}
with ${\cal R}_M^{(3)}$ given by 
\eqref{coniugazione Phi (3) nel lemma} and 
\begin{equation}\label{def Op r 0 3 nella proof}
{\rm Op}( r_0^{(3)})  := - {\cal R}_0^{(I)} + {\cal R}_0^{(II)}\,.
\end{equation}
The estimate \eqref{stima a 1 (3)} follows  by  \eqref{stima a 1 (2)}. 

The estimate \eqref{stima Op r 0 3 statement} on the operator ${\rm Op}(r_0^{(3)})$ follows by the definitions \eqref{coniugazione Phi (3) nel lemma}, \eqref{def cal R 0 (II)}, \eqref{def Op r 0 3 nella proof}, by applying the estimates \eqref{m3Lip}, \eqref{stima a 1 (1)}, \eqref{stima r 0 (2)}, \eqref{stima b (3)}, \eqref{stima Phi (3) tau}, \eqref{norm-increa}, \eqref{Norm Fourier multiplier}, \eqref{norma pseudo moltiplicazione}, \eqref{estimate composition parameters}, \eqref{stima commutator parte astratta} (using the ansatz \eqref{ansatz I delta} with $\mu_0$ large enough). 
Next we estimate the remainder ${\cal R}_M^{(3)}$, defined in \eqref{coniugazione Phi (3) nel lemma}. We only consider the second term in the definition of ${\cal R}_M^{(3)}$, since 
the estimates the first and third terms can be obtained similarly. 
We recall that the operator ${\cal R}_M(Q_{-1}^{kdv}; \o)$ is $\vphi$-independent. For
$m \in {\mathbb S}_+$ and $\lambda, n_1, n_2 \in \N$ with 
$ \lambda \leq \lambda_0 $ and $n_1 + n_2 + \lambda_0\leq M - 2 $,  
one has 
\begin{align}
& \langle D \rangle^{n_1} \partial_{\vphi_m}^\lambda \Big( (\Phi^{(3)} - {\rm Id}_\bot) {\cal R}_M( Q_{-1}^{kdv}; \o) (\Phi^{(3)})^{- 1}  \Big) \langle D \rangle^{n_2}  \label{unimi 100} \\
& = \sum_{\lambda_1 + \lambda_2 = \lambda} C_{\lambda_1, \lambda_2} \langle D \rangle^{n_1} \partial_{\vphi_m}^{\lambda_1} (\Phi^{(3)} - {\rm Id}_\bot)  
{\cal R}_M(Q_{-1}^{kdv}; \o) \partial_{\vphi_m}^{\lambda_2} 
 (\Phi^{(3)})^{- 1}   \langle D \rangle^{n_2} \nonumber \\
& = \sum_{\lambda_1 + \lambda_2 = \lambda} C_{\lambda_1, \lambda_2} \Big( \langle D \rangle^{n_1} \partial_{\vphi_m}^{\lambda_1}  (\Phi^{(3)} - {\rm Id}_\bot)  \langle D \rangle^{- n_1} \Big) \Big( \langle D \rangle^{n_1} {\cal R}_M( Q_{-1}^{kdv}; \o) \langle D \rangle^{n_2} \Big) \Big( \langle D \rangle^{- n_2} \partial_{\vphi_m}^{\lambda_2}  (\Phi^{(3)})^{- 1}   \langle D \rangle^{n_2} \Big) \, . \nonumber
\end{align}
By the estimates \eqref{Norm Fourier multiplier}, \eqref{estimate composition parameters}, \eqref{stima Phi (3) tau} and Lemma \ref{lemma: action Sobolev}, one has 
$$
\begin{aligned}
& \mathfrak M_{\langle D \rangle^{n_1} \partial_{\vphi_m}^{\lambda_1}  (\Phi^{(3)} - {\rm Id}_\bot)  \langle D \rangle^{- n_1}}(s) \lesssim_s |\langle D \rangle^{n_1} \partial_{\vphi_m}^{\lambda_1}  (\Phi^{(3)} - {\rm Id}_\bot)  \langle D \rangle^{- n_1}|_{0, s, 0}^\Lipg \lesssim_{s, M} \e + \| \io \|_{s + \sigma_M(\lambda_0)}^\Lipg\,, \\
& \mathfrak M_{\langle D \rangle^{- n_2} \partial_{\vphi_m}^{\lambda_2}  (\Phi^{(3)})^{- 1}   \langle D \rangle^{n_2} }(s) \lesssim_s |\langle D \rangle^{- n_2} \partial_{\vphi_m}^{\lambda_2}  (\Phi^{(3)})^{- 1}   \langle D \rangle^{n_2} |_{0, s, 0}^\Lipg \lesssim_{s, M} 1 + \| \io \|_{s + \sigma_M(\lambda_0)}^\Lipg\,,
\end{aligned}
$$
and therefore, by Lemmata \ref{composizione operatori tame AB}, \ref{espansione simboli Q kdv - 1} and using \eqref{ansatz I delta},  the operator \eqref{unimi 100} satisfies  \eqref{SM3}.  
The estimates \eqref{stime delta 12 terzo step}, \eqref{stima delta 12 RM (3)} follow by similar arguments. 
\end{proof}

\subsection{Elimination of the $ \vphi $-dependence of the first order term}\label{ulimo step pre redu}

The goal of this section is to remove the $ \vphi $-dependence of the coefficient $  a_1^{(3)} (\vphi) $ 
of the Hamiltonian operator $ {\mathcal L}^{(3)}_\omega $ in \eqref{def:L3}, \eqref{Op:L3}. 
We   conjugate the operator $ {\mathcal L}^{(3)}_\omega $ 
by the variable transformation 
$\Phi^{(4)} \equiv \Phi^{(4)} (\vphi ) $, 
$$
(\Phi^{(4)} w)(\vphi, x) = w(\vphi, x + b^{(4)}  (\vphi)) \, , \quad 
((\Phi^{(4)})^{-1} h)(\vphi, x) = h(\vphi, x -  b^{(4)} (\vphi)) \, , 
$$
where $b^{(4)} (\vphi) $ is a small, real valued, 
periodic  function  
chosen in \eqref{alphi} below. Note that $\Phi^{(4)}$ is the time-one flow of the transport equation $\partial_\tau w = b^{(4)}(\vphi) \partial_x w$. 
Each  $ \Phi^{(4)} (\vphi) $ is a symplectic linear isomorphism of $ H^s_\bot (\T_1)$,
and the  conjugated operator
\be\label{def:L5}
{\mathcal L}_\omega^{(4)} :=  \Phi^{(4)}{\mathcal L}_\omega^{(3)}    \big(\Phi^{(4)} \big)^{- 1} 
\ee
is  Hamiltonian. 

\begin{lemma}\label{lem:4}
Assume that $ \omega \in \mathtt{DC}(\g,\t) $.
Let  $  b^{(4)} (\vphi )  $ be the real valued, periodic function 
\be\label{alphi}
b^{(4)}(\vphi; \o) := - (\om \cdot \pa_\vphi)^{-1}
\big( a_1^{(3)}(\vphi; \o) - m_1   \big)  \, , \quad 
m_1   := \frac{1}{(2 \pi)^{|\Splus|}} \int_{\T^{\Splus}}   a_1^{(3)} (\vphi; \o) \, d \vphi   
\ee
and let $M \in \N$. Then there exists $\sigma_M > 0$ with the following properties:\\
(i) The constant $m_1$ and the function $b^{(4)}$ satisfy
\begin{equation}\label{stima m1}
|m_1|^\Lipg \lesssim_M \e \gamma^{- 2}\,, \quad \| b^{(4)}\|_s^\Lipg \lesssim_{s, M} \gamma^{- 1}\big( \e + \| \io \|_{s + \sigma_M}^\Lipg ) \, , \quad \forall s \geq s_0 \, .
\end{equation}
(ii) The Hamiltonian operator in \eqref{def:L5}
admits an expansion of the form 
\be\label{Op:L4}
{\mathcal L}_\omega^{(4)} =  
\omega \cdot \pa_\vphi -  \big(  m_3   \pa_x^3 
 + m_1  \pa_x   + {\rm Op}( r_0^{(4)}) + Q_{-1}^{kdv} ( D; \om)\big) 
 + {\mathcal R_M^{(4)}}
\ee
where $ r_0^{(4)} := r_0^{(4)} (\vphi, x, \xi ; \o )$
is a pseudo-differential symbol in $S^0 $ 
satisfying for any $s \geq s_0$, 
\begin{equation}\label{estimate r0 (4)}
| {\rm Op}( r_0^{(4)})|_{0, s, 0}^\Lipg \lesssim_{s, M} \e  + \| \iota \|_{s+\sigma_M}^\Lipg \,, \quad \forall s \geq s_0\,. 
\end{equation}
Let $s_1 \geq s_0$ and let  $\breve \io_1, \breve \io_2$ be two tori satisfying \eqref{ansatz I delta} with $\mu_0 \geq s_1 + \sigma_M$. Then 
\begin{equation}\label{stime delta 12 quarto step}
\begin{aligned}
& |\Delta_{12} m_1|\,,\, \| \Delta_{12} b^{(4)} \|_{s_1}   \lesssim_{s_1, M} \| \io_1 - \io_2\|_{s_1 + \sigma_M}, \quad  |\Delta_{12} {\rm Op}( r_0^{(4)}) |_{0, s_1, 0} \lesssim_{s_1, M} \| \io_1 - \io_2\|_{s_1 + \sigma_M}\,. 
\end{aligned}
\end{equation}
(iii) Let $S> \sM$. Then the maps $(\Phi^{(4)})^{\pm 1}$ are $\Lipg$-tame operators 
with a tame constant satisfying 
\begin{equation}\label{stima Phi (4) enunciato}
\mathfrak M_{(\Phi^{(4)})^{\pm 1}}(s) \lesssim_{S, M} 1 + \| \io \|_{s + \sigma_M}^\Lipg, \quad \forall s_0 \leq s \leq S\,. 
\end{equation}
Let $\lambda_0 \in \N$. Then there exists a constant 
$\sigma_M(\lambda_0) > 0$ so that
for any $ \lambda, n_1, n_2 \in \N $ with $ \lambda \leq \lambda_0 $
and $n_1 + n_2 + 2 \lambda_0 \leq M - 3 $, the operator  
$ \pa_{\vphi_m}^\lambda \langle D \rangle^{n_1}{\mathcal R_M^{(4)}} \langle D \rangle^{n_2}$, $ m \in \Splus $, is 
$\Lipg$-tame with a  tame constant satisfying 
\be\label{SM4}
{\mathfrak M}_{\pa_{\vphi_m}^\lambda  \langle D \rangle^{n_1}{\mathcal R_M^{(4)}} \langle D \rangle^{n_2}} (s) \lesssim_{S, M, \lambda_0} \e  + \| \iota \|_{s+\sigma_M(\lambda_0)}^\Lipg\,, \quad \forall \sM \leq s \leq S  \, .
\ee
Let $s_1 \geq \sM$ and let $\breve \io_1, \breve \io_2$ be two tori satisfying \eqref{ansatz I delta} with $\mu_0 \geq s_1 + \sigma_M(\lambda_0)$. Then 
\begin{equation}\label{stima delta 12 RM (4)}
\| \pa_{\vphi_m}^\lambda  \langle D \rangle^{n_1} \Delta_{12} {\mathcal R_M^{(4)}} \langle D \rangle^{n_2} \|_{{\cal B}(H^{s_1})} \lesssim_{s_1, M, \lambda_0} \| \io_1 - \io_2\|_{s_1 + \sigma_M(\lambda_0)}\,. 
\end{equation}
\end{lemma}
\begin{proof}
The estimates \eqref{stima m1} are direct consequences of  \eqref{stima a 1 (3)} and of the ansatz \eqref{ansatz I delta}. Note that
$$
\Phi^{(4)} \circ \omega \cdot \partial_\vphi \circ
 (\Phi^{(4)})^{- 1} = 
 \omega \cdot \partial_\vphi  - 
 \big(\omega \cdot \partial_\vphi b^{(4)} \big) \partial_x 
$$
 and for any pseudo-differential operator ${\rm Op}(a(\vphi, x, \xi))$ a direct calculation shows that 
$$
 \Phi^{(4)}  {\rm Op}(a (\vphi, x, \xi))  \big(\Phi^{(4)} \big)^{- 1} =  {\rm Op}(  a  (\vphi, x + b^{(4)}(\vphi) , \xi)) \, , 
$$
and hence, by recalling \eqref{Op:L3} and by the definition \eqref{alphi}, one obtains \eqref{Op:L4} with
\begin{equation}\label{R M 4 r 0 4}
{\rm Op}\big(r_0^{(4)}(\vphi, x, \xi) \big) 
= {\rm Op}\big(r_0^{(3)}(\vphi, x + b^{(4)}(\vphi), \xi) \big) \, , 
\quad {\cal R}_M^{(4)} := \Phi^{(4)} {\cal R}_M^{(3)} (\Phi^{(4)})^{- 1}\,. 
\end{equation} 
The estimates \eqref{estimate r0 (4)} follow
by Lemma \ref{lemma:LS norms}, using \eqref{stima m1}, \eqref{stima Op r 0 3 statement} and the ansatz \eqref{ansatz I delta}. The estimates \eqref{SM4} for the operator ${\cal R}_M^{(4)}$ 
follow by 
\eqref{SM3}, \eqref{stima m1} arguing as in the proof of the estimates of the remainder ${\cal R}_N(\tau, \vphi)$ 
(with $\beta$ given by $b^{(4)}$) at the end of the proof of Proposition \ref{proposizione astratta egorov}. The estimates \eqref{stima Phi (4) enunciato} follow by  Lemma \ref{lemma:LS norms} and 
\eqref{stima m1}.  The estimates \eqref{stime delta 12 quarto step}, 
\eqref{stima delta 12 RM (4)} follow by similar arguments. 
\end{proof}

\section{KAM reduction of the linearized operator}\label{sec: reducibility}

The goal of this section is to complete the diagonalization
of the Hamiltonian operator ${\mathcal L}_\omega$,
started in Section \ref{linearizzato siti normali}. 
It remains to reduce the Hamiltonian operator
${\mathcal L}_\omega^{(4)}$ in \eqref{Op:L4}.
We are going to apply the KAM-reducibility scheme described in \cite{Berti-Montalto}.

Recall that ${\mathcal L}_\omega^{(4)}$ is an operator
acting on $ H^s_\bot$. It is convenient to rename it as 
\be\label{defL0-red}
{\mathtt L}_0  := \om \cdot \pa_\vphi + \ii {\mathtt D}_0 + {\mathtt R}_0 
\ee
where $ \omega \in \mathtt{DC}(\g,\t) $ 
(cf. \eqref{diofantei in omega})
and in view of \eqref{decokdv}, \eqref{Omega-normal}, \eqref{mu-kdv-om}
\begin{align}\label{op-diago0}
& {\mathtt D}_0  := {\rm diag}_{j \in \Sbot} ( \mu_j^{0}) \, , \quad 
\mu_j^{0} :=  m_3 (2 \pi j)^3 - m_1 2 \pi j - q_j(\omega) \, , \quad
q_j(\omega) := \om_j^{kdv} \big( \nu(\omega), 0 \big) - (2 \pi  j)^3 \, , \\
& \label{def:RoKAM}
{\mathtt R}_0   := - {\rm Op}( r_0^{(4)}) + {\mathcal R_M^{(4)}} \, . 
\end{align}
Note that $ \mu_{-j}^{0} = - \mu_{j}^{0} $ 
for any $j \in S_\bot$.  
By \eqref{Proposition 2.30} we have 
\begin{equation}\label{stima qj omega}
\sup_{j \in \Sbot} |j| |q_j|^{\rm sup}, \sup_{j \in \Sbot} |j| |q_j|^{\rm lip} \lesssim 1\, , 
\end{equation}
and, by  \eqref{m3Lip}, \eqref{stima m1} and $ \e \g^{-3} \leq 1 $, 
\begin{equation}\label{stima mu j 0 - mu j' 0}
|\mu_j^{0} - \mu_{j'}^{0}|^{\rm lip} \lesssim_M   |j^3 - j'^3|  \, , \quad \forall j, j' \in \Sbot \, .  
\end{equation}
The operator ${\mathtt R}_0 $ 
satisfies the tame estimates of Lemma \ref{lem:tame iniziale} below. 
We first {\it fix} the constants 
\be
\begin{aligned}\label{alpha beta}
& {\mathtt b} := [{\mathtt a}] + 2 \in \N \,,
\quad {\mathtt a} := 3 \tau_1 + 1 \, , \quad \tau_1 := 2 \tau  + 1 \, ,\\
& \mu(\mathtt b) := s_0 + \mathtt b + \sigma_M + \sigma_M(\mathtt b) + 1 \,, \quad
 M := 2(s_0 + \mathtt b) + 4 \, , 
\end{aligned}
\ee
where the constants $\sigma_M$, $\sigma_M(\mathtt b)$ are the ones introduced in Lemma \ref{lem:4} and 
where $M$ is related to the order of smoothing of the remainder 
${\mathcal R_M^{(4)}} $ in \eqref{Op:L4} (cf. \eqref{SM4}).
Note that $M$ only depends on the number of frequencies $|\Splus|$ and 
the diophantine constant $\tau$. 

\begin{lemma} \label{lem:tame iniziale}
Let $\mathtt b$ and $M$ defined in \eqref{alpha beta} and 
 $S > \sM $ with  $\sM$ given by \eqref{def sM}. \\
(i) The operators $ {\mathtt R}_0 $, $[  {\mathtt R}_0 , \pa_x ] $,
$ \partial_{\vphi_m}^{s_0}  [  {\mathtt R}_0 , \pa_x ]   $,  
$ \partial_{\vphi_m}^{s_0 +  {\mathtt b}}  {\mathtt R}_0   $, 
$ \partial_{\vphi_m}^{s_0 +  {\mathtt b}}  [  {\mathtt R}_0, \pa_x ] $,
$m \in \Splus $,  are $ \Lipg $-tame with tame constants
\begin{align}\label{tame cal R0 cal Q0}
& {\mathbb M}_0 (s) := \max_{m \in \Splus} \big\{ 
{\mathfrak M}_{  {\mathtt R}_0  }(s),
{\mathfrak M}_{[  {\mathtt R}_0 , \pa_x ]  }(s),
{\mathfrak M}_{ \partial_{\vphi_m}^{s_0}  {\mathtt R}_0  }(s),
{\mathfrak M}_{ \partial_{\vphi_m}^{s_0}  [  {\mathtt R}_0 , \pa_x ]  }(s)  \big\} \, , \\
& \label{tame norma alta cal R0 cal Q0}
{\mathbb M}_0   (s, {\mathtt b})  := \max_{m \in \Splus} \big\{ 
{\mathfrak M}_{ \partial_{\vphi_m}^{s_0 +  {\mathtt b}}  {\mathtt R}_0   }(s),
{\mathfrak M}_{ \partial_{\vphi_m}^{s_0 +  {\mathtt b}}  [  {\mathtt R}_0, \pa_x ]  }(s)  \big\} \, , 
\end{align}
 satisfying,  for any $ \sM \leq s \leq S $, 
\begin{equation} \label{stima mathfrak M s0 b}
 {\mathfrak M}_0(s, \mathtt b) := 
 {\rm max}\{ {\mathbb M}_0 (s), {\mathbb M}_0   (s, {\mathtt b}) \}  \lesssim_S 
 \e + \| \io\|_{s + \mu(\mathtt b)}^\Lipg \, . 
\end{equation}
Assuming that the ansatz \eqref{ansatz I delta} 
holds with $\mu_0 \geq \sM + \mu(\mathtt b)$, 
the latter estimate yields
$ {\mathfrak M}_0(\sM, \mathtt b)   
\lesssim_S \e \gamma^{- 2} $. 
\\ 
(ii) For any two tori  $\breve \io_1, \breve \io_2$ 
satisfying the ansatz \eqref{ansatz I delta},  
one has for any $m \in \mathbb S_+$ and any
$\lambda \in \N$ with $\lambda \leq s_0 +  \mathtt b$
\begin{equation}\label{stima delta 1 2 mathtt R0}
\| \partial_{\vphi_m}^\lambda \Delta_{12} \mathtt R_0 \|_{{\cal B}(H^{\sM})}, \,  \| \partial_{\vphi_m}^\lambda [\Delta_{12} \mathtt R_0, \partial_x] \|_{{\cal B}(H^{\sM})} \lesssim \| \io_1 - \io_2\|_{\sM + \mu(\mathtt b)} \, . 
\end{equation}
\end{lemma}

\begin{proof}
($i$) 
Since the assertions for the various operators are proved in the same way, we restrict ourselves to show
that there are tame constants
${\mathfrak M}_{\partial_{\vphi_m}^{s_0 + \mathtt b} [\mathtt R_0, \partial_x]}(s)$, $m \in \Splus$, satisfying 
the bound in \eqref{stima mathfrak M s0 b}. 
The two operators ${\rm Op}( r_0^{(4)})$ and ${\mathcal R_M^{(4)}}$
in the definition \eqref{def:RoKAM} of ${\mathtt R}_0$ are treated separately.
By Lemma \ref{lemma: action Sobolev} each operator  
$\partial_{\vphi_m}^{s_0 + \mathtt b}[{\rm Op}( r_0^{(4)}), \partial_x ] = - {\rm Op}
\big( \partial_{\vphi_m}^{s_0 + \mathtt b} \partial_x r_0^{(4)} \big) $, $m \in \Splus$,   is 
$ \Lipg$-tame  with a  tame constant satisfying,  for
 $s_0 \leq s \leq S $,
 \begin{equation}\label{tame bound r0 (4)}
\begin{aligned}
{\mathfrak M}_{\partial_{\vphi_m}^{s_0 + \mathtt b}[{\rm Op}( r_0^{(4)}), \partial_x ]}(s) & \stackrel{\eqref{interpolazione parametri operatore funzioni}}{\lesssim_s} \Big| {\rm Op}\Big( \partial_{\vphi_m}^{s_0 + \mathtt b} \partial_x r_0^{(4)} \Big) \Big|_{0, s, 0}^\Lipg \lesssim_s \Big| {\rm Op}(  r_0^{(4)} ) \Big|_{0, s +s_0 +  \mathtt b + 1, 0}^\Lipg \\
& \stackrel{\eqref{estimate r0 (4)}}{\lesssim_s} \e + \| \io\|_{s + s_0 + \mathtt b + 1 + \sigma_M}^\Lipg\,.
\end{aligned} 
\end{equation}
Next we treat $\partial_{\vphi_m}^{s_0 + \mathtt b} [{\cal R}_M^{(4)}, \partial_x]$, $m \in \Splus$. Note that
$$
\begin{aligned}
\partial_{\vphi_m}^{s_0 + \mathtt b} [{\cal R}_M^{(4)}, \partial_x] & =  \partial_{\vphi_m}^{s_0 + \mathtt b}{\cal R}_M^{(4)} \langle D \rangle \langle D \rangle^{- 1}\partial_x -  \langle D \rangle^{- 1} \partial_x \langle D \rangle \partial_{\vphi_m}^{s_0 + \mathtt b}{\cal R}_M^{(4)} \, . 
\end{aligned}
$$
Since there is a tame constant 
${\mathfrak M}_{\langle D \rangle^{- 1}\partial_x}(s)$ bounded by $1$ it then follows by 
 \eqref{SM4} that, for any $ \sM \leq s \leq S $,  
\begin{equation}\label{partial vphi commutator R M 4}
{\mathfrak M}_{\partial_{\vphi_m}^{s_0 + \mathtt b} [{\cal R}_M^{(4)}, \partial_x]}(s) \lesssim_S \e + \| \io \|_{s + \sigma_M(\mathtt b)}^\Lipg\,. 
\end{equation} 
Combining \eqref{tame bound r0 (4)}, \eqref{partial vphi commutator R M 4} and recalling the definition of $\mu(\mathtt b)$ in \eqref{alpha beta} one obtains tame constants 
${\mathfrak M}_{\partial_{\vphi_m}^{s_0 + \mathtt b} 
[\mathtt R_0, \partial_x]}(s)$, $m \in \Splus$,
satisfying the claimed bound. \\
($ii$) The estimate \eqref{stima delta 1 2 mathtt R0} 
 follows by similar arguments using \eqref{stime delta 12 quarto step} and \eqref{stima delta 12 RM (4)} with $s_1 = \sM$. 
\end{proof}

We perform the almost reducibility scheme for $ {\mathtt L}_0 $ along the scale 
\be\label{def Nn}
N_{- 1} := 1\,,\quad N_\nu := N_0^{\chi^\nu}\,, \ \nu \geq 0\,,\quad \chi := 3/2\, ,
\ee
requiring at each induction step the second order Melnikov non-resonance conditions 
\eqref{Omega nu + 1 gamma}. 

\begin{theorem}\label{iterazione riducibilita}
{\bf (Almost reducibility)} 
There exists $ \overline{\tau} := \overline{\tau} (\t, \Splus ) > 0 $ so that for any $ S > \sM $, 
there is $ N_0 := N_0 (S, {\mathtt b})  \in \N$
with the property that if 
\begin{equation}\label{KAM smallness condition1}
N_0^{\overline{\tau}}  {\mathfrak M}_0(\sM, {\mathtt b}) \gamma^{- 1} \leq 1\,, 
\end{equation}
then the following holds for any $ \nu \in \N $:
\begin{itemize}
\item[${\bf(S1)_{\nu}}$] There exists a Hamiltonian operator
${\mathtt L}_\nu$, acting on $ H^s_\bot $ and defined for  
$\omega \in  {\mathtt \Omega}_\nu^\gamma  $, of the form
\begin{equation}\label{cal L nu}
{\mathtt L}_\nu := \Dom  + \ii {\mathtt D}_\nu + {\mathtt R}_\nu \,, 
\quad {\mathtt D}_\nu := {\rm diag}_{j \in \Sbot} \mu_j^\nu\,,\quad  \mu_j^\nu \in \R \, , 
\end{equation}

where for any $j \in \Sbot$, $ \mu_j^\nu $ 
is a $ \Lipg $-function of the form 
\begin{equation}\label{mu j nu}
\mu_j^\nu(\omega) := \mu_j^0(\omega) + r_j^\nu(\omega)\,, 
\end{equation}
with
\begin{equation}\label{stima rj nu}
 \mu_{-j}^{\nu} = - \mu_j^{\nu}  \, , \qquad 
 |r_j^\nu|^\Lipg \leq C(S) \e \gamma^{- 2}\,,
\end{equation}
and where $\mu_j^{(0)}$ is defined in \eqref{op-diago0}. 
If $\nu =0$, $ \mathtt \Omega_\nu^\gamma  $ is  defined 
to be the set
$ {\mathtt \Omega}_0^\gamma := \DC(\gamma, \tau) $ , and if $\nu \geq 1$,
\begin{align}
{\mathtt \Omega}_\nu^\gamma & :=  {\mathtt \Omega}_\nu^\gamma (\io) :=  
\Big\{ \omega \in {\mathtt \Omega}_{\nu - 1}^\gamma  \, : \label{Omega nu + 1 gamma} 
 |\omega \cdot \ell  + \mu_j^{\nu - 1} - \mu_{j'}^{\nu - 1}| \geq  
\gamma \frac{|j^{3} -  j'^{3}|}{ \langle \ell \rangle^{\tau}}, 
 \forall |\ell  | \leq N_{\nu - 1}, j, j' \in  {\mathbb S}^\bot   \Big\} \, . 
\end{align}
The operators $ {\mathtt R}_\nu $ and   
$ \langle\partial_\vphi \rangle^{\mathtt b }  {\mathtt R}_\nu $ 
are $ \Lipg $-modulo-tame  with modulo-tame constants 
\be\label{def:msharp}
{\mathfrak M}_\nu^\sharp (s) := {\mathfrak M}_{{\mathtt R}_\nu}^\sharp (s) \, , \qquad
{\mathfrak M}_\nu^\sharp (s, {\mathtt b}) := 
{\mathfrak M}_{ \langle \pa_\vphi \rangle^{\mathtt b} {\mathtt R}_\nu}^\sharp (s) \, , 
\ee 
satisfying, for some  $C_* (\sM, \mathtt b) > 0 $, for all $s \in [\sM, S] $, 
\begin{equation}\label{stima cal R nu}
{\mathfrak M}_\nu^\sharp (s) \leq C_* (\sM, \mathtt b) {\mathfrak M}_0 (s, {\mathtt b}) N_{\nu - 1}^{- {\mathtt a}}\,,\qquad 
 {\mathfrak M}_\nu^\sharp ( s, \mathtt b) \leq  C_* (\sM, \mathtt b) {\mathfrak M}_0 (s, {\mathtt b}) N_{\nu - 1}\,.
\end{equation}
Moreover, if  $ \nu \geq 1 $ and 
$\omega \in {\mathtt \Omega}_\nu^\gamma$, there exists a Hamiltonian operator ${ \Psi}_{\nu - 1} $ acting on 
$ H^s_\bot $,  so that the corresponding symplectic time one flow  
\begin{equation}\label{Psi nu - 1}
 { \Phi}_{\nu - 1} := \exp ( { \Psi}_{\nu - 1}) 
\end{equation}
conjugates $ {\mathtt L}_{\nu - 1} $ to 
\be\label{coniugionu+1}
{\mathtt L}_\nu = { \Phi}_{\nu - 1} {\mathtt L}_{\nu - 1} { \Phi}_{\nu - 1}^{- 1}\, . 
\ee
The operators $ \Psi_{\nu - 1} $ and $ \langle \partial_\vphi \rangle^{\mathtt b}  \Psi_{\nu - 1} $ 
are $ \Lipg $-modulo-tame  with a modulo-tame constant satisfying, for all $s \in [\sM, S] $,
(with $ \tau_1$, $\mathtt a $ defined in \eqref{alpha beta})
\be\label{tame Psi nu - 1}
 {\mathfrak M}_{\Psi_{\nu - 1}}^\sharp  (s) \leq \frac{C( \sM, \mathtt b)}{\g} 
N_{\nu - 1}^{\tau_1} N_{\nu - 2}^{- \mathtt a} {\mathfrak M}_0 (s, {\mathtt b}) \, , \quad
{\mathfrak M}_{\langle \partial_\vphi \rangle^{\mathtt b}  \Psi_{\nu - 1}}^\sharp  (s) \leq 
\frac{C( \sM, \mathtt b)}{\g} 
N_{\nu - 1}^{\tau_1} N_{\nu - 2}  {\mathfrak M}_0 (s, {\mathtt b}) \, . 
\ee
\item[${\bf(S2)_{\nu}}$]  For any $j \in \mathbb S^\bot$, there exists a Lipschitz extension $\widetilde \mu_j^\nu :  \Omega \to \R$ of $\mu_j^\nu : \mathtt \Omega_\nu^\gamma \to \R$,
where 
$ \widetilde \mu_j^0 = m_3 (2 \pi j)^3 - \widetilde m_1 2 \pi j - q_j(\omega) $ 
(cf. \eqref{op-diago0}) 
and  $ \widetilde m_1 : \Omega \to \R $ is an extension of $ m_1 $
satisfying  $  |\widetilde m_1|^\Lipg \lesssim \e \gamma^{- 2} $;
if $\nu \geq 1 $, 
$$ 
|\widetilde \mu_j^\nu - \widetilde \mu_j^{\nu - 1}|^\Lipg \lesssim \mathfrak M_{\nu - 1}^\sharp(\sM)  \lesssim {\mathfrak M}_0 (\sM, {\mathtt b}) N_{\nu - 1}^{- {\mathtt a}}\,. 
$$
\item[${\bf(S3)_{\nu}}$] Let $\breve \io_1$, 
$\breve \io_2$ be two tori satisfying \eqref{ansatz I delta} with 
$\mu_0 \geq \sM + \mu(\mathtt b)$.  
Then, for all $ \omega \in {\mathtt \Omega}_\nu^{\gamma_1}( \io_1) \cap {\mathtt \Omega}_\nu^{\gamma_2}( \io_2)$
with $\gamma_1, \gamma_2 \in [\gamma/2, 2 \gamma]$, we have  
\begin{align}\label{stima R nu i1 i2}
& \! \! \| |{\mathtt R}_\nu(\io_1) - {\mathtt R}_\nu(\io_2)| \|_{{\cal B}(H^{\sM})} 
 \lesssim_{S}   N_{\nu - 1}^{- \mathtt a}
 \| \io_1 - \io_2\|_{\sM +  \mu(\mathtt b) }, \\ 
& \label{stima R nu i1 i2 norma alta}
\! \! \|  |\langle \partial_\vphi \rangle^{\mathtt b}({\mathtt R}_\nu(\io_1) - {\mathtt R}_\nu(\io_2)) | \|_{{\cal B}(H^{\sM})}
 \lesssim_{S}  N_{\nu - 1} \| \io_1 - \io_2\|_{ \sM +  \mu(\mathtt b) }\,.
\end{align}
Moreover, if $\nu \geq 1$, then for any $j \in \Sbot$, 
\begin{align}\label{r nu - 1 r nu i1 i2}
& \big|(r_j^\nu(\io_1) - r_j^\nu(\io_2)) - (r_j^{\nu - 1}(\io_1) - r_j^{\nu - 1}(\io_2))  \big| \lesssim 
\| |{\mathtt R}_\nu(\io_1) - {\mathtt R}_\nu(\io_2)| \|_{{\cal B}(H^{\sM})} \,, \\
& \ |r_j^{\nu}(\io_1) - r_j^{\nu}(\io_2)| \lesssim_S    \| \io_1 - \io_2  \|_{ \sM  + \mu(\mathtt b) }\,. \label{r nu i1 - r nu i2}
\end{align}
\item[${\bf(S4)_{\nu}}$]  Let $\breve \io_1$, $\breve \io_2$ be two tori as in ${\bf(S3)_{\nu}}$ and $0 < \rho < \gamma/2$. Then 
$$
 C(S) N_{\nu - 1}^\tau \| \io_1 - \io_2 \|_{ \sM + \mu(\mathtt b) } \leq \rho \quad \Longrightarrow \quad 
{\mathtt \Omega}_\nu^\gamma(\io_1) \subseteq {\mathtt \Omega}_\nu^{\gamma - \rho}(\io_2) \, . 
$$
\end{itemize}
\end{theorem}

Theorem \ref{iterazione riducibilita} implies that the symplectic invertible operator
\be\label{defUn}
U_n :=   \Phi_{n - 1} \circ \ldots \circ \Phi_0, \quad n \geq 1\,,
\ee
almost diagonalizes  $ {\mathtt  L}_0 $, meaning that \eqref{cal L infinito} below holds. 
The following corollary of Theorem \ref{iterazione riducibilita} and Lemma \ref{lem:tame iniziale} can be deduced as in \cite{Berti-Montalto}. 

\begin{theorem}\label{Teorema di riducibilita} {\bf (KAM almost-reducibility)}
Assume the ansatz
\eqref{ansatz I delta}  with $ \mu_0 \geq \sM +   \mu (\mathtt b)$. Then for any $S > \sM$ there exist $N_0 := N_0(S, \mathtt b) > 0$, $ 0 < \d_0 := \delta_0(S)  <  1 $,  so that if 
\begin{equation}\label{ansatz riducibilita}
N_0^{\overline{\tau}} \e \gamma^{- 3}  \leq \d_0  
\end{equation}
with $ \overline{\tau} := \overline{\tau} (\tau, \Splus)  $
given by Theorem \ref{iterazione riducibilita},
the following holds:
for any $ n \in \N$ and any $ \omega $ in 
\be\label{Cantor set}
{ {\mathtt \Omega}}_{n + 1}^{ \gamma} := {{\mathtt \Omega}}_{n + 1}^{ \gamma} (\io)  =  \bigcap_{\nu = 0}^{n + 1 } {\mathtt \Omega}_\nu^\gamma
\ee
with $ {\mathtt \Omega}_\nu^\gamma$ defined in \eqref{Omega nu + 1 gamma},
the operator $ U_n  $, introduced in \eqref{defUn}, 
is well defined  and 
${\mathtt L}_{n } := {U}_n {\mathtt L}_0 {U}_n^{- 1}$
satisfies
\begin{equation}\label{cal L infinito}
{\mathtt L}_{n } = 
\Dom  + \ii {\mathtt D}_n + {\mathtt R}_n  
\end{equation}
where $ {\mathtt D}_n $ and $  {\mathtt R}_n $ are defined in \eqref{cal L nu} (with $ \nu = n $).
The operator $ {\mathtt R}_n $ is $ \Lipg $-modulo-tame with a modulo-tame constant
\be\label{stima resto operatore quasi diagonalizzato}
{\mathfrak M}^\sharp_{{\mathtt R}_n}(s)   \lesssim_S   N_{n - 1}^{- {\mathtt a}} (\e + 
\| \iota \|_{s   + \mu (\mathtt b) }^\Lipg) \, , \qquad \forall \sM \leq s \leq S\,.
\ee
Moreover, the operator $ {\mathtt L}_n $ is Hamiltonian,
 $ {U}_n$, $ {U}_n^{- 1} $ are symplectic, and 
$ U_{n}^{\pm 1} -  {\rm Id}_\bot  $ are $ \Lipg$-modulo-tame with a modulo-tame constant satisfying
\begin{equation}\label{stima Phi infinito}
{\mathfrak M}^\sharp_{U_{n}^{\pm 1} -  {\rm Id}_\bot} (s) \lesssim_S  \g^{-1} N_0^{\tau_1} (\e + 
\| \iota \|_{s  + \mu (\mathtt b) }^\Lipg)\,, \quad \forall \sM \leq s \leq S \, , 
\end{equation} 
where ${\rm Id}_\bot$ denotes
the identity operator on $L^2_\bot(\T_1)$ and
$\tau_1$ is defined in \eqref{alpha beta}.  
\end{theorem}

\subsection{Proof of  Theorem \ref{iterazione riducibilita}}

{\sc Proof of ${\bf(S1)}_{0}$}. Properties \eqref{cal L nu}-\eqref{stima rj nu} for $ \nu = 0 $
 follow by  \eqref{defL0-red}-\eqref{op-diago0} with  $r_j^0(\omega) = 0$. 
Moreover also \eqref{stima cal R nu} for $ \nu = 0 $ holds because, 
arguing as in Lemma 7.6 in \cite{Berti-Montalto},  
the following Lemma holds: 
\begin{lemma}\label{lem: Initialization} 
$ {\mathfrak M}_0^\sharp (s) $,  $ {\mathfrak M}_0^\sharp ( s, \mathtt b) \lesssim_{ \mathtt b}    
{\mathfrak M}_0 (s, {\mathtt b})  $ where ${\mathfrak M}_0 (s, {\mathtt b})$ is defined in \eqref{stima mathfrak M s0 b}. 
\end{lemma}

\noindent
{\sc Proof of ${\bf(S2)}_0$}. For any $j \in \mathbb S^\bot$, $\mu_j^{0}$ is defined in \eqref{op-diago0}.
Note that $m_3(\omega)$ and $q_j(\omega)$ are already defined on the whole parameter space $\Omega$.  
By the Kirszbraun Theorem and \eqref{stima m1}
there is an extension  $\widetilde m_1$ on  $ \Omega $ of  $ m_1$ 
satisfying the estimate 
$ |\widetilde m_1|^\Lipg \lesssim \e \gamma^{- 2}$. 
This proves ${\bf (S2)}_0$.

\noindent
{\sc Proof of ${\bf(S3)}_0$}. The estimates \eqref{stima R nu i1 i2}, \eqref{stima R nu i1 i2 norma alta} at $ \nu = 0 $ follows arguing as in the proof of ${\bf(S3)}_0$ in \cite{Berti-Montalto}.   

\noindent 
{\sc Proof of ${\bf(S4)}_{0}$}. By the definition
of  $ \tOm_0^\gamma$ one has
$ \tOm_0^\gamma(\io_1) = \DC(\gamma, \tau) \subseteq \DC(\gamma - \rho, \tau) = \tOm_0^{\gamma - \rho}(\io_2)$. 
\\[1mm]
\noindent
{\bf Iterative reductibility step.}
In what follows we describe how to define $\Psi_\nu$, $\Phi_\nu$, ${\mathtt L}_{\nu + 1}$
etc.,  at the iterative step. To simplify notation we drop the index $\nu$ and write $+$ instead of $\nu + 1$.
So, e.g. we write
${\mathtt  L}$ for ${\mathtt  L}_{\nu }$,
${\mathtt  L}_+$ for ${\mathtt  L}_{\nu +1}$,
$\Psi$ for $\Psi_{\nu}$, etc.
We conjugate $ {\mathtt L} $ by the symplectic time one flow map 
\be\label{defPsi}
\Phi:= \exp( \Psi ) 
\ee
generated by a Hamiltonian vector field $ \Psi $ acting in $ H^s_\bot $. 
By a Lie expansion we get 
\begin{equation}\label{Ltra1}
\begin{aligned}
& \Phi  \mathtt L \Phi^{- 1} = \Phi ( \omega \cdot \partial_\vphi + \ii \mathtt D ) \Phi^{- 1} + \Phi \mathtt R \Phi^{- 1} \\
& = \omega \cdot \partial_\vphi + \ii \mathtt D -  \omega \cdot \partial_\vphi \Psi - \ii [\mathtt D, \Psi ] + \Pi_N \mathtt R + \Pi_N^\bot \mathtt R - \int_0^1 {\rm exp}( \tau \Psi)[\mathtt R, \Psi] {\rm exp}(- \tau \Psi)\, d \tau \\
& \quad + \int_0^1 (1 - \tau) {\rm exp}( \tau \Psi) \Big[\omega \cdot \partial_\vphi \Psi + \ii [\mathtt D, \Psi ] , \Psi \Big] {\rm exp}(- \tau \Psi)\, d \tau 
\end{aligned}
\end{equation}
where the projector $ \Pi_N $ is defined in \eqref{proiettore-oper} and $ \Pi_N^\bot := 
{\rm Id }_\bot - \Pi_N$.
We want to solve the homological equation 
\begin{equation}\label{equazione omologica}
- \Dom\Psi -  \ii [{\mathtt  D}, \Psi] + \Pi_{N} {\mathtt  R} = [{\mathtt  R}] \qquad 
{\rm where}  \qquad [{\mathtt R}]  := 
{\rm diag}_{j \in \Sbot} {\mathtt R}_j^j(0) \, . 
\end{equation}
The solution of \eqref{equazione omologica} is 
\begin{align}\label{shomo1}
& \Psi_j^{j'}(\ell  ) := \begin{cases}
 \dfrac{{\mathtt R}_j^{j'}(\ell  )}{\ii (\omega \cdot \ell  + \mu_j - \mu_{j'})} \qquad \forall (\ell ,  j,  j') \neq (0,  j,  j) \, , \, \,
| \ell  | \leq N \, , \, j, j' \in \mathbb S^\bot \\
0 \qquad \text{otherwise}\, .
\end{cases}  
\end{align}
The denominators in \eqref{shomo1} are different from zero 
for $ \om \in {\mathtt \Omega}_{\nu+1}^\gamma $ 
(cf. \eqref{Omega nu + 1 gamma}).

\begin{lemma} {\bf (Homological equations)}\label{Homological equations tame}
(i) The solution 
$ \Psi  $ of the homological equation 
 \eqref{equazione omologica}, given by \eqref{shomo1}
 for $ \omega \in {\mathtt \Omega}_{\nu+1}^{\gamma}$,
   is a $ \Lipg $-modulo-tame operator with a modulo-tame constant satisfying
\begin{align}\label{stima tame Psi}
{\mathfrak M}_{\Psi }^\sharp (s) \lesssim N^{\tau_1} \g^{-1} {\mathfrak M}^\sharp (s) \, , \quad
{\mathfrak M}_{\langle \partial_\vphi \rangle^{\mathtt b} \Psi  }^\sharp (s)
 \lesssim N^{\tau_1} \g^{-1} {\mathfrak M}^\sharp (s, {\mathtt b})\,,
\end{align}
where $ \tau_1 := 2 \tau + 1$. Moreover $\Psi $ is Hamiltonian.\\
(ii) Let $\breve \io_1 $, $\breve \io_2 $ be two tori and  define 
$ \Delta_{12} \Psi  := \Psi  (\io_2) - \Psi  (\io_1) $. 
If $\gamma/2 \leq \gamma_1, \gamma_2 \leq 2 \gamma$  
then, for any $ \omega  \in {\mathtt \Omega}_{\nu + 1}^{\gamma_1}(\io_1) \cap {\mathtt \Omega}_{\nu + 1}^{\gamma_2}(\io_2)$,
\begin{align}\label{stime delta 12 Psi bassa}
& \!\!\! \!  \| |\Delta_{12} \Psi  | \|_{{\cal B}(H^{\sM})}\leq C 
N^{2 \tau} \gamma^{- 2} \big( \||{\mathtt R}(\io_2)| \|_{{\cal B}(H^{\sM})} \| \io_1 - \io_2 \|_{\sM +  \mu(\mathtt b)} + \| | \Delta_{1 2} {\mathtt R} | \|_{{\cal B}(H^{\sM})} \big) \, , \\  
&
\label{stime delta 12 Psi alta}
\!\!\!  \! \| |\langle \partial_\vphi \rangle^{\mathtt b}\Delta_{12} \Psi  | \|_{{\cal B}(H^{\sM})} \! \lesssim_{\mathtt b} \!
N^{2 \tau} \gamma^{- 2} \big(  \| |\langle \partial_\vphi \rangle^{\mathtt b}{\mathtt R}(\io_2) | \|_{{\cal B}(H^{\sM})} \| \io_1 - \io_2 \|_{ \sM + \mu(\mathtt b)} 
\! + \!  \| | \langle \partial_\vphi \rangle^{\mathtt b}\Delta_{12} {\mathtt R}  | \|_{{\cal B}(H^{\sM})}\big) \, . 
\end{align} 
\end{lemma}

\begin{proof}
Since $ {\mathtt R} $  is Hamiltonian, one infers from
Definition \ref{definition Hamiltonian operators} and 
Lemma \ref{lem:PR}-($iii$) that the operator $ \Psi $ defined in \eqref{shomo1} 
is Hamiltonian as well. 
We now prove \eqref{stima tame Psi}. 
Let $ \om \in {\mathtt \Omega}_{\nu+1}^{\gamma} $. 
By \eqref{Omega nu + 1 gamma},
and the definition of $ \Psi  $ in \eqref{shomo1}, it follows that for any $ (\ell, j, j')  \in \Z^{\Splus} \times \Sbot \times \Sbot $, 
with $ |\ell | \leq N $, $ (\ell, j, j') \neq (0,  j,  j) $, 
\begin{equation}\label{stima elementare Psi R EQ homo}
 |\Psi_j^{j'}(\ell )|  
\lesssim \langle \ell \rangle^\tau \gamma^{-1} |{\mathtt R}_j^{j'}(\ell )| 
\end{equation} and 
$$
\Delta_\om \Psi_j^{j'}(\ell ) = 
\dfrac{\Delta_{\om} {\mathtt R}_j^{j'}(\ell)}{\delta_{\ell j j'}(\om_1)} - {\mathtt R}_j^{j'}(\ell;\om_2) 
\dfrac{\Delta_{\om} \delta_{\ell j j'}}{\delta_{\ell j j'}(\om_1) \delta_{\ell j j'}(\om_2)}\,,\quad 
\delta_{\ell j j'}(\om)  := \ii (\omega \cdot \ell + \mu_j - \mu_{j'})\,.
$$
By \eqref{stima mu j 0 - mu j' 0}, \eqref{mu j nu},  \eqref{stima rj nu} one gets 
$
|\Delta_\omega \delta_{\ell j j'}| \lesssim (\langle \ell \rangle + |j^3 - j'^3| ) |\omega_1 - \omega_2|,
$
and therefore, using also \eqref{Omega nu + 1 gamma},
we deduce that  
\be\label{stima-Psi-R}
|\Delta_\omega \Psi_j^{j'}(\ell )| \lesssim \langle \ell \rangle^\tau \gamma^{- 1} |\Delta_{\om} {\mathtt R}_j^{j'}(\ell)| + \langle \ell \rangle^{2 \tau + 1} \gamma^{- 2}  |{\mathtt R}_j^{j'}(\ell;\om_2) | |\omega_1 - \omega_2| \,.
\ee
Recalling the definition \eqref{CK0-tame}, using
 \eqref{stima elementare Psi R EQ homo}, \eqref{stima-Psi-R},
and arguing as in the proof of the estimates (7.61) in \cite[Lemma 7.7]{Berti-Montalto}, one
then deduces \eqref{stima tame Psi}.
The estimates \eqref{stime delta 12 Psi bassa}-\eqref{stime delta 12 Psi alta} can be obtained by arguing
similarly. 
\end{proof}

By \eqref{Ltra1}--\eqref{equazione omologica} one has
$$
\begin{aligned}
{\mathtt L}_+ & = \Phi {\mathtt L} \Phi^{- 1} =  \Dom  + \ii {\mathtt  D}_+ + {\mathtt  R}_+ 
\end{aligned}
$$
which  proves \eqref{coniugionu+1} and \eqref{cal L nu} at the step $ \nu + 1 $,  with 
\be\label{new-diag-new-rem} 
\begin{aligned}
& \ii {\mathtt  D}_+ := \ii {\mathtt  D} + [{\mathtt  R}] \, , \quad \\
& {\mathtt  R}_+  =  \Pi_N^\bot \mathtt R  - \int_0^1 {\rm exp}( \tau \Psi)[\mathtt R, \Psi] {\rm exp}(- \tau \Psi)\, d \tau     + \int_0^1 (1 - \tau) {\rm exp}( \tau \Psi) \big[\Pi_N \mathtt R - [\mathtt R]  , \Psi \big]  {\rm exp}(- \tau \Psi)\, d \tau \, . 
\end{aligned}
\ee
The operator $ {\mathtt L}_+  $ has the same form as $ \mathtt L $.
More precisely, 
$ {\mathtt  D}_+  $ is diagonal 
and $ {\mathtt  R}_+  $ is the sum of 
an operator supported on high frequencies and  
one which is quadratic in
$ \Psi $ and  $ { \mathtt R} $. 
The new normal form $ {\mathtt  D}_+  $ has the following properties:

\begin{lemma}\label{nuovadiagonale} 
{\bf (New diagonal part)}  
(i) The new normal form 
 is 
\be\label{new-NF}
 {\mathtt  D}_+ =  {\mathtt  D} - \ii [{\mathtt  R}] \,, 
\qquad {\mathtt  D}_+ := {\rm diag}_{j \in \Sbot} \mu_j^+\,,\qquad \mu_j^+ := 
\mu_j + {\mathtt r}_j\in \R \, , 
\ \ {\mathtt r}_j:= -\ii {\mathtt R}_j^j(0)\, , \ \ 
\forall j \in \mathbb S^\bot \, , 
\ee
with 
$$ 
\mu_{-j}^+ = - \mu_j^+ \, , \quad 
|\mu_j^+ - \mu_j|^\Lipg = |\mathtt r_j|^\Lipg \lesssim {\mathfrak M}^\sharp (\sM) \, .
$$
(ii) For any tori $\breve \io_1(\omega) $, $\breve \io_2(\omega)$
and any $ \omega \in {\mathtt \Omega}_{\nu}^{\gamma_1}(\io_1) \cap {\mathtt \Omega}_\nu^{\gamma_2}(\io_2)$, one has 
\be\label{diff:r1r2}
| {\mathtt r}_j (\io_1) -{\mathtt r}_j (\io_2)   |  
\lesssim \| | \Delta_{1 2} {\mathtt R} | \|_{{\cal B}(H^{\sM})}\,.
\ee
\end{lemma}

\begin{proof} 
By the definition \eqref{def:msharp}
of $ {\mathfrak M}^\sharp (\sM) $ and using \eqref{tame-coeff} (with $\sM = s_1 $) we have that $ |\mu_j^+ - \mu_j|^\Lipg \leq |{\mathtt R}_j^j(0) |^\Lipg \lesssim {\mathfrak M}^\sharp (\sM) $.  
Since $ {\mathtt R} (\vphi) $ is Hamiltonian, Lemma \ref{lem:PR} implies that $ {\mathtt r}_j = - \ii {\mathtt R}_j^j(0) $, 
$j \in \mathbb S^\bot$,
are odd in $ j $ and real.
The estimate \eqref{diff:r1r2} 
is proved in the same way by using $|\Delta_{1 2} {\mathtt R}_j^j(0) | \leq C \| | \Delta_{1 2} {\mathtt R} | \|_{{\cal B}(H^{\sM})}$. 
\end{proof}
\noindent
{\bf Induction.} 
Assuming that the statements $({\bf S1})_{\nu} $-$({\bf S4})_{\nu} $  are true for some $ \nu \geq 0 $ we show in this paragraph that 
$({\bf S1})_{\nu + 1}$-$({\bf S4})_{\nu + 1}$ hold.
\\[1mm]
{\sc Proof of $({\bf S1})_{\nu + 1}$}. 
By Lemma \ref{Homological equations tame}, for all $ \om \in {\mathtt \Omega}_{\nu + 1}^\gamma$ the solution $\Psi_{\nu}$
of the homological equation 
 \eqref{equazione omologica}, defined
 in \eqref{shomo1},   is well defined and, 
 by  \eqref{stima tame Psi}, \eqref{stima cal R nu},  
 satisfies the estimates \eqref{tame Psi nu - 1} at $ \nu + 1 $.
In particular, the estimate \eqref{tame Psi nu - 1} for 
$ \nu + 1 $, $ s = \sM $ and \eqref{alpha beta}, \eqref{KAM smallness condition1} imply 
\begin{equation}\label{costante tame Psi nu i (s0)}
{\mathfrak M}^\sharp_{\Psi_{\nu }}(\sM ) \lesssim_{\mathtt b} 
 N_{\nu}^{\tau_1} N_{\nu - 1}^{- \mathtt a} \gamma^{- 1} {\mathfrak M}_0(\sM, {\mathtt b}) \leq 1 \, .
\end{equation}
By Lemma \ref{serie di neumann per maggiorantia} and using again Lemma \ref{Homological equations tame}
one infers that
\begin{equation}\label{stima tame Psi tilde}
\begin{aligned}
{\mathfrak M}^\sharp_{\Phi_\nu^{\pm 1}}(\sM) & \lesssim 1\,, \\
{\mathfrak M}^\sharp_{\langle \partial_\vphi \rangle^{\mathtt b} \Phi_\nu^{\pm 1}}(\sM) & \lesssim 1 + {\mathfrak M}_{\langle \partial_\vphi \rangle^\mathtt b \Psi_\nu}(\sM) \lesssim 1 + N_\nu^{\tau_1} \gamma^{- 1} {\mathfrak M}^\sharp_\nu(\sM, \mathtt b)\,,   \\
{\mathfrak M}^\sharp_{\Phi_\nu^{\pm 1}}(s)  & \lesssim 1 + {\mathfrak M}^\sharp_{\Psi_\nu}(s) \lesssim_s 1 + N_\nu^{\tau_1} \g^{-1} {\mathfrak M}_\nu^\sharp (s)   \\
{\mathfrak M}^\sharp_{\langle \partial_\vphi \rangle^{\mathtt b} \Phi_\nu^{\pm 1}}(s) & \lesssim  1 + {\mathfrak M }_{\langle \partial_\vphi \rangle^{\mathtt b } \Psi_\nu}(s) + {\mathfrak M}_{\Psi_\nu}^\sharp(s)  {\mathfrak M }_{\langle \partial_\vphi \rangle^{\mathtt b } \Psi_\nu}(\sM) \\
& \stackrel{\eqref{KAM smallness condition1}, \eqref{stima cal R nu}, \eqref{stima tame Psi}}{\lesssim} 1 +  N_\nu^{\tau_1} \gamma^{- 1} {\mathfrak M}^\sharp_\nu(s, \mathtt b)  + N_\nu^{2 \tau_1} N_{\nu - 1} \gamma^{- 1} {\mathfrak M}^\sharp_\nu(s)\,. 
\end{aligned}
\end{equation}
By Lemma \ref{nuovadiagonale}, by the estimate \eqref{stima cal R nu} and Lemma \ref{lem:tame iniziale},  the operator $ {\mathtt  D}_{\nu+1} $ is diagonal and its  eigenvalues  
$ \mu_j^{\nu+1} : {\mathtt \Omega}_{\nu+1}^\gamma \to \R $ 
satisfy \eqref{stima rj nu} at $ \nu + 1 $. 

Now we estimate the remainder $ {\mathtt  R}_{\nu + 1} $
defined in  \eqref{new-diag-new-rem}.

\begin{lemma}\label{estimate in low norm} {\bf (Nash-Moser iterative scheme)}
The operator $  {\mathtt R}_{\nu + 1} $ is $ \Lipg $-modulo-tame with a
modulo-tame constant satisfying
\begin{equation}\label{schema quadratico tame}
{\mathfrak M}_{\nu + 1}^\sharp (s ) \lesssim N_\nu^{- {\mathtt b}} {\mathfrak M}_\nu^\sharp (s, {\mathtt b}) 
+  N_\nu^{\tau_1} \gamma^{- 1} {\mathfrak M}_\nu^\sharp (s) {\mathfrak M}_\nu^\sharp (\sM)\,.
\end{equation}
The operator 
$ \langle \partial_\vphi \rangle^{ {\mathtt b}} {\mathtt R}_{\nu + 1} $ is $ \Lipg $-modulo-tame with a modulo-tame constant satisfying 
\be\label{M+Ms}
{\mathfrak M}_{\nu + 1}^\sharp (s, {\mathtt b}) \lesssim_{\mathtt b} 
 {\mathfrak M}_\nu^\sharp (s, \mathtt b) + 
 N_\nu^{\tau_1} \gamma^{- 1} {\mathfrak M}_\nu^\sharp (s, \mathtt b) 
{\mathfrak M}_\nu^\sharp(\sM) + N_\nu^{\tau_1} \gamma^{- 1} {\mathfrak M}_\nu^\sharp (\sM, \mathtt b) {\mathfrak M}_\nu^\sharp (s)  \,.
\ee
\end{lemma}

\begin{proof}
The proof follows by Lemmata \ref{lemma:smoothing-tame},
\ref{interpolazione moduli parametri},
using the estimates  \eqref{stima cal R nu}, \eqref{stima tame Psi}, \eqref{stima tame Psi tilde}. 
\end{proof}

The estimates \eqref{schema quadratico tame}, \eqref{M+Ms},  and  \eqref{alpha beta},  
allow to prove  that also \eqref{stima cal R nu} holds at the step $ \nu + 1 $. It implies (see \cite[Lemma 7.10]{Berti-Montalto})
\begin{lemma}\label{stima M nu + 1 K nu + 1}
$ {\mathfrak M}_{\nu + 1}^	\sharp (s) \leq C_*(\sM, \mathtt b) N_\nu^{- \mathtt a} {\mathfrak M}_0(s, \mathtt b) $ and 
$ {\mathfrak M}_{\nu + 1}^\sharp (s,\mathtt b) \leq C_*(\sM, \mathtt b) N_\nu {\mathfrak M}_0(s, \mathtt b)  $. 
\end{lemma}

\noindent {\sc Proof of $({\bf S2})_{\nu + 1}$.} By Lemma \ref{nuovadiagonale}, for any $j \in \mathbb S^\bot$, $\mu_j^{\nu + 1} = \mu_j^\nu + \mathtt r_j^\nu$ where $ |\mathtt r_j^\nu |^\Lipg \lesssim \mathfrak M_0(\sM, \mathtt b) N_\nu^{- \mathtt a}$.  
Then $({\bf S2})_{\nu + 1}$ follows by defining 
$\widetilde \mu_j^{\nu + 1} 
:= \widetilde \mu_j^\nu + \widetilde {\mathtt r}_j^\nu$
where $\widetilde {\mathtt r}_j^\nu :  \Omega \to \R$ 
is a Lipschitz extension of  ${\mathtt r}_j^\nu$ 
(cf. Kirszbraun extension Theorem).

\smallskip

\noindent {\sc Proof of $({\bf S3})_{\nu + 1}$.} 
The proof follows by induction arguing as in 
the proof of $({\bf S2})_{\nu + 1}$.

\smallskip

\noindent {\sc Proof of $({\bf S4})_{\nu + 1}$.} 
The proof is the same as that of $({\bf S3})_{\nu + 1}$  in 
\cite[Theorem 4.2]{BBM-Airy}. 
 \qed

\subsection{Almost-invertibility  of  $ {\mathcal L}_\om $}\label{quasi invertibilita}

By \eqref{cal L infinito}, for any $ \omega \in \mathtt {\mathtt \Omega}^\gamma_n$, we have that 
$ \mL_0 = \mU_n^{-1} \mL_n \mU_n $ where $\mU_n$ is defined in \eqref{defUn} and thus 
\begin{equation}\label{final semi conjugation}
{\mathcal L}_\omega =  {\cal V}_n^{-1}  {\mathtt L}_n  {\cal V}_n \, ,  
\qquad {\cal V}_n := U_n \Phi^{(4)}  \cdots \Phi^{(1)} \, .
\end{equation}
\begin{lemma} 
There exists $\sigma = \sigma(\tau, \mathbb S_+) > 0$ such that, if \eqref{ansatz riducibilita}   and 
\eqref{ansatz I delta} with $\mu_0 \geq \sM + \mu(\mathtt b) + \sigma$ hold, then the operators ${\cal V}_n^{\pm 1}$ satisfy for any $\sM \leq s \leq S$ the estimate
\begin{equation}\label{stime W1 W2}
\| {\cal V}_n^{\pm 1} h \|_s^\Lipg \lesssim_S \| h \|_{s + \sigma }^\Lipg  + 
N_0^{\tau_1} \gamma^{- 1}\| \iota \|_{s + \mu(\mathtt b) + \sigma}^\Lipg \| h \|_{\sM + \sigma}^\Lipg \,.
\end{equation}
\end{lemma}

\begin{proof}
By the estimates \eqref{stima tame Phi (2) enunciato}, \eqref{stima tame Phi (1) enunciato}, \eqref{stima Phi (3) tau tame nel lemma}, \eqref{stima Phi (4) enunciato}, using Lemmata \ref{composizione operatori tame AB}, 
\ref{lemma operatore e funzioni dipendenti da parametro}, \ref{A versus |A|} and \eqref{stima Phi infinito}. 
\end{proof}

We now decompose the operator $ {\mathtt L}_n $ in \eqref{cal L infinito} as 
\begin{equation}\label{decomposizione bf Ln}
{\mathtt L}_n = {\mathfrak L}_n^{<}  + {\mathtt R}_n + {\mathtt R}_n^\bot   
\end{equation}
where 
\be\label{Rn-bot}
{\mathfrak L}_n^{<} := \Pi_{K_n} \big( \Dom  + \ii {\mathtt D}_n \big) \Pi_{K_n} + \Pi_{K_n}^\bot \, , \quad 
{\mathtt R}_n^\bot := 
\Pi_{K_n}^\bot \big( \Dom  + \ii {\mathtt D}_n \big) \Pi_{K_n}^\bot 
- \Pi_{K_n}^\bot \, , 
\end{equation}
the diagonal operator $ {\mathtt D}_n $ is defined in 
\eqref{cal L nu} (with $ \nu = n $), and 
$ K_n := K_0^{\chi^n} $ 
is the scale of the nonlinear Nash-Moser iterative scheme introduced in 
\eqref{definizione Kn}.  

\begin{lemma} {\bf (First order Melnikov non-resonance conditions)}\label{lem:first-Mel}
For all $ \om  $ in 
\begin{equation}\label{prime di melnikov}
 {\mathtt \Lambda}_{n + 1}^{\gamma}  :=  {\mathtt \Lambda}_{n + 1}^{\gamma} ( \io ) := 
\big\{ \omega \in  \mathtt \Omega  :  |\omega \cdot \ell  
+  \widetilde \mu_j^n| \geq  2\gamma |j|^{3} \langle \ell  \rangle^{- \tau} \,,
\quad \forall | \ell  | \leq K_n\,,\ j \in \mathbb S^\bot \big\} \, ,
\end{equation}
the operator $ {\mathfrak L}_n^< $ in \eqref{Rn-bot} 
is invertible and
\begin{equation}\label{stima tilde cal Dn}
\| ({\mathfrak L}_n^<)^{- 1} g \|_s^\Lipg 
\lesssim \gamma^{- 1} \| g \|_{s + 2 \t + 1 }^\Lipg\, . 
\end{equation}
\end{lemma}
 By \eqref{final semi conjugation}, \eqref{decomposizione bf Ln}, 
Theorem \ref{Teorema di riducibilita}, 
 estimates \eqref{stima tilde cal Dn}, \eqref{stima tilde cal Rn}, \eqref{stime W1 W2}, and using that, 
for all $ b  > 0$,   
\begin{equation}\label{stima tilde cal Rn}
\| {\mathtt R}_n^\bot h \|_{\sM}^\Lipg \lesssim K_n^{- b} \| h \|_{\sM + b + 3}^\Lipg\,,\quad \| {\mathtt R}_n^\bot h\|_s^\Lipg \lesssim \| h \|_{s + 3}^\Lipg \, , 
\end{equation}
we deduce the following theorem, stating the almost-invertibility assumption of  ${\cal L}_\omega$ of Section \ref{sezione approximate inverse}. 

\begin{theorem}\label{inversione parziale cal L omega}
{\bf (Almost-invertibility of $ {\cal L}_\omega $)} 
Let $ {\mathtt a}, {\mathtt b}, M $ as in \eqref{alpha beta} and $S > \sM$. There exists $\sigma = \sigma(\tau, \mathbb S_+) > 0$ such that, 
if \eqref{ansatz riducibilita} and \eqref{ansatz I delta} with $\mu_0 \geq \sM + \mu(\mathtt b) + \sigma$  hold, then, for all 
\begin{equation}\label{Melnikov-invert}
\omega \in  {\bf \Omega}_{n + 1}^{\g}  
:= {\bf \Omega}_{n + 1}^{\g} (\io) 
:= {\mathtt \Omega}_{n + 1}^\gamma  
\cap  {\mathtt \Lambda}_{n + 1}^{\gamma}
\end{equation}
(see \eqref{Cantor set}, \eqref{prime di melnikov}),  
the operator $ {\mathcal L}_\omega$ defined in \eqref{Lomega def} 
can be decomposed as 
\begin{align}\label{splitting cal L omega}
&  {\mathcal L}_\omega  = {\mathcal L}_\omega^{<} + {\cal R}_\omega +  {\cal R}_\omega^\bot \,, \\
& {\mathcal L}_\omega^{<} := {\cal V}_n^{-1} {\mathfrak L}_n^{<} {\cal V}_n \,,\quad 
{\cal R}_\omega := {\cal V}_n^{- 1} {\mathtt R}_n {\cal V}_n
\,,\quad {\cal R}_\omega^\bot := {\cal V}_n^{- 1} {\mathtt R}_n^\bot {\cal V}_n \, , \nonumber 
\end{align}
where  ${\mathcal L}_\omega^{<} $ is invertible and satisfies \eqref{tame inverse} and the operators ${\cal R}_\omega$ and ${\cal R}_\omega^\bot$  satisfy \eqref{stima R omega corsivo}-\eqref{stima R omega bot corsivo bassa}. 
\end{theorem}

\section{Proof of Theorem \ref{main theorem}}\label{sec:NM}
Theorem \ref{main theorem} 
is a consequence of Theorem \ref{iterazione-non-lineare} below
where we construct iteratively a sequence of better and better approximate solutions of the equation
$ {\cal F}_\omega ( \iota, \zeta) = 0$ where ${\cal F}_\omega $  is defined in 
\eqref{operatorF}. 

\subsection{The Nash-Moser iteration}
We consider the finite-dimensional subspaces 
of $L^2_\vphi \times L^2_\vphi \times L^2_\bot$, defined 
for any $n \in \N$ as
$$ 
\mathtt E_n := \big\{ \iota(\vphi ) = (\Theta , y , w) (\vphi) , \ \  
\Theta = \Pi_n \Theta, \ y = \Pi_n y , \ w = \Pi_n w \big\}
$$
where 
$L^2_\vphi = L^2_\vphi(\T_1 \times \R^{\Splus})$ 
(cf. \eqref{def H^s_varphi}) and where
$ \Pi_n:= \Pi_{K_n} : 
L^2_\bot \to \cap_{s \ge 0}H^s_\bot  $ is the projector
(cf. \eqref{def:smoothings})
\be\label{truncation NM}
\Pi_n  : \, w = 
\sum_{\ell  \in \Z^{\Splus},  j \in \Sbot} w_{\ell, j} e^{\ii (\ell  \cdot \ph + 2 \pi jx)} 
 \quad \mapsto \quad \Pi_n w 
:= \sum_{|(\ell ,j)| \leq K_n} 
w_{\ell,  j} e^{\ii (\ell  \cdot \ph + 2 \pi jx)}  
\ee
with $ K_n = K_0^{\chi^n} $ (cf. \eqref{definizione Kn}) and also
denotes the corresponding one on $L^2_\varphi,$ given by
 $L^2_\varphi \to \cap_{s \ge 0} H^s_{\varphi},$ 
 $p = \sum_{\ell  \in \Z^{\Splus}} p_{\ell} 
 e^{\ii \ell  \cdot \ph}    \mapsto 
 \sum_{|\ell| \leq K_n} p_{\ell} e^{\ii \ell  \cdot \ph} $.
Note that $ \Pi_n $, $n \ge 1$, are smoothing operators
for the Sobolev spaces $H^s_\bot$. In particular  $ \Pi_n $ and
$ \Pi_n^\bot := {\rm Id} - \Pi_n $
satisfy the smoothing properties \eqref{p3-proi}.
For the Nash-Moser Theorem \ref{iterazione-non-lineare},
stated below, we introduce the constants
\begin{align}
& \label{costanti nash moser 2} \!\! \!  
 \overline \sigma := \max \{ \sigma_1, \sigma_2 \}\, ,  
 \qquad {\mathtt b} := [{\mathtt a}] + 2\, ,
 \qquad {\mathtt a} = 3 \tau_1 + 1\, ,
 \qquad  \tau_1 = 2 \tau  + 1 \, ,
  \qquad \chi = 3/ 2 \, ,\\
& \label{costanti nash moser}
\!\! \! {\mathtt a}_1 :=  {\rm max}\{12 \overline \sigma + 13, \, p \tau + 3 + \chi(\mu(\mathtt b) + 2 \overline \sigma) \},  \quad 
  \mathtt a_2 := \chi^{- 1} \mathtt a_1  - \mu(\mathtt b) - 2 \overline \sigma  \, , \\
&  \label{costanti nash moser 1}
\!\! \!  {\mathtt b}_1 := {\mathtt a}_1 + \mu({\mathtt b}) +  3 \overline \sigma + 4  +  \frac23 \mu_1 \,,  
\qquad  \mu_1 := 3( \mu({\mathtt b}) + 2\overline \sigma + 2 ) + 1 \, , \quad S := \sM + {\mathtt b}_1 \, ,  
\end{align}
where 
 $ \sigma_1 $ is  defined in Lemma \ref{stime tame campo ham per NM AI}, 
  $ \sigma_2 $  in Theorem \ref{thm:stima inverso approssimato},  
and $ {\mathtt a} $, $\mu(\mathtt b)$  in \eqref{alpha beta}.   
The number $ p $ is the exponent in \eqref{NnKn} 
and is requested to satisfy  
\be\label{cond-su-p}
p {\mathtt a} > (\chi - 1 ) {\mathtt a}_1 + \chi (\overline \sigma + 4) = \frac12 {\mathtt a}_1 + \frac32 (\overline \sigma + 4) \, . 
\ee
In view of the definition \eqref{costanti nash moser}
of $ {\mathtt a}_1 $, we can define $p := p(\tau, \Splus) $ as 
\be\label{choice:p}
p := \frac{12 \overline \sigma + 17 + \chi(\mu(\mathtt b) + 2 \overline \sigma)}{ \mathtt a} \, . 
\ee
We denote by 
$ \|  W \|_{s}^\Lipg := \max\{ \|  \iota\|_{s}^\Lipg, |  \zeta |^\Lipg \}  $ 
the norm of a function  
$$
W := ( \iota, \zeta ): {\mathtt \Omega} \to
\big( H^{s}_\vphi  \times H^{s}_\vphi \times H^{s}_\bot \big) \times \R^{\Splus}
 \, , 
\ \  \omega \mapsto W (\omega) = ( \iota(\omega), \zeta (\omega)) \, .   
$$
The following 
Nash-Moser Theorem can be proved in a by now standard way  
as in \cite{Berti-Montalto}, \cite{BBHM}. 

\begin{theorem}\label{iterazione-non-lineare} 
{\bf (Nash-Moser)} 
There exist $ 0 < \d_0 < 1$, $ C_* > 0 $ so that if
\begin{equation}\label{nash moser smallness condition}  
\e K_0^{\tau_2} < \d_0 , \quad \tau_2 := {\rm max}\{ p \overline{\tau} + 3, \ 4 \overline \sigma + 4 + {\mathtt a}_1  \} \, , 
\quad K_0 := \gamma^{- 1}, \quad \gamma:= \e^{\frak a}\,,\quad 0 < {\frak a} < \frac{1}{ \tau_2}\,,
\end{equation}
where $ \overline{\tau} := \overline{\tau}(\tau, \Splus)$ is  defined in Theorem \ref{iterazione riducibilita}, 
 then the following holds for all $n \in \N$: 
\begin{itemize}
\item[$({\cal P}1)_{n}$] 
Let $\tilde W_0 := (0, 0)$. For $n \ge 1,$
there exists a   $\Lipg$-function $\tilde W_n : 
\R^{\Splus} \to \mathtt E_{n -1} \times \R^{\Splus} $, 
$ \omega  \mapsto \tilde W_n (\omega) 
:=  (\tilde \iota_n, \tilde \zeta_n  ) $,   
satisfying 
\begin{equation}\label{ansatz induttivi nell'iterazione}
\| \tilde W_n \|_{\sM + \mu({\mathtt b}) + \overline \sigma}^\Lipg \lesssim  \e  \gamma^{- 2}\,. 
\end{equation}
Let $\tilde U_n := U_0 + \tilde W_n$ where $ U_0 := (\vphi,0,0, 0)$.
For $ n \geq 1 $, the difference $\tilde H_n := \tilde U_{n} - \tilde U_{n-1}$, 
,  satisfies
\begin{equation}  \label{Hn}
\|\tilde H_1 \|_{\sM + \mu({\mathtt b}) + \overline \sigma}^\Lipg \lesssim	  \e \gamma^{- 2} \,,
 \qquad \| \tilde H_{n} \|_{ \sM + \mu({\mathtt b}) + \overline \sigma}^\Lipg \lesssim   \e \gamma^{- 2} K_{n - 1}^{- \mathtt a_2} \,,\quad \text{ for } n \ge 2\, . 
\end{equation}
\item[$({\cal P}2)_{n}$] Let ${\cal G}_{0} := \Omega$ and define for 
$n \ge 1$,
\be\label{def:cal-Gn}
 {\cal G}_{n}  :=  {\cal G}_{n-1} \cap \ 
 {\bf \Omega}_{n  }^{ \gamma}({\tilde \io}_{n-1}) \, , 
\ee
where $  {\bf \Omega}_{n}^{ \gamma}({\tilde \io}_{n-1}) $ is defined in \eqref{Melnikov-invert}. 
Then for any $\omega \in {\cal G}_{n}$ 
\be\label{P2n}
\| {\cal F}_\omega (\tilde U_n) \|_{ \sM}^\Lipg  \leq C_* \e K_{n - 1}^{- {\mathtt a}_1} \, , \qquad
K_{-1} := 1.
\ee
\item[$({\cal P}3)_{n}$] \emph{(High norms)\quad } 
$ \| \tilde W_n \|_{ \sM + {\mathtt b}_1}^\Lipg 
\leq C_* \e   K_{n - 1}^{\mu_1},$ $\forall \omega  \in {\cal G}_{n}$.
\end{itemize}
\end{theorem}
\begin{proof}
We argue by induction.
To simplify notation, we write within this proof $\| \cdot \|$
for $\| \cdot \|^{\Lipg}$. 

\smallskip
\noindent
{\sc Step 1:} \emph{Proof of} $({\cal P}1, \mP 2, \mP 3)_0$.
Note that $({\cal P}1)_0$ and $({\cal P}3)_0$ are trivially satisfied
and hence it remains to verify \eqref{P2n} at $ n = 0 $.
By \eqref{operatorF}, 
\eqref{splitting ham forma normale + perturbazione}, \eqref{mu-kdv-om}, and
Lemma \ref{stime tame campo ham per NM AI},
there exists $C_* > 0$ large enough so that
$\|{\cal F}_\omega (U_0) \|_{ \sM}^\Lipg   \le  \e C_*$.

\smallskip
\noindent
{\sc Step 2:} \emph{Proof of the induction step.}
Assuming that $({\cal P}1, \mP 2, \mP 3)_n$ hold for some $n \geq 0$, 
we have to prove that $({\cal P}1, \mP 2, \mP3)_{n+1}$ hold.
We are going to define the approximation $ \tilde U_{n+1} $ by a modified Nash-Moser scheme.
To this aim, we prove the almost-approximate invertibility of the linearized operator 
\begin{equation}\label{def L_n}
L_n := L_n(\omega) := d_{\io,\zeta} {\cal F}_\omega( \tilde \io_n(\omega)) 
\end{equation} 
by applying Theorem \ref{thm:stima inverso approssimato} to ${ L}_n(\omega) $.
To prove that the inversion assumptions \eqref{inversion assumption}-\eqref{tame inverse} 
hold, we apply Theorem \ref{inversione parziale cal L omega} with $ \io = \tilde \io_n $.

By choosing $ \e $ small enough it follows by 
\eqref{nash moser smallness condition} 
that $N_0 = K_0^p = \g^{-p}= \e^{-p \mathtt a} $ 
satisfies the requirement of 
Theorem \ref{inversione parziale cal L omega} and that 
the smallness condition \eqref{ansatz riducibilita} holds.
Therefore Theorem \ref{inversione parziale cal L omega} applies, 
and we deduce that \eqref{inversion assumption}-\eqref{tame inverse} hold 
for all $ \omega \in {\bf \Omega}_{n + 1}^{\gamma}(\tilde \io_n)$, see \eqref{Melnikov-invert}.

Now we apply Theorem \ref{thm:stima inverso approssimato} to the linearized operator 
$ L_n(\omega) $  with $ {\Omega}_o = {\bf \Omega}_{n + 1}^{\gamma}(\tilde \io_n)$ and $S = \sM  + \mathtt b_1$, see \eqref{costanti nash moser 1}.
It implies the existence of
an almost-approximate inverse ${\bf T}_n  := { \bf T}_n (\omega, {\tilde \io}_n(\omega))$ 
satisfying
\begin{align}
\| {\bf T}_n g  \|_s 
& \lesssim_{\sM + \mathtt b_1} \gamma^{-2 } \big( \| g \|_{s + \overline \sigma} 
+ K_0^{\tau_1  p} \gamma^{- 1} \| \tilde \iota_n \|_{s + \mu({\mathtt b}) + \overline \sigma } \| g \|_{\sM + \overline \sigma }\big) \, , 
\quad \forall \sM  \leq s \leq \sM + \mathtt b_1 \,, 
\label{stima Tn} 
\end{align}
where we used that $\overline \sigma  \geq \sigma_2$ 
(cf. \eqref{costanti nash moser 2}), 
$\sigma_2$ is the loss of regularity constant appearing in the estimate \eqref{stima inverso approssimato 1}, and $N_0 = K_0^p$. Furthermore, by \eqref{nash moser smallness condition}, \eqref{ansatz induttivi nell'iterazione} one obtains that 
\begin{equation}\label{stima i tilde K0 bassa}
K_0^{\tau_1  p} \gamma^{- 1} \| \tilde W_n \|_{\sM + \mu({\mathtt b}) + \overline \sigma } \leq 1 \, ,
\end{equation}
therefore \eqref{stima Tn} specialized for $s = \sM$ becomes 
\begin{align}
\| {\bf T}_n  g \|_{\sM} 
& 
\lesssim_{ \mathtt b_1} \gamma^{-2 } \| g \|_{\sM + \overline \sigma}\,. 
\label{stima Tn norma bassa}
\end{align}
For all $\omega \in {\cal G}_{n + 1} = {\cal G}_n \cap {\bf \Lambda}_{n + 1}^{\gamma}(\tilde \io_n)$ (see \eqref{def:cal-Gn}), we define 
\begin{equation}\label{soluzioni approssimate}
U_{n + 1} := \tilde U_n + H_{n + 1}\,, \qquad 
H_{n + 1} :=
( \widehat \iota_{n+1}, \widehat \zeta_{n+1}) :=  - {\bf \Pi}_{n } {\bf T}_n \Pi_{n } {\cal F}_\omega(\tilde U_n) 
\in \mathtt E_n \times \R^{\Splus}  
\end{equation}
where  $ {\bf \Pi}_n $ is defined by (see \eqref{truncation NM})
\be\label{proiettore modificato}
 {\bf \Pi}_n ({\iota}, \zeta) := (\Pi_n \iota, \zeta)\,,
 \quad {\bf \Pi}_n^\bot (\iota, \zeta) := (\Pi_n^\bot \iota, 0)\,,\quad \forall (\iota, \zeta)\,.
\ee
We show that the iterative scheme in \eqref{soluzioni approssimate} is rapidly converging. 
We write  
\begin{equation}\label{def:Qn} 
{\cal F}_\omega(U_{n + 1}) =  {\cal F}_\omega(\tilde U_n) + L_n H_{n + 1} + Q_n 
\end{equation} 
where $ L_n := d_{\io,\zeta} {\cal F}_\omega (\tilde U_n)$ 
and $Q_n$ is defined by \eqref{def:Qn}.
Then, by the definition of $ H_{n+1} $ in \eqref{soluzioni approssimate}, 
we have (recall also \eqref{proiettore modificato})
\begin{align}
{\cal F}_\omega(U_{n + 1}) & = 
 {\cal F}_\omega (\tilde U_n) - L_n {\bf \Pi}_{n } {\bf T}_n \Pi_{n } {\cal F}_\omega (\tilde U_n) + Q_n \nonumber \\
 & = 
 {\cal F}_\omega (\tilde U_n) - L_n  {\bf T}_n \Pi_{n } {\cal F}_\omega (\tilde U_n) + L_n  {\bf \Pi}_n^\bot  {\bf T}_n \Pi_{n } {\cal F}_\omega (\tilde U_n)
 + Q_n \nonumber\\ 
 & = \Pi_{n }^\bot {\cal F}_\omega (\tilde U_n) + R_n + Q_n + P_n  
\label{relazione algebrica induttiva}
\end{align}
where 
\begin{equation}\label{Rn Q tilde n}
R_n := L_n  {\bf \Pi}_n^\bot  {\bf T}_n \Pi_{n }{\cal F}_\omega ( \tilde U_n) \,,
\qquad 
P_n := -  ( L_n {\bf T}_n - {\rm Id}) \Pi_{n } {\cal F}_\omega ( \tilde U_n)\,.
\end{equation}
We first note that for any $ \omega \in \Omega$,  $s \geq \sM $
one has by the triangular inequality,  
\eqref{operatorF}, Lemma \ref{stime tame campo ham per NM AI}, and
\eqref{costanti nash moser 2}, \eqref{ansatz induttivi nell'iterazione} 
\begin{equation}\label{F tilde Un W tilde n}
\| {\cal F}_\omega(\tilde U_n)\|_s 
\lesssim_s \|{\cal F}_\omega (U_0)\|_s + \| {\cal F}_\omega (\tilde U_n) - {\cal F}_\omega (U_0)\|_s 
\lesssim_s \e + \| \tilde W_n\|_{s + \overline \sigma}  
\end{equation}
and, by \eqref{ansatz induttivi nell'iterazione}, \eqref{nash moser smallness condition}, 
\eqref{P2n}
\begin{equation}\label{gamma - 1 F tilde Un}
K_0^{\tau_1 p} \gamma^{- 1} \| {\cal F}_\omega(\tilde U_n)\|_{\sM} \leq 1\, . 
\end{equation}
We now prove the following inductive estimates of Nash-Moser type.

\begin{lemma} \label{lemma:2017.0504.1}
For all $\omega \in {\cal G}_{n + 1}$ we have, setting 
$ \mu_2 :=  \mu({\mathtt b}) +  3 \overline \sigma + 3 $,  
\begin{align}
&  \| {\cal F}_\omega (U_{n + 1})\|_{\sM}  
\lesssim_{\sM + {\mathtt b}_1} 
K_n^{\mu_2 - {\mathtt b}_1} (\e + \| \tilde W_n \|_{\sM + \mathtt b_1}) 
+ K_n^{4 \overline \sigma + 4} \| {\cal F}_\omega (\tilde U_n)\|_{\sM}^2 
+ \e  K_{n - 1}^{- p {\mathtt a} } K_n^{\overline \sigma + 4} \| {\cal F}_\omega (\tilde U_n) \|_{\sM} 
\label{F(U n+1) norma bassa} \\ 
& \label{U n+1 alta}
\| W_1 \|_{\sM+ {\mathtt b}_1} 
\lesssim_{\sM+ {\mathtt b}_1}  K_0^2 \e  \, , \qquad 
\| W_{n + 1}\|_{\sM + {\mathtt b}_1} \lesssim_{\sM + {\mathtt b}_1} 
K_n^{\mu({\mathtt b}) + 2\overline \sigma + 2}  (\e  +  \| \tilde W_n\|_{\sM + {\mathtt b}_1}  )\, , \ n \geq 1 \, . 
\end{align}
\end{lemma}

\begin{proof} 
We first estimate $ H_{n +1} $ defined in  \eqref{soluzioni approssimate}.

\smallskip

\noindent
{\bf Estimates of $ H_{n+1} $.}
 By \eqref{soluzioni approssimate} and \eqref{p3-proi}, 
\eqref{stima Tn}, \eqref{ansatz induttivi nell'iterazione},  we get 
\begin{align}
\|  H_{n + 1} \|_{\sM + {\mathtt b}_1} 
& \lesssim_{\sM + {\mathtt b}_1}  \gamma^{- 2} 
\big( K_n^{\overline \sigma } \|{\cal F}_\omega (\tilde U_n) \|_{\sM + {\mathtt b}_1} + 
K_n^{\mu({\mathtt b}) + 2 \overline \sigma } K_0^{\tau_1 p} \gamma^{- 1} \|\tilde \iota_n \|_{\sM + {\mathtt b}_1}\| {\cal F}_\omega (\tilde U_n)\|_{\sM }  \big)
\nonumber \\
& \stackrel{\eqref{F tilde Un W tilde n}, \eqref{gamma - 1 F tilde Un}}{\lesssim_{\sM + {\mathtt b}_1}} 
K_n^{\mu(\mathtt b) + 2 \overline \sigma } \g^{-2} \big( \e  +  \| \tilde W_n \|_{\sM + \mathtt b_1} \big) \\
& \stackrel{\gamma^{- 1} = K_0 \leq K_n}{\lesssim_{\sM + \mathtt b_1}} K_n^{\mu(\mathtt b) + 2 \overline \sigma + 2} \big( \e  +  \| \tilde W_n \|_{\sM + \mathtt b_1} \big) \, ,  \label{H n+1 alta} 
\\
\label{H n+1 bassa}
\|  H_{n + 1}\|_{\sM} 
& \stackrel{\eqref{stima Tn norma bassa}}
{\lesssim_{\sM + \mathtt b_1}} \gamma^{-2}K_{n}^{\overline \sigma} \| {\cal F}_\omega (\tilde U_n)\|_{\sM} \, .
\end{align}
Next we  estimate the terms $ Q_n $ in \eqref{def:Qn} and $ P_n , R_n $ in \eqref{Rn Q tilde n} in $ \| \ \|_{\sM} $ norm. 

\smallskip

\noindent
{\bf Estimate of $ Q_n $.}
By \eqref{ansatz induttivi nell'iterazione}, \eqref{soluzioni approssimate}, 
\eqref{p3-proi},
\eqref{H n+1 bassa}, \eqref{P2n},
and since $ \chi 2 \overline \s - \mathtt{a}_1 \leq 0$  
(see \eqref{costanti nash moser}),  
we deduce that 
$\| \tilde W_n + t H_{n+1} \|_{\sM +  \overline \sigma} \lesssim \e \gamma^{- 2} 
K_0^{2 \ov{\sigma} } $ for all $t \in [0,1]$. Since
$ \gamma^{- 1} = K_0 $, by \eqref{nash moser smallness condition}
we can apply Lemma \ref{stime tame campo ham per NM AI} and by Taylor's formula, using
\eqref{def:Qn}, 
\eqref{operatorF}, 
\eqref{H n+1 bassa},
\eqref{p3-proi},
and $\gamma^{- 1} = K_0 \leq K_n$, we get
\begin{equation} \label{Qn norma bassa}
\| Q_n \|_{\sM} 
\lesssim_{\sM + \mathtt b_1}  \| H_{n+1} \|_{\sM+\overline \sigma}^2 
\lesssim_{\sM + \mathtt b_1} K_n^{4 \overline \sigma + 4}  \| {\cal F}_\omega (\tilde U_n) \|_{\sM}^2\, . 
\end{equation}
{\bf Estimate of $ P_n $.} 
By \eqref{splitting per approximate inverse}, 
$L_n {\bf T}_n - {\rm Id} = {\cal P}({\tilde \io}_n ) 
+ {\cal P}_\omega( {\tilde \io}_n )
+ {\cal P}_\omega^\bot ( {\tilde \io}_n )$. 
Accordingly, we decompose $ P_n $ in \eqref{Rn Q tilde n} as
$P_n = - P_n^{(1)} - P_{n , \omega} - P_{n, \omega}^\bot$,
where 
\[ 
P_n^{(1)} := \Pi_n {\cal P}({\tilde \io}_n ) \Pi_n {\cal F}_\omega (\tilde U_n), \qquad  
P_{n, \omega} := \Pi_n {\cal P}_\omega( {\tilde \io}_n ) \Pi_n {\cal F}_\omega (\tilde U_n), \qquad 
P_{n, \omega}^\bot := \Pi_n {\cal P}_\omega^\bot({\tilde \io}_n ) \Pi_n {\cal F}_\omega (\tilde U_n).
\] 
By \eqref{p3-proi}, 
\begin{equation} \label{2017.0404.1}
\begin{aligned}
\| {\cal F}_\omega (\tilde U_n) \|_{\sM+ \overline \sigma} 
& \leq \|\Pi_n {\cal F}_\omega (\tilde U_n) \|_{\sM  + \overline \sigma} 
+ \|\Pi_n^\bot {\cal F}_\omega (\tilde U_n) \|_{\sM  + \overline \sigma}  \\
& \leq K_n^{\overline \sigma} ( \| {\cal F}_\omega (\tilde U_n)\|_{\sM} 
+ K_n^{- \mathtt b_1} \|{\cal F}_\omega (\tilde U_n) \|_{\sM + \mathtt b_1} ).
\end{aligned}
\end{equation}
By 
\eqref{stima inverso approssimato 2},
\eqref{stima i tilde K0 bassa},
\eqref{2017.0404.1}, 
and using that \eqref{F tilde Un W tilde n}, \eqref{gamma - 1 F tilde Un}, $\gamma^{- 1} = K_0 \leq K_n$
we obtain
\begin{align} 
\| P_n^{(1)} \|_{\sM}  
& \lesssim_{\sM + \mathtt b_1}
\g^{-3} K_n^{2 \overline \sigma } \| {\cal F}_\omega (\tilde U_n)\|_{\sM} 
( \| {\cal F}_\omega(\tilde U_n)\|_{\sM} 
+ K_n^{- {\mathtt b}_1} \| {\cal F}_\omega (\tilde U_n) \|_{\sM + {\mathtt b}_1} ) 
\notag \\ & 
\lesssim_{\sM + \mathtt b_1} 
 K_n^{2 \overline \sigma + 3 } \| {\cal F}_\omega (\tilde U_n)\|_{\sM} 
( \| {\cal F}_\omega (\tilde U_n)\|_{\sM} + K_n^{\overline \sigma - {\mathtt b}_1}
(\e + \| \tilde W_n\|_{\sM + \mathtt b_1} ) ) \nonumber \\
& \lesssim_{s_0 + \mathtt b_1}   K_n^{2 \overline \sigma + 3 } \| {\cal F}_\omega (\tilde U_n)\|_{\sM}^2 +  K_n^{3 \overline \sigma + 3 - \mathtt b_1 } (\e + \| \tilde W_n\|_{\sM + \mathtt b_1} )\,. 
\label{Q n 1 bassa}
\end{align}
By 
\eqref{stima cal G omega}, 
\eqref{stima i tilde K0 bassa}, \eqref{ansatz induttivi nell'iterazione},
\eqref{p3-proi},
we have
\begin{equation} \label{Q n omega bassa}
\| P_{n, \omega} \|_{\sM}   \lesssim_{\sM + \mathtt b_1} 
\e \gamma^{- 4 } N_{n - 1}^{- {\mathtt a}} K_n^{\overline \sigma } \| {\cal F}_\omega (\tilde U_n) \|_{\sM} \stackrel{\gamma^{- 1} = K_0 \leq K_n }{\lesssim_{s_0 + \mathtt b_1}} 
\e  N_{n - 1}^{- {\mathtt a}} K_n^{\overline \sigma + 4} \| {\cal F}_\omega (\tilde U_n) \|_{\sM}\,,
\end{equation}
where $\mathtt{a}$ is  in \eqref{costanti nash moser 2}.
By 
\eqref{stima cal G omega bot bassa}, 
\eqref{p3-proi},
\eqref{costanti nash moser 1},
\eqref{P2n}, \eqref{gamma - 1 F tilde Un}
and then using
\eqref{F tilde Un W tilde n}, 
 $\gamma^{- 1} = K_0 \leq K_n$,
we get
\begin{align}
\| P_{n, \omega}^\bot\|_{\sM} 
& \lesssim_{\sM + \mathtt b_1}  K_{n}^{\mu({\mathtt b}) + 2 \overline \s - {\mathtt b}_1} 
\g^{-2 } ( \| {\cal F}_\omega (\tilde U_n) \|_{\sM + {\mathtt b}_1} + \e \| \tilde W_n \|_{\sM + {\mathtt b}_1}) \nonumber \\ 
& \lesssim_{\sM + \mathtt b_1}  K_{n}^{\mu({\mathtt b}) + 3 \overline \sigma + 2 - {\mathtt b}_1}
 (\e + \| \tilde W_n \|_{\sM + \mathtt b_1 }). 
\label{Q n omega bot bassa} 
\end{align}

\noindent
{\bf Estimate of $ R_n $.}  
By the definition \eqref{def L_n} of $L_n$ one has that
for any $\widehat U = (\widehat \io, \widehat \zeta)$,
$L_n \widehat U $ is given by
\begin{align}
L_n \widehat U   & = \omega \cdot \partial_\vphi \widehat \io - 
d_\io X_{{\cal H}_\e}\big( (\vphi, 0, 0) + \tilde\io_n \big)[\widehat \io] - (0, \widehat \zeta, 0) \nonumber\\
& \stackrel{\eqref{splitting ham forma normale + perturbazione}}{=}\omega \cdot \partial_\vphi \widehat \io - d_\io X_{{\cal N}}\big( (\vphi, 0, 0) + \tilde\io_n \big)[\widehat \io] - 
 d_\io X_{{\cal P}_\e}\big( (\vphi, 0, 0) + \tilde\io_n \big)[\widehat \io] - (0, \widehat \zeta, 0) 
 \end{align}
 where we recall that $d_\io X_{{\cal N}}\big( (\vphi, 0, 0) + \tilde\io_n \big)[\widehat \io] = \big( \Omega_{\mathbb S_+}^{kdv}(\mu)[\widehat y], \, 0\, , \Omega^{kdv}(\mu, D)[\widehat w] \big)$.  
 By the estimate of $d_\io X_{{\cal P}_\e}$ 
 of Lemma \ref{stime tame campo ham per NM AI}, one then obtains
  $\| L_n \widehat U \|_{\sM} \lesssim \| \widehat U\|_{\sM + \overline \sigma}$. 
Using  
\eqref{Rn Q tilde n},
\eqref{stima Tn},
\eqref{ansatz induttivi nell'iterazione}, \eqref{p3-proi}
and then 
\eqref{stima i tilde K0 bassa}, \eqref{F tilde Un W tilde n},
\eqref{gamma - 1 F tilde Un}, $\gamma^{- 1} = K_0 \leq K_n$,
we get 
\begin{align} 
\| R_n\|_{\sM} 
& \lesssim_{\sM + \mathtt b_1} K_n^{ \mu({\mathtt b}) + 3 \overline  \s  + 2  - {\mathtt b}_1} 
(\e + \| \tilde W_n\|_{\sM  + \mathtt b_1} ). 
\label{stima Rn norma bassa}
\end{align}
{\bf Estimate of $ {\cal F}_\omega(U_{n + 1}) $.} 
By 
\eqref{relazione algebrica induttiva},
\eqref{p3-proi},
\eqref{F tilde Un W tilde n},
\eqref{Qn norma bassa}, 
\eqref{Q n 1 bassa}-\eqref{Q n omega bot bassa}, 
\eqref{stima Rn norma bassa},
\eqref{ansatz induttivi nell'iterazione}, 
we get \eqref{F(U n+1) norma bassa}. 
By \eqref{soluzioni approssimate} and \eqref{stima Tn} 
we now deduce the bound \eqref{U n+1 alta} 
for $ W_1 := H_1 $. Indeed 
$$ 
\| W_1 \|_{\sM + {\mathtt b}_1} = \| H_1 \|_{\sM + {\mathtt b}_1}   \lesssim_{\sM + {\mathtt b}_1} 
\g^{-2} \| {\cal F}_\omega(U_0)\|_{\sM + {\mathtt b}_1 + \overline \sigma} 
{\lesssim_{\sM + {\mathtt b}_1}  } \,\,\,
\e \gamma^{- 2} \, \,\,  \stackrel{\gamma^{- 1} = K_0}{\lesssim} K_0^2 \e \, .
$$ 
Estimate \eqref{U n+1 alta} for $ W_{n+1} := \tilde W_n + H_{n+1} $, $ n \geq 1 $,  
follows by \eqref{H n+1 alta}. 
\end{proof}

By Lemma \ref{lemma:2017.0504.1} we get the following lemma, where for clarity 
we  write $ \| \cdot \|^\Lipg_s $ instead of $\| \cdot	 \|_s$ as above. 

\begin{lemma}\label{lemma:quadra}
For any  $ \omega \in {\cal G}_{n + 1}$
\begin{align}
& \qquad \quad \| {\cal F}_\omega(U_{n + 1}) \|_{\sM }^\Lipg  
\leq C_* \e K_n^{- \mathtt a_1} \label{stima F u n + 1 induttiva} \, , \qquad 
\| W_{n + 1} \|_{\sM  + \mathtt b_1}^\Lipg 
\leq C_* K_n^{\mu_1} \e   \, ,  \\ 
& \label{stima H n+1 lemma}
\| H_1 \|_{\sM  + \mu(\mathtt b) + \overline \sigma}^\Lipg \lesssim  \e \g^{-2} \, , \quad
\| H_{n + 1}\|_{\sM  + \mu(\mathtt b) + \overline \sigma}^\Lipg 
\lesssim   \e \gamma^{- 2} K_n^{\mu(\mathtt b) + 2 \overline \sigma } K_{n - 1}^{- \mathtt a_1} \, ,
\  	\ n \geq 1 \,.
\end{align}
\end{lemma}

\begin{proof}
First note that, by  \eqref{def:cal-Gn},  if  
$ \omega \in  {\cal G}_{n + 1}$, then 
$ \omega \in  {\cal G}_n$ and so  \eqref{P2n} and the inequality in $({\cal P}3)_{n}$ holds. 
Then the first inequality in \eqref{stima F u n + 1 induttiva} follows by 
\eqref{F(U n+1) norma bassa},
$ ({\cal P}2)_{n}$, $({\cal P}3)_{n}$,  $\gamma^{- 1} = K_0 \leq K_n $,
and by \eqref{costanti nash moser}, \eqref{costanti nash moser 1}, \eqref{cond-su-p}-\eqref{choice:p}. 
For $ n = 0 $ we use also  \eqref{nash moser smallness condition}. 

The second inequality in \eqref{stima F u n + 1 induttiva} for $n=0$ follows directly from 
the bound for $W_1$ in \eqref{U n+1 alta}, since $\mu_1 \geq 2$,
see \eqref{costanti nash moser 1} and $ C_* > 0$ large enough 
(i.e., $\e$ small enough); 
the second inequality in \eqref{stima F u n + 1 induttiva} for $n \geq 1$ is proved inductively 
by taking \eqref{U n+1 alta}, $({\cal P}3)_{n} $, and the choice of $ \mu_1 $ in \eqref{costanti nash moser 1} into account 
and by choosing $ K_0 $ large enough. 

Since  $ H_1  = W_1 $, the first inequality in \eqref{stima H n+1 lemma} follows since $\| H_1 \|_{\sM + \mu(\mathtt b) + \overline \sigma} \lesssim \gamma^{- 2} \|{\cal F}_\omega(U_0) \|_{\sM  + \mu(\mathtt b) + 2 \overline \sigma} \lesssim \e \gamma^{- 2}$.
If $ n \geq 1 $, estimate \eqref{stima H n+1 lemma} follows 
by  \eqref{p3-proi},
\eqref{H n+1 bassa} and \eqref{P2n}.
\end{proof}

Denote by $ \tilde H_{n + 1}$ a $ \Lipg$-extension 
of $ (H_{n + 1})_{|{\cal G}_{n+1}} $ to the whole set $\Omega$
of parameters, provided by the Kirzbraun theorem.   
Then $ \tilde H_{n + 1}$ satisfies the same bound as $ H_{n+1} $ in \eqref{stima H n+1 lemma} and therefore, by  the definition 
of $ {\mathtt a}_2 $ in \eqref{costanti nash moser}, the estimate \eqref{Hn} holds at $ n + 1 $.

Finally we define the functions
$$
\tilde W_{n+1} := \tilde W_{n} + \tilde H_{n + 1} \, , \quad 
\tilde U_{n + 1} := \tilde U_n + \tilde H_{n + 1} = U_0 + \tilde W_n + \tilde H_{n + 1} = U_0 + \tilde W_{n + 1}\,,
$$
which are defined for all $\omega \in \Omega $. Note that 
for any $ \omega \in {\cal G}_{n + 1} $,
$ \tilde W_{n + 1} = W_{n + 1} $, 
$ \tilde U_{n + 1} = U_{n + 1} $.  
Therefore  $({\cal P}2)_{n + 1}$, $({\cal P}3)_{n + 1}$ are proved by Lemma \ref{lemma:quadra}.
Moreover by \eqref{Hn}, which at this point has been proved up to the step $n + 1 $, 
we have 
$$
\| \tilde W_{n + 1} \|_{\sM  + \mu({\mathtt b}) + \overline \sigma}^\Lipg 
\leq {\mathop \sum}_{k = 1}^{n + 1} \| \tilde H_k \|_{\sM  + \mu({\mathtt b}) + \overline \sigma}^\Lipg 
\leq C_*  \e  \gamma^{-2}
$$
and thus \eqref{ansatz induttivi nell'iterazione} holds also 
at the step $n + 1$. This completes the proof of Theorem \ref{iterazione-non-lineare}.
\end{proof}

We now deduce Theorem \ref{main theorem}.  
Let $ \gamma = \e^{\frak a}  $ with $ \frak a \in (0, {\frak a}_0) $ and $ {\frak a}_0 := 1 / \tau_2 $ where $\tau_2$ is defined in \eqref{nash moser smallness condition}.
Then the smallness condition \eqref{nash moser smallness condition} holds for $ 0 < \e < \e_0 $ small enough and Theorem \ref{iterazione-non-lineare} applies.   
Passing to the limit for $n \to \infty$ we deduce the existence of a function $U_\infty(\omega) = (\breve \io_\infty(\omega), \zeta_\infty(\omega))$, $\omega \in \Omega$, such that ${\cal F}_\omega(U_\infty(\omega)) = 0$ for any $\omega$ in the set
\be\label{defGinfty}
\bigcap_{n \geq 0} {\cal G}_n = 
\mG_0 \cap \bigcap_{n \geq 1} 
{\bf \Omega}_{n + 1}^{\g}(\tilde \io_{n-1}) 
\stackrel{\eqref{Melnikov-invert}}{=} 
\mG_0 \cap \Big[ \bigcap_{n \geq 1}  \tLm_{n}^{\gamma}(\tilde \io_{n-1}) \Big] \cap 
 \Big[ \bigcap_{n \geq 1}   \mathtt \Omega_{n}^{\gamma}(\tilde \io_{n-1}) \Big]\, .
\ee
Moreover
\begin{equation}\label{U infty - U n}
\|  U_\infty -  U_0 \|_{\sM + \mu(\mathtt b) + \overline \sigma}^\Lipg \lesssim \e \gamma^{- 2} \,, \quad \| U_\infty - {\tilde U}_n \|_{\sM   + \mu({\mathtt b}) + \overline \sigma}^\Lipg \lesssim  \e \gamma^{-2} K_{n }^{- \mathtt a_2}\,, \ \  n \geq 1 \, .
\end{equation}
Formula \eqref{zeta Z inequality} implies that $\zeta_\infty(\omega ) = 0$ for $\omega$ belonging to the set \eqref{defGinfty}, and therefore 
$ \breve \io_\omega := \breve \io_\infty (\omega) $ is an invariant torus for the Hamiltonian vector field $X_{{\cal H}_\e}$ filled by quasi-periodic solutions with frequency $\omega$.
 It remains only to prove the measure estimate 
 \eqref{measure estimate Omega in Theorem 4.1}. 

\subsection{Measure estimates}\label{proof-41}

Arguing as in \cite{Berti-Montalto} one proves the following two lemmata.
 
\begin{lemma} \label{lemma inclusione cantor riccardo 1}
The set 
\begin{equation}\label{cantor finale 1 riccardo}
{\cal  G}_\infty 
:= \mG_0 \cap \Big[ \bigcap_{n \geq 1} \tLm_n^{2 \gamma}( \io_\infty) \Big] 
\cap \Big[ \bigcap_{n \geq 1} \mathtt \Omega_n^{2 \gamma}(\io_\infty)  \Big]
\end{equation} 
is contained in ${\cal G}_n$ for any $n \ge 0$, and hence  $ {\cal  G}_\infty  \subseteq  \bigcap_{n \geq 0 } {\cal G}_n $. 
\end{lemma}

For any $j \in \mathbb S^\bot$, the sequence 
$\widetilde \mu_j^n : \Omega \to \R$, $n \ge 0$, in 
Theorem \ref{iterazione riducibilita}-${\bf(S2)}_n $ 
is a Cauchy sequence with respect to the norm $ | \cdot |^\Lipg$.
We denote the limit by $\mu_j^\infty $, 
\begin{equation}\label{autovalori finali riccardo}
\mu_j^\infty := \lim_{n \to \infty} \widetilde \mu_j^n(\io_\infty) \, ,  \quad j \in \mathbb S^\bot \, . 
\end{equation}
By Theorem \ref{iterazione riducibilita}
one has for any $j \in \mathbb S^\bot$,
\begin{equation}\label{distanza-rnrinfty}
\mu_{-j}^\infty = - \mu_j^\infty\, , \qquad
| \mu_j^\infty - \widetilde \mu_j^n(\io_\infty)|^\Lipg \lesssim  
 \e \gamma^{- 2}  N_{n - 1}^{- {\mathtt a}} \, ,  \ n \geq 0 \, .
 \end{equation}
\begin{lemma} \label{lemma inclusione cantor riccardo 2}
The {\it set}  
\begin{align}\label{Cantor set infinito riccardo}
\Omega_\infty^\gamma & := \Big\{ \omega \in \mathtt{DC}(4 \gamma, \tau) :  |\omega \cdot \ell + \mu_j^\infty - \mu_{j'}^\infty| \geq \frac{4 \gamma |j^3 - j'^3|}{\langle \ell \rangle^\tau}, \    \forall (\ell, j, j') \in \Z^{\mathbb S_+} \times  \Sbot \times  \Sbot ,\,  
\nonumber \\
& \qquad \qquad \quad 
\qquad \qquad |\omega \cdot \ell + \mu_j^\infty | \geq \frac{4 \gamma |j|^3}{\langle \ell \rangle^\tau }, \ \forall (\ell, j) \in \Z^{\mathbb S_+} \times \mathbb S^\bot    \Big\}
\end{align}
is contained in ${\cal  G}_\infty$,  $ \Omega_\infty^\gamma \subseteq {\cal  G}_\infty $, 
where $\mG_\infty$ is defined in \eqref{cantor finale 1 riccardo}. 
\end{lemma}

In view of Lemma \ref{lemma inclusione cantor riccardo 1} and \ref{lemma inclusione cantor riccardo 2}, it suffices to estimate the Lebesgue measure
$|\Omega \setminus \Omega_\infty^\gamma|$ of  
$\Omega \setminus \Omega_\infty^\gamma$. 

\begin{proposition}{\bf (Measure estimates)}\label{stima Omega - Omega infty} 
Let $\tau > |\mathbb S_+| + 2$. Then there is $ {\mathfrak a} \in (0,1)$ so that for $ \e \g^{-3} $ sufficiently small, one has
$ | \Omega \setminus \Omega_\infty^\gamma | \lesssim  \gamma^{\mathfrak a} $.
\end{proposition}

The remaining part of this section is devoted to prove Proposition \ref{stima Omega - Omega infty}. 
By \eqref{Cantor set infinito riccardo}, we have 
\begin{equation}\label{complementare Omega infty}
\Omega \setminus \Omega_\infty^\gamma = 
\Omega \setminus \mathtt{DC}(4 \gamma, \tau)  \,\, \cup \,
\bigcup_{\begin{subarray}{c}
(\ell, j, j') \in \Z^{\mathbb S_+} \times {\mathbb S}^\bot \times {\mathbb S}^\bot 
 (\ell, j, j') \neq (0, j, j) \end{subarray}}  {\cal R}_{\ell, j, j'} 
\,\, \cup \,
 \bigcup_{(\ell, j) \in \Z^{\mathbb S_+} \times {\mathbb S}^\bot} {\cal Q}_{\ell, j}
\end{equation}
where ${\cal R}_{\ell,j,j'}$, ${\cal Q}_{\ell, j}$
denote the 'resonant' sets  
\begin{align}\label{def risonante 2}
{\cal R}_{\ell,j,j'} & := \Big\{   \omega \in \mathtt{DC}(4 \gamma, \tau) \ : \ | \om \cdot \ell + \mu_j^\infty - \mu_{j'}^\infty | < \frac{4\gamma |j^3 - j'^3|}{
\langle \ell \rangle^\tau}  \Big\}  \, ,\\
\label{risonante 1}
{\cal Q}_{\ell, j} & := \Big\{ \omega \in \mathtt{DC}(4 \gamma, \tau) : |\omega \cdot \ell + \mu_j^\infty|  < \frac{4 \gamma |j|^3}{\langle \ell \rangle^\tau} \Big\}\,. 
\end{align}
Note that $ {\cal R}_{\ell, j, j } = \emptyset $. 
Furthermore, it is well known that $|\Omega \setminus \mathtt{DC}(4 \gamma, \tau)| \lesssim \gamma$.
In order to prove Proposition \ref{stima Omega - Omega infty} we shall use the following asymptotic properties of  $ \mu_j^\infty(\omega)$. 
For any $ \omega $ in $\mathtt{DC}(4\gamma, \tau) $,
we have   $\widetilde \mu_j^{0} (\io_\infty) = \mu_j^{0} (\io_\infty) $  and 
we write  $ \mu_j^\infty(\omega) = \mu_j^0 (\io_\infty)  + r_j^\infty(\omega)   $,
where by \eqref{op-diago0}, $ m_3^\infty := m_3 ( \io_\infty ) $, 
$ m_1^\infty := m_1 ( \io_\infty ) $,  
$$
\mu_j^0(\io_\infty) = 
m_3^\infty (\omega) (2 \pi j)^3 - m_1^\infty (\omega ) 2 \pi j - q_j (\omega) \, .
$$ 
On $\mathtt{DC}(4 \gamma, \tau)$, the following estimates hold
\begin{equation}\label{exp1s1}
\begin{aligned}
&  |m_3^\infty  + 1|^\Lipg \stackrel{\eqref{m3Lip}}{\lesssim } \e\, , \qquad
|m_1^\infty |^\Lipg \stackrel{\eqref{stima m1}}{\lesssim} \e \gamma^{- 2},
\\
& \sup_{j \in \Sbot} |j| |q_j|^{\rm sup}, \sup_{j \in \Sbot} |j| |q_j|^{\rm lip} \stackrel{\eqref{stima qj omega}}{\lesssim } 1, 
\qquad |r_j^\infty|^\Lipg 
\stackrel{\eqref{distanza-rnrinfty}}{\lesssim } \e \gamma^{- 2} \, . 
\end{aligned}
\end{equation}
From the latter estimates one infers the following standard lemma
see \cite[Lemma 5.3]{BBM-Airy}). 
\begin{lemma}\label{lem:mes1}
$(i)$ 
If  $ {\cal R}_{\ell,j,j'} \neq \emptyset $, 
then  $ |j^3 - j'^3 | \leq C \langle \ell \rangle $ for some  
$ C > 0 $. In particular one has
 $ j^2 + j'^2 \leq C \langle \ell \rangle $. 
 
 \noindent
 $(ii)$ If ${\cal Q}_{\ell,  j} \neq \emptyset$, then $|j|^3 \leq C \langle \ell \rangle$ for some  $C > 0$. 
\end{lemma}
Lemma \ref{lem:mes1} can be used to estimate 
$|{\cal R}_{\ell,j,j'} |$ and $|{\cal Q}_{\ell,j}|$ 
for $|\ell|$ sufficiently large.
\begin{lemma}\label{lem:mes2}
$(i)$ 
If $ {\cal R}_{\ell,j,j'} \neq \emptyset $, then there exists 
$ C_1  > 0 $ with the following property: 
if $ |\ell| \geq C_1 $, then $|{\cal R}_{\ell,j,j'} | \lesssim  
\gamma |j^3 - j'^3| \langle \ell \rangle^{- (\tau + 1)}$.  

\noindent
$(ii)$ If $ {\cal Q}_{\ell,j} \neq \emptyset $, 
then there exists $ C_1  > 0 $ with the following property:  
if $ |\ell| \geq C_1 $, then 
$|{\cal Q}_{\ell,j} | \lesssim \gamma |j|^3 \langle \ell \rangle^{- (\tau+1)}$.  
\end{lemma}
\begin{proof}
We only prove item $(i)$ since item $(ii)$ can be proved in a similar way. Assume that ${\cal R}_{\ell, j, j'} \neq \emptyset$.
Let  $ \bar \omega $ such that 
$\bar \omega \cdot \ell = 0$ and introduce the real valued function 
$s \mapsto \phi_{\ell,j,k} (s)$, 
$$
\phi_{\ell,j,j'} (s) :=  
f_{\ell,j,j'} \big( \bar \om + s \frac{\ell}{|\ell|} \big) \, , 
\qquad f_{\ell,j,j'} (\om) 
:= \om \cdot \ell + \mu_j^\infty(\omega) - \mu_{j'}^\infty(\omega) \, . 
$$
Using that by Lemma \ref{lem:mes1},   $|j^3 - j'^3| \leq C \langle \ell \rangle$, one infers from \eqref{exp1s1} that, for
$\e \gamma^{- 2}$ small enough and  $|\ell| \geq C_1$
with $C_1$ large enough, 
$ |\phi_{\ell,j,j'} (s_2) - \phi_{\ell,j,j'} (s_1)| \geq \frac{|\ell|}{2} |s_2 - s_1| $. 
Since $\mathtt{DC}(4 \gamma, \tau)$ is bounded 
one sees by standard arguments that
$$
\big| \big\{ s \in \R \, : \, 
\bar \om + s \frac{\ell}{|\ell|} \in {\cal R}_{\ell, j, j'} \big\} \big|
\lesssim \gamma |j^3 - j'^3| \langle \ell \rangle^{- (\tau + 1)}\, .
$$
The claimed estimate then follows by applying Fubini's theorem. 
\end{proof}

It remains to estimate the Lebesgue measure of the resonant sets 
${\cal R}_{\ell, j, j'}$ and ${\cal Q}_{\ell, j}$ 
for $|\ell| \leq C_1$. 
\begin{lemma} \label{lem:G1} Assume that $|\ell| \leq C_1$ and that
$\e \gamma^{- 3}$ is small enough. Then the following holds:

\noindent
$(i)$ If  ${\cal R}_{\ell, j, j'} \neq \emptyset$, then there are 
constants
$ \mathfrak a \in (0,1) $ and $C_2 > 0$ so that $|j|, |j'| \leq C_2$ and 
$ | {\cal R}_{\ell,j,j'} | \lesssim \gamma^{\mathfrak a} $. 

\noindent
$(ii)$ If ${\cal Q}_{\ell, j } \neq \emptyset$ then there are
constants 
$ \mathfrak a \in (0,1) $ and  $C_2 > 0$ so that $|j| \leq C_2$ and $ | {\cal Q}_{\ell,j} | \lesssim  \gamma^{\mathfrak a} $. 
\end{lemma}

\begin{proof} 
We only prove item $(i)$ since item $(ii)$ can
be proved in a similar way. If $|\ell| \leq C_1$ and ${\cal R}_{\ell, j, j'} \neq \emptyset$, Lemma \ref{lem:mes1}-$(i)$ implies that there is a constant $C_2 $ such that $ |j|, |j'| \leq j^2 + j'^2 \leq C_2$. 
For $ \e \g^{-3}$ small enough one sees,  using \eqref{exp1s1},
the definition \eqref{op-diago0} of $\mu_j^0$, and 
the bounds $|\ell| \leq C_1, |j|, |j'| \leq C_2 $, that $|\mu_j^\infty - \omega^{kdv}_j| \lesssim \e \gamma^{- 2} \lesssim \gamma$,
implying that for some constant $C_3 > 0,$
\be\label{mes-Wei}
{\cal R}_{\ell,j,j'} \subset \big\{ \om \in \Omega \ : \ 
|\om \cdot \ell + \om^{kdv}_j ( \ac(\om),0 )  - \om^{kdv}_{j'} (\ac(\om),0)  | \leq  
C_3 \gamma
  \big\} \, .
\ee
By Lemma \ref{Proposition 2.3}, the function 
$ \om \mapsto  \om \cdot \ell + \om^{kdv}_j ( \ac(\om),0 )  - \om^{kdv}_{j'} (\ac(\om),0)$ is real analytic and not identically zero.
Hence by the Weierstrass preparation theorem 
(cf. the proof of \cite[Lemma 9.7]{BKM}), we deduce that 
the measure of the set on the right hand side of \eqref{mes-Wei} is smaller than $ \g^{\mathfrak a}$ for some 
$ \mathfrak a \in (0, 1) $ and $ \gamma $ small enough.
\end{proof}
 By \eqref{complementare Omega infty} and Lemmata \ref{lem:mes2}--\ref{lem:G1}
 we deduce that 
$$
|\Omega \setminus \Omega_\infty^\gamma| \lesssim \gamma^{\mathfrak a} + \gamma \sum_{\begin{subarray}{c}
|\ell | \geq C_1, |j|, |j'| \leq C \langle \ell \rangle
\end{subarray}} \langle \ell \rangle^{- \tau} \lesssim \gamma^{\mathfrak a}\,,
$$
where we used the assumption that $\tau - 2>|\mathbb S_+|$.
This concludes the proof of Proposition \ref{stima Omega - Omega infty}.


\begin{thebibliography}{10}




\bibitem{BBHM} 
{\sc P. Baldi, M. Berti, E. Haus, R. Montalto}, 
\emph{Time quasi-periodic gravity water waves
in finite depth}, Inventiones Math. 214 (2), 739--911, 2018.

\bibitem{BBM-Airy} {\sc P. Baldi, M. Berti, R. Montalto}, {\em KAM for quasi-linear
and fully nonlinear forced perturbations of Airy equation}, 
Math. Annalen 359, 471--536, 2014. 


\bibitem{BBM-auto} {\sc P. Baldi, M. Berti, R. Montalto}, {\em KAM for autonomous
quasi-linear perturbations of KdV}, 
Ann. Inst. H. Poincar\'e Analyse Non. Lin. 33, no. 6, 1589--1638, 2016. 

\bibitem{BBMmKdV} {\sc P. Baldi, M. Berti, R. Montalto}, 
\emph{KAM for autonomous quasi-linear perturbations of mKdV}, 
 \newblock {Bollettino Unione Matematica Italiana},  9, 143--188, 2016.

\bibitem{BGMR1}
{\sc D. Bambusi, B. Gr\'ebert, A. Maspero, D. Robert},
{\it Growth of Sobolev norms for abstract linear Schr\"odinger Equations}. J. Eur. Math. Soc. (JEMS), in press 2017.


\bibitem{BBiP1}
{\sc M. Berti, L. Biasco, M. Procesi}, {\it KAM theory for the Hamiltonian DNLW}, 
Ann. Sci. \'Ec. Norm. Sup\'er. (4),  
Vol. 46, fascicule 2, 301--373, 2013.

\bibitem{BBiP2}
{\sc M. Berti, L. Biasco, M. Procesi}, {\it KAM theory for the reversible derivative wave equation},
Arch. Rational Mech. Anal.,  212, 905--955, 2014. 

\bibitem{BB13} {\sc M. Berti, P. Bolle,}  
{\it A Nash-Moser approach to KAM theory}, Fields Institute Communications, 
special volume ``Hamiltonian PDEs and Applications'', 
255--284,  2015. 



\bibitem{BKM} {\sc M. Berti, T. Kappeler, R. Montalto},
{\em Large KAM tori for perturbations of the defocusing NLS equation}, Ast\'erisque 403, 
viii+148 pp., 2018.  


\bibitem{Berti-Montalto} {\sc M. Berti, R. Montalto}, 
{\em Quasi-periodic standing wave solutions of gravity capillary standing water waves}, to appear in Memoirs of the Amer. Math. Society, MEMO 891, arXiv:1602.02411, 2016.

\bibitem{BK} {\sc R. Bikbaev, S. Kuksin},
{\em On the parametrization of finite-gap solutions by frequency vector and wave number vectors and a theorem by I. Krichever}, Lett. Math. Phys. 28, 115-122, 1993.

\bibitem{DKN} {\sc B. Dubrovin, I. Krichever, S. Novikov,}
{\em Integrable Systems I}, in {\em Dynamical Systems IV},
Encyclopedia of Mathematical Sciences vol. 4, 
V. Arnold, S. Novikov (eds.), 173--280, Springer, 1990.

\bibitem{EGK}
{\sc H. Eliasson, B. Gr\'ebert, S. Kuksin,}
{\it KAM for the nonlinear beam equation},  Geom. Funct. Anal. Vol. 26, 1588--1715, 2016.

\bibitem{EK} 
{\sc H. Eliasson, S. Kuksin,}
{\it KAM for non-linear Schr\"odinger equation},  Annals of Math.,
172, 371-435, 2010.

\bibitem{FT} {\sc  L. Faddeev, L. Takhtajan,}
{\em Hamiltonian methods in the theory of solitons}, 
 Springer-Verlag, 1987.
 

\bibitem{FP} {\sc R. Feola, M. Procesi},  
{\it Quasi-periodic solutions for fully nonlinear forced
reversible Schr\"odinger equations}, J. Diff. Eq., 259, no. 7, 3389--3447, 2015.


\bibitem{GGKM} {\sc C. Gardner, J. Greene, M. Kruskal, 
R. Miura,} {\em Korteweg-deVries equation and generalization. VI. Methods for exact solution}, Comm. Pure Appl. Math. 27, 97--133, 1974.

\bibitem{Giu} 
{\sc F. Giuliani}, 
{\it Quasi-periodic solutions for quasi-linear generalized KdV equations}, 
J. Differential Equations 262, 10, 5052--5132, 2017. 

 

 \bibitem{Kappeler-Montalto-pseudo} {\sc T. Kappeler, R. Montalto,} {\em Normal form coordinates for the KdV equation having expansions in terms of pseudo-differential operators}, Comm. Math Phys. DOI: 10.1007/s00220-019-03498-1, 2019.

\bibitem{KP} {\sc T. Kappeler, J. P{\"o}schel}, {\em KdV {\&} KAM}, Springer-Verlag, 2003.

\bibitem{KST} {\sc T. Kappeler, B. Schaad, P. Topalov}, 
{\em Qualitative features of periodic solutions of KdV}, 
Comm. Part. Diff. Eqs. 38, (9), 1626--1673, 2013.

\bibitem{KST2} {\sc T. Kappeler, B. Schaad, P. Topalov}, {\em Semi-linearity of the nonlinear Fourier transform of the defocusing NLS equation}, Int. Math. Res. Notices,
 7212--7229, 2016.

\bibitem{KZ} {\sc M. Kruskal, N. Zabusky,} {\em Interactions
of `solitons" in a collisionless plasma and the recurrence
of initial states,} Phys.Rev.Lett. 15, 240--243, 1965. 

\bibitem{K2-KdV} {\sc S. Kuksin}, {\it A KAM theorem for equations of the Korteweg-de Vries type},
Rev. Math. Phys., 10, 3, 1--64, 1998.

\bibitem{K} {\sc S. Kuksin}, {\em Analysis of Hamiltonian PDEs}, Oxford
University Press, 2000.

\bibitem{KPe}
{\sc S. Kuksin, G. Perelman}, 
{\it Vey theorem in infinite dimensions 
and its application to KdV},
 Discrete Contin. Dyn. Syst. 27, no. 1, 1-24, 2010.

\bibitem{Lax} {\sc P. Lax,} {\em Integrals of nonlinear equations of evolution and solitary waves}, Comm. Pure Appl. Math. 21, 467--490, 1968.

\bibitem{LY}  {\sc J. Liu, X. Yuan},
{\it A KAM Theorem for Hamiltonian Partial Differential
Equations with Unbounded Perturbations}, Comm. Math. Phys, 307 (3), 629--673, 2011.
 


\bibitem{MGK} {\sc R. Miura, C. Gardner, M. Kruskal,}  {\em Korteweg-de Vries equation and generalizations. II. Existence of conservation laws and constants of motion}, J. Mathematical Phys. 9, 1204--1209, 1968. 



\bibitem{ZGY} {\sc J. Zhang, M.  Gao, X.  Yuan}, 
{\it KAM tori for reversible partial differential equations},
Nonlinearity 24, 1189--1228, 2011. 

\end{thebibliography}
\end{document}